\documentclass[11pt,a4paper]{amsbook}

\usepackage{amsfonts}

\usepackage[T1]{fontenc}

\usepackage{enumerate}

\usepackage[abbrev,backrefs]{amsrefs}

\usepackage[utf8]{inputenc}

\usepackage{amssymb}

\usepackage{palatino}

\usepackage{color}
\newbox\barleftbox 
\newbox\barrightbox 
{\setlength\fboxsep{0pt}\setlength\fboxrule{0pt} 
 \global\setbox\barleftbox=\hbox{%
   \colorbox[gray]{.8}{\hskip 4pt\relax\vrule height3pt width 0pt}%
   \hskip4pt} 
 \dimen0=\ht\barleftbox \advance\dimen0-1pt 
 \ht\barleftbox=\dimen0 
 \dp\barleftbox=-1pt 
 \global\setbox\barrightbox=\hbox{%
   \hskip4pt 
   \colorbox[gray]{.8}{\hskip 4pt\relax\vrule height3pt width 0pt}} 
 \dimen0=\ht\barrightbox \advance\dimen0-1pt 
 \ht\barrightbox=\dimen0 
 \dp\barrightbox=-1pt 
} 
\newcommand{\leftstrip}[1]{%
   \valign{##\cr 
           \leaders\copy\barleftbox\vfill\cr 
           \vbox{\hsize\marginparwidth\advance\hsize-8pt 
                 \raggedright\sffamily\footnotesize #1}\cr 
   } 
} 
\newcommand{\rightstrip}[1]{%
   \valign{##\cr 
           \vbox{\hsize\marginparwidth\advance\hsize-8pt 
                 \raggedright\sffamily\footnotesize #1}\cr 
           \leaders\copy\barrightbox\vfill\cr 
   } 
} 
\let\oldmarginpar\marginpar 
\renewcommand\marginpar[1]{%
  \oldmarginpar[\leftstrip{#1}]{\rightstrip{#1}}}

\renewcommand{\le}{\leqslant}
\renewcommand{\leq}{\leqslant}
\renewcommand{\ge}{\geqslant}
\renewcommand{\geq}{\geqslant}
\renewcommand{\div}{\operatorname{div}}
\newcommand{\abs}[1]{\mathopen\lvert#1\mathclose\rvert}

\newcommand{\norm}[1]{\mathopen\lVert#1\mathclose\rVert}
\newcommand{\bignorm}[1]{\mathopen\big\lVert#1\mathclose\big\rVert}
\newcommand{\Bignorm}[1]{\mathopen\Big\lVert#1\mathclose\Big\rVert}

\newcommand{\N}{{\mathbb N}}
\newcommand{\R}{{\mathbb R}}

\newcommand{\cM}{\mathcal{M}}
\newcommand{\cH}{\mathcal{H}}
\newcommand{\e}{{\mathrm{e}}}
\DeclareMathOperator{\supp}{supp}

\DeclareMathOperator{\diam}{diam}

\DeclareMathOperator{\sgn}{sgn}

\DeclareMathOperator{\capt}{cap}
\newcommand{\st}{\mid}
\newcommand{\dif}{\,\mathrm{d}}

\newcommand{\lc}{_\mathrm{c}}
\newcommand{\ld}{_\mathrm{d}}
\newcommand{\la}{_\mathrm{a}}
\newcommand{\ls}{_\mathrm{s}}
\newcommand{\loc}{_\mathrm{loc}}

\renewcommand{\atopwithdelims}[2]{\genfrac{\{}{\}}{0pt}{}{#1}{#2}}
\newcommand{\meas}[1]{\left| #1 \right|}
\newcommand{\Quasicontinuous}[1]{\widehat{#1}}

\theoremstyle{plain}
\newtheorem{proposition}{Proposition}[chapter]
\newtheorem{lemma}[proposition]{Lemma}
\newtheorem{sublemma}[proposition]{Sublemma}

\newtheorem{corollary}[proposition]{Corollary}

\theoremstyle{definition}
\newtheorem{definition}[proposition]{Definition}

\theoremstyle{remark}
\newtheorem{remark}[proposition]{Remark}

\newtheorem{openproblem}[proposition]{Open problem}
\newtheorem{claim}{Claim}
\newtheorem*{Claim}{Claim}
\newcommand{\Newclaim}{\setcounter{claim}{0}}

\long\def\comment#1{\marginpar{\raggedright\small #1}}

\newcommand{\hide}[1]{}

\usepackage{refcount}
\newcounter{cte}
\newcommand{\Constant}{\refstepcounter{cte} C_{\thecte}}
\newcommand{\NewConstant}{\setcounter{cte}{1} C_{\thecte}}
\newcommand{\SameConstant}{C_{\thecte}}

\numberwithin{equation}{section}

\usepackage{color}
\usepackage[backref=page]{hyperref}
\hypersetup{colorlinks=true, linkcolor=black, urlcolor=black, citecolor=black}

\title[Elliptic equations involving measures]{Selected problems on elliptic equations involving measures}

\begin{document}

\frontmatter

\author{Augusto C. Ponce}

\address{
\begin{center}
Université catholique de Louvain\break\indent
Institut de recherche en math{\'e}matique et physique\break\indent
Chemin du cyclotron 2\break\indent
1348 Louvain-la-Neuve\break\indent
Belgium\break\indent
\texttt{Augusto.Ponce@uclouvain.be}\break\indent
\end{center}
\vfill{}
\begin{center}
Mémoire lauréat du Concours annuel de 2012 de la Classe des Sciences de l'Académie royale de Belgique
portant sur une contribution originale à la théorie des problèmes elliptiques dont les données sont des mesures.\\[1cm]
This monograph became the core of the book\\ \emph{Elliptic PDEs, Measures and Capacities: From the Poisson equation to Nonlinear Thomas-Fermi Problems} which has received the\\[10pt]
2014 EMS Monograph Award\\[10pt]
\noindent
and has been published by the European Mathematical Society in the series EMS Tracts in Mathematics.
\end{center}
}

\maketitle
\tableofcontents
\mainmatter

\chapter*{Foreword}

A couple of years ago I was asked about what papers one should read in the field of elliptic equations involving measures.
I do not remember precisely which ones I mentioned, but one my favorites is my work with H.~Brezis and M.~Marcus where we introduced the concept of reduced measure~\cite{BreMarPon:07}.
This paper is one of my first contributions in the domain.  
Although our work was published in 2007, it was already completed beginning 2004 and several preliminary versions started to circulate at that time.

What are reduced measures good for? 
The notion of reduced measure clarifies several common properties which are shared by nonlinear Dirichlet problems with absorption term, or equivalently with nonlinearities satisfying the sign condition.
The existence of the reduced measure can be obtained via a nonlinear Perron method and according to the fundamental property of reduced measures, the reduced measure is the largest good measure for the Dirichlet problem.

In 2004, I went to Rome and then I got in touch with another school also working on elliptic problems involving measures, around L.~Boccardo, L.~Orsina and A.~Porretta.
I have learned from them that proofs should rely as least as possible on uniqueness of solutions.

I have now the impression that each paper I wrote subsequently was influenced by this point of view and I have gradually developed an alternative approach to prove the main result on reduced measures.
The modification I have in mind emphasizes the role of the Perron method and is in the spirit of Chacon and Rosenthal's biting lemma.

I have finally decided to assemble part of my contribution in the domain.
The evolution is scattered in several papers and some of the ideas have never been published.
In the meantime, my taste for writing down proofs has also evolved so
I found this would also be an opportunity to rewrite some of the proofs in a unified style.
This task was harder than I could possibly have imagined.

In any case, \emph{this is not a monograph about reduced measures.}
My aim is to show how one can adapt classical techniques in elliptic partial differential equations --- maximum principles, Perron method, method of sub and supersolutions ---  in a setting where very little regularity is available.
Yet, the fact that we shall be working with solutions in \(L^1(\Omega)\) is just what one needs to talk about weak formulations and we can free ourselves from regularity issues.

I would like to thank Isabelle, Clément and Raphaël for their patience during the preparation of the manuscript.

\bigskip
\hfill Louvain-la-Neuve, March 2012



\chapter{Introduction}

This monograph concerns the linear Dirichlet problem
\[
\left\{
\begin{alignedat}{2}
- \Delta u & = \mu \quad & &  \text{in \(\Omega\),}\\
u & = 0 \quad	& & \text{on \(\partial\Omega\),}
\end{alignedat}
\right.
\]
where \(\Omega \subset \R^N\) is a smooth bounded domain.
We are interested in the case where \(\mu\) is just an \(L^1\) function and, more generally, a finite Borel measure in \(\Omega\).

Equations involving measure data appear naturally in the study of properties of the Laplacian:
\begin{enumerate}[\((a)\)]
\item Green's function satisfies the linear Dirichlet problem with a Dirac mass;
\item subsolutions and supersolutions are in general functions whose Laplacian are locally finite measures;
\item if \(u\) is a smooth function vanishing on the boundary \(\partial\Omega\), then \(\Delta u^+\) is a finite measure which cannot be identified as a function; more generally, for every convex function \(\Phi : \R \to \R\),  \(\Delta \Phi(u)\) is a finite measure.
\end{enumerate}

A solution of the linear Dirichlet problem is an \(L^1\) function which satisfies the equation in the weak sense and admissible test functions are smooth functions vanishing on the boundary \(\partial\Omega\); see definition~\ref{definitionSolutionLinearDirichletProblem}.
This notion was originally introduced by Littman, Stampacchia and Weinberger~\cite{LitStaWei:1963}.

According to Stampacchia's regularity theory, every solution of the linear Dirichlet problem belongs to the Sobolev spaces \(W_0^{1, q}(\Omega)\) for \(1 \le q < \frac{N}{N-1}\), but we just miss the critical exponent \(\frac{N}{N-1}\);  see proposition~\ref{prop3.1}.
This is an important difference with respect to the Calderón-Zy\-gmund \(L^p\) theory which tells that if \(\mu \in L^p(\Omega)\) for some \(1 < p < +\infty\), then the solution of the linear Dirichlet problem belongs to \(W^{2, p}(\Omega)\).
In our case, if \(\mu \in L^1(\Omega)\), then the solution need not belong to \(W^{2, 1}(\Omega) \subset W^{1, \frac{N}{N-1}}(\Omega)\) by the Sobolev-Gagliardo-Nirenberg imbedding theorem.

\medskip
The linear theory is an important tool to understand the nonlinear Dirichlet problem
\[
\left\{
\begin{alignedat}{2}
- \Delta u + g(u) & = \mu \quad & &  \text{in \(\Omega\),}\\
u & = 0 \quad	& & \text{on \(\partial\Omega\),}
\end{alignedat}
\right.
\]
where \(g : \R \to \R\) is a continuous function.
Some pioneering contributions to nonlinear problems with \(L^1\) or measure data are due to Brezis and Strauss~\cite{BreStr:73}, Lieb and Simon~\cite{LieSim:77} and Bénilan and Brezis~\cites{BenBre:04,Bre:80,Bre:82}.
The motivation for studying such problems is beautifully discussed in the preface of \cite{BenBre:04} by H.~Brezis.

Important tools to study the nonlinear Dirichlet problem in this weak setting are
\begin{enumerate}[\((a)\)]
\item \emph{maximum principles} for subsolutions and supersolutions of the linear Dirichlet problem;
\item \emph{Kato's inequality}, which which gives comparison principles for the nonlinear Dirichlet problem; 
\item \emph{method of sub and supersolutions} and the \emph{Perron method}, to obtain existence of solutions and extremal solutions.
\end{enumerate}

There are two typical assumptions we have in mind for the nonlinearity \(g\). 
The first assumption concerns the sign of the nonlinearity:
\begin{enumerate}[\(-\)]
\item \emph{sign condition:} for every \(t \in \R\), \(g(t)t \ge 0\),
\end{enumerate}
This condition implies for example that the nonlinear Dirichlet problem is of absorption type, meaning that for every solution \(u\),
\[
\int\limits_\Omega \abs{g(u)} \le \int\limits_\Omega \dif\abs{\mu}.
\]
When \(\mu \in L^2(\Omega)\), the functional associated to the nonlinear Dirichlet problem is bounded from below in \(W_0^{1, 2}(\Omega)\) and minimizers always exist; see proposition~\ref{propositionExistenceMinimizer}.

The second assumption concerns the integrability of \(g(w)\) for a given function \(w\):
\begin{enumerate}[\(-\)]
\item \emph{integrability condition:} for every \(\underline w, w, \overline w \in L^1(\Omega)\) such that \(\underline w \le w \le \overline w\) in \(\Omega\), if \(g(\underline w) \in L^1(\Omega)\) and if \(g(\overline w) \in L^1(\Omega)\), then \(g(w) \in L^1(\Omega)\).
\end{enumerate}
The integrability condition guarantees that the method of sub and supersolution holds; see proposition~\ref{propositionMethodSubSuperSolutions}.
This assumption is satisfied when the nonlinearity \(g\) is a monotone function or, as in the case of the Chern-Simons scalar equation, when \(g\) is of the form \(g(t) = \e^t(\e^t - 1)\)~\cite{Yan:01}.

\medskip
The study of the nonlinear Dirichlet problem with measure data turns out to be more subtle than with \(L^1\) data.
It was observed in 1975 by B\'enilan and Brezis~\cites{BenBre:04,Bre:80,Bre:82} 
that if $N \geq 3$ and 
\[
g(t) = |t|^{p-1} t
\]
with $p\geq \frac{N}{N-2}$, then the nonlinear Dirichlet problem has no solution when \(\mu\) is a Dirac mass; see proposition~\ref{propositionNonExistenceBenilanBrezis}.
They also proved that if $p< \frac{N}{N-2}$ and $N\geq 2$, then the nonlinear Dirichlet problem has a solution for any finite measure $\mu$; see proposition~\ref{propositionExistenceBenilanBrezis}.

Later, Baras and Pierre \cite{BarPie:84} characterized all
measures $\mu$ for which the nonlinear Dirichlet problem admits a solution for a nonlinearity of the form $g(t) = |t|^{p-1} t$.
Their necessary and sufficient condition for the existence of a solution 
when $p\geq \frac{N}{N-2}$ can be expressed in terms of the \(W^{2, p'}\) capacity; see proposition~\ref{propositionExistenceBarasPierre}.

The case of exponential nonlinearities of the form 
\[
g(t) = \e^t - 1
\]
was studied by Vázquez~\cite{Vaz:83} in dimension \(N = 2\) and more recently by Bartolucci, Leoni, Orsina and Ponce~\cite{BLOP:05} in dimension \(N \ge 3\).
The solution in this case is related to the Hausdorff measure \(\cH^{N-2}\); see proposition~\ref{propositionExistenceSolutionExponentialDimension2} and proposition~\ref{propositionExistenceSolutionExponentialDimension3}.
One of the tools used in the proof is related to the uniform convergence of the Hausdorff outer measures \(\cH_\delta^s\) to the Hausdorff measure \(\cH^s\) on sets with finite Hausdorff measure as \(\delta\) tends to zero; see proposition~\ref{propositionUniformConvergenceHausdorff}.

Brezis, Marcus and Ponce~\cite{BreMarPon:07} introduced the concept of \emph{reduced measure} in order to analyze the nonexistence
mechanism behind the nonlinear Dirichlet problem and to describe what happens if one forces the problem to have a solution in cases where the problem refuses to have one.  

The approach developed in \cite{BreMarPon:07} was to introduce an approximation scheme.
For example, the measure $\mu$ is kept fixed and $g$ is truncated; 
alternatively, the nonlinearity $g$ is kept fixed and $\mu$ is approximated via convolution.  
It was originally observed by Brezis~\cite{Bre:83a} that if $N\geq 3$, $g(t) =
|t|^{p-1} t$, with $p\geq \frac{N}{N-2}$, and $\mu$ is a Dirac mass, then all natural approximations $u_n$ of the nonlinear Dirichlet problem converge to $0$. 
However, $0$ is not a solution corresponding to a Dirac mass.

In this monograph, we adopt a different approach to define the reduced measure.
If the nonlinear Dirichlet problem with datum \(\mu\) has a subsolution, then by the Perron method the largest subsolution exists; see proposition~\ref{propositionExistenceReducedMeasure}.
In \cite{BreMarPon:07}, the largest subsolution \(u^*\) was obtained as the limit of an approximation scheme --- which involves the resolution of infinitely many nonlinear Dirichlet problems ---, while the Perron method does not require to solve any Dirichlet problem at all: the largest subsolution is obtained as the supremum of all subsolutions; see proposition~\ref{propositionPerronMethod}.
The reduced measures is then defined as
\[
\mu^* = - \Delta u^* + g(u^*).
\]

The fundamental property of the reduced measure insures that \(\mu^*\) is the largest measure less than or equal to \(\mu\) for which the nonlinear Dirichlet problem has a solution; see proposition~\ref{propositionReducedMeasure}.
As a consequence, we deduce that if \(g\) satisfies the sign condition and if \(\mu \ge 0\), then \(\mu^* \ge 0\).
Note that in this case \(u^* \ge 0\) since \(0\) is a subsolution and \(u^*\) is the largest one, but it is not obvious that
\[
- \Delta u^* + g(u^*) \ge 0.
\]


\chapter{Variational approach}

We prove existence of variational solutions of the nonlinear Dirichlet problem
\[
\left\{
\begin{alignedat}{2}
- \Delta u + g(u) & = \mu \quad & &  \text{in \(\Omega\),}\\
u & = 0 \quad	& & \text{in \(\partial\Omega\),}
\end{alignedat}
\right.
\]
when \(g : \R \to \R\) is a continuous function satisfying the sign condition.
By variational, we mean that the solution lies in the Sobolev space in \(W_0^{1, 2}(\Omega)\) and is obtained as a critical point of a functional.

Our approach relies on the observation that the equation
\[
- \Delta u + g(u) = \mu \quad \text{in \(\Omega\),}
\]
is the Euler-Lagrange equation in \(W_0^{1, 2}(\Omega)\) of the functional
\[
E(u) = \frac{1}{2} \int\limits_\Omega \abs{\nabla u}^2 + \int\limits_{\Omega} G(u) - \int\limits_\Omega u \mu,
\]
where \(G : \R \to \R\) is the function defined for \(t \in \R\) by
\[
G(t) = \int_0^t g(s) \dif s.
\]

\section{Minimizers}

If the nonlinearity \(g\) satisfies the sign condition, then for every \(t \in \R\), 
\[
G(t) \ge 0.
\] 
Thus, for every \(\mu \in L^2(\Omega)\), the functional \(E\) is bounded from below in \(W_0^{1, 2}(\Omega)\) and the existence of a minimizer of \(E\) follows from a standard minimization technique:

\begin{proposition}
\label{propositionExistenceMinimizer}
Let \(g : \R \to \R\) be a continuous function satisfying the sign condition.
If \(\mu \in L^2(\Omega)\), then there exists \(u \in W_0^{1, 2}(\Omega)\) such that for every \(v \in W_0^{1, 2}(\Omega)\),
\[
E(u) \le E(v).
\]
\end{proposition}

We first need an estimate that insures that minimizing sequences are bounded in \(W_0^ {1, 2}(\Omega)\):

\begin{lemma}
Let \(g : \R \to \R\) be a continuous function.
If \(g\) satisfies the sign condition and if \(\mu \in L^2(\Omega)\), then for every \(v \in W_0^{1, 2}(\Omega)\),
\[
\norm{v}_{W_0^{1, 2}(\Omega)}^2 \le C \big(E(v) + \norm{v}_{L^2(\Omega)}^2 \big),
\]
for some constant \(C > 0\) depending on \(N\) and \(\Omega\).
\end{lemma}

\begin{proof}
Let \(v \in W_0^{1, 2}(\Omega)\).
Since \(g\) satisfies the sign condition, \(G(v) \ge 0\), and we have
\[
\frac{1}{2} \int\limits_\Omega \abs{\nabla v}^2 \le E(v) + \int\limits_\Omega v \mu.
\]
For every \(\epsilon > 0\), by the Young inequality there exists a constant \(C_\epsilon > 0\) such that
\[
v \mu \le \epsilon v^2 + C_\epsilon \mu^2.
\]
Thus,
\[
\frac{1}{2} \int\limits_\Omega \abs{\nabla v}^2 \le E(v) + \epsilon \int\limits_\Omega v^2 + C_\epsilon \int\limits_\Omega  \mu^2.
\]
By the Poincaré inequality,
\[
\int\limits_\Omega v^2 \le \NewConstant \int\limits_\Omega \abs{\nabla v}^2
\]
and this implies
\[
\Big(\frac{1}{2} - \epsilon\SameConstant\Big) \int\limits_\Omega \abs{\nabla v}^2 \le E(v) + C_\epsilon \int\limits_\Omega  \mu^2.
\]
Choosing \(\epsilon > 0\) such that
\[
\frac{1}{2} - \epsilon\SameConstant > 0,
\]
the conclusion follows by applying once again the Poincaré inequality.
\end{proof}

\begin{proof}[Proof of proposition~\ref{propositionExistenceMinimizer}]
Let \((u_n)_{n \in \N}\) be a sequence in \(W_0^{1, 2}(\Omega)\). 
If \((E(u_n))_{n \in \N}\) is bounded in \(\R\), then by the previous lemma the sequence \((u_n)_{n \in \N}\) is bounded in \(W_0^{1, 2}(\Omega)\). 
By the Rellich-Kondrachov compactness theorem, there exists a subsequence \((u_{n_k})_{k \in \N}\) converging in \(L^2(\Omega)\) to some function \(u\).
Since \(\mu \in L^2(\Omega)\),
\[
\int\limits_\Omega u \mu = \lim_{k \to \infty}{\int\limits_\Omega u_{n_k} \mu}.
\]
We may also assume that subsequence \((u_{n_k})_{k \in \N}\) converges almost everywhere to \(u\). 
Thus, by Fatou's lemma,
\[
\int\limits_\Omega G(u) \le \liminf_{k \to \infty}{\int\limits_\Omega G(u_{n_k})}.
\]
Since \((u_n)_{n \in \N}\) is bounded in \(W_0^{1, 2}(\Omega)\), \(u \in W_0^{1, 2}(\Omega)\) and
\[
\int\limits_\Omega \abs{\nabla u}^2 \le \liminf_{k \to \infty}{\int\limits_\Omega \abs{\nabla u_{n_k}}^2};
\]
see \cite{Wil:07}*{théorème~21.13}.
We conclude that
\[
E(u) \le \liminf_{k \to \infty}{E(u_{n_k})}.
\]
We now choose \((u_n)_{n \in \N}\) to be a minimizing sequence of the functional \(E\).
In this case,
\[
E(u) \le \lim_{n \to \infty}{E(u_{n})} = \inf_{v \in W_0^{1, 2}(\Omega)}{E(v)}.
\]
Since \(u \in W_0^{1, 2}(\Omega)\), equality holds and \(u\) is a minimizer of \(E\).
\end{proof}

Although this approach for solving the Dirichlet problem leads to a minimization problem which can be easily solved, it is rather annoying to show that \(u\) satisfies the Euler-Lagrange equation in the weak form, meaning that for every \(v \in W_0^{1, 2}(\Omega)\),
\[
\int\limits_\Omega \nabla u \cdot \nabla v + \int\limits_\Omega g(u) v = \int\limits_\Omega v \mu.
\]

The difficulty in considering the variation 
\[
t \in \R \longmapsto E(u + tv)
\]
of the functional \(E\) around a minimizer \(u\) comes from the term \(G(u + tv)\), which need not be an \(L^1\) function for every \(v \in W_0^{1, 2}(\Omega)\) and for every \(t \in \R\) close to \(0\). 
We could restrict ourselves to test functions \(v \in W_0^{1, 2}(\Omega) \cap L^\infty(\Omega)\), but the same obstruction arises.

\section{Euler-Lagrange equation}

It is easier to obtain solutions of the Euler-Lagrange equation when \(\mu\) is for instance an \(L^\infty\) function: 

\begin{proposition}
\label{propositionExistenceSolutionEulerLagrange}
Let \(g : \R \to \R\) be a continuous function satisfying the sign condition.
If \(\mu \in L^\infty(\Omega)\), then the Euler-Lagrange equation associated to the functional \(E\) has a solution \(u \in W_0^{1, 2}(\Omega)\) such that 
\[
\norm{g(u)}_{L^\infty(\Omega)} \le \norm{\mu}_{L^\infty(\Omega)}.
\]
\end{proposition}

We begin by establishing the following lemma:

\begin{lemma}
Let \(g : \R \to \R\) be a continuous function satisfying the sign condition and let \(\kappa \in \R\). 
Given \(\mu \in L^\infty(\Omega)\), let \(u \in W_0^{1, 2}(\Omega)\) be a minimizer of the functional \(E\) over \(W_0^{1, 2}(\Omega)\) . 
\begin{enumerate}[\((i)\)]
\item If for every \(t \ge \kappa\), \(g(t) \ge \norm{\mu}_{L^\infty(\Omega)}\), then \(u \le \kappa\) in \(\Omega\).
\item If for every \(t \le \kappa\), \(g(t) \le -\norm{\mu}_{L^\infty(\Omega)}\), then \(u \ge \kappa\) in \(\Omega\).
\end{enumerate}
\end{lemma}

The proof of the lemma relies on the idea that if we truncate a function at the level \(\kappa\), then we reduce its energy \(E\).

\begin{proof}
We establish the first assertion.
Given \(v \in W_0^{1, 2}(\Omega)\) and \(\kappa \in \R\), let \(v_\kappa : \Omega \to \R\) be the function defined by
\[
v_\kappa = \min{\{v, \kappa\}}.
\]
Note that 
\[
\nabla v_\kappa =
\begin{cases}
\nabla v	& \text{in \(\{v \le \kappa\}\),}\\
0			& \text{in \(\{v > \kappa\}\),}
\end{cases}
\]
and if \(\kappa \ge 0\), then 
\[
v_\kappa \in W_0^{1, 2}(\Omega).
\]

\begin{Claim}
If the set \({\{v > \kappa\}}\) is not negligible with respect to the Lebesgue measure and if for every \(t \ge \kappa\), 
\[
g(t) \ge \norm{\mu}_{L^\infty(\Omega)},
\]
then
\[
E(v_\kappa) < E(v).
\]
\end{Claim}

Assuming the claim, we may conclude the proof of assertion \((i)\). Indeed, if \(u \in W_0^{1, 2}(\Omega)\) is a minimizer of \(E\), then \(E(u_\kappa) \ge E(u)\). By the claim, the set \(\{u > \kappa\}\) is negligible, which means that \(u \le \kappa\) almost everywhere in \(\Omega\).

\begin{proof}[Proof of the claim]
Since the set \({\{v > \kappa\}}\) is not negligible, we have
\[
\int\limits_\Omega \abs{\nabla v_\kappa}^2 < \int\limits_\Omega \abs{\nabla v}^2.
\]
Moreover, 
\[
\begin{split}
\int\limits_\Omega G(v_\kappa) 
& = \int\limits_\Omega G(v) - \int\limits_\Omega \big[G(v) - G(v_\kappa)\big]\\
& = \int\limits_\Omega G(v) -\int\limits_{\{v > \kappa\}} \big[ G(v) - G(\kappa) \big].
\end{split}
\]
For every \(s > \kappa\),
\[
G(s) - G(\kappa) \ge  (s - \kappa) \inf_{t \in (\kappa, s)}{g(t)} \ge  (s - \kappa) \inf_{t \in (\kappa, +\infty)}{g(t)}.
\]
Thus,
\[
\int\limits_\Omega G(v_\kappa) 
\le \int\limits_\Omega G(v) - \inf_{t \in (\kappa, +\infty)}{g(t)} \int\limits_{\{v > \kappa\}} (v - \kappa).
\]
We also have
\[
\int\limits_\Omega v_\kappa \mu
= \int\limits_\Omega v \mu - \int\limits_{\{v > \kappa\}} (v - \kappa) \mu.
\]
Thus,
\[
E(v_\kappa)
< E(v) - \int\limits_{\{v > \kappa\}} \left(\inf_{t \in (\kappa, +\infty)}{g(t)} - \mu \right)(v - \kappa).
\]
Since for every \(t \ge \kappa\), \(g(t) \ge \norm{\mu}_{L^\infty(\Omega)}\), we have
\[
\inf_{t \in (\kappa, +\infty)}{g(t)} - \mu \ge 0.
\]
This implies that \(E(v_\kappa) < E(v)\).
\end{proof}

The proof of assertion \((i)\) is complete. The proof of the other assertion is similar.
\end{proof}

\begin{proof}[Proof of proposition~\ref{propositionExistenceSolutionEulerLagrange}]
If for every \(t \in \R\), \(\abs{g(t)} < \norm{\mu}_{L^\infty(\Omega)}\), then \(g\) is bounded and every minimizer of \(E\) over \(W_0^{1, 2}(\Omega)\) satisfies the Euler-Lagrange equation and the conclusion follows.

\medskip
We now assume that for every \(t \le 0\), 
\[
g(t) > - \norm{\mu}_{L^\infty(\Omega)}
\]
and that there exists \(\kappa > 0\) such that 
\[
g(\kappa) = \norm{\mu}_{L^\infty(\Omega)}.
\]

Let \(g_\kappa : \R \to \R\) be the function defined for \(t \in \R\) by
\[
g_\kappa(t) = 
\begin{cases}
g(t)	& \text{if \(t \le \kappa\),}\\
g(\kappa) & \text{if \(t > \kappa\).}
\end{cases}
\]
In particular, \(g_\kappa\) is bounded.
Denote by \(E_\kappa\) the functional associated to this nonlinearity:
\[
E_\kappa(u) = \frac{1}{2} \int\limits_\Omega \abs{\nabla u}^2 + \int\limits_{\Omega} G_\kappa (u) - \int\limits_\Omega u \mu.
\]
The minimization problem over \(W_0^{1, 2}(\Omega)\) associated to \(E_\kappa\) has a solution \(u \in W_0^{1, 2}(\Omega)\) and since \(g_\kappa\) is bounded this function satisfies for every \(v \in W_0^{1, 2}(\Omega)\),
\[
\int\limits_\Omega \nabla u \cdot \nabla v + \int\limits_\Omega g_\kappa (u) v = \int\limits_\Omega v \mu.
\]
By the previous lemma, \(u \le \kappa\) in \(\Omega\). Thus, \(g_\kappa(u) = g(u)\) in \(\Omega\) and this implies that \(u\) is a solution of the Euler-Lagrange equation associated to the functional \(E\).

\medskip
The two remaining cases concerning the nonlinearity \(g\), namely when \(g\) bounded from above by \(\norm{\mu}_{L^\infty(\Omega)}\) but not bounded from below by \(-\norm{\mu}_{L^\infty(\Omega)}\) and when \(g\) is not bounded from above by \(\norm{\mu}_{L^\infty(\Omega)}\) nor bounded from below by \(-\norm{\mu}_{L^\infty(\Omega)}\) can be proved in a similar way.
\end{proof}

An alternative approach to establish existence of solutions of the nonlinear Dirichlet problem when \(\mu \in L^\infty(\Omega)\) is to apply the method of sub and supersolutions (proposition~\ref{propositionMethodSubSuperSolutions}). Take for instance a nonpositive function \(\underline{v} \in C^\infty(\overline\Omega)\) such that 
\[
- \Delta \underline{v} \le \mu
\]
and a nonnegative function \(\overline{v} \in C^\infty(\overline\Omega)\) such that 
\[
- \Delta \underline{v} \ge \mu.
\]
If \(g\) satisfies the sign condition, then \(\underline{v}\) is a subsolution and \(\overline{v}\) is a supersolution of the nonlinear Dirichlet problem. The method of sub and supersolutions implies the existence of a weak solution \(u\) such that
\[
\underline{v} \le u \le \overline{u}.
\]
In particular, \(u \in W_0^{1, 1}(\Omega) \cap L^\infty(\Omega)\) and, by the interpolation inequality (lemma~\ref{lemmaInterpolationLinfty}), this solution belongs to \(W_0^{1, 2}(\Omega)\).

The minimization strategy still holds when \(\mu\) belongs to \(L^{\frac{2N}{N+2}}(\Omega)\) --- and more generally when \(\mu\) belongs to the dual space \((W^{1, 2}_0(\Omega))'\) --- but one cannot hope to go beyond the exponent \(\frac{2N}{N+2}\) since the functional \(E\) need not be bounded from below. 
In this sense, the variational approach to obtain solutions the nonlinear Dirichlet problem
\[
\left\{
\begin{alignedat}{2}
- \Delta u + g(u) & = \mu \quad & &  \text{in \(\Omega\),}\\
u & = 0 \quad	& & \text{in \(\partial\Omega\),}
\end{alignedat}
\right.
\]
when \(\mu\) is a measure or even an \(L^1\) function seems hopeless.

A different strategy consists in searching for solutions using an approximation procedure. 
For instance, when \(\mu \in L^1(\Omega)\) we may consider a sequence of functions \((\mu_n)_{n \in \N}\) in \(L^\infty(\Omega)\) converging to \(\mu\) in \(L^1(\Omega)\) --- e.g.~a sequence of truncates of \(\mu\) --- and then investigate what happens to the solutions \(u_n\) of the approximated problems. 
We will implement this strategy in the next chapter.


\chapter{$L^1$ data versus measure data}

We investigate a major difference between Dirichlet problems with \(L^1\) data and measure data.

The existence and regularity theory for the \emph{linear} Dirichlet problem
\[
\left\{
\begin{alignedat}{2}
- \Delta u & = \mu	\quad && \text{in \(\Omega\),}\\
u & = 0 	\quad && \text{on \(\partial\Omega\),}
\end{alignedat}
\right.
\]
has no difference whether \(\mu\) is an \(L^1\) function or a finite measure in the sense that solutions always exist in both cases, the estimates satisfied by the solutions in one case or in the other are the same, and there is no gain in regularity if we know that \(\mu\) is an \(L^1\) function rather than a finite measure.

Concerning the \emph{nonlinear} Dirichlet problem
\[
\left\{
\begin{alignedat}{2}
- \Delta u + g(u) & = \mu	\quad && \text{in \(\Omega\),}\\
u & = 0 	\quad && \text{on \(\partial\Omega\),}
\end{alignedat}
\right.
\]
the situation is radically different.
When \(g : \R \to \R\) satisfies the sign condition, solutions of the nonlinear Dirichlet problem always exist if \(\mu\) is an \(L^1\) function \cites{BreStr:73,GalMor:84}.
This is no longer true for measures:
Bénilan and Brezis \cites{BenBre:04} discovered that the nonlinear Dirichlet problem need not have a solution if \(\mu\) is a Dirac mass and \(g\) has polynomial growth.

\section{Linear case}

Given an \(L^1\) function or more generally a finite Borel measure \(\mu\) in \(\Omega\), we consider the linear Dirichlet problem
\[
\left\{
\begin{alignedat}{2}
- \Delta u & = \mu	\quad && \text{in \(\Omega\),}\\
u & = 0 	\quad && \text{on \(\partial\Omega\).}
\end{alignedat}
\right.
\]
We adopt the notion of weak solution introduced by Littman, Stampacchia and Weinberger~\cite{LitStaWei:1963}:

\begin{definition}
\label{definitionSolutionLinearDirichletProblem}
Let \(\mu \in \cM(\Omega)\).
A function \(u : \Omega \to \R\) is a \emph{solution} of the linear Dirichlet problem if
\begin{enumerate}[\((i)\)]
\item \(u \in L^1(\Omega)\), 
\item for every \(\zeta \in C_0^\infty(\overline\Omega)\),
meaning that \(\zeta \in C^\infty(\overline\Omega)\) and \(\zeta = 0\) on \(\partial\Omega\),
\[
- \int\limits_\Omega u \Delta\zeta + \int\limits_\Omega g(u)\zeta 
= \int\limits_\Omega \zeta \dif\mu.
\]
\end{enumerate}
\end{definition}

We denote by \(\cM(\Omega)\) the vector space of finite Borel measures in \(\Omega\). 
This space equipped with the norm
\[
\norm{\mu}_{\cM(\Omega)} = \abs{\mu}(\Omega)
\]
is a Banach space.

The boundary data is encoded in this weak formulation since we are not simply using as test functions smooth functions with compact support, but we allow our test functions to have a nontrivial normal derivative on the boundary.
An equivalent definition of solution is to require that
\begin{enumerate}[\((i')\)]
\item \(u \in W_0^{1, 1}(\Omega)\),
\item for every \(\varphi \in C_c^\infty(\Omega)\), meaning that \(\varphi \in C^\infty(\overline\Omega)\) and \(\varphi\) has compact support in \(\Omega\),
\[
\int\limits_\Omega \nabla u \cdot \nabla\varphi
= \int\limits_\Omega \varphi \dif\mu.
\]
\end{enumerate}
The equivalence of both definitions will be discussed later (see corollary~\ref{corollaryWeakFormulationEquivalence}).

This second, equivalent, definition of solution of the linear Dirichlet problem has the advantage of making more transparent the meaning of the boundary condition in the weak formulation since its is encoded in the sense of traces in \(W^{1, 1}(\Omega)\). 
A disadvantage is that one has to make sure each time that \(\nabla u \in L^1(\Omega)\) and that \(u \in W_0^{1, 1}(\Omega)\).
We believe it is better to stick to the first one, since it is easier to use, although the second one might be more appealing to those which are used to variational methods.

\begin{proposition}
\label{propositionExistenceLinearDirichletProblem}
For every \(\mu \in \cM(\Omega)\), the linear Dirichlet problem with datum \(\mu\) has a unique solution \(u\) and
\[
\norm{u}_{L^1(\Omega)} \le C \norm{\mu}_{\cM(\Omega)},
\]
for some constant \(C > 0\) depending on \(N\) and \(\Omega\).
\end{proposition}

This proposition makes no distinction between measure data and \(L^1\) data since
every \(f \in L^1(\Omega)\) can be identified with the finite measure \(\mu\) defined for every Borel set \(A \subset \Omega\) by
\[
\mu(A) = \int\limits_A f.
\]
In this case,
\[
\norm{\mu}_{\cM(\Omega)} = \int\limits_\Omega \abs{f} = \norm{f}_{L^1(\Omega)}
\]
and for every bounded Borel measurable function \(\psi : \Omega \to \R\),
\[
\int\limits_\Omega \psi \dif \mu = \int\limits_\Omega \psi f.
\]
This last equality may by established by approximating \(\psi\) by a sequence of simple functions.

\medskip

The proof of the proposition relies on the following estimate:

\begin{lemma}
Given \(\mu \in L^2(\Omega)\), let \(u \in W_0^{1, 2}(\Omega)\)
be the solution of the linear Dirichlet problem with datum \(\mu\).
Then,
\[
\norm{u}_{L^1(\Omega)} \le C \norm{\mu}_{L^1(\Omega)},
\]
for some constant \(C > 0\) depending on \(N\) and \(\Omega\).
\end{lemma}

\begin{proof}
On the one hand, for every \(v \in W_0^{1, 2}(\Omega)\) such that \(\Delta v \in L^2(\Omega)\),
\[
- \int\limits_\Omega u \Delta  v 
= \int\limits_\Omega \nabla u \cdot \nabla v 
= \int\limits_\Omega v \mu.
\]
Thus,
\[
\bigg| \int\limits_\Omega u \Delta  v \bigg| \le \norm{v}_{L^\infty(\Omega)} \norm{\mu}_{L^1(\Omega)}.
\]
On the other hand, for every \(h \in L^\infty(\Omega)\),
if \(v \in W_0^{1, 2}(\Omega)\) is the solution of the linear Dirichlet problem
\[
\left\{
\begin{alignedat}{2}
- \Delta v & = h \quad & &  \text{in \(\Omega\),}\\
v & = 0 \quad	& & \text{on \(\partial\Omega\),}
\end{alignedat}
\right.
\]
then by the maximum principle,
\[
\norm{v}_{L^\infty(\Omega)} \le C \norm{h}_{L^\infty(\Omega)}.
\]
Applying the previous estimate with this function \(v\), we have for every \(h \in L^\infty(\Omega)\),
\[
\bigg| \int\limits_\Omega u h \bigg| \le C \norm{h}_{L^\infty(\Omega)} \norm{\mu}_{L^1(\Omega)}.
\]
Using this estimate with \(h = \sgn{u}\), the conclusion follows.
\end{proof}

The previous lemma is based on a disguised duality argument: starting from the estimate
\[
\norm{u}_{L^\infty(\Omega)} \le C \norm{\Delta u}_{L^\infty(\Omega)},
\]
which follows from the maximum principle, we deduce that
\[
\norm{v}_{L^1(\Omega)} \le C \norm{\Delta v}_{L^1(\Omega)},
\]
with the same constant \(C\).

\medskip

We give a separate proof of proposition~\ref{propositionExistenceLinearDirichletProblem} when \(\mu \in L^1(\Omega)\).
The existence of solutions for finite measures is more subtle and requires an additional compactness argument.

\begin{proof}[Proof of proposition~\ref{propositionExistenceLinearDirichletProblem} when \(\mu \in L^1(\Omega)\)]
Given a sequence of functions \((\mu_n)_{n \in \N}\) in \(L^\infty(\Omega)\), for each \(n \in \N\) let \(u_n \in W_0^{1, 2}(\Omega)\) be the solution of the linear Dirichlet problem with datum \(\mu_n\).
By linearity of the Laplacian, for every \(m, n \in \N\) the function \(u_m - u_n\) satisfies the linear Dirichlet problem with datum \(\mu_m - \mu_n\).
Thus, by the previous lemma,
\[
\norm{u_m - u_n}_{L^1(\Omega)} \le C \norm{\mu_m - \mu_n}_{L^1(\Omega)}.
\]
Choosing the sequence \((\mu_n)_{n \in \N}\) so that it converge to \(\mu\) in \(L^1(\Omega)\), then \((\mu_n)_{n \in \N}\) is a Cauchy sequence in \(L^1(\Omega)\).
By the estimate above, \((u_n)_{n \in \N}\) is also a Cauchy sequence in \(L^1(\Omega)\), whence it converges in \(L^1(\Omega)\) to a function \(u\).
Since for every \(n \in \N\) and for every \(\zeta \in C_0^\infty(\overline\Omega)\),
\[
- \int\limits_\Omega u_n \Delta\zeta 
= \int\limits_\Omega \zeta \mu_n,
\]
we deduce that
\[
- \int\limits_\Omega u \Delta\zeta 
= \int\limits_\Omega \zeta \mu.
\]
Therefore, \(u\) is a solution of the Dirichlet problem.

Since for every \(n \in \N\),
\[
\norm{u_n}_{L^1(\Omega)} \le C \norm{\mu_n}_{L^1(\Omega)},
\]
we also have
\[
\norm{u}_{L^1(\Omega)} \le C \norm{\mu}_{L^1(\Omega)} =  C \norm{\mu}_{\cM(\Omega)}.
\]
Uniqueness of the solution follows from this estimate and from the linearity of the Laplacian.
The proof of the proposition is complete.
\end{proof}

The above proof relies on the fact that \(\mu\) was an \(L^1\) function.
The difficulty in adapting the proof to measures arises when trying to strongly approximate a finite measure \(\mu\) by a sequence of functions \((\mu_n)_{n \in \N}\) in the strong sense in \(\cM(\Omega)\),
\[
\lim_{n \to \infty}{\norm{\mu_n - \mu}_{\cM(\Omega)}} = \lim_{n \to \infty}{\int\limits_\Omega \abs{\mu_n - \mu}} = 0.
\] 
Indeed, if such approximation is possible, then by the triangle inequality \((\mu_n)_{n \in \N}\) is a Cauchy sequence in \(L^1(\Omega)\), in which case \(\mu\) coincides with an \(L^1\) function.

The problem here is that we are asking too much from the sequence \((\mu_n)_{n \in \N}\).
Indeed, for every \(n \in \N\) and for every \(\zeta \in C_0^\infty(\overline\Omega)\),
\[
- \int\limits_\Omega u_n \Delta\zeta 
= \int\limits_\Omega \zeta \mu_n.
\]
Thus, we only need convergence to \(\mu\) in the following sense: for every \(\psi \in C_0(\overline\Omega)\), meaning that \(\psi : \overline\Omega \to \R\) is continuous and \(\psi = 0\) on \(\partial\Omega\),
\[
\lim_{n \to \infty}{\int\limits_\Omega \psi \mu_n} = \int\limits_\Omega \psi \dif\mu .
\]
We say in this case that the sequence \((\mu_n)_{n \in \N}\) converges weakly to \(\mu\) in the sense of measures.

Every finite measure can be approximated by smooth functions in this sense:

\begin{lemma}
\label{lemmaWeakApproximation}
For every \(\mu \in \cM(\Omega)\), there exists a sequence \((\mu_n)_{n \in \N}\) in \(C^\infty(\overline\Omega)\) converging weakly to \(\mu\) in the sense of measures and such that
\[
\lim_{n \to \infty}{\norm{\mu_n}_{L^1(\Omega)}} = \norm{\mu}_{\cM(\Omega)}.
\]
\end{lemma}

\begin{proof}
Given a sequence of smooth mollifiers \((\rho_n)_{n \in \N}\), for every \(n \in \N\) let \(\mu_n : \overline{\Omega} \to \R\) be the function defined for \(x \in \overline\Omega\) by
\[
\mu_n(x) = \int\limits_{\Omega} \rho_n(x - y) \dif\mu(y).
\]
This is the convolution of \(\rho_n\) with  \(\mu\); in particular, \(\mu_n \in C^\infty(\overline{\Omega})\).

If \(\rho_n\) is an even function, then by Fubini's theorem we have for every \(\psi \in C_0(\overline\Omega)\),
\[
\int\limits_\Omega \psi \mu_n = \int\limits_\Omega \bigg(\int\limits_{\Omega} \rho_n(x - y) \psi(x) \dif x \bigg) \dif\mu(y) = \int_\Omega \rho_n * \psi \dif\mu.
\]
Since the sequence \((\rho_n * \psi)_{n \in \N}\) converges uniformly to \(\psi\) in \(\Omega\), we deduce the weak convergence of \((\mu_n)_{n \in \N}\) to \(\mu\) in the sense of measures.

In particular, by the lower semicontinuity of the norm under weak convergence,
\[
\norm{\mu}_{\cM(\Omega)} \le \liminf_{n \to \infty}{\norm{\mu_n}_{L^1(\Omega)}}.
\]
Since for every \(n \in \N\),
\[
\norm{\mu_n}_{L^1(\Omega)} \le \norm{\mu}_{\cM(\Omega)},
\]
the conclusion follows.
\end{proof}

\begin{proof}[Proof of proposition~\ref{propositionExistenceLinearDirichletProblem}]
Given a sequence  \((\mu_n)_{n \in \N}\) of functions in \(C^\infty(\overline\Omega)\), for each \(n \in \N\) let \(u_n \in C_0^{\infty}(\overline\Omega)\) be the solution of the Dirichlet problem with datum \(\mu_n\).

If the sequence \((\mu_n)_{n \in \N}\) is bounded in \(L^1(\Omega)\), then by the estimate given by corollary~\ref{corollaryEstimateSobolevCInfty} below, the sequence \((u_n)_{n \in \N}\) is bounded in the Sobolev space \(W^{1, q}(\Omega)\) for every \(1 \le q < \frac{N}{N-1}\).
By the Rellich-Kondrachov compactness theorem, we may extract a subsequence \((u_{n_k})_{k \in \N}\) which converges in \(L^1(\Omega)\) to some function \(u\).

We choose the sequence \((\mu_n)_{n \in \N}\) converging weakly to \(\mu\) in the sense of measures in \(\Omega\). 
Then, for every \(\zeta \in C_0^\infty(\overline\Omega)\),
\[
- \int\limits_\Omega u \Delta\zeta 
= \int\limits_\Omega \zeta \mu,
\]
whence \(u\) is a solution of the linear Dirichlet problem with datum \(\mu\).

For every \(n \in \N\),
\[
\norm{u_n}_{L^1(\Omega)} \le C \norm{\mu_n}_{L^1(\Omega)}.
\]
By the previous lemma we may choose the sequence \((\mu_n)_{n \in \N}\) such that
\[
\lim_{n \to \infty}{\norm{\mu_n}_{L^1(\Omega)}} = \norm{\mu}_{\cM(\Omega)}.
\]
Thus, \(u\) satisfies the estimate in the statement.
Uniqueness of the solution follows from this estimate and from the linearity of the Laplacian.
The proof of the proposition is complete.
\end{proof}

There is an alternative proof of existence of solutions based on an estimate given by Green's function.
Indeed, if \(\mu \in C^\infty(\overline\Omega)\), then the solution \(u\) of the linear Dirichlet problem has the integral representation
\[
u(x) = \int\limits_\Omega G(x, y) \mu(y) \dif y,
\]
where for every \(x \in \Omega\), \(G(x, \cdot)\) is the solution of the linear Dirichlet problem
\[
\left\{
\begin{alignedat}{2}
- \Delta G(x, \cdot) & = \delta_x	\quad && \text{in \(\Omega\),}\\
G(x, \cdot) & = 0 	\quad && \text{on \(\partial\Omega\).}
\end{alignedat}
\right.
\]
For \(N \ge 3\), we have for every \(x, y \in \Omega\) such that \(x \ne y\)
\[
\abs{G(x, y)} \le \frac{\NewConstant}{\abs{x - y}^{N-2}}.
\]
Thus,
\[
\abs{u(x)} \le \SameConstant \int\limits_\Omega \frac{\mu(y)}{\abs{x - y}^{N-2}} \dif y.
\]
It then follows from the Young inequality \cite{Bre:11}*{theorem~4.15} that for every \(1 \le p < \frac{N}{N-2}\),
\[
\norm{u}_{L^p(\Omega)} \le \Constant \norm{\mu}_{L^1(\Omega)}.
\]
Thus, the sequence of solutions \((u_n)_{n \in \N}\) in the proof of the proposition is bounded in \(L^p(\Omega)\) for every \(1 \le p < \frac{N}{N-2}\).
In particular, we may extract a subsequence which converges weakly in \(L^1(\Omega)\) and the conclusion follows.

\medskip
We can also give a direct proof of uniqueness without relying on the estimate in the proposition as follows.
If \(u, v \in L^1(\Omega)\) satisfy the Dirichlet problem with same datum \(\mu\), then for every \(\zeta \in C_0^\infty(\overline\Omega)\),
\[
\int\limits_\Omega (u - v) \Delta\zeta = 0.
\]
Thus, for every \(f \in C^\infty(\overline\Omega)\),
\[
\int\limits_\Omega (u - v) f = 0.
\]
Take a sequence \((f_n)_{n \in \N}\) in \(C^\infty(\overline\Omega)\) such that \((f_n)_{n \in \N}\) is uniformly bounded and converges to \(\sgn{(u - v)}\) almost everywhere in \(\Omega\). Applying the previous identity with \(f = f_n\), it follows from the dominated convergence theorem that
\[
\int\limits_\Omega \abs{u - v} = 0.
\]
Thus, \(u = v\) in \(\Omega\).

\medskip
In the proof of proposition~\ref{propositionExistenceLinearDirichletProblem}, 
the solution \(u\) of the linear Dirichlet problem is obtained as the limit of a bounded sequence in the Sobolev space \(W_0^{1, q}(\Omega)\). Thus, \(u\) itself belongs to \(W_0^{1, q}(\Omega)\) for every \(1 \le q \le \frac{N}{N-1}\).
In particular, \(u\) vanishes on \(\partial\Omega\) in the sense of traces.

It is possible to recover the Dirichlet boundary condition without explicitly making reference to Sobolev spaces. 
The proof of the next proposition relies on the weak formulation of the solution:

\begin{proposition}
\label{propositionDirichletBoundaryCondition}
For every \(\mu \in \cM(\Omega)\), if \(u\) is the solution of the linear Dirichlet problem
with datum \(\mu\), then for every \(\epsilon > 0\),
\[
\int\limits_{\{x \in \Omega : d(x, \partial\Omega) < \epsilon\}} \abs{u} \le C \epsilon^2 \norm{\mu}_{\cM(\Omega)},
\]
for some constant \(C > 0\) depending on \(\Omega\). In particular,
\[
\lim_{\epsilon \to 0} \frac{1}{\epsilon} \int\limits_{\{x \in \Omega : d(x, \partial\Omega) < \epsilon\}} \abs{u} = 0.
\]
\end{proposition}

\begin{proof}
We first assume that \(\mu\) is a nonnegative measure.
In this case, by the weak maximum principle (proposition~\ref{propositionWeakMaximumPrinciple}), \(u\) is nonnegative.

Let \(\Phi : [0, +\infty) \to \R\) be a smooth bounded function. 
Given \(\epsilon > 0\) and given \(\zeta \in C_0^\infty(\Omega)\), let \(\zeta_\epsilon : \overline\Omega \to \R\) be defined by
\[
\zeta_\epsilon = \epsilon \Phi\Big( \frac{\zeta}{\epsilon} \Big).
\]
On the one hand, 
\[
\Delta \zeta_\epsilon = \Phi'\Big( \frac{\zeta}{\epsilon} \Big) \Delta\zeta + \frac{1}{\epsilon} \Phi''\Big( \frac{\zeta}{\epsilon} \Big) \abs{\nabla\zeta}^2.
\]
Thus, if \(\Phi\) is nondecreasing and \(\zeta\) is superharmonic in \(\Omega\),
\[
\Delta \zeta_\epsilon \le \frac{1}{\epsilon} \Phi''\Big( \frac{\zeta}{\epsilon} \Big) \abs{\nabla\zeta}^2.
\]
On the other hand,
if \(\Phi(0) = 0\), then \(\zeta_\epsilon \in C_0^\infty(\overline\Omega)\). Thus, \(\zeta_\epsilon\) is an admissible test function and
\[
- \int\limits_\Omega u \Delta\zeta_\epsilon = \int\limits_\Omega \zeta_\epsilon \dif\mu \le \NewConstant \epsilon \norm{\mu}_{\cM(\Omega)}.
\]
We then have
\[
- \int\limits_\Omega u  \,\Phi''\Big( \frac{\zeta}{\epsilon} \Big) \abs{\nabla\zeta}^2 \le \SameConstant \epsilon^2 \norm{\mu}_{\cM(\Omega)}.
\]

If the function \(\Phi\) is concave, then the integrand in the left-hand side is nonpositive.
Assume in addition that \(\Phi''(0) < 0\).
Then, by continuity of \(\Phi''\), there exists \(c_1 > 0\) and \(\delta > 0\) such that
\[
- \Phi'' \ge c_1 \quad \text{in \([0, \delta]\).}
\]
If \(\zeta\) is superharmonic and \(\zeta \ne 0\), then by the classical strong maximum principle \(\zeta > 0\) in \(\Omega\), whence by the Hopf lemma, \(\abs{\nabla\zeta} > 0\).
Thus, there exists \(\overline{\epsilon} > 0\) and \(c_2 > 0\) such that
\[
\abs{\nabla\zeta} \ge c_2 \quad \text{in \(\{x \in \Omega : d(x, \partial\Omega) \le \overline\epsilon\}\).}
\]
We deduce that
\[
- \Phi''\Big( \frac{\zeta}{\epsilon} \Big) \abs{\nabla\zeta}^2 \ge c_3 \quad 
\text{in \(\big\{x \in \Omega : \zeta(x) \le \epsilon\delta \ \text{and}\ d(x, \partial\Omega) \le \overline\epsilon\big\}\).}
\]

Since \(\zeta\) vanishes on \(\partial\Omega\), by the mean value inequality we have for every \(x \in \Omega\),
\[
\zeta(x) \le \norm{D\zeta}_{L^{\infty}(\Omega)} d(x, \partial\Omega).
\]
Choosing a positive superharmonic function \(\zeta \in C_0^\infty(\Omega)\) such that
\[
\norm{D\zeta}_{L^{\infty}(\Omega)} \le \delta,
\]
it follows that for every \(0 < \epsilon \le \overline{\epsilon}\), if \(x \in \Omega\) is such that \(d(x, \partial\Omega) < \epsilon\), then
\[
\zeta(x) \le \epsilon\delta \quad \text{and} \quad d(x, \partial\Omega) \le \overline\epsilon.
\]
Therefore,
\[
 \int\limits_{\{x \in \Omega : d(x, \partial\Omega) < \epsilon\}} c_3 u \le \SameConstant \epsilon^2 \norm{\mu}_{\cM(\Omega)}.
\]
Thus, the estimate holds for every \(0 < \epsilon \le \overline\epsilon\). 
For \(\epsilon > \overline{\epsilon}\), the conclusion follows from the estimate
\[
\norm{u}_{L^1(\Omega)} \le C \norm{\mu}_{\cM(\Omega)}.
\]
This concludes the proof when \(u\) is nonnegative.

\medskip
If the function \(u\) is not nonnegative, then we may proceed as follows.
Take any nonnegative measure \(\nu \in \cM(\Omega)\) such that
\[
-\nu \le \mu \le \nu.
\]
By the weak maximum principle (proposition~\ref{propositionWeakMaximumPrinciple}), the solution \(v\) of the linear Dirichlet problem with datum \(\nu\) is nonnegative and satisfies
\[
- v \le u \le v.
\]
Thus, for every \(\epsilon > 0\),
\[
\int\limits_{\{x \in \Omega : d(x, \partial\Omega) < \epsilon\}} \abs{u} \le  \int\limits_{\{x \in \Omega : d(x, \partial\Omega) < \epsilon\}} v \le C \epsilon^2 \norm{\nu}_{\cM(\Omega)}.
\]
Choosing for instance \(\nu = \abs{\mu}\), the conclusion follows.
\end{proof}

The previous proposition actually gives a better information near the boundary than saying that the solution belongs to \(W_0^{1, q}(\Omega)\) for \(1 \le q < \frac{N}{N-1}\). 
In fact, a function in this Sobolev space satisfy the limit in the statement but the integral need not be controlled by a term of the order of \(\epsilon^2\).

\medskip

More generally, one  may consider linear Dirichlet problems of the type
\[
\left\{
\begin{alignedat}{2}
- \Delta u & = \mu	\quad && \text{in \(\Omega\),}\\
u & = \nu 	\quad && \text{on \(\partial\Omega\),}
\end{alignedat}
\right.
\]
where \(\nu\) is an \(L^1\) function or a finite measure on \(\partial\Omega\).
This problem and its nonlinear version were studied by Brezis in 1972 in an unpublished work when \(\nu \in L^1(\partial\Omega)\) \citelist{\cite{GmiVer:91}*{lemma~4.1} \cite{Ver:04}*{lemma~2.5}}.
The existence of solutions of the linear Dirichlet problem with boundary measure data follows from an estimate of the solutions in \(L^p(\Omega)\) for \(1 \le p < \frac{N}{N-1}\) using the Poisson kernel \cite{MarVer:98a}*{lemma~1.4}.
The boundary trace in this case is attained in the weak sense \cite{Ver:04}*{theorem~2.13}.

\section{Nonlinear case}

The nonlinear Dirichlet problem 
\[
\left\{
\begin{alignedat}{2}
- \Delta u + g(u) & = \mu	\quad && \text{in \(\Omega\),}\\
u & = 0 	\quad && \text{on \(\partial\Omega\),}
\end{alignedat}
\right.
\]
may behave differently according to whether \(\mu\) is an \(L^1\) function or a finite measure.
The meaning of solution is an adaptation of the linear case:

\begin{definition}
Let \(g : \R \to \R\) be a continuous function and let \(\mu \in \cM(\Omega)\).
A function \(u : \Omega \to \R\) is a \emph{solution} of the nonlinear Dirichlet problem if
\begin{enumerate}[\((i)\)]
\item \(u \in L^1(\Omega)\), 
\item \(g(u) \in L^1(\Omega)\),
\item for every \(\zeta \in C_0^\infty(\overline\Omega)\),
\[
- \int\limits_\Omega u \Delta\zeta + \int\limits_\Omega g(u)\zeta = \int\limits_\Omega \zeta \dif\mu.
\]
\end{enumerate}
\end{definition}

To give a flavor of what happens in this case, let us momentarily consider the case where the nonlinearity \(g\) is \emph{nondecreasing}.

\begin{proposition}
\label{propositionBrezisStrauss}
Let \(g : \R \to \R\) be a nondecreasing continuous function.
For every \(\mu \in L^1(\Omega)\), the nonlinear Dirichlet problem with nonlinearity \(g\) and datum \(\mu\) has a solution.
\end{proposition}

We explain the original argument of Brezis and Strauss~\cite{BreStr:73}*{theorem~1} who initiated the study of solutions of nonlinear problems with \(L^1\) data at a time when mostly variational solutions were considered. 

\begin{proof}
For every \(m, n \in \N\), we subtract the equation satisfied by \(u_m\) by the equation satisfied by \(u_n\) and we get
\[
-\Delta(u_m - u_n) + g(u_m) - g(u_n) = \mu_m - \mu_n.
\]
Using \(\sgn{(u_m - u_n)}\) as a test function --- or to be more precise a regularized version of the sign function --- we deduce that
\[
\int\limits_\Omega [g(u_m) - g(u_n)] \sgn{(u_m - u_n)} \le \int\limits_\Omega (\mu_m - \mu_n) \sgn{(u_m - u_n)},
\]
since we formally have
\[
- \int\limits_\Omega \Delta (u_m - u_n) \sgn{(u_m - u_n)} \ge 0.
\]
This is a consequence of the fact that the sign function is nondecreasing \cite{BreMarPon:07}*{proposition~4.B.2} (see also lemma~\ref{lemmaEstimateAbsorption} and corollary~\ref{corollaryContractionSign} below).
Using the monotonicity of \(g\), we conclude that
\[
\norm{g(u_m) - g(u_n)}_{L^1(\Omega)} \le \norm{\mu_m - \mu_n}_{L^1(\Omega)}.
\]
Since \((\mu_n)_{n \in \N}\) is a Cauchy sequence in \(L^1(\Omega)\), this implies that \((g(u_n))_{n \in \N}\) is also a Cauchy sequence in \(L^1(\Omega)\). Thus, \((g(u_n))_{n \in \N}\) converges in \(L^1(\Omega)\).

Convergence in \(L^1(\Omega)\) of the sequence \((u_n)_{n \in \N}\) is now a consequence of the linear estimates above.
Indeed, for every \(m, n \in \N\),
\[
\begin{split}
\norm{u_m - u_n}_{L^1(\Omega)} 
&\le C \norm{(\mu_m - g(u_m)) - (\mu_n - g(u_n))}_{L^1(\Omega)}\\
&\le C \norm{g(u_m) - g(u_n)}_{L^1(\Omega)} + C \norm{\mu_m - \mu_n}_{L^1(\Omega)}.
\end{split}
\]
Thus,
\[
\norm{u_m - u_n}_{L^1(\Omega)} \le 2C \norm{\mu_m - \mu_n}_{L^1(\Omega)}.
\] 
This estimate implies that the sequence \((u_n)_{n \in \N}\) is also a Cauchy sequence in \(L^1(\Omega)\). 
Thus, \((u_n)_{n \in \N}\) converges in \(L^1(\Omega)\) to some function \(u\), whence \((g(u_n))_{n \in \N}\) converges in \(L^1(\Omega)\) to \(g(u)\).
Therefore, \(u\) solves the nonlinear Dirichlet problem with datum \(\mu \in L^1(\Omega)\).
\end{proof}

\medskip
The existence of solutions of the Dirichlet problem for every \(\mu \in L^1(\Omega)\) when the nonlinearity \(g\) is nondecreasing relies on a contraction phenomenon.
There is an alternative argument due to Gallouët and Morel~\cite{GalMor:84} for nonlinearities \(g\) satisfying only the sign condition.
Their strategy consists in showing that the sequence \((g(u_n))_{n \in \N}\) is equi-integrable and this relies on the following estimate valid for every \(\kappa \ge 0\),
\[
\int\limits_{\{\abs{u_n} > \kappa\}} \abs{g(u_n)} \le \int\limits_{\{\abs{u_n} > \kappa\}} \abs{\mu_n}.
\]

One might hope to use the same kind of argument to solve the nonlinear Dirichlet problem for an arbitrary finite measure \(\mu\). 
This is not possible.
Indeed, the approximation of a measure by functions can be done in the weak sense of measures but not strongly and already in the study of the linear Dirichlet problem we relied on a compactness argument to obtain the existence of solutions.
Moreover, there are measures \(\mu\) for which the nonlinear Dirichlet problem does not have a solution. 

The first example of such surprising phenomenon was discovered by 
Bénilan and Brezis \citelist{\cite{BenBre:04}*{remark~A.4} \cite{Bre:82}*{théorème~1}} and shows a major difference between the nonlinear \(L^1\) theory and the nonlinear measure theory:

\begin{proposition}
\label{propositionNonExistenceBenilanBrezis}
Let  \(a \in \Omega\). If \(p \ge \frac{N}{N-2}\), then the nonlinear Dirichlet problem
\[
\left\{
\begin{alignedat}{2}
-\Delta u + \abs{u}^{p-1}u & = \delta_a \quad && \text{in \(\Omega\),}\\
u & = 0 \quad && \text{on \(\partial\Omega\),}
\end{alignedat}
\right.
\]
has no solution.
\end{proposition}

\begin{proof}
We assume for simplicity that \(a = 0\) and \(\Omega\) is the unit ball \(B_1\) centered at \(0\). 
Assume by contradiction that this problem has a solution. Let \(\varphi \in C_c^\infty(\R^N)\) be such that \(\supp{\varphi} \subset B_1\). Given \(n \in \N_*\), let \(\varphi_n : B_1 \to \R\) be the function defined for \(x \in B_1\) by 
\[
\varphi_n(x) = \varphi(nx).
\]
Using \(\varphi_n\) as a test function for the equation, we have
\[
- \int\limits_{B_1} u \Delta\varphi_n + \int\limits_{B_1} \abs{u}^{p-1}u \varphi_n = \int\limits_{B_1} \varphi_n \delta_0 = \varphi(0).
\]
By the dominated convergence theorem, 
\[
\lim_{n \to \infty}{\int\limits_{B_1} \abs{u}^{p-1}u \varphi_n} = 0.
\]
Moreover,
\[
\int\limits_{B_1} u \Delta \varphi_n = \int\limits_{B_{\frac{1}{n}}} u \Delta \varphi_n = n^2 \int\limits_{B_{\frac{1}{n}}} u \Delta \varphi(nx) \dif x.
\]
By the Hölder inequality and by a change of variable, we deduce that
\[
\begin{split}
\bigg| \int\limits_{B_1} u \Delta \varphi_n \bigg| 
& \le n^2 \bigg( \int\limits_{B_{\frac{1}{n}}} \abs{u}^p \bigg)^\frac{1}{p} \bigg( \int\limits_{B_{\frac{1}{n}}} \abs{\Delta \varphi(nx)}^{p'} \dif x  \bigg)^{\frac{1}{p'}}\\
& = n^{2-\frac{N}{p'}}  \bigg( \int\limits_{B_{\frac{1}{n}}} \abs{u}^p \bigg)^\frac{1}{p} \bigg( \int\limits_{B_{1}} \abs{\Delta \varphi}^{p'} \bigg)^{\frac{1}{p'}}
\end{split}
\]
Note that \(2-\frac{N}{p'} \le 0\) if and only if \(p \ge \frac{N}{N-2}\) and thus by the dominated convergence theorem,
\[
\lim_{n \to \infty}{\int\limits_{B_1} u \Delta \varphi_n} = 0.
\]
We deduce that \(\varphi(0) = 0\). In order to get a contradiction, it suffices to take \(\varphi\) such that \(\varphi(0) \ne 0\).
\end{proof}


\chapter{Linear regularity theory}

We prove Stampacchia's regularity theory~\cite{Sta:65} for the linear Dirichlet problem
\[
\left\{
\begin{alignedat}{2}
- \Delta u & = \mu	\quad && \text{in \(\Omega\),}\\
u & = 0 	\quad && \text{on \(\partial\Omega\),}
\end{alignedat}
\right.
\]
which asserts that if \(\mu\) is an \(L^1\) function or a finite measure, then the solution \(u\) belongs to the Sobolev space \(W_0^{1, q}(\Omega)\) for every \(1 \le q < \frac{N}{N - 1}\).
When \(\mu\) is an \(L^1\) function, the solution \(u\) need not belong to \(W^{2, 1}(\Omega)\), which would have been the natural counterpart of the Calderón-Zygmund regularity theory.

\section{Main estimates}

The main result of this section concerning the linear Dirichlet problem is due to Stampacchia~\cite{Sta:65}*{théorème~9.1}:

\begin{proposition}
\label{prop3.1}
Let $\mu \in \cM(\Omega)$.
If \(u\) is the solution of the linear Dirichlet problem with datum \(\mu\), then for every $1 \leq q < \frac{N}{N-1}$, $u \in W_0^{1,q}(\Omega)$ and the following estimate holds
$$
\|D u\|_{L^q(\Omega)} \leq C \|\mu\|_{\cM(\Omega)},
$$
for some constant \(C > 0\) depending on \(q\), \(N\) and \(\Omega\).
\end{proposition}

In particular, by the Sobolev imbedding we have for every $1 \leq p < \frac{N}{N-2}$, $u \in L^p(\Omega)$ and
$$
\|u\|_{L^{p}(\Omega)} \leq C' \|\mu\|_{\cM(\Omega)}.
$$

The estimate of proposition~\ref{prop3.1} will be obtained using a duality argument which relies on the following estimate due to Stampacchia~\cite{Sta:58}*{proposizione~5.1}:

\begin{lemma}
\label{lemmaInequalityStampacchia}
Let $F \in C^\infty(\overline{\Omega}; \R^N)$.
If $v \in C_0^\infty(\overline\Omega)$ is the solution of the linear Dirichlet problem
\begin{equation*}
\left\{
\begin{alignedat}{2}
- \Delta v & = \div{F} && \quad \text{in } \Omega,\\
v & = 0 && \quad \text{on } \partial\Omega,
\end{alignedat}
\right.
\end{equation*}
then for every $p > N$,
\begin{equation*}
\left\| v \right\|_{L^\infty(\Omega)} \leq C \|F\|_{L^p(\Omega)},
\end{equation*}
for some constante \(C > 0\) depending on \(p\), \(N\) and \(\Omega\).
\end{lemma}

For the convenience of the reader, we follow the strategy of the proof from \cite{HarSta:66}*{lemma~7.3}.

\begin{proof}
We shall assume that \(N \ge 3\).
The case of dimension \(N = 2\) requires a small modification concerning the Sobolev inequality.

Given $\kappa > 0$, let \(S_\kappa : \R \to \R\) be the function defined for \(t \in \R\) by
\[
S_\kappa(t) = 
\begin{cases}
t+\kappa 	& \text{if \(t < - \kappa\),}\\
0 			& \text{if \(- \kappa \le t \le \kappa\),}\\
t-\kappa 	& \text{if \(t > \kappa\),}\\
\end{cases}
\]  

Take $S_\kappa(v)$ as a test function of the Dirichlet problem. 
On the one hand, by the H\"older inequality, we have for every \(2 \le p < + \infty\),
\begin{equation*}
\|D S_\kappa(v)\|_{L^2(\Omega)} 
\leq \|F\|_{L^2(\{\abs{v} > \kappa\})}
\leq \|F\|_{L^p(\Omega)} \meas{\{\abs{v} > \kappa\}}^{\frac12 - \frac1p}.
\end{equation*}
On the other hand, by the H\"older inequality and by the Sobolev inequality,
\begin{equation*}
\begin{split}
\|S_\kappa(v)\|_{L^1(\Omega)}
& \le \|S_\kappa(v)\|_{L^{\frac{2N}{N- 2}}(\Omega)} \meas{\{\abs{v} > \kappa\}}^{\frac12 + \frac1N}\\
& \leq \NewConstant \|D S_\kappa(v)\|_{L^2} \meas{\{\abs{v} > \kappa\} }^{\frac12 + \frac1N}.
\end{split}
\end{equation*}
Combining both estimates, we deduce that for every \(\kappa > 0\),
\begin{equation*}
\|S_{\kappa}(v)\|_{L^1(\Omega)} \leq \SameConstant \|F\|_{L^p(\Omega)} \meas{\{\abs{v} > \kappa\}}^\alpha,
\end{equation*}
where $\alpha = 1 + \frac1N - \frac1p$. 

\begin{Claim}
If $\alpha > 1$, then
\begin{equation*}
\|v\|_{L^\infty(\Omega)} 
\leq  \Constant \|F\|_{L^p(\Omega)}^{\frac1\alpha} \|v\|_{L^1(\Omega)}^{1 - \frac1\alpha}.
\end{equation*}
\end{Claim}

Assuming the claim, we can conclude the proof of the lemma.
Note that if \(p > N\), then \(\alpha > 1\) and the claim applies.
Since $ \|v\|_{L^1(\Omega)} \leq |\Omega| \|v\|_{L^\infty(\Omega)}$, we deduce from the claim that
\[
\|v\|_{L^\infty(\Omega)} \leq \Constant \|F\|_{L^p(\Omega)},
\]
which is the estimate we wanted to establish. 

It remains to prove the claim:

\begin{proof}[Proof of the claim]
By Cavalieri's principle \cite{Wil:07}*{corollaire~6.34},
\begin{equation*}
\|S_\kappa(v)\|_{L^1(\Omega)} = \int_0^{+\infty} \meas{\{\abs{S_\kappa(v)} > s\}} \dif s = \int_\kappa^{+\infty} \meas{\{\abs{v} > t \}} \dif t.
\end{equation*}
Therefore,
\[
\int_\kappa^{+\infty} \meas{\{\abs{v} > t \}} \dif t 
\le \SameConstant \|F\|_{L^p(\Omega)} \meas{\{\abs{v} > \kappa\}}^\alpha.
\] 
Let \(H : [0, +\infty) \to \R\) be the function defined for \(\kappa \ge 0\) by
\[
H(\kappa) = \int_\kappa^{+\infty} \meas{\{\abs{v} > t \}} \dif t.
\]
Note that \(H\) is nonincreasing and in view of the estimate we have for almost every \(\kappa \ge 0\),
\[
- H'(\kappa) = \meas{\{\abs{v} > \kappa\}} \ge \left(\frac{H(\kappa)}{\SameConstant \|F\|_{L^p(\Omega)}} \right)^\frac{1}{\alpha}.
\]
Integrating this inequality, we conclude that if \(\alpha > 1\), then there exists \(\kappa_0 \ge 0\) such that \(H(\kappa_0) = 0\) and
\[
\kappa_0 \le \Constant \|F\|_{L^p(\Omega)}^{\frac1\alpha} {H(0)}^{1 - \frac1\alpha}
\]
Since \(\norm{w}_{L^\infty(\Omega)} \le \kappa_0\) and \(H(0) = \norm{v}_{L^1(\Omega)}\), the claim follows.
\end{proof}

The proof of the lemma is complete.
\end{proof}

\begin{corollary}
\label{corollaryEstimateSobolevCInfty}
Let \(f \in C^\infty(\overline\Omega)\). If \(u \in C_0^\infty(\overline\Omega)\) is the solution of the linear Dirichlet problem
\[
\left\{
\begin{alignedat}{2}
- \Delta u & = f	\quad && \text{in \(\Omega\),}\\
u & = 0 	\quad && \text{on \(\partial\Omega\),}
\end{alignedat}
\right.
\]
then for every \(1 \le q < \frac{N}{N-1}\),
\[
\|D u\|_{L^q(\Omega)} \leq C \|f\|_{L^1(\Omega)},
\]
for some constant \(C > 0\) depending on \(q\), \(N\) and \(\Omega\).
\end{corollary}

\begin{proof}
On the one hand, by Green's identity we have for every \(u \in C_0^\infty(\overline\Omega)\) and for every \(v \in C_0^\infty(\overline\Omega)\),
\[
\int\limits_\Omega u \Delta v = \int\limits_\Omega v \Delta u.
\]
On the other hand, by the divergence theorem we have for every \(F \in C^\infty(\overline\Omega; \R^N)\) and for every \(u \in C_0^\infty(\overline\Omega)\), 
\[
\int\limits_\Omega \nabla u \cdot F = - \int\limits_\Omega u \div{F}.
\]
Thus, if \(u\) is the solution of the linear Dirichlet problem with datum \(f\) and if \(v\) is the solution of the linear Dirichlet problem with datum \(\div{F}\),
\[
\int\limits_\Omega \nabla u \cdot F = \int\limits_\Omega v f.
\]
This yields
\[
\bigg| \int\limits_\Omega \nabla u \cdot F \bigg| \le \norm{v}_{L^\infty(\Omega)} \norm{f}_{L^1(\Omega)}.
\]
For every \(1 \le q < \frac{N}{N-1}\), by Stampacchia's estimate (lemma~\ref{lemmaInequalityStampacchia}) we get
\[
\bigg| \int\limits_\Omega \nabla u \cdot F \bigg| \le C  \norm{F}_{L^{q'}(\Omega)} \norm{f}_{L^1(\Omega)}.
\]
Since this estimate holds for every \(F \in C^\infty(\overline\Omega; \R^N)\), by the Riesz representation theorem we get
\[
\norm{D u}_{L^q(\Omega)} \le C \norm{f}_{L^1(\Omega)}.
\qedhere
\]
\end{proof}

By the Riesz representation theorem, every element \(T\) in the dual space \((W_0^{1, q}(\Omega))'\) can be written in the form
\[
T = \div{F}
\]
for some \(F \in L^{q'}(\Omega)\) in the sense that for every \(w \in W_0^{1, q}(\Omega)\),
\[
T(w) = \int\limits_\Omega F \cdot \nabla w.
\]
In addition,
\[
\norm{T}_{(W_0^{1, q}(\Omega))'} = \norm{F}_{L^{q'}(\Omega)}.
\]
The previous corollary is based on the duality between the estimates
\[
\norm{u}_{L^\infty(\Omega)} \le C \norm{\Delta u}_{(W_0^{1, q}(\Omega))'}
\]
and
\[
\norm{v}_{W_0^{1, q}(\Omega)} \le C \norm{\Delta u}_{L^1(\Omega)}.
\]

\medskip
By definition, \(W_0^{1, q}(\Omega)\) is the closure of \(C_c^\infty(\Omega)\) in \(W^{1, q}(\Omega)\).
We have the following characterizations of an element in \(W_0^{1, q}(\Omega)\):

\begin{lemma}
Let \(u : \Omega \to \R\) be a measurable function. 
For every \(q \ge 1\), the following conditions are equivalent:
\begin{enumerate}[\((i)\)]
\item \(u \in W_0^{1, q}(\Omega)\),
\item the extension \(U : \R^N \to \R\) defined by
\[
U(x) =
\begin{cases}
u(x)	& \text{if \(x \in \Omega\),}\\
0		& \text{if \(x \not\in \Omega\),}
\end{cases}
\]
belongs to \(W^{1, q}(\R^N)\),
\item there exists \(G \in L^1(\Omega; \R^N)\) such that for every \(F \in C^\infty(\overline{\Omega}; \R^N)\),
\[
\int\limits_\Omega u \div{F} = \int\limits_\Omega G \cdot F,
\]
in which case \(\nabla u = -G\).
\end{enumerate}
\end{lemma}

In the proof of Stampacchia's regularity result we use the third characterization.
The existence of \(G\) is obtained using the Riesz representation theorem and Stampacchia's estimate (lemma~\ref{lemmaInequalityStampacchia}).

\begin{proof}[Proof of proposition~\ref{prop3.1}]
On the one hand, for every \(\zeta \in C_0^\infty(\overline\Omega)\),
\[
- \int\limits_\Omega u \Delta \zeta
= \int\limits_\Omega \zeta \dif\mu.
\]
Thus,
\[
\bigg| \int\limits_\Omega u \Delta \zeta \bigg| \le \norm{\zeta}_{L^\infty(\Omega)} \norm{\mu}_{\cM(\Omega)}.
\]
On the other hand, for every \(F \in C^\infty(\overline\Omega; \R^N)\),
if \(v \in C_0^\infty(\overline\Omega)\) is the solution of the linear Dirichlet problem with datum \(\div{F}\), then for every \(1 \le q < \frac{N}{N-1}\), we have \(q' > N\) and by Stampacchia's estimate,
\[
\norm{v}_{L^\infty(\Omega)} \le C \norm{F}_{L^{q'}(\Omega)}.
\]
Applying the previous estimate with test function \(v\),
\[
\bigg| \int\limits_\Omega u \div{F} \bigg| 
\le C \norm{F}_{L^{q'}(\Omega)} \norm{\mu}_{\cM(\Omega)}.
\]
Thus, the functional
\[
F \in C^\infty(\overline\Omega; \R^N) \longmapsto \int\limits_\Omega u \div{F}
\]
admits a unique extension as a continuous linear functional in \(L^{q'}(\Omega; \R^N)\).

By the Riesz representation theorem, for \(q' < +\infty\) there exists a unique function \(G \in L^q(\Omega; \R^N)\) such that for every \(F \in C^\infty(\overline\Omega; \R^N)\),
\[
\int\limits_\Omega u \div{F} = \int\limits_\Omega G \cdot F.
\]
Thus, \(u \in W_0^{1, q}(\Omega)\) and \(\nabla u = -G\).
In particular, the conclusion also holds for \(q = 1\).

Since for every \(F \in C^\infty(\overline\Omega; \R^N)\),
\[
\bigg| \int\limits_\Omega \nabla u \cdot F \bigg| \le C \norm{F}_{L^{q'}(\Omega)} \norm{\mu}_{\cM(\Omega)},
\]
we conclude that for every \(1 \le q < \frac{N}{N-1}\),
\[
\norm{D u}_{L^q(\Omega)} \le C \norm{\mu}_{\cM(\Omega)}
\]
The proof of the proposition is complete.
\end{proof}

An alternative approach to Stampacchia's regularity result is to prove simultaneously existence and regularity of solutions following the strategy of the proof of proposition~\ref{propositionExistenceLinearDirichletProblem}.
By uniqueness of the solution, this implies that \emph{every} solution of the linear Dirichlet problem has the required regularity.

\medskip

The next corollary clarifies the meaning of the boundary condition that is implicit in Littman-Stampacchia-Weinberger's formulation of solution of the Dirichlet problem:

\begin{corollary}
\label{corollaryWeakFormulationEquivalence}
For every \(\mu \in \cM(\Omega)\), \(u\) is a solution of the linear Dirichlet problem with datum \(\mu\) if and only if \(u \in W_0^{1, 1}(\Omega)\) and the equation \(- \Delta u = \mu\) is satisfied in the sense of distributions in \(\Omega\), meaning that for every \(\varphi \in C_c^\infty(\Omega)\),
\[
\int\limits_\Omega \nabla u \cdot \nabla \varphi = \int\limits_\Omega \varphi \dif\mu.
\]
\end{corollary}

\begin{proof}
If \(u\) is a solution of the Dirichlet problem, then the equation is satisfied in the sense of distributions and by Stampacchia's regularity result (proposition~\ref{prop3.1}), \(u \in W_0^{1, 1}(\Omega)\). 
Conversely, assume that \(u \in W_0^{1, 1}(\Omega)\) and the equation is satisfied in the sense of distributions in \(\Omega\). On the one hand by an approximation argument on the test function  (see for instance the proof of proposition~4.B.1 in \cite{BreMarPon:07}), we have for every \(\zeta \in C_0^\infty(\overline\Omega)\),
\[
\int\limits_\Omega \nabla u \cdot \nabla\zeta = \int\limits_\Omega \zeta \dif\mu.
\]
On the other hand, if \(u \in W_0^{1, 1}(\Omega)\), then
\[
\int\limits_\Omega \nabla u \cdot \nabla\zeta = - \int\limits_\Omega u \Delta\zeta,
\]
whence \(u\) is a solution of the linear Dirichlet problem with datum \(\mu\).
\end{proof}

Stampacchia's original result concerns not only the Laplacian but any elliptic linear operator with bounded coefficients.
Overall, the proof of Stampacchia's regularity result for measure or \(L^1\) data is simpler compared to Calderón-Zygmund's \(L^p\) regularity theory which relies on singular integral estimates.

We recall that the classical Calderón-Zygmund \(L^p\) theory for the linear Dirichlet problem
\[
\left\{
\begin{alignedat}{2}
- \Delta u & = \mu	\quad && \text{in \(\Omega\),}\\
u & = 0 	\quad && \text{on \(\partial\Omega\),}
\end{alignedat}
\right.
\]
asserts that if \(1 < p  < +\infty\) and if \(\mu \in L^p(\Omega)\), then \(u \in W^{2, p}(\Omega) \cap W^{1, p}_0(\Omega)\) and the following estimate holds,
\[
\norm{u}_{W^{2, p}(\Omega)} \le C \norm{\mu}_{L^p(\Omega)};
\]
see \cite{GilTru:83}*{theorem~9.15 and lemma~9.17}.
This inequality is false for \(p = 1\) or \(p = +\infty\). 
In particular, if \(\mu \in L^1(\Omega)\), it may happen that \(u \not\in W^{2, 1}(\Omega)\).

We explain why the Calderón-Zygmund estimate cannot hold for \(p = 1\), by giving the argument in dimension \(N \ge 3\).
If the estimate
\[
\norm{u}_{W^{2, 1}(\Omega)} \le C \norm{\mu}_{L^1(\Omega)}
\]
were correct, then by the Sobolev-Gagliardo-Nirenberg inequality we would have
\[
\norm{u}_{L^\frac{N}{N-2}(\Omega)} \le C' \norm{\mu}_{L^1(\Omega)}.
\]
By an approximation argument, this inequality would also hold with \(\mu\) replaced by a Dirac mass \(\delta_a\), in which case \(\norm{\mu}_{L^1(\Omega)}\) should be understood as the total mass of the Dirac mass, namely \(1\). 
But this is not possible since  if \(\Omega\) is the unit ball \(B_1\) centered at \(0\), then the solution of the linear Dirichlet problem with datum \(\delta_0\) is given by 
\[
u(x) = c_N \left( \frac{1}{\norm{x}^{N-2}} - 1 \right),
\]
but \(u \not\in L^{\frac{N}{N-2}}(\Omega)\).

The situation for \(p = 1\) can be even worse than one might expect.
By a remarkable counterexample of Ornstein \cite{Orn:62}*{theorem~1},
even the stronger inequality
\[
\norm{D^2 u}_{L^1(\Omega)} \le \sum_{i=1}^N \Bignorm{\frac{\partial^2 u}{\partial x_i^2}}_{L^1(\Omega)}
\]
is false.



\section{Weak $L^p$ estimates} 

We now present an improvement of Stampacchia's regularity theory for solutions of the linear Dirichlet problem in terms of weak \(L^p\) estimates, as an alternative to the duality argument of the previous section.
We have seen that if \(u\) is the solution of the linear Dirichlet problem with datum \(\mu\), then the estimates
\[
\norm{u}_{L^\frac{N}{N-2}(\Omega)} \le C \norm{\mu}_{L^1(\Omega)}
\]
and
\[
\norm{D u}_{L^\frac{N}{N-1}(\Omega)} \le C' \norm{\mu}_{L^1(\Omega)}
\]
are false.
We shall see that they have a true counterpart for weak \(L^p\) spaces.

Such weak \(L^p\) estimates can already be found for instance in the appendix of a paper by Bénilan, Brezis and Crandall~\cite{BenBreCra:75} and their proof is based on the Newtonian potential generated by the measure \(\mu\) in the whole space \(\R^N\).
The proof we present below follows a different strategy developed by Boccardo and Gallouët~\cite{BocGal:89}. 
Compared to the estimates we have already established, these weak \(L^p\) estimates  have the advantage that the constants involved do not depend on the domain.

\begin{proposition}
\label{propositionWeakLpEstimates}
Let \(N \ge 3\). For every \(\mu \in \cM(\Omega)\), if \(u\) the solution of the linear Dirichlet problem with datum \(\mu\), then for every \(t > 0\),
\[
\meas{\{\abs{u} > t\}}^\frac{N-2}{N} \le \frac{C}{t} \norm{\mu}_{\cM(\Omega)}
\]
and
\[
\meas{\{\abs{\nabla u} > t\}}^\frac{N-1}{N} \le \frac{C'}{t} \norm{\mu}_{\cM(\Omega)}
\]
for some constants \(C, C' > 0\) depending only on the dimension \(N\).
\end{proposition}

We recall that a measurable function \(u : \Omega \to \R\) is a weak \(L^p\) function if for every \(t > 0\),
\[
t \meas{\{\abs{u} > t\}}^\frac{1}{p} \le C.
\]
If \(u\) is an \(L^p\) function then by the Chebyshev inequality it is also a weak \(L^p\) function but the converse is false. However, by Cavalieri's principle  \cite{Wil:07}*{corollaire~6.34}, for every \(1 \le q < +\infty\),
\[
\int\limits_\Omega \abs{u}^q = q \int_0^{+\infty} t^{q - 1} \meas{\{\abs{u} > t\}} \dif t,
\]
which implies that if \(u\) is a weak \(L^p\) function, then it is an \(L^q\) function for every \(1 \le q < p\).

\medskip

The key-point in Boccardo-Gallouët's argument is the use of the truncation function \(T_{\kappa} : \R \to \R\) defined for \(s \in \R\) by
\[
T_{\kappa}(s) =
\begin{cases}
-\kappa	& \text{if \(s < -\kappa\),}\\
s			& \text{if \(-\kappa \le s \le \kappa\),}\\
\kappa	& \text{if \(s > \kappa\).}
\end{cases}
\]
The argument relies on the following interpolation estimate:

\begin{lemma}
\label{lemmaInterpolationLinfty}
For every \(\mu \in \cM(\Omega)\), if \(u\) the solution of the linear Dirichlet problem with datum \(\mu\),
then for every \(\kappa > 0\), we have \(T_\kappa (u) \in W^{1, 2}_0(\Omega)\) and
\[
\norm{D T_\kappa (u)}_{L^2(\Omega)} \le \kappa^\frac{1}{2} \norm{\mu}_{\cM(\Omega)}^\frac{1}{2}.
\]
\end{lemma}

\begin{proof}
We first establish the estimate when \(u \in W_0^{1, 2}(\Omega)\).
In this case, for every \(\kappa > 0\) we have \(T_\kappa(u) \in W_0^{1, 2}(\Omega)\) and
\[
\abs{\nabla T_\kappa(u)}^2 = \nabla T_\kappa(u) \cdot \nabla u.
\]
Using \(T_\kappa(u)\) as a test function in the equation satisfied by \(u\) we get
\[
\int\limits_\Omega \abs{\nabla T_\kappa(u)}^2 = \int\limits_\Omega \nabla T_\kappa(u) \cdot \nabla u = - \int\limits_\Omega T_\kappa(u) \Delta u \le \kappa \norm{\Delta u}_{L^1(\Omega)}.
\]
This gives the estimate when \(u \in W_0^{1, 2}(\Omega)\).

Given \(\mu \in \cM(\Omega)\), we consider a sequence of functions \((\mu_n)_{n \in \N}\) in \(C_0^\infty(\overline{\Omega})\) converging weakly to \(\mu\) in the sense of measures as in the approximation lemma (lemma~\ref{lemmaWeakApproximation}).
If \(u_n\) denotes the solution of the linear Dirichlet problem with datum \(\mu_n\), then for every \(n \in \N\),
\[
\norm{D T_\kappa (u_n)}_{L^2(\Omega)} \le \kappa^\frac{1}{2} \norm{\mu_n}_{L^ 1(\Omega)}^\frac{1}{2}.
\]
Since the sequence \((T_\kappa (u_n))_{n \in \N}\) is bounded in \(W_0^{1, 2}(\Omega)\) and \((u_n)_{n \in \N}\) converges to \(u\) for instance in \(L^1(\Omega)\), the conclusion follows.
\end{proof}

We deduce from this lemma that for every \(u \in W_0^{1, 1}(\Omega)\)  such that \(u \in L^\infty(\Omega)\) and \(\Delta u \in L^1(\Omega)\), we have \(u \in W^{1, 2}_0(\Omega)\) and
\begin{equation}\label{equationInterpolationLinftyBis}
\norm{Du}_{L^2(\Omega)} \le \norm{u}_{L^\infty(\Omega)}^\frac{1}{2} \norm{\Delta u}_{L^1(\Omega)}^\frac{1}{2}.
\end{equation}
This inequality is an easy case of the Gagliardo-Nirenberg interpolation inequality \cites{Gag:59,Nir:59},
\[
\norm{Du}_{L^{2p}(\R^N)} \le C \norm{u}_{L^\infty(\R^N)}\norm{D^2 u}_{L^{p}(\R^N)}
\]
for every \(1 \le p < +\infty\).

\begin{proof}[Proof of proposition~\ref{propositionWeakLpEstimates}]
We begin with the first inequality. By the interpolation inequality (lemma~\ref{lemmaInterpolationLinfty}), for every \(t > 0\), \(T_t(u) \in W^{1, 2}_0(\Omega)\). Thus, by the Sobolev inequality,
\[
\norm{T_t(u)}_{L^\frac{2N}{N-2}(\Omega)} \le \NewConstant \norm{D T_t(u)}_{L^2(\Omega)}.
\]
By the Chebyshev inequality,
\[
t \meas{\{\abs{u} > t\}}^\frac{N-2}{2N} \le \norm{T_t(u)}_{L^\frac{2N}{N-2}(\Omega)}
\]
while by the previous interpolation inequality,
\[
\norm{D T_t(u)}_{L^2(\Omega)} \le t^\frac{1}{2} \norm{\mu}_{\cM(\Omega)}^\frac{1}{2}.
\]
We deduce that
\[
t \meas{\{\abs{u} > t\}}^\frac{N-2}{2N} \le \SameConstant t^\frac{1}{2} \norm{\mu}_{\cM(\Omega)}^\frac{1}{2}.
\]
This gives the first estimate. 

We now establish the second estimate. 
Note that for every \(s > 0\) and \(t > 0\),
\[
\{\abs{\nabla u} > t\} \subset \atopwithdelims{\abs{\nabla u} > t}{\abs{u} \le s} \cup \{\abs{u} > s\}.
\]
Thus,
\[
\meas{\{\abs{\nabla u} > t\}} \le \meas{\atopwithdelims{\abs{\nabla u} > t}{\abs{u} \le s}} + \meas{ \{\abs{u} > s\}}.
\]
We have already an estimate of the second term in the right-hand side. In order to deal with the first one, 
we first note that
\[
\atopwithdelims{\abs{\nabla u} > t}{\abs{u} \le s} = \{\abs{\nabla T_s(u)} > t\}.
\]
By the Chebyshev inequality and by the interpolation inequality (lemma~\ref{lemmaInterpolationLinfty}),
\[
t \meas{\atopwithdelims{\abs{\nabla u} > t}{\abs{u} \le s}}^\frac{1}{2} \le \norm{\nabla T_s(u)}_{L^2(\Omega)} \le s^\frac{1}{2} \norm{\mu}_{\cM(\Omega)}^\frac{1}{2}.
\]
We then have
\[
\meas{\{\abs{\nabla u} > t\}} \le \frac{s}{t^2} \norm{\mu}_{\cM(\Omega)} + \frac{\Constant}{s^\frac{N}{N-2}} \norm{\mu}_{\cM(\Omega)}^\frac{N}{N-2}.
\]
Minimizing the right-hand side with respect to \(s\), we obtain the second estimate in the statement.
\end{proof}

The counterpart of the estimates of proposition~\ref{propositionWeakLpEstimates} in dimension \(N = 2\) are
\[
\meas{\{\abs{u} > t\}} \le C\meas{\Omega} \e^{- C t/{\norm{\Delta u}_{\cM(\Omega)}}}
\]
and
\[
\meas{\{\abs{\nabla u} > t\}}^2 \le \frac{C'}{t} \norm{\Delta u}_{\cM(\Omega)}.
\]

The first one can be obtained as in the previous proof by replacing the Sobolev inequality by the Trudinger inequality \citelist{\cite{Tru:67} \cite{GilTru:83}*{theorem~7.15}},
\[
\int\limits_\Omega \e^{\alpha (\frac{v}{\norm{D v}_{L^2}})^2} \le C''\meas{\Omega}.
\]

The second estimate is more subtle.
In \cite{BenBreCra:75}*{lemma~A.14} it relies on the integral representation
\[
\nabla u(x) = \frac{1}{2\pi} \int\limits_{\Omega} \nabla_x G(x, y) \dif\mu(y),
\]
where \(G\) denotes the Green's function, which satisfies the estimate
\[
\abs{\nabla_x G(x, y)} \le \frac{C'''}{\abs{x - y}}.
\]

In view of the strategy of Boccardo and Gallouët this argument is unsatisfactory in the sense that it relies on the linearity of the Laplacian and on the integral representation of solutions of the Poisson equation. 
More recent proofs using the structure of the Laplacian can be found in \citelist{\cite{DolHunMul:00}*{theorem~1.1} \cite{KilShaZho:08}*{theorem~1.2}  \cite{Min:07}*{theorem~1.4}} but they are not as elementary as in the case of dimension \(N \ge 3\).

\section{Compactness in Sobolev spaces}

By Stampacchia's estimate (proposition~\ref{prop3.1}), we have the following compactness result in Lebesgue spaces:

\begin{proposition}
\label{propositionCompactnessLp}
Let \((\mu_n)_{n \in \N}\) be a sequence in \(\cM(\Omega)\), and for each \(n \in \N\) let \(u_n\) be the solution of the linear Dirichlet problem with datum \(\mu_n\).
If \((\mu_n)_{n \in \N}\) is bounded in  \(\cM(\Omega)\), then there exists a subsequence \((u_{n_k})_{k \in \N}\) which converges strongly in \(L^p(\Omega)\) for every \(1 \le p < \frac{N}{N-2}\).
\end{proposition}

\begin{proof}
By proposition~\ref{prop3.1}, the sequence \((u_n)_{n \in \N}\) is bounded in \(W^{1, q}_0(\Omega)\) for every \(1 \le q < \frac{N}{N-1}\).
The conclusion follows from the Rellich-Kondrachov compactness theorem.
\end{proof}

A stronger result is actually true:

\begin{proposition}
\label{propositionCompactnessW1p}
Let \((\mu_n)_{n \in \N}\) be a sequence in \(\cM(\Omega)\), and for each \(n \in \N\) let \(u_n\) be the solution of the linear Dirichlet problem with datum \(\mu_n\).
If \((\mu_n)_{n \in \N}\) is bounded in  \(\cM(\Omega)\), then there exists a subsequence \((u_{n_k})_{k \in \N}\) which converges strongly in \(W_0^{1, q}(\Omega)\) for every \(1 \le q < \frac{N}{N-1}\).
\end{proposition}

This result is due to Boccardo and Gallou\"et~\cite{BocGal:89}*{assertion~(21)} and relies on the following inequality which is implicitly proved in \cite{BocGal:89}:

\begin{lemma}
For every \(\mu \in \cM(\Omega)\), if \(u\) the solution of the linear Dirichlet problem with datum \(\mu\),
then for every \(t > 0\),
\[
\meas{\{\abs{\nabla u} > t\}}^2 \le \frac{4}{t^2} \norm{u}_{L^1(\Omega)} \norm{\mu}_{\cM(\Omega)}.
\]
\end{lemma}

\begin{proof}
Let \(t > 0\). For every \(s > 0\),
\[
\{\abs{\nabla u} > t\} \subset \atopwithdelims{\abs{\nabla u} > t}{\abs{u} \le s} \cup \{\abs{u} > s\}.
\]
Thus,
\[
\meas{\{\abs{\nabla u} > t\}} \le \meas{\atopwithdelims{\abs{\nabla u} > t}{\abs{u} \le s}} + \meas{ \{\abs{u} > s\}}.
\]
By the Chebyshev inequality and the interpolation inequality (lemma~\ref{lemmaInterpolationLinfty}),
\[
\meas{\atopwithdelims{\abs{\nabla u} > t}{\abs{u} \le s}} = \meas{\{\abs{\nabla T_s(u)} > t\}} \le \frac{s}{t^2} \norm{\mu}_{\cM(\Omega)}.
\]
By the Chebyshev inequality, we also have
\[
\meas{ \{\abs{u} > s\}} \le \frac{1}{s} \norm{u}_{L^1(\Omega)}.
\]
Combining both estimates we get
\[
\meas{\{\abs{\nabla u} > t\}} \le \frac{s}{t^2} \norm{\mu}_{\cM(\Omega)} + \frac{1}{s} \norm{u}_{L^1(\Omega)}.
\]
Minimizing with respect to \(s\) we obtain the estimate.
\end{proof}

\begin{proof}[Proof of proposition~\ref{propositionCompactnessW1p}]
Passing to a subsequence if necessary, we may assume that \((u_n)_{n \in \N}\) is a Cauchy sequence in \(L^1(\Omega)\). 
By the estimate we have just proved, for every \(t > 0\) and for every \(m, n \in \N\),
\[
\meas{\{\abs{\nabla u_m - \nabla u_n} > t\}}^2 
\le \frac{4}{t^2} \norm{u_m - u_n}_{L^1(\Omega)} \norm{\Delta u_m - \Delta u_n}_{\cM(\Omega)}.
\]
Thus, \((\nabla u_n)_{n \in \N}\) is a Cauchy sequence with respect to the Lebesgue measure.
Since for every \(1 \le q < \frac{N}{N-1}\) this sequence is bounded in \(L^{q}(\Omega)\), by interpolation in Lebesgue spaces we deduce that it converges in \(W^{1, q}_0(\Omega)\) for every such exponent \(q\).
\end{proof}

From the equivalent formulation of weak solutions in terms of the Sobolev space \(W_0^{1, 1}(\Omega)\)  (corollary~\ref{corollaryWeakFormulationEquivalence}), 
the vector space
\[
X(\Omega) = \big\{ u \in W^{1, 1}_0(\Omega) : \Delta u \in \cM(\Omega) \big\}
\]
equipped with the norm \(\norm{\Delta u}_{\cM(\Omega)}\) is a Banach space.
From Boccardo-Gallouët's compactness result, \(X(\Omega)\)
is compactly imbedded into \(W_0^{1, q}(\Omega)\) for every \(1 \le q < \frac{N}{N-1}\).

In view of the interpolation inequality in \(W_0^{1, 2}(\Omega)\) (lemma~\ref{lemmaInterpolationLinfty}) and of Boccardo-Gallouët's compactness result, one might wonder whether \(X(\Omega) \cap L^\infty(\Omega)\) is compactly imbedded in \(W_0^{1, 2}(\Omega)\), but this is not true by a counterexample of Cioranescu and Murat~\cite{CioMur:97}*{example~2.1}.


\chapter{Maximum principles}

We present counterparts of the standard maximum principles in the setting of Littman-Stampacchia-Weinberger's weak solutions.
We keep track of the information of what happens near the boundary by using  \(C_0^\infty(\overline\Omega)\) test functions --- which vanish on the boundary but possibly have nonzero normal derivatives ---, rather than \(C_c^\infty(\Omega)\) test functions --- which have compact support in \(\Omega\).

\section{Weak maximum principle}

We begin with the substitute of the classical weak maximum principle in the setting of weak solutions:

\begin{proposition}
\label{propositionWeakMaximumPrinciple}
Let \(u \in L^1(\Omega)\).
\begin{enumerate}[\((i)\)]
\item If \(\Delta u \ge 0\) in the sense of \((C_0^\infty(\overline\Omega))'\), then \(u \le 0\) in \(\Omega\).
\item If \(\Delta u \le 0\) in the sense of \((C_0^\infty(\overline\Omega))'\), then \(u \ge 0\) in \(\Omega\).
\end{enumerate}
\end{proposition}

For instance, by \(\Delta u \ge 0\) in the sense of \((C_0^\infty(\overline\Omega))'\) we mean that
for every \(\zeta \in C_0^\infty(\overline\Omega)\) such that \(\zeta \ge 0\) in \(\overline\Omega\), 
\[
\int\limits_\Omega u \Delta \zeta \ge 0.
\]

\begin{proof}
We only need to establish assertion \((i)\).
For every \(\zeta \in C_0^\infty(\overline\Omega)\) such that \(\zeta \ge 0\) in \(\overline\Omega\), 
\[
\int\limits_\Omega u \Delta \zeta \ge 0.
\]
For every \(f \in C^\infty(\overline\Omega)\), let \(\zeta \in C_0^\infty(\overline\Omega)\) be the solution of the linear Dirichlet problem
\[
\left\{
\begin{alignedat}{2}
- \Delta \zeta & = f	&& \quad \text{in \(\Omega\),}\\
\zeta & = 0	&& \quad \text{in \(\partial\Omega\).}
\end{alignedat}
\right.
\]
If \(f \ge 0\) in \(\overline\Omega\), then \(\zeta\) is superharmonic, whence \(\zeta \ge 0\) in \(\Omega\). We then deduce that
\[
\int\limits_\Omega u f \le 0.
\]
Since this inequality holds for every \(f \in C^\infty(\Omega)\) such that \(f \ge 0\) in \(\overline\Omega\), we may take a sequence \((f_n)_{n \in \N}\) of such functions converging almost everywhere to the characteristic function \(\chi_{\{u > 0\}}\) and such that \((f_n)_{n \in \N}\) is bounded in \(L^\infty(\Omega)\). By the dominated convergence theorem we deduce that
\[
\int\limits_{\{u > 0\}} u \le 0.
\]
Therefore, \(u \le 0\) almost everywhere in \(\Omega\).
\end{proof}

It is useful to have a condition that allows to pass from an inequality in the sense of distributions to an inequality in the sense of \((C_0^\infty(\overline\Omega))'\). 
Stated differently, we want to find an assumption which insures that a subsolution of the \emph{equation}
\[
- \Delta u = \mu \quad \text{in \(\Omega\),}
\]
is a subsolution of the \emph{Dirichlet problem}
\[
\left\{
\begin{alignedat}{2}
- \Delta u & = \mu	\quad && \text{in \(\Omega\),}\\
u & = 0 	\quad && \text{on \(\partial\Omega\).}
\end{alignedat}
\right.
\]

This can be done under some additional information that insures that \(u\) is nonpositive on the boundary \(\partial\Omega\):

\begin{proposition}
\label{propositionDistributionC0Infty}
Let \(u \in L^1(\Omega)\) and \(\mu \in \cM(\Omega)\).
The following assertions are equivalent:
\begin{enumerate}[\((i)\)]
\item \(-\Delta u \le \mu \) in the sense of \((C_0^\infty(\overline\Omega))'\),
\item \(-\Delta u \le \mu \) in the sense of distributions in \(\Omega\) and
\[
\lim_{\epsilon \to 0} \frac{1}{\epsilon} \int\limits_{\{x \in \Omega : d(x, \partial\Omega) < \epsilon\}} u^+ = 0.
\]
\end{enumerate}
\end{proposition}

The direct implication relies on the fact that a \emph{solution} of the linear Dirichlet problem satisfies the limit above.

\begin{proof}[Proof of proposition~\ref{propositionDistributionC0Infty}: \((i) \Rightarrow (ii)\)]
We only need to show that the limit holds.

Given \(\overline{\mu} \in \cM(\Omega)\), let \(\overline{u}\) be the solution of the Dirichlet problem
\[
\left\{
\begin{alignedat}{2}
- \Delta \overline{u} & = \overline{\mu}	&& \quad \text{in \(\Omega\),}\\
\overline{u} & = 0	&& \quad \text{on \(\partial\Omega\).}
\end{alignedat}
\right.
\]
Thus,
\[
\Delta (u - \overline{u}) \ge \overline{\mu} - \mu
\]
in the sense of \((C_0^\infty(\overline\Omega))'\). 
Choosing \(\overline{\mu}\) such that
\[
\overline{\mu} - \mu \ge 0,
\]
it follows from the weak maximum principle (proposition~\ref{propositionWeakMaximumPrinciple})
that \(u \le \overline{u}\).
Assuming in addition that
\[
\overline{\mu} \ge 0,
\]
then by the weak maximum principle we also have \(\overline{u} \ge 0\). We deduce that
\[
u^+ = \max{\{u, 0\}} \le \overline{u}.
\]
Thus, for every \(\epsilon > 0\),
\[
0
\le
\int\limits_{\{x \in \Omega : d(x, \partial\Omega) < \epsilon\}} u^+ 
\le  \int\limits_{\{x \in \Omega : d(x, \partial\Omega) < \epsilon\}} \overline{u}.
\]
Since \(\overline{u}\) is the solution of a Dirichlet problem with zero boundary data, we deduce that (see proposition~\ref{propositionDirichletBoundaryCondition}),
\[
\lim_{\epsilon \to 0} \frac{1}{\epsilon}\int\limits_{\{x \in \Omega : d(x, \partial\Omega) < \epsilon\}} u^+ = 0.
\]
This concludes the proof.
\end{proof}

In order to prove the reverse implication, we need a local version of Stampacchia's regularity result:

\begin{lemma}
\label{lemmaLocalizationMeasure}
If \(u \in L\loc^1(\Omega)\) is such that \(\Delta u \in \cM\loc(\Omega)\), then for every \(\varphi \in C_c^\infty(\Omega)\), \(u \varphi \in W_0^{1, 1}(\Omega)\), \(\Delta(u \varphi) \in \cM(\Omega)\) and
\[
\Delta(u \varphi) = \varphi \Delta u + 2 \nabla u \cdot \nabla\varphi + u \Delta\varphi.
\]
\end{lemma}

This localization lemma is useful in the sense that it reduces the study of a function \(u \in L\loc^1(\Omega)\) such that \(\Delta u \in \cM\loc(\Omega)\) to another one where the function vanishes on the boundary, in which case we may use Stampacchia's regularity theory for the linear Dirichlet problem (proposition~\ref{prop3.1}).
In particular, we deduce that every function \(u \in L^1(\Omega)\) such that \(\Delta u \in \cM(\Omega)\) belongs to \(W\loc^{1, 1}(\Omega)\) and even to \(W\loc^{1, q}(\Omega)\) for every \(1 \le q < \frac{N}{N-1}\).

\begin{proof}[Proof of lemma~\ref{lemmaLocalizationMeasure}]
We first show that \(u \in W\loc^{1, 1}(\Omega)\). 
To this purpose, we may assume that \(\Delta u \in \cM(\Omega)\).
Let \(v\) be the solution of the linear Dirichlet problem
\[
\left\{
\begin{alignedat}{2}
- \Delta v & = \Delta u	&& \quad \text{in \(\Omega\),}\\
v & = 0 && \quad \text{on \(\partial\Omega\).}
\end{alignedat}
\right.
\]
Then, for every \(\zeta \in C_0^\infty(\overline{\Omega})\),
\[
\int\limits_\Omega (u+v) \Delta\zeta = 0.
\]
On the one hand, by the classical Weyl's lemma \cite{Wey:40}, \(u + v\) is a harmonic function,
whence \(u + v \in C^\infty(\Omega)\).
On the other hand, by Stampacchia's regularity theory (proposition~\ref{prop3.1}), \(v \in W_0^{1, 1}(\Omega)\). 
We deduce that \(u \in W\loc^{1, 1}(\Omega)\) as claimed.
In particular, for every \(\varphi \in C_c^\infty(\Omega)\), \(u \varphi \in W_0^{1, 1}(\Omega)\). 

We now prove the formula of \(\Delta(u \varphi)\).
For every \(\psi \in C_c^\infty(\Omega)\),
\[
\varphi \Delta\psi = \Delta(\varphi \psi) - 2 \nabla \varphi \cdot \nabla\psi - \psi \Delta\varphi.
\]
Since \(\Delta u \in \cM\loc(\Omega)\),
\[
\int\limits_\Omega u \varphi \Delta\psi 
= \int\limits_\Omega \varphi \psi \, \Delta u - 2 \int\limits_\Omega u \nabla \varphi \cdot \nabla\psi - \int\limits_\Omega u \psi \Delta\varphi.
\]
Since \(u \in W\loc^{1, 1}(\Omega)\), we also have
\[
\int\limits_\Omega u \nabla \varphi \cdot \nabla\psi
= - \int\limits_\Omega \nabla u \cdot \nabla \varphi \, \psi
- \int\limits_\Omega  u \Delta\varphi \, \psi.
\]
Combining both identities,
\[
\int\limits_\Omega u \varphi \Delta\psi 
= \int\limits_\Omega \psi \big(\varphi\Delta u + 2 \nabla u \cdot \nabla \varphi + u \Delta\varphi \big).
\]
We conclude that
\[
\Delta(u\varphi) = \varphi\Delta u + 2 \nabla u \cdot \nabla \varphi + u \Delta\varphi,
\]
in the sense of distributions in \(\Omega\).
Since the right-hand side belongs to \(\cM(\Omega)\), the conclusion follows.
\end{proof}

We also need the following version of Green's identity and the proof is based on approximation of \(u\) via convolution with smooth mollifiers:

\begin{lemma}
Let \(u \in W\loc^{1, 1}(\Omega)\) be such that \(\Delta u \in \cM\loc(\Omega)\).
For every smooth open set \(\omega \Subset \Omega\) and for every \(\zeta \in C_0^\infty(\overline\omega)\),
\[
\int\limits_\omega u \Delta\zeta 
= \int\limits_\omega \zeta \Delta u
+ \int\limits_{\partial\omega} u \frac{\partial\zeta}{\partial n}.
\]
\end{lemma}

We may now complete the proof of the proposition:

\begin{proof}[Proof of proposition~\ref{propositionDistributionC0Infty}: \((ii) \Rightarrow (i)\)]
Let \(\zeta \in C_0^\infty(\overline\Omega)\) be such that \(\zeta \ge 0\) in \(\Omega\). For every regular value \(t > 0\) of \(\zeta\), 
the set \(\{\zeta > t\}\) is smooth and \(\{\zeta > t\} \Subset \Omega\). 
By the localization lemma above, \(u \in W\loc^{1, 1}(\Omega)\).
By Green's identity we have for every \(\psi \in C_0^\infty(\overline{\{\zeta > t\}})\),
\[
- \int\limits_{\{\zeta > t\}} u \Delta\psi 
\le \int\limits_{\{\zeta > t\}}  \psi \dif\mu - \int\limits_{\partial\{\zeta > t\}} u \frac{\partial\psi}{\partial n}.
\]
Applying the inequality above with test function \(\psi = \zeta - t\), we have
\[
- \int\limits_{\{\zeta > t\}} u \Delta\zeta 
\le \int\limits_{\{\zeta > t\}}  (\zeta - t) \dif\mu - \int\limits_{\partial\{\zeta > t\}} u \frac{\partial\zeta}{\partial n}.
\]

For every \(x \in \partial\{\zeta > t\}\),
\[
\frac{\partial\zeta}{\partial n}(x) = - \abs{\nabla\zeta(x)}.
\]
Moreover, since \(t\) is a regular value of \(\zeta\), \(\partial\{\zeta > t\} = \{\zeta = t\}\). Thus,
\[
- \int\limits_{\partial\{\zeta > t\}} u \frac{\partial\zeta}{\partial n}
 = \int\limits_{\{\zeta = t\}} u \abs{\nabla\zeta}
 \le \int\limits_{\{\zeta = t\}} u^+ \abs{\nabla\zeta},
\]
which gives
\[
- \int\limits_{\{\zeta > t\}} u \Delta\zeta 
\le \int\limits_{\{\zeta > t\}}  (\zeta - t) \dif\mu + \int\limits_{\{\zeta = t\}} u^+ \abs{\nabla\zeta}.
\]

\begin{Claim}
If \(\zeta \in C_0^\infty(\overline\Omega)\) is such that \(\zeta > 0\) in \(\Omega\) and \(\frac{\partial\zeta}{\partial n} < 0\) on \(\partial\Omega\), then
\[
\liminf_{t \to 0}{\int\limits_{\{\zeta = t\}} u^+ \abs{\nabla\zeta}} = 0.
\]
\end{Claim}

We assume the claim and conclude the proof of the proposition.

Assume that \(\zeta\) satisfies the assumptions of the claim.
By Sard's lemma we may choose a sequence \((t_n)_{n \in \N}\) of regular values of \(\zeta\) converging to \(0\) such that
\[
\lim_{n \to \infty}{\int\limits_{\{\zeta = t_n\}} u^+ \abs{\nabla\zeta}} = 0.
\]
Since for every \(n \in \N\),
\[
\int\limits_{\{\zeta > t_n\}} u \Delta\zeta \ge - \int\limits_{\{\zeta = t_n\}} u^+ \abs{\nabla\zeta} + \int\limits_{\{\zeta > t_n\}}  (\zeta - t_n) \dif\mu,
\]
by the dominated convergence theorem we deduce that
\[
\int\limits_{\Omega} u \Delta\zeta \ge \int\limits_{\Omega}  \zeta \dif\mu.
\]
This estimate was established assuming that \(\zeta > 0\) in \(\Omega\) and that \(\frac{\partial\zeta}{\partial n} < 0\) in \(\partial\Omega\). 
For an arbitrary \(\zeta \in C_0^\infty(\overline\Omega)\) such that \(\zeta \ge 0\) in \(\overline\Omega\), we proceed as follows.
Let \(\overline{\zeta} \in C_0^\infty(\overline{\Omega})\) be a function satisfying the assumptions of the claim.
For every \(\epsilon > 0\), the function \(\zeta + \epsilon\overline{\zeta}\) also satisfies the assumptions of the claim.
Thus,
\[
\int\limits_{\Omega} u \Delta(\zeta + \epsilon\overline{\zeta})
 \ge \int\limits_{\Omega}  (\zeta + \epsilon\overline{\zeta}) \dif\mu.
\]
Letting \(\epsilon\) tend to \(0\), the conclusion follows.

We are left to establish the claim:

\begin{proof}[Proof of the claim]
By the co-area formula \cite{Giu:84}*{theorem~1.23}, for every \(\alpha > 0\) and for every \(\epsilon > 0\),
\[
\int_0^{\alpha\epsilon} \bigg(\int\limits_{\{\zeta = t\}} u^+ \abs{\nabla\zeta} \bigg) \dif t
= \int\limits_{\{0 < \zeta < \alpha\epsilon\}} u^+ \abs{\nabla\zeta}^{2}.
\]
Since \(\zeta > 0\) in \(\Omega\) and \(\frac{\partial\zeta}{\partial n} < 0\) in \(\partial\Omega\), there exists \(\alpha > 0\) such that for every \(\epsilon > 0\),
\[
\{0 < \zeta < \alpha\epsilon\} \subset \{x \in \Omega : d(x, \partial\Omega) < \epsilon\}.
\]
Thus, 
\[
\int_0^{\alpha\epsilon} \bigg( \int\limits_{\{\zeta = t\}} u^+ \abs{\nabla\zeta} \bigg) \dif t
\le C \int\limits_{\{x \in \Omega : d(x, \partial\Omega) < \epsilon\}} u^+.
\]
Hence, by assumption on the integral in the right-hand side,
\[
\lim_{\epsilon \to 0}{\frac{1}{\epsilon}\int_0^{\alpha\epsilon}
\bigg( \int\limits_{\{\zeta = t\}} u^+ \abs{\nabla\zeta}} \bigg) \dif t = 0,
\]
from which the claim follows.
\end{proof}

The proof of the proposition is complete.
\end{proof}

There is an alternative proof of the implication \((ii) \Rightarrow (i)\) of proposition~\ref{propositionDistributionC0Infty}. 
First, one shows that there exist two nonpositive measures $\nu \in \cM(\partial\Omega)$ and $\lambda \in \cM\loc(\Omega)$ such that for every \(\zeta \in C_0^\infty(\overline\Omega)\) with \(\zeta \ge 0\) in \(\Omega\),
$$
- \int\limits_\Omega u \Delta\zeta = \int\limits_\Omega \zeta \dif\mu + \int\limits_\Omega \zeta \dif \lambda - \int\limits_{\partial\Omega} \frac{\partial\zeta}{\partial n} \dif\nu.
$$
The existence of $\lambda$ is 
a consequence of a property of positive distributions (see lemma~\ref{lemmaPositiveDistributions} below) and is fairly straightforward. 
The existence of $\mu$ is a consequence of the Herglotz theorem concerning positive
harmonic functions  \cite{Ver:04}*{theorem~2.13}.
One may then use the same strategy as in the proof of \cite{MarVer:98a}*{lemma~1.5}.

In the case of measure boundary trace, \(\nabla u \not\in L^1(\Omega)\), but some estimates of \(\nabla u\) in \(L^1\) spaces with weights are available~\cite{DiaRak:09}.


\section{Variants of Kato's inequality}

Kato's inequality \cite{Kat:72}*{lemma~A} can be seen as a replacement of the maximum principle for functions which need not be twice differentiable or at least that do not belong to the Hilbert space \(W^{1, 2}(\Omega)\).
The original motivation of Kato was to study properties of solutions of the Schrödinger equation which need not be variational.

This inequality can be stated as follows:

\begin{proposition}
If \(u \in L^1(\Omega)\) is such that \(\Delta u \in L^1(\Omega)\), then
\[
\Delta u^+ \ge \chi_{\{u > 0\}} \Delta u
\]
in the sense of distributions in \(\Omega\).
\end{proposition}

\begin{proof}[Sketch of the proof]
The first step of the proof of Kato's inequality relies on the observation that when the function \(u\) is smooth, then for every smooth function \(\Phi : \R \to \R\),
\[
\Delta \Phi(u) = \Phi'(u) \Delta u + \Phi''(u) \abs{\nabla u}^2
\]
and if in addition \(\Phi\) is \emph{convex}, then
\[
\Delta \Phi(u) \ge \Phi'(u) \Delta u.
\]
The next step consists in approximating \(u \in L^1(\Omega)\) by smooth functions --- for instance via convolution --- in which case we may apply the inequality above. In the last step, we approximate the function \(t \in \R \mapsto t^+\) by convex smooth functions.
\end{proof}

The main ingredient in the proof of the classical maximum principle relies on the fact that if \(u\) attains its maximum at some point \(x \in \Omega\), then \(\Delta u(x) \le 0\). 
Similarly,  if \(u\) attains its minimum at \(x \in \Omega\), then \(\Delta u(x) \ge 0\).

This information is encoded in Kato's inequality as follows.
Given a twice differentiable function \(u : \Omega \to \R\), we have for every \(x \in \Omega\),
\[
\Delta u^+(x) =
\begin{cases}
\Delta u(x)	& \text{if \(u(x) > 0\),}\\
0				& \text{if \(u(x) < 0\).}
\end{cases}
\]
If \(u(x) = 0\), then \(\Delta u^+(x)\) need not be well defined in the classical sense.
Since in this case \(x\) is a minimum point for \(u^+\), which is a nonnegative function, we could formally say that \(\Delta u^+(x) \ge 0\). 
We arrive at the following pointwise statement of  Kato's inequality,
\[
\Delta u^+(x) \ge \chi_{\{u > 0\}}(x) \Delta u(x).
\]
This geometric interpretation of Kato's inequality has been pointed out by Yanyan Lin  in a personal communication.
Connections with the problem of removable singularities have been investigated in \cites{DavPon:03,DavPon:07}.

\medskip

The assumption \(\Delta u \in L^1(\Omega)\) is not stable in the context of Kato's inequality. 
Indeed, \(\Delta u^+\) is a locally finite measure in \(\Omega\) which in general is not an \(L^1\) function.
For instance, for every \(a < c < b\), the function \(u : (a, b) \to \R\) defined by \(u(x) = x - c\) satisfies \(\Delta u^+ = (u^+)'' = \delta_c\) in the sense of distributions in \((a, b)\).
The property that \(\Delta u^+\) is a locally finite measure is a more general fact concerning positive distributions (see lemma~\ref{lemmaPositiveDistributions}).

Kato's inequality is usually applied to a solution of some equation, in which case the assumption \(\Delta u \in L^1(\Omega)\) is probably enough.
However, when dealing with subsolutions, or in our case where the equation itself involves measures, such assumption on \(\Delta u\) is restrictive.

In order to have a counterpart of Kato's inequality when \(\Delta u\) is a measure, one should first understand the meaning of the product \(\chi_{\{u > 0\}} \Delta u\).
This is a delicate issue.
Indeed, if \(u\) and \(v\) are two functions which coincide almost everywhere in \(\Omega\), then \(\Delta u\) and \(\Delta v\) coincide as distributions but \(\chi_{\{u > 0\}} \Delta u\) and \(\chi_{\{v > 0\}} \Delta v\) need not coincide.

\medskip
We propose three ways of handling the product \(\chi_{\{u > 0\}} \Delta u\) in Kato's inequality when \(\Delta u\) need not be an \(L^1\) function.

The \emph{first strategy} consists in eliminating the function \(\chi_{\{u > 0\}}\) as follows.
If \(u\) is smooth, then
\[
\chi_{\{u > 0\}} \Delta u \ge \min{\{\Delta u, 0\}}.
\]
Using this trick in the proof of Kato's inequality we get the following proposition:

\begin{proposition}
\label{propositionKatoAncona}
If \(u \in L^1(\Omega)\) is such that \(\Delta u \in \cM(\Omega)\), then
\[
\Delta u^+ \ge \min{\{\Delta u, 0\}}
\]
in the sense of distributions in \(\Omega\).
\end{proposition}

The \emph{second strategy} consists in replacing \(\Delta u\) by some \(L^1\) function smaller than \(\Delta u\). Assuming that \(u\) is smooth and that \(\Delta u \ge f\) for some \(f \in L^1(\Omega)\), then
\[
\chi_{\{u > 0\}} \Delta u \ge \chi_{\{u > 0\}} f.
\]
Using this inequality in the proof of Kato's inequality we get the following proposition:

\begin{proposition}
\label{propositionKatoVariant}
Let \(f \in L^1(\Omega)\). If \(u \in L^1(\Omega)\) is such that
\[
\Delta u \ge f
\]
in the sense of distributions in \(\Omega\),
then
\[
\Delta u^+ \ge \chi_{\{u > 0\}} f
\]
in the sense of distributions in \(\Omega\).
\end{proposition}

These two statements or a suitable combination of them suffice for most purposes.

\medskip
Kato's inequality implies the following generalization for the maximum of two functions:

\begin{proposition}
If \(u_1, u_2 \in L^1(\Omega)\) are such that \(\Delta u_1, \Delta u_2 \in L^1(\Omega)\), then
\[
\Delta \max{\{u_1, u_2\}} \ge \chi_{\{u_1 > u_2\}} \Delta u_1 + \chi_{\{u_1 < u_2\}} \Delta u_2 + \chi_{\{u_1 = u_2\}}  \frac{\Delta u_1 + \Delta u_2}{2}
\]
in the sense of distributions in \(\Omega\).
\end{proposition}

This more general version actually follows from the standard Kato's inequality by noticing that for every \(a_1, a_2 \in \R\),
\[
\max{\{a_1, a_2\}} = \frac{a_1 + a_2 + \abs{a_1 - a_2}}{2}
\]
and
\[
\abs{a_1 - a_2} = (a_1 - a_2)^+ + (a_2 - a_1)^+.
\]

We shall need the following variant in the spirit of proposition~\ref{propositionKatoVariant}:

\begin{proposition}
\label{propositionKatoLocalMax}
Let \(f_1, f_2 \in L^1(\Omega)\). If \(u_1, u_2 \in L^1(\Omega)\) are such that
\[
\Delta u_1 \ge f_1 \quad \text{and} \quad \Delta u_2 \ge f_2
\]
in the sense of distributions in \(\Omega\),
then
\[
\Delta \max{\{u_1, u_2\}} \ge \chi_{\{u_1 > u_2\}} f_1 + \chi_{\{u_2 > u_1\}} f_2 + \chi_{\{u_1 = u_2\}}  \frac{f_1+f_2}{2}
\]
in the sense of distributions in \(\Omega\).
\end{proposition}

We cannot deduce this statement from proposition~\ref{propositionKatoVariant} since the functions \(u_1 - u_2\) and \(u_2 - u_1\) need not satisfy the assumptions of proposition~\ref{propositionKatoVariant} unless \(\Delta(u_1 - u_2)\) is an \(L^1\) function.
The proof of the proposition relies on a smooth approximation of the function \(\max\).

\begin{proof}
Given a smooth function \(\Phi : \R \to \R\), let \(w : \Omega \to \R\) be the function defined by
\[
w = \frac{u_1 + u_2 + \Phi(u_1 - u_2)}{2}.
\]
Assume that \(u_1\) and \(u_2\) are smooth. If \(\Phi\) is convex, then
\[
\Delta w \ge \frac{1 + \Phi'(u_1 - u_2)}{2} \Delta u_1 + \frac{1 - \Phi'(u_1 - u_2)}{2} \Delta u_2.
\]
If for every \(t \in \R\), \(-1 \le \Phi'(t) \le 1\), then the coefficients of \(\Delta u_1\) and \(\Delta u_2\) are nonnegative. Thus, by assumption,
\[
\Delta w \ge \frac{1 + \Phi'(u_1 - u_2)}{2} f_1 + \frac{1 - \Phi'(u_1 - u_2)}{2} f_2.
\]

Let \(\omega \Subset \Omega\) be an open set and let \(\rho_\epsilon\) be a smooth mollifier such that \(\omega + \supp{\rho_\epsilon} \subset \Omega\). We have
\[
\Delta (\rho_\epsilon * u_1) \ge \rho_\epsilon * f_1
\quad
\text{and}
\quad
\Delta (\rho_\epsilon * u_2) \ge \rho_\epsilon * f_2
\]
pointwisely in \(\omega\). Let \(w_\epsilon : \omega \to \R\) be the function defined by
\[
w_\epsilon = \frac{\rho_\epsilon * u_1 + \rho_\epsilon * u_2 + \Phi(\rho_\epsilon * u_1 - \rho_\epsilon * u_2)}{2}.
\]
Thus,
\[
\Delta w_\epsilon \ge \frac{1 + \Phi'(\rho_\epsilon * u_1 - \rho_\epsilon * u_2)}{2} f_1 + \frac{1 - \Phi'(\rho_\epsilon * u_1 - \rho_\epsilon * u_2)}{2} f_2
\]
pointwisely in \(\omega\). Letting \(\epsilon\) tend to \(0\), we deduce that
\[
\Delta w \ge \frac{1 + \Phi'(u_1 - u_2)}{2} f_1 + \frac{1 - \Phi'(u_1 - u_2)}{2} f_2
\]
in the sense of distributions in \(\omega\). Since \(\omega \Subset \Omega\) is arbitrary, this inequality holds in the sense of distributions in \(\Omega\).

In order to get the conclusion, we consider an approximation of the absolute value by smooth convex functions. More precisely, let \((\Phi_n)_{n \in \N}\) be a sequence of smooth convex functions such that
\begin{enumerate}[\((a)\)]
\item for every \(t \in \R\), \(-1 \le \Phi_n(t) \le 1\),
\item for every \(t < 0\), \(\lim\limits_{n \to \infty}{\Phi_n'(t)} = -1\),
\item for every \(t > 0\), \(\lim\limits_{n \to \infty}{\Phi_n'(t)} = 1\),
\item \(\lim\limits_{n \to \infty}{\Phi_n'(0)} = 0\).
\end{enumerate}
The conclusion follows by applying the dominated convergence theorem.
\end{proof}

We now return to our discussion about the product \(\chi_{\{u > 0\}} \Delta u\) that appears in Kato's inequality. 
The \emph{third strategy} consists in giving a meaning to \(\chi_{\{u > 0\}} \Delta u\) by choosing a good representative of \(u\). 

We first observe that since \(\Delta u\) is a finite measure, \(u\) has a quasicontinuous representative \(\Quasicontinuous{u}\) with respect to the \(W^{1, 2}\) capacity and \(\Quasicontinuous{u}\) is unique up to sets of zero \(W^{1, 2}\) capacity (see proposition~\ref{propositionExistenceQuasicontinuousRepresentativeLaplacian}).

Next, every measure $\mu$ has a unique Lebesgue decomposition as a sum of two measures,
\[
\mu = \mu\ld + \mu\lc,
\]
where the measure \(\mu\ld\) is a diffuse measure with respect to the capacity \(\capt_{W^{1, 2}}\), meaning that for any Borel set $A \subset \Omega$ such that $\capt_{W^{1, 2}}{(A)} = 0$, 
\[
\mu\ld(A) = 0,
\]
and \(\mu\lc\) is a measure concentrated in some Borel set \(E \subset \Omega\) of zero \(W^{1, 2}\) capacity, meaning that 
\[
\abs{\mu\lc}(\Omega \backslash E) = 0.
\]
The proof of such decomposition in terms of the \(W^{1, 2}\) capacity is the same as for the classical decomposition in terms of the Lebesgue measure, in which case one measure is absolutely continuous with respect to Lebesgue measure and the other measure is singular with respect to Lebesgue measure  \cite{BreMarPon:07}*{lemma~4.A.1}.

Using the notation above, we have the following counterpart of Kato's inequality due to Brezis and Ponce~\cite{BrePon:04}*{theorem~1.1} (see also \cites{Olo:06,DalMurOrsPri:99}):

\begin{proposition}
\label{propositionKatoDiffuse}
If \(u \in L^1(\Omega)\) is such that \(\Delta u \in \cM(\Omega)\), then \(\Delta u^+ \in \cM\loc(\Omega)\) and the diffuse part of \(\Delta u^+\) with respect to the \(W^ {1, 2}\) capacity satisfies
\[
(\Delta u^+)\ld \ge \chi_{\{\Quasicontinuous{u} > 0\}} (\Delta u)\ld.
\]
\end{proposition}

The inequality above should be understood in the sense of measures, meaning that for every Borel set \(A \subset \Omega\),
\[
(\Delta u^+)\ld(A) \ge (\Delta u)\ld(A \cap \{\Quasicontinuous{u} > 0\}).
\]
Thus, for every \(\varphi \in C_c^\infty(\Omega)\) such that \(\varphi \ge 0\) in \(\Omega\),
\[
\int\limits_\Omega \varphi \, (\Delta u^+)\ld \ge \int\limits_{\{u > 0\}} \varphi \, (\Delta u)\ld.
\]

The proof of this version of Kato's inequality requires some additional work and relies on a characterization of diffuse measures as elements of \(L^1(\Omega) + (W_0^{1, 2}(\Omega))'\) due to Boccardo, Orsina and Gallouët~\cite{BocGalOrs:96} (see corollary~\ref{corollaryDecompositionBoccardoGallouetOrsina}~\((\ref{itemDecompositionBoccardoGallouetOrsina})\) below).

Compared to Kato's inequality, there is a piece of the inequality missing which concerns the measure \((\Delta u^+)\lc\). 
We shall return to this issue after discussing the inverse maximum principle (see corollary~\ref{corollaryKatoConcentrated}).

\medskip
We deduce the following consequence for the maximum of two functions:

\begin{corollary}
If \(u_1, u_2 \in L^1(\Omega)\) are such that \(\Delta u_1, \Delta u_2 \in \cM(\Omega)\), then \(\Delta \max{\{u_1, u_2\}} \in \cM\loc(\Omega)\) and the diffuse part of \(\Delta \max{\{u_1, u_2\}}\) with respect to the \(W^ {1, 2}\) capacity satisfies
\begin{multline*}
(\Delta \max{\{u_1, u_2\}})\ld\\
\ge \chi_{\{\Quasicontinuous u_1 > \Quasicontinuous u_2\}} (\Delta u_1)\ld + \chi_{\{\Quasicontinuous u_1 < \Quasicontinuous u_2\}} (\Delta u_2)\ld + \chi_{\{\Quasicontinuous u_1 = \Quasicontinuous u_2\}}  \frac{(\Delta u_1)\ld + (\Delta u_2)\ld}{2}.
\end{multline*}
\end{corollary}

We now return to the issue of what it can be said about \(\Delta u^+\).
The classical result for positive distributions asserts that \(\Delta u^+\) is a locally finite measure:

\begin{lemma}
\label{lemmaPositiveDistributions}
Let \(\mu \in \cM\loc(\Omega)\). If \(F : C_c^\infty(\Omega) \to \R\) is a linear functional such that \(F \ge \mu\) in the sense of distributions in \(\Omega\), then there exists a unique \(\lambda \in \cM\loc(\Omega)\) such that \(F = \lambda\) in the sense of distributions in \(\Omega\).
\end{lemma}

\begin{proof}
We may assume that \(\mu = 0\).
Given a nonnegative function \(\varphi \in C_c^\infty(\Omega)\), for every \(\psi \in C^\infty(\overline\Omega)\) the function \((\norm{\psi}_{L^\infty(\Omega)} - \psi)\varphi\) is an admissible function in the inequality \(F \ge 0\). 
By linearity of \(F\) we have
\[
F(\psi\varphi) \le F(\varphi) \norm{\psi}_{L^\infty(\Omega)}.
\]
Since this property holds for every \(\psi \in C^\infty(\overline\Omega)\), we have
\[
\abs{F(\psi\varphi)} \le F(\varphi) \norm{\psi}_{L^\infty(\Omega)}.
\]

Given an open subset \(\omega \Subset \Omega\), we choose \(\varphi \in C_c^\infty(\Omega)\) such that \(\varphi = 1\) in \(\omega\). 
For every \(\psi \in C_c^\infty(\omega)\), we then have \(\psi \varphi = \psi\) and in addition \(\norm{\psi}_{L^\infty(\Omega)} = \norm{\psi}_{L^\infty(\omega)}\). Therefore,
\[
\abs{F(\psi)} \le F(\varphi) \norm{\psi}_{L^\infty(\omega)}.
\]
Thus, the linear functional \(F\) has a unique extension as a continuous linear function on the Banach space \(C_0(\overline\omega)\) of continuous functions vanishing on \(\partial\omega\). 
By the Riesz representation theorem \cite{Fol:99}*{theorem~7.17}, there exists a unique finite measure \(\lambda_\omega\) in \(\omega\) such that for every \(\psi \in C_c^\infty(\omega)\),
\[
F(\psi) = \int\limits_\omega \psi \dif\lambda_\omega.
\]
Using a partition of unit of \(\Omega\) the conclusion follows.
\end{proof}

It is natural to ask whether under the assumptions of Kato's inequality, \(\Delta u^+\) is a finite measure.
The answer is negative even if \(u\) is harmonic in \(\Omega\) and has a continuous extension to \(\overline\Omega\)  \cite{BrePon:08}*{proposition~A.1}.
However, the answer is affirmative if \(u\) satisfies the linear Dirichlet problem
\[
\left\{
\begin{alignedat}{2}
- \Delta u & = \mu	\quad && \text{in \(\Omega\),}\\
u & = 0 	\quad && \text{on \(\partial\Omega\),}
\end{alignedat}
\right.
\]
in which case we have the estimate
\[
\norm{\Delta u^+}_{\cM(\Omega)} \le \norm{\Delta u}_{\cM(\Omega)}.
\]
More generally, the following is true:

\begin{proposition}
\label{propositionKatoEstimate}
Let \(\mu \in \cM(\Omega)\).
If \(u\) is the solution of the linear Dirichlet problem with datum \(\mu\), then for every convex Lipschitz continuous function \(\Phi : \R \to \R\), we have \(\Delta \Phi(u) \in \cM(\Omega)\) and
\[
\norm{\Delta \Phi(u)}_{\cM(\Omega)} \le 2C \norm{\Delta u}_{\cM(\Omega)},
\]
where \(C > 0\) is any constant satisfying the Lipschitz condition of \(\Phi\).
\end{proposition}

This proposition does not appear in \cite{BrePon:08}, but all the ingredients are proved in that paper.
The result relies on the estimate
\begin{equation}
\label{equationBrezisPonce}
\norm{\Delta (u - a)^+}_{\cM(\Omega)} \le 2 \norm{\Delta u}_{\cM(\Omega)}
\end{equation}
valid for every \(a \in \R\). 
The constant \(2\) is needed here because if \(a > 0\), then the function \((u - a)^+\) is compactly supported in \(\Omega\), which doubles the total mass of its Laplacian.
This inequality is obtained by combining \cite{BrePon:08}*{theorem~1.1} and \cite{BrePon:08}*{theorem~1.2}.

\begin{proof}[Proof of proposition~\ref{propositionKatoEstimate}]
We first consider the case where \(\Phi\) is convex and piecewise affine.

Let us temporarily assume that \(\Phi\) is monotone, say nondecreasing, and \(\Phi'(t) = 0\) for every \(t \le a_1\).
In this case, there exists finitely many numbers \(a_1 < \dotsc < a_\ell\) in \(\R\),  \(\alpha_1, \dotsc, \alpha_\ell\) in \([0, +\infty)\) and \(\beta\) in \(\R\) such that for every \(t \in \R\),
\[
\Phi(t) = \beta + \sum_{i=1}^\ell \alpha_i(t - a_i)^+.
\]
Thus, \(\Delta \Phi(u) \in \cM(\Omega)\), and by the triangle inequality and  inequality~\eqref{equationBrezisPonce} above,
\[
\norm{\Delta\Phi(u)}_{\cM(\Omega)} 
\le \sum_{i=1}^\ell \alpha_i \norm{\Delta(u - a_i)^+}_{\cM(\Omega)}  
\le 2 \sum_{i=1}^\ell \alpha_i \norm{\Delta u}_{\cM(\Omega)} .
\]
Note that in this case, the Lipschitz constant of \(\Phi\) is given by \(\sum\limits_{i=1}^\ell \alpha_i\).

When \(\Phi\) is convex and piecewise affine but not necessarily monotone, we take 
\[
\underline{c} = \inf{\Phi'} 
\quad \text{and} \quad
\overline{c} = \sup{\Phi'}. 
\]
There exist convex and piecewise affine functions \(\Phi_1 : \R \to \R\) and \(\Phi_2 : \R \to \R\) such that 
\begin{enumerate}[\((a)\)]
\item for every \(t \in \R\),
\[
\Phi(t) = \Phi_1(t) + \Phi_2(t) + \frac{ \overline{c} + \underline{c}}{2} \, t,
\]
\item \(\Phi_1\) is nonincreasing and \(\Phi_2\) is nondecreasing,
\item there exists \(a_0 \in \R\) such that for every \(t \ge a_0\), 
\[
\Phi_1'(t) = 0
\]
and for every \(t \le a_0\), 
\[
\Phi_2'(t) = 0.
\]
\end{enumerate}

The functions \(\Phi_1\) and \(\Phi_2\) have Lipschitz constant at most equal to \(\frac{ \overline{c} - \underline{c}}{2}\).
In view of the previous case, we have for \(j \in \{1, 2\}\), \(\Delta\Phi_j(u) \in \cM(\Omega)\) and
\[
\norm{\Delta\Phi_j(u)}_{\cM(\Omega)} \le (\overline{c} - \underline{c}) \norm{\Delta u}_{\cM(\Omega)}.
\]
Thus, \(\Delta\Phi(u) \in \cM(\Omega)\) and by the triangle inequality,
\[
\begin{split}
\norm{\Delta\Phi(u)}_{\cM(\Omega)} 
& \le 2 (\overline{c} - \underline{c}) \norm{\Delta u}_{\cM(\Omega)} + \frac{ \abs{\overline{c} + \underline{c}}}{2} \norm{\Delta u}_{\cM(\Omega)}\\
& \le 2 \max{\{\abs{\overline{c}}, \abs{\underline{c}}\}} \norm{\Delta u}_{\cM(\Omega)}.
\end{split}
\]
This implies the result when \(\Phi\) is convex and piecewise affine.
The case of a convex Lipschitz continuous function \(\Phi\) follows by approximating \(\Phi\) by a sequence of convex and piecewise affine functions \((\Phi_n)_{n \in \N}\) converging uniformly to \(\Phi\) on bounded subsets of \(\R\).
\end{proof}

In contrast of what happens to the composition of Lipschitz continuous functions with functions in the Sobolev spaces \(W^{1, p}\) \citelist{\cite{BocMur:82}*{theorem~4.2} \cite{AmbDal:90}*{theorem~2.1}}, it is not possible to have an explicit formula of \(\Delta\Phi(u)\) in terms of \(\Delta u\) and \(\nabla u\), even when \(\Delta u\) is a smooth function, since \(\Delta\Phi(u)\) need not be a function. 
This raises the following question:

\begin{openproblem}
Let \(u \in L^1(\Omega)\) and let  \(\Phi : \R \to \R\) be a convex Lipschitz continuous function.
If \(\Delta u \in L^1(\Omega)\) and  \(\Delta \Phi(u) \in L^1(\Omega)\), is it true that
\[
\Delta\Phi(u) = \Phi'(u) \Delta u + \Phi''(u) \abs{\nabla u}^2
\]
pointwisely in \(\Omega\)?
\end{openproblem} 

The answer is affirmative if \(u \in W^{2, 1}(\Omega) \cap L^\infty(\Omega)\) and \(\Phi'\) is Lipschitz continuous by the property of composition of functions is Sobolev spaces; see also \cite{DalMurOrsPri:99}.
A delicate point is already to show that \(\Phi''(u)\) is well-defined.

A related question consists in showing that if \(\Delta u \in L^1(\Omega)\), then \(
\Delta u = 0\) almost everywhere in the set \(\{\nabla u = 0\}\).
This result does not seem to be explicitly proved in the literature; see however \cite{Haj:96}.

\medskip
The use of Kato's inequality in applications strongly rely on the linear structure of the Laplacian. 
An alternative approach to quasilinear equations --- in particular involving the \(p\)-Laplacian --- is to use suitable truncates of the solution as test functions.
This is a delicate issue and has been investigated in \cites{BenBocGalGarPieVaz:95,DalMurOrsPri:99}.


\section{Inverse maximum principle}

Maximum principles are based on the idea that if \(\Delta u\) is nonpositive, then the function \(u\) itself should be nonnegative, assuming that the right boundary conditions are satisfied.
The inverse maximum principle we present in this section goes the other way around in the sense that if \(u\) is nonnegative, then there is part of \(\Delta u\) which must be nonpositive. The values of \(u\) on the boundary are irrelevant in this case.

The following statement of the inverse maximum principle is due to Dupaigne and Ponce~\cite{DupPon:04}*{theorem~3} (see also \cite{DalDal:99}*{lemma~3.5}):

\begin{proposition}
\label{propositionInverseMaximumPrinciple}
Let $u \in L^1(\Omega)$ be such that \(\Delta u \in \cM(\Omega)\). If $u \ge 0$ in \(\Omega\), then the concentrated part of \(\Delta u\) with respect to the \(W^{1, 2}\) capacity satisfies
\[
(\Delta u)\lc \le 0.
\]
\end{proposition}

Under the assumptions of the inverse maximum principle, it is not possible to say anything about the diffuse part of \(\Delta u\) with respect to the \(W^{1, 2}\) capacity.

The inverse maximum principle can be better understood by keeping in mind a well-known fact in potential theory which says that measures which are concentrated on sets of zero \(W^{1, 2}\) capacity --- polar sets in the language of potential theory --- generate unbounded potentials. 
A precise statement may be found for instance in \cite{Hel:69}*{theorem~7.35}.
Whether the potential achieves the value \(+\infty\) or \(-\infty\) dictates the sign of the measure.

As an example in dimension \(2\), let \(u \in L^1(\Omega)\) be such that \(\Delta u = \alpha \delta_a\) for some point \(a \in \Omega\) and for some \(\alpha \in \R_*\).
In a neighborhood of \(a\), the function \(u\) behaves like a multiple of the fundamental solution,
\[
u(x) \sim \frac{\alpha}{2\pi} \log{\abs{x - a}}.
\]
In this case, \(\alpha \delta_a\) is a concentrated measure in dimension \(2\) and if \(u \ge 0\), then \(\alpha \le 0\).
This is a typical behavior one should expect from the inverse maximum principle.

\medskip
We first establish a couple ingredients that will be needed in the proof of the inverse maximum principle.

\begin{lemma}
\label{lemmaMeasureDistribution}
If \(\mu, \nu \in \cM(\Omega)\), then \(\mu \le \nu\) in the sense of measures if and only if \(\mu \le \nu\) in the sense of distributions.
\end{lemma}

The inequality \(\mu \le \nu\) in the sense of measures means that for every Borel set \(A \subset \Omega\),
\[
\mu(A) \le \nu(A),
\]
while in the sense of distributions,  \(\mu \le \nu\) means that for every \(\varphi \in C_c^\infty(\Omega)\) such that \(\varphi \ge 0\) in \(\overline{\Omega}\),
\[
\int\limits_\Omega \varphi \dif\mu \le \int\limits_{\Omega} \varphi \dif\nu.
\]

\begin{proof}
If \(\mu \le \nu\) in the sense of measures, then for every nonnegative simple measurable function \(f : \Omega \to \R\),
\[
\int\limits_\Omega f \dif \mu \le \int\limits_\Omega f \dif \nu.
\]
Given a nonnegative function \(\varphi \in C_c^\infty(\Omega)\), we take a nondecreasing sequence of nonnegative simple measurable functions converging pointwisely to \(\varphi\) and the direct implication follows.

Conversely, assume that \(\mu \le \nu\) in the sense of distributions. For every nonnegative function \(\varphi \in C_c^\infty(\Omega)\) we have
\[
\int\limits_\Omega \varphi \dif\mu \le \int\limits_{\Omega} \varphi \dif\nu.
\]
Given a compact set \(K \subset \Omega\), take a sequence \((\varphi_n)_{n \in \N}\) of nonnegative functions in \(C_c^\infty(\Omega)\) such that
\begin{enumerate}[\((a)\)]
\item \((\varphi_n)_{n \in \N}\) is bounded in \(L^\infty(\Omega)\),
\item \((\varphi_n)_{n \in \N}\) converges pointwisely to \(0\) in \(\Omega\setminus K\),
\item \((\varphi_n)_{n \in \N}\) converges pointwisely to \(1\) in \(K\).
\end{enumerate}
By the dominated convergence theorem we get
\[
\mu(K) = \int\limits_K \dif\mu \le \int\limits_K \dif\nu = \nu(K).
\]
Since \(K\) is an arbitrary compact subset of \(\Omega\), by inner regularity of finite measures \cite{Fol:99}*{theorem~7.8} the inequality holds for every Borel subset of \(\Omega\).
\end{proof}

The following result is due to Grun-Rehomme~\cite{Gru:77}:

\begin{lemma}
\label{lemmaGrunRehommeW12}
If \(u \in W^{1, 2}(\Omega)\) is such that \(\Delta u \in \cM(\Omega)\), then \(\Delta u\) is a diffuse measure with respect to the \(W^{1, 2}\) capacity.
\end{lemma}

\begin{proof}
For every \(\varphi \in C_c^\infty(\Omega)\),
\[
\int\limits_\Omega \varphi \Delta u 
= -\int\limits_\Omega \nabla \varphi \cdot \nabla u.
\]
Given a compact set \(K \subset \Omega\) such that \(\capt_{W^{1, 2}}{(K)} = 0\), let \((\varphi_n)_{n \in \N}\) be a sequence in \( C_c^\infty(\Omega)\) such that 
\begin{enumerate}[\((a)\)]
\item \((\varphi_n)_{n \in \N}\) converges to \(0\) in \(W^{1, 2}(\Omega)\)
\item for every \(n \in \N\), \(\varphi_n = 1\) in \(K\),
\item \((\varphi_n)_{n \in \N}\) converges pointwisely to \(0\) in \(\Omega \setminus K\),
\item for every \(n \in \N\), \(0 \le \varphi_n \le 1\) in \(\Omega\).
\end{enumerate}
On the one hand, by the dominated convergence theorem,
\[
\lim_{n \to \infty}{\int\limits_\Omega \varphi_n \Delta u} = \int\limits_K \Delta u = (\Delta u)(K).
\]
On the other hand, by convergence of the sequence \((\varphi_n)_{n \in \N}\) in \(W^{1, 2}(\Omega)\),
\[
\lim_{n \to \infty}{\int\limits_\Omega \nabla \varphi_n \cdot \nabla u} = 0.
\]
Therefore,
\[
(\Delta u)(K) = 0.
\]
Since \(K\) is an arbitrary compact subset of \(\Omega\), by inner regularity of finite Borel measures, the equality holds for every Borel subset of \(\Omega\) with zero \(W^{1, 2}\) capacity.
\end{proof}

\begin{proof}[Proof of proposition~\ref{propositionInverseMaximumPrinciple}]
We first establish the inverse maximum principle when \(u \in W_0^{1, 1}(\Omega)\). 

For every \(\kappa > 0\), by the interpolation inequality (lemma~\ref{lemmaInterpolationLinfty}) we have \(T_\kappa(u) \in W_0^{1, 2}(\Omega)\). 
Since \(u \ge 0\), 
\[
T_\kappa(u) = \kappa - (\kappa - u)^+ .
\] 
Thus, by Kato's inequality (proposition~\ref{propositionKatoAncona}) applied to the function \(\kappa - u\),
\[
\Delta T_\kappa(u) \le \max{\{\Delta u, 0\}} = (\Delta u)^+
\]
in the sense of distributions in \(\Omega\), whence in the sense of measures.

Given a Borel set \(E \subset \Omega\) such that \(\capt_{W^{1, 2}}{(E)} = 0\), by Grun-Rehomme's lemma, the measure \(\Delta T_\kappa(u)\) in the left-hand side does not charge \(E\).
In particular,
\[
\Delta T_\kappa(u) \le 0
\]
in the sense of measures in \(E\). We also have
\[
\Delta T_\kappa(u) \le (\Delta u)^+
\]
in the sense of measures in \(\Omega \setminus E\).
Using the additivity of measures, we may combine both inequalities and deduce that
\[
\Delta T_\kappa(u) \le \chi_{\Omega \setminus E} (\Delta u)^+
\]
in the sense of measures in \(\Omega\). 

Thus, for every \(\varphi \in C_c^\infty(\Omega)\) such that \(\varphi \ge 0\) in \(\Omega\),
\[
\int\limits_\Omega T_\kappa(u) \Delta\varphi \le \int\limits_{\Omega \setminus E} \varphi \, (\Delta u)^+.
\]
Let \(\kappa\) tend to infinity in this inequality. By the dominated convergence theorem,
\[
\int\limits_\Omega u \Delta \varphi = \int\limits_{\Omega \setminus E} \varphi \, (\Delta u)^+.
\]
In other words,
\[
\Delta u \le \chi_{\Omega \setminus E} (\Delta u)^+
\]
in the sense of distributions in \(\Omega\), whence in the sense of measures. In particular, 
\[
(\Delta u)(E) \le (\Delta u)^+ (E \setminus E) = 0.
\]
Since this property holds for every Borel set \(E \subset \Omega\) with zero \(W^{1, 2}\) capacity, we conclude that \((\Delta u)\lc \le 0\).

In order to consider the case of a function \(u \in L^1(\Omega)\) which need not belong to \(W_0^{1, 1}(\Omega)\), we may proceed as follows.
For every nonnegative function \(\varphi \in C_c^\infty(\Omega)\), by the localization lemma (lemma~\ref{lemmaLocalizationMeasure}) the function \(u \varphi\) satisfies the assumptions of the inverse maximum principle and in addition belongs to \(W_0^{1, 1}(\Omega)\). 
Thus, 
\[
(\Delta (u\varphi))\lc \le 0.
\] 
Given a subdomain \(\omega \Subset \Omega\), we may take \(\varphi = 1\) in \(\omega\), whence 
\[
\Delta (u\varphi) = \Delta u
\]
in \(\omega\) and this implies that \((\Delta u)\lc \le 0\) in \(\omega\). 
Since \(\omega\) is arbitrary the conclusion follows.
\end{proof}

We may now return to Kato's inequality when \(\Delta u\) is a measure to explain what happens to the concentrated part of the measure \(\Delta u^+\) with respect to the \(W^{1, 2}\) capacity.

\begin{corollary}
\label{corollaryKatoConcentrated}
If \(u \in L^1(\Omega)\) is such that \(\Delta u \in \cM(\Omega)\), then \(\Delta u^+ \in \cM\loc(\Omega)\) and the concentrated part of \(\Delta u^+\) with respect to the \(W^{1, 2}\) capacity satisfies
\[
(\Delta u^+)\lc = \min{\{(\Delta u)\lc, 0\}}.
\]
\end{corollary}

\begin{proof}
By Kato's inequality (proposition~\ref{propositionKatoAncona}),
\[
\Delta u^+ \ge \min{\{\Delta u, 0\}}
\]
in the sense of distributions, whence in the sense of measures. 
Comparing the concentrated parts on both sides, we get
\[
(\Delta u^+)\lc \ge \min{\{(\Delta u)\lc, 0\}}.
\]
Since \(u^+ \ge 0\) and \(u^+ \ge u\), by the inverse maximum principle we have
\[
(\Delta u^+)\lc \le 0
\quad \text{and} \quad
(\Delta u^+)\lc \le (\Delta u)\lc.
\]
Thus,
\[
(\Delta u^+)\lc \le \min{\{(\Delta u)\lc, 0\}}.
\]
Therefore, equality of measures must hold.
\end{proof}

Since
\[
\Delta u^+ = (\Delta u^+)\ld + (\Delta u^+)\lc,
\]
we may assemble the information given by proposition~\ref{propositionKatoDiffuse} and the corollary above and write
\[
\Delta u^+ \ge \chi_{\{u > 0\}} (\Delta u)\ld + \min{\{(\Delta u)\lc, 0\}}. 
\]
Thus, for every \(\varphi \in C_c^\infty(\Omega)\) such that \(\varphi \ge 0\) in \(\Omega\),
\[
\int\limits_\Omega  u^+ \Delta\varphi \ge \int\limits_{\{u > 0\}} \varphi \, (\Delta u)\ld + \int\limits_\Omega \varphi \min{\{(\Delta u)\lc, 0\}}. 
\]

\medskip
We also have the following counterpart for the maximum of two functions:

\begin{corollary}
If \(u_1, u_2 \in L^1(\Omega)\) are such that \(\Delta u_1, \Delta u_2 \in \cM(\Omega)\), then \(\Delta \max{\{u_1, u_2\}} \in \cM\loc(\Omega)\) and the concentrated part of \(\Delta \max{\{u_1, u_2\}}\) with respect to the \(W^{1, 2}\) capacity satisfies
\[
(\Delta \max{\{u_1, u_2\}})\lc = \min{\{(\Delta u_1)\lc, (\Delta u_2)\lc\}}.
\]
\end{corollary}


\section{Strong maximum principle}

The strong maximum principle asserts that if $\Omega \subset \R^N$ is a connected open set and if $u : \Omega \to \R$ is a nonnegative smooth function such that
\[
-\Delta u \geq 0,
\]
then either $u = 0$ in \(\Omega\) or $u>0$ in $\Omega$. 
Another formulation of the same fact says that if for some point $a \in \Omega$, $u(a)=0$, then  $u = 0$ in $\Omega$. 

By the Harnack inequality~\cite{Sta:65}*{corollaire~8.1}, the same conclusion holds when the Laplace operator $-\Delta$ is replaced by $-\Delta + V$ where $V \in L^p(\Omega)$ for some $p>\frac{N}{2}$.
The conclusion fails without this assumption on \(V\). 
For instance, given \(a \in \Omega\) the function \(u : \Omega \to \R\) defined by $u(x)=|x - a|^2$ satisfies the equation
\[
- \Delta u + Vu = 0
\] 
with \(V(x) = \frac{2N}{|x - a|^2}\) but in this case \(V \not\in  L^{N/2}(\Omega)\).

Under a weaker integrability condition on $V$, if the function $u$ vanishes on a \emph{larger} set, then one would still hope to conclude that $u = 0$ in \(\Omega\). 
This type of result was obtained by B{\'e}nilan and Brezis~\cite{BenBre:04}*{theorem~C.1} in the case where $V \in L^1(\Omega)$ and $\supp{u}$ is a compact subset of $\Omega$. 
Their strong maximum principle has been further extended by Ancona~\cite{Anc:79}*{theorem~9}:

\begin{proposition}
\label{propositionStrongMaximumPrinciple}
Let $\Omega \subset \R^N$ be a connected open set and let \(V \in L^1(\Omega)\). 
If \(u \in L^1(\Omega) \cap C(\Omega)\) is a nonnegative function such
\[
\Delta u \le Vu
\]
in the sense of distributions in \(\Omega\) and if \(u\) vanishes in a set of positive \(W^{1, 2}\) capacity, then \(u = 0\) in \(\Omega\).
\end{proposition}

Ancona's proof of the strong maximum principle above relies on tools from potential theory. 
We present an alternative proof due to Brezis and Ponce~\cite{BrePon:03} which is based on the following estimate:

\begin{lemma}
\label{lemmaEstimateBrezisPonce}
Let \(V \in L^1(\Omega)\) and let \(u \in L^1(\Omega)\) be a nonnegative function such that \(Vu \in L^1(\Omega)\). If
\[
\Delta u \le Vu
\]
in the sense of distributions in \(\Omega\), then \(\log{(1 + u)} \in W\loc^{1, 2}(\Omega)\) and for every \(\varphi \in C_c^\infty(\Omega)\),
\[
\int\limits_\Omega \abs{\nabla\log(1 + u)}^2 \varphi^2 \le C \int\limits_\Omega (V^+ \varphi^2 + \abs{D\varphi}^2).
\]
\end{lemma}

\begin{proof}
We first assume that \(u \in W\loc^{1, 2}(\Omega)\).
For every \(\varphi \in C_c^\infty(\Omega)\) such that \(\varphi \ge 0\) in \(\Omega\),
\[
- \int\limits_\Omega \nabla u \cdot \nabla \varphi = \int\limits_\Omega (\Delta u) \varphi \le \int\limits_\Omega V u \varphi.
\]
By an approximation argument on the test function, we deduce that for every function \(v \in W_0^{1, 2}(\Omega) \cap L^\infty(\Omega)\) with compact support in \(\Omega\) and such that \(v \ge 0\) in \(\Omega\),
\[
- \int\limits_\Omega \nabla u \cdot \nabla v \le \int\limits_\Omega V u v.
\]

In view of the pointwise identity,
\[
\nabla\log{(1 + u)} = \frac{\nabla u}{1 + u},
\]
we will apply this estimate with \(v = \varphi^2/(1 + u)\), where \(\varphi \in C_c^\infty(\Omega)\).

\medskip
For every \(\varphi \in C_c^\infty(\Omega)\),
\begin{equation*}
\frac{\nabla u}{(1 + u)^2} \varphi^2 = - \nabla \left( \frac{\varphi^2}{1 + u} \right) + 2 \frac{\varphi}{1 + u} \nabla\varphi.
\end{equation*}
Taking the scalar product against \(\nabla u\) on both sides, we get
\[
\frac{\abs{\nabla u}^2}{(1 + u)^2} \varphi^2 = 
- \nabla u \cdot \nabla \left( \frac{\varphi^2}{1 + u} \right) + 2 \frac{(\nabla u) \varphi}{1 + u} \cdot \nabla\varphi.
\]
For every \(\epsilon > 0\), there exists \(\NewConstant > 0\) depending on \(\epsilon\) such that
\[
\left|\frac{(\nabla u) \varphi}{1 + u} \cdot \nabla\varphi \right| \le \epsilon \frac{\abs{\nabla u}^2}{(1 + u)^2} \varphi^2 + \SameConstant \abs{\nabla \varphi}^2.
\]
Thus,
\[
(1 - 2\epsilon) \frac{\abs{\nabla u}^2}{(1 + u)^2} \varphi^2 
\le - \nabla u \cdot \nabla \left( \frac{\varphi^2}{1 + u} \right) + 2 \SameConstant \int\limits_\Omega \abs{\nabla \varphi}^2.
\]
Choosing \(0 < \epsilon < \frac{1}{2}\), we get
\[
\frac{\abs{\nabla u}^2}{(1 + u)^2} \varphi^2 
\le - \Constant \nabla u \cdot \nabla \left( \frac{\varphi^2}{1 + u} \right) + \Constant \abs{\nabla \varphi}^2.
\]
Therefore,
\[
\abs{\nabla\log(1 + u)}^2 \varphi^2 
\le - C_2 \nabla u \cdot \nabla \left( \frac{\varphi^2}{1 + u} \right) + \SameConstant \abs{\nabla \varphi}^2.
\]

By the estimate in the beginning of the proof with  \(v = \varphi^2/(1 + u)\),
\[
- \int\limits_\Omega \nabla u \cdot \nabla \left( \frac{\varphi^2}{1 + u} \right) 
\le \int\limits_\Omega V u \frac{\varphi^2}{1 + u} \le \int\limits_\Omega V^+ \varphi^2.
\]
We conclude that
\[
\int\limits_\Omega \abs{\nabla\log(1 + u)}^2 \varphi^2 
\le C_2 \int\limits_\Omega V^+ \varphi^2 + \SameConstant \int\limits_\Omega \abs{\nabla \varphi}^2.
\]
Since \(0 \le \log(1 + u) \le u\) in \(\Omega\), we deduce that \(\log{(1 + u)} \in W\loc^{1, 2}(\Omega)\) under the assumption that \(u \in W\loc^{1, 2}(\Omega)\).

We now explain how to drop the additional assumption on \(u\).
Since \(\Delta u \in \cM\loc(\Omega)\), by the localization lemma (lemma~\ref{lemmaLocalizationMeasure}) and by the interpolation inequality (lemma~\ref{lemmaInterpolationLinfty}), for every \(\kappa > 0\) the truncated function \(T_\kappa(u)\) belongs to \(W\loc^{1, 2}(\Omega)\). 
Since \(u \ge 0\),
\[
T_\kappa(u) = \kappa - (\kappa - u)^+.
\]
By assumption on \(\Delta u\), we also have
\[
\Delta (\kappa - u) \ge - Vu
\]
in the sense of distributions in \(\Omega\). 
Thus, by Kato's inequality (proposition~\ref{propositionKatoVariant}) applied to the function \(\kappa - u\),
\[
\Delta (\kappa - u)^+ \ge - \chi_{\{u < \kappa\}} Vu
\]
in the sense of distributions in \(\Omega\), whence
\[
\Delta T_\kappa(u) \le \chi_{\{u < \kappa\}} Vu \le V^+ T_\kappa(u)
\]
in the sense of distributions in \(\Omega\). 
Thus, \(T_\kappa(u)\) also satisfies the assumptions of the lemma and, in addition, \(T_\kappa(u) \in W\loc^{1, 2}(\Omega)\). 
We deduce that for every \(\varphi \in C_c^\infty(\Omega)\),
\[
\int\limits_\Omega \abs{\nabla \log(1 + T_\kappa(u))}^2 \varphi^2 \le C \int\limits_\Omega (V^+ \varphi^2 + \abs{D\varphi}^2).
\]
Letting \(\kappa\) tend to infinity, we have the conclusion.
\end{proof}

We need the following variant of the Poincaré inequality for functions vanishing on a set of positive \(W^{1, 2}\) capacity:

\begin{lemma}
Let \(A \subset \Omega\) be a Borel set such that \(\capt_{W^{1, 2}}(A) > 0\). If \(\Omega\) is connected, then there exists \(C > 0\) such that for every \(u \in W^{1, 2}(\Omega) \cap C(\Omega)\), if \(u\) vanishes in \(A\), then 
\[
\norm{u}_{L^2(\Omega)} \le C \norm{Du}_{L^2(\Omega)}.
\]
\end{lemma}

\begin{proof}
Arguing by contradiction, if the inequality is not true, then there exists a sequence \((u_n)_{n \in \N}\) is \(W^{1, 2}(\Omega) \cap C(\Omega)\) such that for every \(n \in \N\), \({u_n} = 0\) in \(A\) and the sequence \((u_n)_{n \in \N}\) converges strongly in \(W^{1, 2}(\Omega)\) to some function \(u\) such that \(\norm{u}_{L^2(\Omega)} > 0\) and \(\norm{Du}_{L^2(\Omega)} = 0\). Since \(\Omega\) is connected, \(u\) must be a constant function. Passing to a subsequence if necessary, we may assume that \((u_n)_{n \in \N}\) converges pointwisely to \(u\) in \(\Omega\) except in a set of zero \(W^{1, 2}\) capacity (see proposition~\ref{propositionPointwiseConvergence} and remark~\ref{remarkPointwiseConvergence}). Thus, \(u\) is a constant function such that \(u = 0\) in \(A\), whence \(u = 0\) in \(\Omega\). This contradicts the fact that \(\norm{u}_{L^2(\Omega)} > 0\).
\end{proof}

\begin{proof}[Proof of proposition~\ref{propositionStrongMaximumPrinciple}]
For every \(\delta > 0\), the function \(u/\delta\) satisfies the assumptions of the previous lemma.
Thus, for every \(\omega \Subset \Omega\), there exists a constant \(\NewConstant > 0\) such that
\[
\int\limits_\omega \left|{\nabla \log\left(1 + \frac{u}{\delta}\right)}\right|^2 \le \SameConstant.
\]
Since the function \(\log\big(1 + \frac{{u}}{\delta}\big)\) vanishes in a set of positive \(W^{1, 2}\) capacity in \(\Omega\), we may choose a connected open set \(\omega \Subset \Omega\) having the same property. Thus, by the Poincaré inequality above,
\[
\int\limits_\omega \left|{\log\left(1 + \frac{u}{\delta}\right)}\right|^2 \le \Constant \int\limits_\omega \left|{\nabla \log\left(1 + \frac{u}{\delta}\right)}\right|^2 \le \Constant.
\]
For every \(t > 0\), by the Chebyshev inequality we then have
\[
\meas{\omega \cap \{u > t\}} \left|{\log\left(1 + \frac{t}{\delta}\right)}\right|^2 \le \SameConstant.
\]
Letting \(\delta\) tend to \(0\), we deduce that for every \(t > 0\), \(\meas{\omega \cap \{u > t\}} = 0\). Thus, \(u = 0\) in \(\omega\), from which the conclusion follows.
\end{proof}

In proposition~\ref{propositionStrongMaximumPrinciple}, the assumption that \(u\) is continuous can be relaxed. Indeed, the result still holds if \(u \in L^1(\Omega)\) is a nonnegative function such that \(Vu \in L^1(\Omega)\). 
In this case, the assumption that \(u\) vanishes in a set of zero capacity has to be understood in terms of the quasicontinuous representative \(\Quasicontinuous{u}\) of \(u\).

There are interesting questions related to proposition~\ref{propositionStrongMaximumPrinciple}:

\begin{openproblem}
Let \(V: \Omega \to \R\) be a measurable function. Assume that $u \in L^1(\Omega) \cap C(\Omega)$ is a nonnegative function such \(Vu \in L^1(\Omega)\) and
\[
\Delta u \le Vu
\]
in the sense of distributions in \(\Omega\).
If $u$ has compact support in \(\Omega\), can one replace the assumption $V \in
L^1(\Omega)$ by a weaker condition --- for example $V^{1/2} \in L^1(\Omega)$ or $V^{1/2} \in L^p(\Omega)$ for some $p > 1$ --- and still conclude that $u = 0$ in $\Omega$?
\end{openproblem}

We cannot hope to go below the assumption $V^{1/2} \in L^1(\Omega)$. 
For instance the $C^2$ function $u : \R^N \to \R$ given by
$$
u(x) =
\begin{cases}
\big(1 - |x|^2 \big)^4  & \text{if } |x| \leq 1,\\
0                       & \text{if } |x| > 1,
\end{cases},
$$ 
satisfies 
\[
- \Delta u + V u = 0,
\]
where $V : \R^N \to \R$ is a function such that for $|x| \lesssim 1$,
\[
V(x) \sim \frac{1}{(1-|x|)^2}.
\] 
In this case, $V^\alpha \in L^1(\R^N)$ for every $0 < \alpha < 1/2$, but $V^{1/2} \not\in L^1(\R^N)$.

\begin{openproblem}
Let $\Omega \subset \R^N$ be a connected open set and let \(V \in L^q(\Omega)\) for some \(1 \le q < +\infty\). 
Assume that $u \in L^1(\Omega) \cap C(\Omega)$ is a nonnegative function such
\[
\Delta u \le Vu
\]
in the sense of distributions in \(\Omega\).
If \(u\) vanishes in a set of positive \(W^{1, 2q}\) capacity --- or positive \(W^{2, q}\) capacity ---, is it true that \(u = 0\) in \(\Omega\)?
\end{openproblem}

Proposition~\ref{propositionStrongMaximumPrinciple} shows that the answer is affirmative when $q = 1$. The \(W^{1, 2}\) capacity can also be seen as a limit of the \(W^{2, q}\) capacities as \(q\) tends to \(1\) (see lemma~\ref{lemmaCapacityEquivalence}).
The answer is also affirmative when $q > \frac{N}{2}$ by the strong maximum principle mentioned above: if $q > \frac{N}{2}$ and $a \in \R^N$ is any point, then
$\capt_{W^{1, 2q}}{(\{a\})} > 0$ and $\capt_{W^{2, q}}{(\{a\})} > 0$.


\chapter{Extremal solutions}

We investigate the existence of a smallest and a largest solution of the nonlinear Dirichlet problem
\begin{equation*}
\left\{
\begin{alignedat}{2}
- \Delta u + g(u) & = \mu &&\quad \text{in } \Omega,\\
u & = 0 &&\quad \text{on } \partial\Omega.
\end{alignedat}
\right.
\end{equation*}

We construct natural candidates for such solutions via the Perron method by taking the largest subsolution \(\overline{u}\) and the smallest supersolution \(\underline{u}\) of the Dirichlet problem.
Using the method of sub and supersolutions we show that \(\overline{u}\) and \(\underline{u}\) are indeed solutions of the nonlinear Dirichlet problem.

\section{Perron method}

The Perron method~\cite{Per:1923} is based on the idea of thinking of a solution as the largest possible subsolution.
This is related to Poincaré's \emph{balayage method}~\cites{Poi:1887,Poi:1890} which gives a way of achieving a solution as the limit of a nondecreasing sequence of subsolutions.

The notion of subsolution we consider is based on Littman-Stampacchia-Weinberger's notion of weak solution:

\begin{definition}
\label{definitionSubsolution}
Let \(g : \R \to \R\) be a continuous function and let \(\mu \in \cM(\Omega)\).
A function $u : \Omega \to \R$ is a \emph{subsolution} of the nonlinear Dirichlet problem if 
\begin{enumerate}[\((i)\)]
\item \(u \in L^1(\Omega)\),
\item $g(u) \in L^1(\Omega)$,
\item \(- \Delta u + g(u) \le \mu\) in the sense of \((C_0^\infty(\overline\Omega))'\).
\end{enumerate}
\end{definition}


Thus, by a subsolution \(u\) of the nonlinear Dirichlet problem we mean that for every \(\zeta \in C_0^\infty(\overline\Omega)\) such that \(\zeta \ge 0\) in \(\overline\Omega\),
\begin{equation*}
- \int\limits_\Omega u \Delta \zeta + \int\limits_\Omega g(u) \zeta \le \int\limits_\Omega \zeta \dif \mu.
\end{equation*}
The boundary condition \(u \le 0\) on \(\partial\Omega\) is implicitly encoded in this weak formulation and the precise meaning of what we mean by this boundary condition is given by proposition~\ref{propositionDistributionC0Infty} above. 

The Perron method relies on the following proposition:

\begin{proposition}
\label{propositionPerronMethod}
Let \(g : \R \to \R\) be a continuous function and let \(\mu \in \cM(\Omega)\). Given \(w \in L^1(\Omega)\), if the nonlinear Dirichlet problem has a subsolution $\overline{v} \le w$ in \(\Omega\), then there exists \(\overline{u} \in L^1(\Omega)\) such that
	\begin{enumerate}[\((i)\)]
	\item \(\overline{u} \le w\) in \(\Omega\),
	\item for every subsolution \(v\) such that \(v \le w\) in \(\Omega\), \(v \le \overline{u}\) in \(\Omega\),
	\item there exists a nondecreasing sequence of subsolutions \((u_n)_{n \in \N}\) such that for every \(n \in \N\), \(u_n \le w\) in \(\Omega\), and \((u_n)_{n \in \N}\) converges to \(\overline{u}\) in \(L^1(\Omega)\).
	\end{enumerate}	 
\end{proposition}

Formally, we could say that for \(x \in \Omega\),
\[
\overline{u}(x) = \sup{\Big\{v(x)  :  \text{$v$ is a subsolution such that \( v \le w\)} \Big\}},
\]
but one should be careful about the meaning of the supremum in the right-hand side since its possibly taken over uncountable sets of subsolutions.

Without additional assumptions on the nonlinearity \(g\) and on the function \(w\), it need not be true that \(\overline{u}\) is a subsolution of the Dirichlet problem.
For instance, we could have \(g(\overline{u}) \not\in L^1(\Omega)\) and even when \(g(\overline{u}) \in L^1(\Omega)\), the weak formulation need not hold for every admissible test function.

\medskip

We first prove a property about the maximum of subsolutions.

\begin{lemma}
Let \(g : \R \to \R\) be a continuous function  and let \(\mu \in \cM(\Omega)\). 
If \(v_1\) and \(v_2\) are subsolutions of the nonlinear Dirichlet problem, then \(\max{\{v_1, v_2\}}\) is also a subsolution.
\end{lemma}

We only require the nonlinearity \(g\) to be continuous and the same proof still works for a nonlinearity which is given by a Carathéodory function \(g : \Omega \times \R \to \R\).

\begin{proof}
For every function \(w : \Omega \to \R\),
\[
\max{\{v_1, v_2\}} = \max{\{v_1 - w, v_2 - w\}} + w.
\]
Taking \(w\) to be the solution of the linear Dirichlet problem
\[
\left\{
\begin{alignedat}{2}
- \Delta w & = \mu	&& \quad \text{in \(\Omega\),}\\
w & = 0	&& \quad \text{on \(\partial\Omega\),}
\end{alignedat}
\right.
\]
we have
\[
\Delta(v_1 - w) \ge - g(v_1) \quad \text{and} \quad  \Delta(v_2 - w) \ge - g(v_2)
\]
in the sense of \((C_0^\infty(\overline\Omega))'\), whence in the sense of distributions in \(\Omega\).
Applying Kato's inequality (proposition~\ref{propositionKatoLocalMax}) to the functions \(v_1 - w\) and \(v_2 - w\), we have
\[
\Delta \max{\{v_1 - w, v_2 - w\}} \ge -g(\max{\{v_1, v_2\}})
\]
in the sense of distributions in \(\Omega\). Thus,
\[
- \Delta \max{\{v_1, v_2\}} + g(\max{\{v_1, v_2\}}) \le \mu
\]
in the sense of distributions in \(\Omega\).

In order to show that \(\max{\{v_1, v_2\}}\) is a subsolution of the nonlinear Dirichlet problem, we need to study the behavior of this function near the boundary of \(\Omega\).
Note that
\[
(\max{\{v_1, v_2\}})^+ \le (v_1)^+ + (v_2)^+.
\]
Thus, for every \(\epsilon > 0\),
\begin{multline*}
\frac{1}{\epsilon}\int\limits_{\{x \in \Omega : d(x, \partial\Omega) < \epsilon\}} (\max{\{v_1, v_2\}})^+\\
\le \frac{1}{\epsilon}\int\limits_{\{x \in \Omega : d(x, \partial\Omega) < \epsilon\}} (v_1)^+ + \frac{1}{\epsilon}\int\limits_{\{x \in \Omega : d(x, \partial\Omega) < \epsilon\}} (v_2)^+.
\end{multline*}
Since \(v_1\) and \(v_2\) are both subsolutions of the nonlinear Dirichlet problem, by the direct implication of proposition~\ref{propositionDistributionC0Infty} both terms in the right hand side of this estimate converge to \(0\) as \(\epsilon\) converges to \(0\). 
Thus,
\[
\lim_{\epsilon \to 0} \frac{1}{\epsilon}\int\limits_{\{x \in \Omega : d(x, \partial\Omega) < \epsilon\}} (\max{\{v_1, v_2\}})^+ = 0.
\]
Using the reverse implication of proposition~\ref{propositionDistributionC0Infty}, the conclusion follows.
\end{proof}

\begin{proof}[Proof of proposition~\ref{propositionPerronMethod}]
Let
$$
S = \sup{\bigg\{ \int\limits_\Omega v  :  \text{$v$ is a subsolution such that \(v \leq w\)} \bigg\}}.
$$
Since the supremum is taken over a nonnempty set and \(w \in L^1(\Omega)\), \(S\) is finite. 

Let \((v_n)_{n \in \N}\) be a sequence of subsolutions of the Dirichlet problem such that for every \(n \in \N\), \(v_n \le w\) in \(\Omega\), and
\[
\lim_{n \to \infty}{\int\limits_\Omega v_n} =  S.
\]
We construct a nondecreasing sequence of subsolutions \((u_n)_{n \in \N}\) having the same properties. 
We define this sequence by induction as follows: let \(u_0 = v_0\) and for every \(n \in \N_*\), 
\[
u_n = \max{\{u_{n-1}, v_n\}}.
\]
Then, for every \(n \in \N\), \(u_n \le w\) and by the previous lemma \(u_n\) is a subsolution of the Dirichlet problem.
Since
\[
\int\limits_\Omega v_n \le \int\limits_\Omega u_n \le S,
\]
we have
\begin{equation*}
\lim_{n \to \infty}{\int\limits_\Omega u_n} = S.
\end{equation*}

The pointwise limit \(\overline{u}\) of the sequence \((u_n)_{n \in \N}\) satisfies the required properties. 
First of all, \(\overline{u} \leq w\) in \(\Omega\) and by the monotone convergence theorem,
$$
S = \lim_{n \to \infty}{\int\limits_\Omega u_n} = \int\limits_\Omega \overline{u}.
$$

If \(v\) is a subsolution of the Dirichlet problem, then \((\max{\{u_n, v\}})_{n \in \N}\) is a nondecreasing sequence of subsolutions converging pointwisely to \(\max{\{\overline{u}, v\}}\). 
If in addition \(v \le w\), then
\[
\int\limits_\Omega u_n \le \int\limits_\Omega \max{\{u_n, v\}} \le S.
\]
By the monotone convergence theorem, we deduce that
\[
\int\limits_\Omega \overline{u} = S = \int\limits_\Omega \max{\{\overline{u}, v\}}.
\]
Since \(\overline{u} \le \max{\{\overline{u}, v\}}\) and both integrals coincide, we deduce that \(\overline{u} = \max{\{\overline{u}, v\}}\) almost everywhere in \(\Omega\).
Therefore, \(v \le \overline{u}\).
This concludes the proof of the proposition.
\end{proof}

By analogy with the definition of subsolution, we may introduce the notion of supersolution:

\begin{definition}
\label{definitionSupersolution}
Let \(g : \R \to \R\) be a continuous function and let \(\mu \in \cM(\Omega)\).
A function $u : \Omega \to \R$  is a \emph{supersolution} of the nonlinear Dirichlet problem if 
\begin{enumerate}[\((i)\)]
\item \(u \in L^1(\Omega)\),
\item $g(u) \in L^1(\Omega)$,
\item \(- \Delta u + g(u) \ge \mu\) in the sense of \((C_0^\infty(\overline\Omega))'\).
\end{enumerate}
\end{definition}

The counterpart of the Perron method for supersolutions is the following:

\begin{proposition}
Let \(g : \R \to \R\) be a continuous function and let \(\mu \in \cM(\Omega)\).
Given \(w \in L^1(\Omega)\), if the nonlinear Dirichlet problem has a supersolution $\underline{v} \ge w$ in \(\Omega\), then there exists \(\underline{u} \in L^1(\Omega)\) such that
	\begin{enumerate}[\((i)\)]
	\item \(\underline{u} \ge w\) in \(\Omega\),
	\item for every supersolution \(v\) such that \(v \ge w\) in \(\Omega\), \(v \ge \underline{u}\) in \(\Omega\),
	\item there exists a nonincreasing sequence of supersolutions \((u_n)_{n \in \N}\) such that for every \(n \in \N\), \(u_n \ge w\) in \(\Omega\), and \((u_n)_{n \in \N}\) converges to \(\underline{u}\) in \(L^1(\Omega)\).
	\end{enumerate}	
\end{proposition}

The proof in this case is based on a property about the minimum of supersolutions:

\begin{lemma}
Let \(g : \R \to \R\) be a continuous function  and let \(\mu \in \cM(\Omega)\). 
If \(v_1\) and \(v_2\) are supersolutions of the nonlinear Dirichlet problem, then \(\min{\{v_1, v_2\}}\) is also a supersolution.
\end{lemma}


\section{Method of sub and supersolutions}

The method of sub and supersolutions relies on the idea that between a subsolution and a supersolution of a nonlinear Dirichlet problem there should be a solution.

This method can be traced back to Picard~\cite{Pic:1890} --- using a monotone iteration scheme --- and Peano~\cite{Pea:1890}.
Perron~\cite{Per:1915} and Scorza Dragoni~\cite{Sco:1931}
explicitly used subsolutions and supersolutions to prove existence of solutions of first and second order ordinary differential equations. 
Concerning elliptic partial differential equations, Nagumo~\cite{Nag:54} proved the method of sub and supersolutions using Schauder's fixed point theorem and this is the strategy we adopt. 

We stick to the definition of subsolution and supersolution based on Littman-Stampacchia-Weinberger's weak solutions (definition~\ref{definitionSubsolution} and definition~\ref{definitionSupersolution}).
The main advantage is that it completely separates the issue of \emph{existence} of solutions from the question of \emph{regularity} of solutions. 

The method of sub and supersolution relies on the following proposition:

\begin{proposition}
\label{propositionMethodSubSuperSolutions}
Let \(g : \R \to \R\) be a continuous function satisfying the integrability condition and let \(\mu \in \cM(\Omega)\). 
If the nonlinear Dirichlet problem has a subsolution $\underline{v}$ and a supersolution $\overline{v}$ and if \(\underline{v} \leq \overline{v}\) in \(\Omega\), then there exists a solution $u$ such that \(\underline{v} \leq u \leq \overline{v}\) in \(\Omega\).
\end{proposition}

This formulation is due to Montenegro and Ponce~\cite{MonPon:07}*{theorem~1.1}.

\medskip

We recall the integrability condition:

\begin{definition}
Let \(g : \R \to \R\) be a continuous function.
We say that \(g\) satisfies the \emph{integrability condition} if
for every \(\underline w, w, \overline w \in L^1(\Omega)\) such that \(\underline w \le w \le \overline w\) in \(\Omega\), if \(g(\underline w) \in L^1(\Omega)\) and if \(g(\overline w) \in L^1(\Omega)\), then \(g(w) \in L^1(\Omega)\).
\end{definition}

In the proof of the method of sub and supersolution, the integrability condition is just what we need to \emph{define} an operator \(G : L^1(\Omega) \to L^1(\Omega)\) whose fixed points satisfy the nonlinear Dirichlet problem.
It is useful to have an equivalent formulation of the integrability condition, which insures the \emph{continuity} of \(G\):

\begin{lemma}
\label{lemmaIntegrabilityConditionDominatedConvergence}
Let \(g : \R \to \R\) be a continuous function.
Then, \(g\) satisfies the integrability condition if and only if
for every \(\underline w, \overline w \in L^1(\Omega)\) such that \(\underline w \le \overline w\) in \(\Omega\) and \(g(\underline w), g(\overline w)\in L^1(\Omega)\), there exists \(h \in L^1(\Omega)\) such that for every \(w \in L^1(\Omega)\), if  \(\underline w \le w \le \overline w\) in \(\Omega\), then \(\abs{g(w)} \le h\) in \(\Omega\).
\end{lemma}

\begin{proof}
The reverse implication is clear. 
In order to see why the integrability condition implies the condition above, for every \(x \in \Omega\) let \(\tilde w(x)\) be the smallest real number in the interval \([\underline{w}(x), \overline{w}(x)]\) such that
\[
\abs{g(\tilde w(x))} = \max_{t \in [\underline{w}(x), \overline{w}(x)]}{\abs{g(t)}}.
\]
In particular, if \(w : \Omega \to \R\) is a measurable function such that \(\underline w \le w \le \overline w\) in \(\Omega\), then \(\abs{g(w)} \le \abs{g(\tilde w)}\) in \(\Omega\).
One shows that the function \(w\) defined in this way is measurable \citelist{ \cite{KraLad:54} \cite{KraZabPusSob:76}*{lemma~1.2}}. 
Since \(\underline{w} \le \tilde w \le \overline{w}\), by the integrability condition, we have \(g(\tilde w) \in L^1(\Omega)\) and the condition above holds with \(h = \abs{g(\tilde w)}\).
\end{proof}

\begin{corollary}
\label{corollaryIntegrabilityConditionMonotoneConvergence}
Let \(g : \R \to \R\) be a continuous function satisfying the integrability condition.
If \((u_n)_{n \in \N}\) is a monotone sequence converging to \(u\) in \(L^1(\Omega)\) and if for every \(n \in \N\), \(g(u_n) \in L^1(\Omega)\), then \((g(u_n))_{n \in \N}\) converges to \(g(u)\) in \(L^1(\Omega)\) if and only if \(g(u) \in L^1(\Omega)\).
\end{corollary}

\begin{proof}
The direct implication is clear.
For the reverse implication we may assume that the sequence \((u_n)_{n \in \N}\) is nondecreasing. Thus, for every \(n \in \N\), \(u_0 \le u_n \le u\). By assumption, \(g(u_0) \in L^1(\Omega)\) and \(g(u) \in L^1(\Omega)\). By the previous lemma, there exists \(h \in L^1(\Omega)\) such that for every \(n \in \N\), \(\abs{g(u_n)} \le h\). The conclusion follows from the dominated convergence theorem.
\end{proof}

In the proof of the method of sub and supersolutions, we need a variant of Kato's inequality~\citelist{\cite{BreCazMarRam:96} \cite{BreMarPon:07}} which takes into account the behavior of the function on the boundary:

\begin{lemma}
\label{lemmaKatoBoundary}
Let \(f \in L^1(\Omega)\). If \(u \in L^1(\Omega)\) is such that
\[
\Delta u \ge f 
\]
in the sense of \((C_0^\infty(\overline\Omega))'\), then
\[
\Delta u^+ \ge \chi_{\{u > 0\}} f
\]
in the sense of \((C_0^\infty(\overline\Omega))'\).
\end{lemma}

This lemma appears in \cite{BreCazMarRam:96}*{lemma~2} when \(u\) is a solution of the Dirichlet problem.
We follow the strategy of the proof of \cite{BreMarPon:07}*{proposition~4.B.5}, where this result is proved in full generality.

\begin{proof}
Applying the direct implication of proposition~\ref{propositionDistributionC0Infty},
\[
\lim_{\epsilon \to 0} \frac{1}{\epsilon} \int\limits_{\{x \in \Omega : d(x, \partial\Omega) < \epsilon\}} u^+ = 0.
\]
By Kato's inequality (proposition~\ref{propositionKatoVariant}),
\[
\Delta u^+ \ge \chi_{\{u > 0\}} f
\]
in the sense of distributions in \(\Omega\).
The reverse implication of proposition~\ref{propositionDistributionC0Infty} gives the conclusion.
\end{proof}

The proof of proposition~\ref{propositionMethodSubSuperSolutions} is based on Schauder's fixed point theorem.
We first modify the nonlinearity \(g\) using the subsolution and the supersolution in such a way that a solution of the modified Dirichlet problem is a solution of the original problem.

\begin{proof}[Proof of proposition~\ref{propositionMethodSubSuperSolutions}]
Let \(\tilde g : \Omega \times \R \to \R\) be the function defined for \((x, t) \in \Omega \times \R\) by
\[
\tilde g(x,t) :=
\begin{cases}
g(\underline{v}(x)) & \text{if } t < \underline{v}(x),\\
g(t) & \text{if } \underline{v}(x) \leq t \leq \overline{v}(x),\\
g(\overline{v}(x)) & \text{if } t > \overline{v}(x).
\end{cases}
\]
Then, \(\tilde g : \Omega \times \R \to \R\) is a Carath\'eodory function and
by the integrability condition, for every \(v \in L^1(\Omega)\),
\[
\tilde g(\cdot, v) \in L^1(\Omega).
\]

\Newclaim
\begin{claim}
If $u$ satisfies the Dirichlet problem
\[
\left\{
\begin{alignedat}{2}
- \Delta u & = \mu -  \tilde g(\cdot, u) &&\quad \text{in } \Omega,\\
u & = 0 &&\quad \text{on } \partial\Omega,
\end{alignedat}
\right.
\]
then
\begin{equation*}
\underline{v} \leq u \leq \overline{v},
\end{equation*}
whence \(\tilde g(\cdot, u) = g(u)\) and \(u\) is a solution of the original Dirichlet problem.
\end{claim}

\begin{proof}[Proof of the claim]
We show that \(u \leq \overline{v}\) in \(\Omega\); the
proof of the inequality \(\underline{v} \leq u\) is similar.

For every \(\zeta \in C_0^\infty(\overline\Omega)\) such that \(\zeta \ge 0\) in \(\Omega\),
\[
\int\limits_\Omega (u - \overline{v}) \Delta\zeta 
\ge \int\limits_{\Omega} \big[ \tilde g(\cdot, u) - g(\overline{v}) \big] \zeta
= \int\limits_{\Omega} \chi_{\{u \le \overline{v}\}} \big[ g(u) - g(\overline{v}) \big] \zeta.
\]
Applying Kato's inequality up to the boundary (lemma~\ref{lemmaKatoBoundary}) to the function $u - \overline{v}$, we get for every \(\zeta \in C_0^\infty(\overline\Omega)\) such that \(\zeta \ge 0\) in \(\Omega\),
\[
\int\limits_\Omega (u - \overline{v})^+ \Delta\zeta \ge 0.
\]
Thus,
\[
\int\limits_\Omega (u - \overline{v})^+ \le 0.
\]
We deduce that \((u - \overline{v})^+ = 0\) in \(\Omega\), whence \(u \le \overline{v}\) in \(\Omega\).
\end{proof}

By the integrability condition, for every \(v \in L^1(\Omega)\) we have \(\tilde g(\cdot, v) \in L^1(\Omega)\). 
Thus, there exists a unique solution of the linear Dirichlet problem
\[
\left\{
\begin{alignedat}{2}
- \Delta u & = \mu - \tilde g(\cdot, v) &&\quad \text{in } \Omega,\\
u & = 0 &&\quad \text{on } \partial\Omega.
\end{alignedat}
\right.
\]
Let \(G : L^1(\Omega) \to L^1(\Omega)\) be the map defined for every \(v \in L^1(\Omega)\) by 
\[
G(v) = u,
\]
where \(u\) is the solution of the linear Dirichlet problem above.

\begin{claim}
The map \(G\) is continuous in \(L^1(\Omega)\). 
\end{claim}

\begin{proof}[Proof of the claim]
By the integrability condition and lemma~\ref{lemmaIntegrabilityConditionDominatedConvergence}, there exists \(h \in L^1(\Omega)\) such that for every \(v \in L^1(\Omega)\) and for every \(x \in \Omega\), 
\[
\abs{\tilde g(x, v(x))} \le h(x).
\]

Let \((v_n)_{n \in \N}\) be a sequence converging to some function \(v\) in \(L^1(\Omega)\).
For every \(n \in \N\) and for every \(x \in \Omega\), 
\[
\abs{\tilde g(x, v_n(x))} \le h(x).
\]
Thus, by the dominated convergence theorem, the sequence \((\tilde g(\cdot, v_n))_{n \in \N}\) converges to \(\tilde g(\cdot, v)\) in \(L^1(\Omega)\). 
By the linear \(L^1\) estimate (proposition~\ref{propositionExistenceLinearDirichletProblem}), the sequence \((G(v_n))_{n \in \N}\) converges to \(G(v)\) in \(L^1(\Omega)\). Therefore, \(G\) is continuous.
\end{proof}

\begin{claim}
The set \(G(L^1(\Omega))\) is bounded and relatively compact in \(L^1(\Omega)\).
\end{claim}

\begin{proof}[Proof of the claim]
Since for every \(v \in L^1(\Omega)\),
\[
\norm{\mu - \tilde g(\cdot, v)}_{\cM(\Omega)} \le \norm{\mu}_{\cM(\Omega)} + \norm{\tilde g(\cdot, v)}_{\cM(\Omega)} \le \norm{\mu}_{\cM(\Omega)} + \norm{h}_{L^1(\Omega)}.
\]
By the linear elliptic estimate (proposition~\ref{propositionExistenceLinearDirichletProblem}),
\[
\norm{G(v)}_{L^{1}(\Omega)} \le C \big( \norm{\mu}_{\cM(\Omega)} + \norm{h}_{L^1(\Omega)} \big).
\]
Thus, the set \(G(L^1(\Omega))\) is bounded in \(L^1(\Omega)\)
By Stampacchia's regularity theory (proposition~\ref{prop3.1}),
\[
\norm{DG(v)}_{L^{1}(\Omega)} \le C \big( \norm{\mu}_{\cM(\Omega)} + \norm{h}_{L^1(\Omega)} \big).
\]
Hence, by the Rellich-Kondrachov compactness theorem, the set \(G(L^1(\Omega))\) is relatively compact in \(L^1(\Omega)\). 
\end{proof}

It follows from Schauder's fixed point theorem that $G$ has a fixed point $u \in L^1(\Omega)$. 
By the first claim, \(u\) is a solution of the original Dirichlet problem.
The proof of the method of sub and supersolutions is complete.
\end{proof}

An inspection of the proof shows that the method of sub and supersolutions holds if the integrability condition on \(g\) is satisfied for functions between the subsolution \(\underline{v}\) and the supersolution \(\overline{v}\).
More precisely, for every \(v \in L^1(\Omega)\) such that \(\underline{v} \le v \le \overline{v}\), we have \(g(v) \in L^1(\Omega)\).

The conclusion of the proposition also holds if the nonlinearity \(g : \R \to \R\) is replaced by a Carathéodory function \(g : \Omega \times \R \to \R\)  \cite{MonPon:07}.

\medskip

From the method of sub and supersolutions we can infer a condition that guarantees that the functions \(\overline{u}\) and \(\underline{u}\) given by the Perron method satisfy the nonlinear Dirichlet problem.

\begin{corollary}
Let \(g : \R \to \R\) be a continuous function satisfying the integrability condition and let \(\mu \in \cM(\Omega)\).
If the nonlinear Dirichlet problem has a subsolution $\underline{v}$ and a supersolution $\overline{v}$ and if \(\underline{v} \leq \overline{v}\) in \(\Omega\), then
\begin{enumerate}[\((i)\)]
\item among all subsolutions which are less than or equal to \(\overline{v}\),
there exists a largest subsolution \(\overline{u}\),
\item \(\overline{u}\) is a solution of the nonlinear Dirichlet problem. 
\end{enumerate} 
\end{corollary}

Since every solution is a subsolution, we deduce form this corollary that \(\overline{u}\) is the largest solution which is less than or equal to \(\overline{v}\).

\begin{proof}
Let \(\overline{u}\) be the function given by the Perron method (proposition~\ref{propositionPerronMethod}) with \(w = \overline{v}\).

\Newclaim
\begin{claim}
\(\overline{u}\) is a subsolution of the Dirichlet problem.
\end{claim}

\begin{proof}[Proof of the claim]
Let \((u_n)_{n \in \N}\) be a nondecreasing sequence of subsolutions such that \(u_n \le \overline{v}\) in \(\Omega\) and \((u_n)_{n \in \N}\) converges in \(L^1(\Omega)\) to \(\overline{u}\). 
In particular, for every \(n \in \N\),
\[
u_0 \le u_n \le \overline{v}. 
\]
Since \(g(u_0) \in L^1(\Omega)\) and \(g(\overline{v}) \in L^1(\Omega)\), and since \(g\) satisfies the integrability condition, by lemma~\ref{lemmaIntegrabilityConditionDominatedConvergence} there exists \(h \in L^1(\Omega)\) such that for every \(n \in \N\),
\[
\abs{g(u_n)} \le h.
\]
By the dominated convergence theorem, \(g(\overline{u}) \in L^1(\Omega)\) and the sequence \((g(u_n))_{n \in \N}\) converges in \(L^1(\Omega)\) to \(g(\overline{u})\). 
Thus, \(\overline{u}\) is a subsolution of the nonlinear Dirichlet problem.
\end{proof}

\begin{claim}
\(\overline{u}\) is a solution of the Dirichlet problem.
\end{claim}

\begin{proof}[Proof of the claim]
Since \(\overline{u}\) is a subsolution and \(\overline{u} \le \overline{v}\), by the method of sub and supersolution, the Dirichlet problem has a solution \(u\) such that \(\overline{u} \le u \le \overline{v}\). In particular, \(u\) is a subsolution and since \(\overline{u}\) is the largest subsolution less than or equal to \(\overline{v}\), we have \(u \le \overline{u}\). Thus, \(u = \overline{u}\), whence \(\overline{u}\) is a solution of the Dirichlet problem.
\end{proof}

The proof of the corollary is complete.
\end{proof}

We state the following counterpart of the previous corollary for the largest supersolution:

\begin{corollary}
Let \(g : \R \to \R\) be a continuous function satisfying the integrability condition and let \(\mu \in \cM(\Omega)\).
If the nonlinear Dirichlet problem has a subsolution $\underline{v}$ and a supersolution $\overline{v}$ and if \(\underline{v} \leq \overline{v}\) in \(\Omega\), then 
\begin{enumerate}[\((i)\)]
\item among all supersolutions which are greater than or equal to \(\underline{v}\), there exists a smallest supersolution \(\underline{u}\),
\item \(\underline{u}\) is a solution of the nonlinear Dirichlet problem. 
\end{enumerate} 
\end{corollary}




\chapter{Polynomial nonlineartity}

We investigate the existence of solutions of the nonlinear Dirichlet problem
\[
\left\{
\begin{alignedat}{2}
- \Delta u + g(u) & = \mu &&\quad \text{in } \Omega,\\
u & = 0 &&\quad \text{on } \partial\Omega,
\end{alignedat}
\right.
\]
when the nonlinearity \(g : \R \to \R\) has polynomial growth: for every \(t \in \R\),
\[
\abs{g(t)} \le C(\abs{t}^p + 1),
\]
for some  \(C > 0\) and \(p > 0\).

\section{Subcritical growth}

The first result in this direction is due to Bénilan and Brezis~\cite{BenBre:04}*{theorem~A.1} and was announced in \citelist{\cite{Bre:80}*{lemme~3} \cite{Bre:82}*{théorème~2}}.

\begin{proposition}
\label{propositionExistenceBenilanBrezis}
Let \(g : \R \to \R\) be a continuous function satisfying the sign condition and such that
for every \(t \in \R\),
\[
\abs{g(t)} \le C(\abs{t}^p + 1),
\]
for some \(C > 0\) and \(p > 0\).
If \(p < \frac{N}{N-2}\), then for every \(\mu \in \cM(\Omega)\) the nonlinear Dirichlet problem has a solution.
\end{proposition}

We need the following absorption estimate:

\begin{lemma}
\label{lemmaEstimateAbsorption}
Let \(g : \R \to \R\) be a continuous function satisfying the sign condition.
For every \(\mu \in \cM(\Omega)\), if \(u\) is a solution of the nonlinear Dirichlet problem, then
\[
\norm{g(u)}_{L^1(\Omega)} \le \norm{\mu}_{\cM(\Omega)}.
\]
\end{lemma}

In this section we apply the lemma with \(\mu \in L^\infty(\Omega)\), 
in which case we know that there exists a solution \(u \in W_0^{1, 2}(\Omega) \cap L^\infty(\Omega)\) arising from the variational formulation (see proposition~\ref{propositionExistenceSolutionEulerLagrange}) and the Euler-Lagrange equation is satisfied with test functions in \(W_0^{1, 2}(\Omega)\).
For such functions \(u\), the absorption estimate has a direct proof.

Formally, the proof of the estimate consists in using \(\sgn{u}\) as a test function.
Since \(\sgn{u}\) is not a legitimate test function, we use a Lipschitz approximation of the sign function instead.
\begin{proof}[Proof of lemma~\ref{lemmaEstimateAbsorption} when \(u \in W_0^{1, 2}(\Omega)\)]
Given \(\epsilon > 0\), let \(S_\epsilon : \R \to \R\) be the function defined for \(t \in \R\) by
\[
S_\epsilon(t) =
\begin{cases}
-1			& \text{if \(t < -\epsilon\)},\\
t/\epsilon	& \text{if \(-\epsilon \le t \le \epsilon\)},\\
1			& \text{if \(t > \epsilon\)}.
\end{cases}
\]
Using \(S_\epsilon (u)\) as a test function in the Euler-Lagrange equation, we get
\[
\int\limits_\Omega \abs{\nabla u}^2 S_\epsilon'(u) + \int\limits_\Omega g(u) S_\epsilon(u) = \int\limits_\Omega S_\epsilon(u) \dif\mu.
\]
Since the function \(S_\epsilon\) is nondecreasing and bounded in absolute value by \(1\),
\[
\int\limits_\Omega g(u) S_\epsilon(u) \le \int\limits_\Omega \dif\abs{\mu} = \norm{\mu}_{\cM(\Omega)}.
\]
Letting \(\epsilon\) tend to zero, we get
\[
\int\limits_\Omega g(u) \sgn{u} \le \norm{\mu}_{\cM(\Omega)}.
\]
By the sign condition, \(g(u) \sgn{u} = \abs{g(u)}\) and the estimate follows.
\end{proof}

Since we will need the absorption estimate in full generality later, we prove it as stated:

\begin{proof}[Proof of lemma~\ref{lemmaEstimateAbsorption}]
Let \((f_n)_{n \in \N}\) be a sequence in \(L^\infty(\Omega)\) converging to \(g(u)\)
and let \((\mu_n)_{n \in \N}\) be a sequence in \(L^\infty(\Omega)\) such that \((\mu_n)_{n \in \N}\) is bounded in \(L^1(\Omega)\) and converges weakly to \(\mu\) in the sense of measures in \(\Omega\).

Given \(n \in \N\), let \(u_n\) be the solution of the linear Dirichlet problem
\[
\left\{
\begin{alignedat}{2}
- \Delta u_n & = \mu_n - f_n &&\quad \text{in } \Omega,\\
u_n & = 0 &&\quad \text{on } \partial\Omega.
\end{alignedat}
\right.
\]
Then, \(u_n \in W_0^{1, 2}(\Omega) \cap L^\infty(\Omega)\) and using \(S_\epsilon(u_n)\) as a test function we deduce as above that
\[
\int\limits_\Omega \abs{\nabla u_n}^2 S_\epsilon'(u_n) + \int\limits_\Omega f_n S_\epsilon(u_n) = \int\limits_\Omega S_\epsilon(u_n) \mu_n.
\]
Thus,
\[
\int\limits_\Omega f_n S_\epsilon(u_n) \le \int\limits_\Omega \abs{\mu_n} = \norm{\mu_n}_{L^ 1(\Omega)}.
\]

By Stampacchia's regularity theory (proposition~\ref{prop3.1}), the sequence \((u_n)_{n \in \N}\) is bounded in \(W_0^{1, 1}(\Omega)\).
Thus, by the Rellich-Kondrachov compactness theorem, there exists a subsequence \((u_{n_k})_{k \in \N}\) which converges in \(L^1(\Omega)\) to some function \(v\).
Since \(u\) and \(v\) satisfy the same linear Dirichlet problem with datum \(\mu - g(u)\), by uniqueness of solution of the linear problem we have \(u = v\).
This implies that the sequence \((u_n)_{n \in \N}\) converges to \(u\) in \(L^1(\Omega)\).

By the dominated convergence theorem, the sequence \((f_n S_\epsilon(u_n))_{n \in \N}\) converges to \(g(u)S_\epsilon(u)\) in \(L^1(\Omega)\). Thus,
\[
\int\limits_\Omega g(u)S_\epsilon(u) \le \liminf_{n \to \infty}{\norm{\mu_n}_{L^ 1(\Omega)}}.
\]
Choosing a sequence \((\mu_n)_{n \in \N}\) such that (see lemma~\ref{lemmaWeakApproximation})
\[
\lim_{n \to \infty}{\norm{\mu_n}_{L^ 1(\Omega)}} = \norm{\mu}_{\cM(\Omega)},
\]
we deduce that for every \(\epsilon > 0\),
\[
\int\limits_\Omega g(u)S_\epsilon(u) \le  \norm{\mu}_{\cM(\Omega)}.
\]
Thus,
\[
\int\limits_\Omega g(u) \sgn{u} \le  \norm{\mu}_{\cM(\Omega)}.
\]
By the sign condition, the estimate follows.
\end{proof}

The proof of Bénilan-Brezis's existence result relies on Stampacchia's regularity theory:

\begin{proof}[Proof of proposition~ \ref{propositionExistenceBenilanBrezis}]
Let \((\mu_n)_{n \in \N}\) be a sequence in \(L^\infty(\Omega)\) such that \((\mu_n)_{n \in \N}\) is bounded in \(L^1(\Omega)\) and converges weakly to \(\mu\)
in the sense of measures in \(\Omega\).
Since \(\mu_n \in L^\infty(\Omega)\), the nonlinear Dirichlet problem with datum \(\mu_n\) has a solution \(u_n\) (see proposition~\ref{propositionExistenceSolutionEulerLagrange}). 
Thus, for every \(\zeta \in C_0^\infty(\overline\Omega)\),
\[
- \int\limits_\Omega u_n \Delta\zeta + \int\limits_\Omega g(u_n)\zeta =  \int\limits_\Omega \zeta \mu_n.
\]

Since the sequence \((\mu_n)_{n \in \N}\) is bounded in \(L^1(\Omega)\), by the absorption estimate the sequence \((\mu_n - g(u_n))_{n \in \N}\) is also bounded in \(L^1(\Omega)\).
By Stampacchia's regularity theory (proposition~\ref{propositionCompactnessLp}), there exists a subsequence \((u_{n_k})_{k \in \N}\) converging to some function \(u\) in \(L^r(\Omega)\) for every \(1 \le r < \frac{N}{N-2}\).
Since \(p < \frac{N}{N-2}\), it follows in particular that the sequence \((g(u_{n_k}))_{k \in \N} \) converges to \(g(u)\) in \(L^1(\Omega)\). 
As \(k\) tends to infinity, we deduce that for every \(\zeta \in C_0^\infty(\overline\Omega)\),
\[
- \int\limits_\Omega u \Delta\zeta + \int\limits_\Omega g(u)\zeta =  \int\limits_\Omega \zeta \dif\mu.
\]
The proof is complete.
\end{proof}



\section{Critical and supercritical growth}

By the counterexample of Bénilan and Brezis (proposition~\ref{propositionNonExistenceBenilanBrezis}), existence of solutions of the nonlinear Dirichlet problem for every measure is not true if the nonlinearity has polynomial growth with power \(p \ge \frac{N}{N-2}\). The main result in this case is due to  Baras and Pierre~\cite{BarPie:84}*{theorem~4.1}:

\begin{proposition}
\label{propositionExistenceBarasPierre}
Let \(g : \R \to \R\) be a continuous function satisfying the sign condition, the integrability condition and such that for every \(t \in \R\),
\[
\abs{g(t)} \le C(\abs{t}^p + 1),
\]
for some \(C > 0\) and \(p > 0\).
If \(p \ge \frac{N}{N-2}\), then the nonlinear Dirichlet problem has a solution for every \(\mu \in \cM(\Omega)\) such that \(\mu \ll \capt_{W^{2, p'}}\).
\end{proposition}

The notation \(\mu \ll \capt_{W^{2, p'}}\) means that \(\mu\) is diffuse with respect to the \(W^{2, p'}\) capacity, or equivalently that \(\mu\) does not charge sets of zero \(W^{2, p'}\) capacity (see definition~\ref{definitionDiffuseMeasure}).

This statement is consistent with Bénilan-Brezis's existence result for \(p < \frac{N}{N-2}\). 
Indeed, if \(1 \le p < \frac{N}{N-2}\), then \(p' > \frac{N}{2}\) and as a consequence of Morrey's imbedding theorem the only set with \(W^{2, p'}\)  capacity zero is the empty set;
thus every measure is diffuse with respect to the \(W^{2, p'}\) capacity.

\medskip

We now present some ingredients in the proof of proposition~\ref{propositionExistenceBarasPierre}.

\begin{lemma}
\label{lemmaEstimationLp}
Given \(\mu \in \cM(\Omega)\), let \(v\) be the solution of the linear Dirichlet problem
\[
\left\{
\begin{alignedat}{2}
- \Delta v & = \mu &&\quad \text{in } \Omega,\\
v & = 0 &&\quad \text{on } \partial\Omega.
\end{alignedat}
\right.
\]
For every \(1 < p < +\infty\), \(v \in L^p(\Omega)\) if and only if there exists \(C > 0\) such that for every \(\zeta \in C_0^\infty(\overline\Omega)\),
\[
\bigg| \int\limits_\Omega \zeta \dif\mu \bigg| \le C \norm{\zeta}_{W^{2, p'}(\Omega)}
\]
\end{lemma}

The proof of lemma~\ref{lemmaEstimationLp} relies on the Calderón-Zygmund  estimates and on the Riesz representation theorem:

\begin{proof}
We begin with the direct implication:
For every \(\zeta \in C_0^\infty(\overline\Omega)\),
\[
\int\limits_\Omega \zeta \dif\mu = - \int\limits_\Omega v \Delta\zeta.
\]
If \(v \in L^p(\Omega)\), then by the Hölder inequality,
\[
\bigg| \int\limits_\Omega \zeta \dif\mu \bigg| \le \norm{v}_{L^p(\Omega)} \norm{\Delta\zeta}_{L^{p'}(\Omega)} \le \NewConstant \norm{v}_{ L^p(\Omega)} \norm{\zeta}_{W^{2, p'}(\Omega)}
\]
and the conclusion follows with \(C = \SameConstant \norm{v}_{ L^p(\Omega)}\).

Conversely, if \(\mu\) satisfies the estimate, then by the Calderón-Zygmund estimate we have for every \(\zeta \in C_0^\infty(\overline\Omega)\),
\[
\bigg| \int\limits_\Omega v \Delta\zeta \bigg| \le \Constant \norm{\zeta}_{W^{2, p'}(\Omega)}
\le \Constant \norm{\Delta \zeta}_{L^{p'}(\Omega)}.
\]
Given \(f \in C^\infty(\overline\Omega)\), let \(\zeta\) be the solution of the linear Dirichlet problem
\[
\left\{
\begin{alignedat}{2}
- \Delta \zeta & = f &&\quad \text{in } \Omega,\\
\zeta & = 0 &&\quad \text{on } \partial\Omega.
\end{alignedat}
\right.
\]
In particular, \(\zeta \in C_0^\infty(\overline\Omega)\) and by the estimate we have established,
we deduce that for every \(f \in C^\infty(\overline\Omega)\),
\[
\bigg| \int\limits_\Omega v f \bigg| \le \SameConstant  \norm{f}_{L^{p'}(\Omega)}.
\]
By the Riesz representation theorem, we deduce that \(v \in L^p(\Omega)\).
\end{proof}

Since the vector space \(C_0^\infty(\overline\Omega)\) is dense in \(W^{2, p'}(\Omega) \cap W_0^{1, p'}(\Omega)\) with respect to the \(W^{2, p'}\) norm, the assumption of the lemma amounts to saying that \(\mu\) has a \emph{unique} extension as a continuous linear function on \(W^{2, p'}(\Omega) \cap W_0^{1, p'}(\Omega)\).
We shall say in this case that 
\[
\mu \in (W^{2, p'}(\Omega) \cap W_0^{1, p'}(\Omega))'.
\]

\begin{corollary}
\label{corollarySolutionW-2p}
Let \(g : \R \to \R\) be a continuous function satisfying the sign condition and such that for every \(t \in \R\),
\[
\abs{g(t)} \le C(\abs{t}^p + 1),
\]
for some \(C > 0\) and \(1 < p < +\infty\).
For every \(\mu \in \cM(\Omega)\) such that 
\[
\max{\{\mu, 0\}}, \min{\{\mu, 0\}} \in (W^{2, p'}(\Omega) \cap W_0^{1, p'}(\Omega))',
\]
the nonlinear Dirichlet problem with datum \(\mu\) has a solution.
\end{corollary}

\begin{proof}
Let \(\overline v\) be the solution of the linear Dirichlet problem with datum \(\max{\{\mu, 0\}}\). 
Since \(\max{\{\mu, 0\}} \in (W^{2, p'}(\Omega) \cap W_0^{1, p'}(\Omega))'\), by the previous lemma we have \(\overline{v} \in L^p(\Omega)\). 
By the weak maximum principle (proposition~\ref{propositionWeakMaximumPrinciple}), \(\overline v \ge 0\) in \(\Omega\). 
The sign condition yields \(g(\overline{v}) \ge 0\), whence \(\overline v\) is a supersolution of the nonlinear Dirichlet problem with datum \(\mu\).

Similarly, if \(\underline v\) is the solution of the linear Dirichlet problem with datum \(\min{\{\mu, 0\}}\), then \(\underline v \in L^p(\Omega)\), \(\underline v \le 0\) in \(\Omega\) and  \(\underline v\) is a subsolution of the nonlinear Dirichlet problem with datum \(\mu\).

Since \(\underline{v}\) and \(\overline{v}\) both belong to \(L^p(\Omega)\), for every \(v \in L^1(\Omega)\) such that \(\underline{v} \le v \le \overline{v}\), \(g(v) \in L^1(\Omega)\).
The method of sub and supersolution (proposition~\ref{propositionMethodSubSuperSolutions}) yields a solution \(u\) of the Dirichlet problem with datum \(\mu\) such that \(\underline{v} \le u \le \overline{v}\).
\end{proof}

We characterize measures which are diffuse with respect to the \(W^{2, p'}\) capacity in terms of measures whose positive and negative parts belong to \((W^{2, p'}(\Omega) \cap W_0^{1, p'}(\Omega))'\).
This is done by the following proposition:

\begin{proposition}
\label{propositionIncreasingSequenceCapacityBis}
Let \(1 < p < +\infty\) and let \(\mu \in \cM(\Omega)\) be a nonnegative measure. 
Given an open set \(\Omega \subset \R^N\), let \(V \subset W^{2, p'}(\Omega)\) be a closed vector subspace containing \(C_c^\infty(\Omega)\). 
If \(\mu \ll \capt_{W^{2, p'}}\), then there exists a sequence \((\mu_n)_{n \in \N}\) in \(\cM(\Omega)\) of nonnegative measures such that
\begin{enumerate}[\((i)\)]
\item for every \(n \in \N\), \(\mu_n \in V'\),
\item the sequence \((\mu_n)_{n \in \N}\) is nondecreasing and converges strongly to \(\mu\) in \(\cM(\Omega)\).
\end{enumerate}
\end{proposition}

This proposition is established in the next section for \(\Omega = \R^N\). 
The proof for an arbitrary open set requires minor changes and is based
on the following variant of lemma~\ref{lemmaIncreasingSequenceDiffuseMeasures} whose proof relies on the Hahn-Banach theorem:

\begin{lemma}
Let \(1 < p < +\infty\) and let \(\mu \in \cM(\Omega)\) be a nonnegative measure. 
Given an open set \(\Omega \subset \R^N\), let \(V \subset W^{2, p'}(\Omega)\) be a closed vector subspace containing \(C_c^\infty(\Omega)\). 
If \(\mu \ll \capt_{W^{2, p'}}\), then for every \(\epsilon > 0\) there exists \(\nu  \in \cM(\Omega)\) such that 
\begin{enumerate}[\((i)\)]
\item \(\nu \in V'\),
\item \(0 \le \nu \le \mu\),
\item \(\norm{\mu - \nu}_{\cM(\Omega)} \le \epsilon\). 
\end{enumerate}
\end{lemma}

The proof of the following comparison principle relies on Kato's inequality:

\begin{lemma}
\label{lemmaEstimateComparison}
Let \(g : \R \to \R\) be a continuous function satisfying the sign condition.
Given \(\mu, \nu \in \cM(\Omega)\), let \(u\) be a solution of the nonlinear Dirichlet problem with datum \(\mu\) and let \(v\)  be the solution of the linear Dirichlet problem with datum \(\nu\).
\begin{enumerate}[\((i)\)]
\item If \(\mu \le \nu\) and if \(v \ge 0\), then \(u \le v\) in \(\Omega\).
\item If \(\mu \ge \nu\) and if \(v \le 0\), then \(u \ge v\) in \(\Omega\).
\end{enumerate}
\end{lemma}

\begin{corollary}
\label{corollaryEstimateComparison}
Let \(g : \R \to \R\) be a continuous function satisfying the sign condition.
Given \(\mu \in \cM(\Omega)\), let \(u\) be a solution of the nonlinear Dirichlet problem with datum \(\mu\).
\begin{enumerate}[\((i)\)]
\item If \(\mu \le 0\), then \(u \le 0\) in \(\Omega\).
\item If \(\mu \ge 0\), then \(u \ge 0\) in \(\Omega\).
\end{enumerate}
\end{corollary}

We give a direct proof of the corollary without using the lemma and then we prove the lemma in full generality:

\begin{proof}[Proof of corollary~\ref{corollaryEstimateComparison}]
We prove the first assertion.
By assumption,
\[
\Delta u = g(u) - \mu \ge g(u)
\]
in the sense of \((C_0^\infty(\overline\Omega))'\).
Thus, by Kato's inequality up to the boundary (lemma~\ref{lemmaKatoBoundary}) and by the sign condition,
\[
\Delta u^+ \ge \chi_{\{u > 0\}} g(u) \ge 0
\]
in the sense of \((C_0^\infty(\overline\Omega))'\).
The conclusion follows from the weak maximum principle (proposition~\ref{propositionWeakMaximumPrinciple}).
\end{proof}

\begin{proof}[Proof of lemma~\ref{lemmaEstimateComparison}]
We prove the first assertion.
By assumption,
\[
- \Delta u = - g(u) + \mu 
\]
and
\[
- \Delta v = \nu
\]
in the sense of \((C_0^\infty(\overline\Omega))'\).
Since \(\mu \le \nu\),
\[
\Delta (u - v) = g(u) + (\nu - \mu) \ge g(u)
\]
in the sense of \((C_0^\infty(\overline\Omega))'\).
Thus, by Kato's inequality up to the boundary (lemma~\ref{lemmaKatoBoundary}),
\[
\Delta (u - v)^+ \ge \chi_{\{u > v\}} g(u)
\]
in the sense of \((C_0^\infty(\overline\Omega))'\).

The assumption \(v \ge 0\) implies that \(\{u > v\} \subset \{u > 0\}\). Since \(g\) satisfies the sign condition, we deduce that 
\[
g(u) \ge 0 \quad \text{in \(\{u > v\}\),}
\]
whence
\[
\Delta (u - v)^+ \ge 0
\]
in the sense of \((C_0^\infty(\overline\Omega))'\).
By the weak maximum principle (proposition~\ref{propositionWeakMaximumPrinciple}), we deduce that \((u - v)^+ \le 0\) and thus \(u \le v\).
\end{proof}

\begin{proof}[Proof of proposition~\ref{propositionExistenceBarasPierre}]
Since \(\max{\{\mu, 0\}} \ll \capt_{W^{2, p}}\), we may apply proposition~\ref{propositionIncreasingSequenceCapacityBis} above with \(V = W^{2, p'}(\Omega) \cap W_0^{1, p'}(\Omega)\) to the measure \(\max{\{\mu, 0\}}\).
We deduce that there exists a nondecreasing sequence \((\overline\mu_n)_{n \in \N}\) of nonnegative measures such that for every \(n \in \N\), 
\[
\overline\mu_n \in (W^{2, p'}(\Omega) \cap W_0^{1, p'}(\Omega))'
\]
and \((\overline\mu_n)_{n \in \N}\) converges strongly to \(\max{\{\mu, 0\}}\) in \(\cM(\Omega)\).

\begin{Claim}
There exists a nondecreasing sequence \((\overline u_n)_{n \in \N}\) in \(L^1(\Omega)\) such that for every \(n \in \N\), \(\overline{u}_n\) is a solution of the nonlinear Dirichlet problem with datum \(\overline{\mu}_n\).
\end{Claim}

We assume momentarily the claim and we complete the proof of the proposition.

By the comparison principle (lemma~\ref{lemmaEstimateComparison}), every function \(\overline{u}_n\) is bounded from above by the solution of the linear Dirichlet problem with datum \(\max{\{\mu, 0\}}\).
Thus, by the monotone convergence theorem the sequence \((\overline u_n)_{n \in \N}\) converges in \(L^1(\Omega)\) to its pointwise limit \(\overline{u}\).
By the absorption estimate (lemma~\ref{lemmaEstimateAbsorption}), the sequence \((g(\overline u_n))_{n \in \N}\) is bounded in \(L^1(\Omega)\). 
By the integrability condition, the sequence \((g(\overline u_n))_{n \in \N}\) converges to \(g(\overline u)\) in \(L^1(\Omega)\) (see corollary~\ref{corollaryIntegrabilityConditionMonotoneConvergence}).

Since for every \(n \in \N\) and for every \(\zeta \in C_0^\infty(\overline\Omega)\),
\[
- \int\limits_\Omega \overline{u}_n \Delta\zeta + \int\limits_\Omega g(\overline{u}_n)\zeta = \int\limits_\Omega \zeta \dif\overline{\mu}_n,
\]
as \(n\) tends to infinity we conclude that \(\overline u\) is a solution of the nonlinear Dirichlet problem with datum \(\max{\{\mu, 0\}}\).

Similarly, the nonlinear Dirichlet problem with datum \(\min{\{\mu, 0\}}\) has a solution \(\underline{u}\). 
Since \(\underline{u} \le 0 \le \overline{u}\), by the method of sub and supersolution (proposition~\ref{propositionMethodSubSuperSolutions}) we deduce that the nonlinear Dirichlet problem with datum \(\mu\) has a solution.

We are left to prove the claim:

\begin{proof}[Proof of the claim]
For each \(n \in \N\), denote by \(\overline{v}_n\) the solution of the linear Dirichlet problem with datum \(\overline{\mu}_n\).
Since \(\overline\mu_n \ge 0\), by the weak maximum principle (proposition~\ref{propositionWeakMaximumPrinciple}), \(\overline{v}_n \ge 0\).
By the sign condition and by lemma~\ref{lemmaEstimationLp}, \(\overline{v}\) is a supersolution of the nonlinear Dirichlet problem with datum \(\overline{\mu}_n\).

We proceed by induction as follows.
Since \(0\) is a subsolution and \(\overline{v}_0\) is a supersolution of the nonlinear Dirichlet problem with datum \(\overline{\mu}_0\),
by the method of sub and supersolution (proposition~\ref{propositionMethodSubSuperSolutions}), there exists a solution \(\overline{u}_0\)  with datum \(\overline{\mu}_0\) such that \(0 \le \overline{u}_0 \le \overline{v}_0\).

Assume that \(n \in \N_*\) and that we have defined \(\overline{u}_{n - 1}\).
By the comparison principle (lemma~\ref{lemmaEstimateComparison}), \(\overline{u}_{n - 1} \le \overline{v}_n\).
Since \(\overline{u}_{n - 1}\) is a subsolution and \(\overline{v}_n\) is a supersolution of the nonlinear Dirichlet problem with datum \(\overline{\mu}_n\), by the method of sub and supersolution, there exists a solution \(\overline{u}_n\)  with datum \(\overline{\mu}_n\) such that \(\overline{u}_{n - 1} \le \overline{u}_n \le \overline{v}_n\).

The sequence \((\overline{u}_n)_{n \in \N}\) defined in this way has the required properties.
\end{proof}

This concludes the proof.
\end{proof}

In the proof of the proposition, we use the fact that \(\max{\{\mu, 0\}}\) is the strong limit of a \emph{nondecreasing} sequence of good measures --- i.e.~measures for which the nonlinear Dirichlet problem has a solution --- in order to show that \(\max{\{\mu, 0\}}\) is also a good measure.
More generally, for any sequence of good measures \((\nu_n)_{n \in \N}\) converging strongly to \(\nu\) in \(\cM(\Omega)\), \(\nu\) is a good measure, but the only proof we know under the assumptions of the proposition relies on reduced measures (see corollary~\ref{corollaryGoodMeasuresClosed}).
When the nonlinearity \(g\) is nondecreasing, this is a consequence
of the contraction property discussed in the proof of proposition~\ref{propositionBrezisStrauss}.

\medskip

Proposition~\ref{propositionExistenceBarasPierre} characterizes all measures for which the nonlinear Dirichlet problem has a solution when the nonlinearity is given by
\[
g(t) = \abs{t}^{p-1}t.
\] 
Indeed, if the Dirichlet problem with this nonlinearity has a solution, then there exists \(u \in L^p(\Omega)\) such that 
\[
\mu = - \Delta u + \abs{u}^{p-1}u
\]
in the sense of distributions in \(\Omega\).
Since \(u \in L^p(\Omega)\), for every \(\varphi \in C_c^\infty(\Omega)\) by the Hölder inequality we have
\[
\bigg| \int\limits_\Omega u \Delta\varphi \bigg| \le C \norm{u}_{L^p(\Omega)} \norm{\varphi}_{W^{2, p'}(\Omega)}.
\]
Thus, \(\Delta u \in (W^{2, p'}(\Omega) \cap W_0^{1, p'}(\Omega))'\), whence the measure \(\abs{\Delta u}\) is diffuse with respect to the \(W^{2, p'}\) capacity (see remark following lemma~\ref{lemmaGrunRehommeWkp}).

Since \(\abs{u}^{p} \in L^1(\Omega)\) and sets with zero \(W^{2, p'}\) capacity have zero Lebesgue measure, the measure associated with \(\abs{u}^{p}\) is also diffuse with respect to the \(W^{2, p'}\) capacity. Therefore, \(\mu \ll \capt_{W^{2, p'}}\).

\medskip

We can also consider the nonlinear Dirichlet problem with measure data on the boundary:
\[
\left\{
\begin{alignedat}{2}
-\Delta u + g(u) & = 0 && \quad \text{in }  \Omega,\\
u & = \nu && \quad \text{on } \partial \Omega.
\end{alignedat}
\right.
\]
The study of this problem when $g$ is a polynomial nonlinearity was initiated by Gmira and Véron~\cite{GmiVer:91} and has vastly expanded in recent years; see the papers of Marcus and Véron \cites{MarVer:98, MarVer:98a, MarVer:01, MarVer:04, MarVer:09}. 
Important motivations coming from the
theory of probability --- and the use of probabilistic methods --- have
reinvigorated the whole subject; see the pioneering papers of
Le Gall~\cites{Leg:95,Leg:97}, the books of Dynkin~\cites{Dyn:02,Dyn:04}, and the numerous references therein.

One of the reasons for studying separately the nonlinear Dirichlet problem with interior measure data and boundary measure data is that these conditions uncouple~\cite{BrePon:05}*{theorem~6}, which means that we may consider each type of data at a time.


\section{Strong approximation of diffuse measures}

We give characterizations of measures which are diffuse with respect to the \(W^{k, p}\) capacity in terms of elements in the dual space \((W^{k, p}(\R^N))'\).
By diffuse we mean that the measure \(\mu\) does not charge sets of zero \(W^{k, p}\) capacity. 
The precise definition is the following:

\begin{definition}
\label{definitionDiffuseMeasure}
Let \(\mu \in \cM(\Omega)\).
We say that \(\mu\) is a diffuse measure with respect to the \(W^{k, p}\) capacity if for every Borel set \(A \subset \R^N\) such that \(\capt_{W^{k, p}}(A) = 0\), then \(\abs{\mu}(A) = 0\).
\end{definition} 

We denote this property by 
\[
\mu \ll \capt_{W^{k, p}}.
\]

The following characterization of diffuse measures is due to Baras and Pierre~\cite{BarPie:84}*{lemme~4.2}:

\begin{proposition}
\label{propositionIncreasingSequenceCapacity}
Let \(k \in \N_*\), \(1 < p < +\infty\) and let \(\mu \in \cM(\R^N)\) be a nonnegative measure.
If \(\mu \ll \capt_{W^{k, p}}\), then there exists a sequence \((\mu_n)_{n \in \N}\) in \(\cM(\R^N)\) of nonnegative measures such that
\begin{enumerate}[\((i)\)]
\item for every \(n \in \N\), \(\mu_n \in (W^{k, p}(\R^N))'\),
\item the sequence \((\mu_n)_{n \in \N}\) is nondecreasing and converges strongly to \(\mu\) in \(\cM(\Omega)\).
\end{enumerate}
\end{proposition}

In the statement above, the assumption \(\mu \in (W^{k, p}(\R^N))'\) means that there exists a constant \(C \ge 0\) such that for every \(\varphi \in C_c^\infty(\R^N)\),
\[
\bigg| \int\limits_{\R^N} \varphi \dif\nu \bigg| \le C \norm{\varphi}_{W^{k, p}(\R^N)}.
\]
Since \(C_c^\infty(\R^N)\) is dense in \(W^{k, p}(\R^N)\), \(\nu\) admits a unique extension as an element in \((W^{k, p}(\R^N))'\).

\medskip

The proposition above has the following equivalent statement in terms of a series of nonnegative measures:

\begin{proposition}
\label{propositionIncreasingSequenceCapacitySeries}
Let \(k \in \N_*\), \(1 < p < +\infty\) and let \(\mu \in \cM(\R^N)\) be a nonnegative measure. 
If \(\mu \ll \capt_{W^{k, p}}\), then there exists a sequence \((\nu_n)_{n \in \N}\) in \(\cM(\R^N)\)  of nonnegative measures such that 
\begin{enumerate}[\((i)\)]
\item for every \(n \in \N\), \(\nu_n \in (W^{k, p}(\R^N))'\),
\item \(\mu = \sum\limits_{n = 0}^\infty \nu_n\) in \(\cM(\R^N)\).
\end{enumerate}
\end{proposition}

Proposition~\ref{propositionIncreasingSequenceCapacitySeries} --- or its equivalent form, proposition~\ref{propositionIncreasingSequenceCapacity} --- characterizes diffuse measures in the sense that if \(\mu\) is a nonnegative measure such that \(\mu = \sum\limits_{n = 0}^\infty \nu_n\), with \(\nu_n\) as in the statement, then \(\mu\) is diffuse with respect to the \(W^{k, p}\) capacity since each measure \(\nu_n\) is diffuse.
This is a consequence of the following lemma:

\begin{lemma}
\label{lemmaGrunRehommeWkp}
Let \(k \in \N_*\), \(1 < p < +\infty\) and let \(\mu \in \cM(\R^N)\) be a nonnegative measure.
If \(\mu \in (W^{k, p}(\R^N))'\), then \(\mu \ll \capt_{W^{k, p}}\).
\end{lemma}

\begin{proof}
Let \(K \subset \R^N\) be a compact set.
For every \(\varphi \in C_c^\infty(\R^N)\) such that \(\varphi \ge 1\) in \(K\),
\[
0 \le \mu(K) \le \int\limits_K \varphi \dif\mu \le \int\limits_{\R^N} \varphi^+ \dif\mu. 
\]
By lemma~\ref{lemmaMaximumTestFunctions}, there exists \(\psi \in C_c^\infty(\R^N)\) such that
\[
\varphi^+ = \max{\{\varphi, 0\}} \le \psi
\]
and
\[
\norm{\psi}_{W^{k, p}(\R^N)} \le C \norm{\varphi}_{W^{k, p}(\R^N)}.
\]
Thus,
\[
0 \le \mu(K) \le \int\limits_{\R^N} \psi \dif\mu. 
\]
Since \(\mu \in (W^{k, p}(\R^N))'\),
\[
0 \le \mu(K) \le \NewConstant \norm{\psi}_{W^{k, p}(\R^N)} \le  \Constant \norm{\varphi}_{W^{k, p}(\R^N)}.
\]
Taking the infimum over \(\varphi\), we deduce that
\[
0 \le \mu(K) \le \SameConstant \capt_{W^{k, p}}{(K)}.
\]
Thus, for every compact set \(K\) such that \(\capt_{W^{k, p}}{(K)} = 0\), we have \(\mu(K) = 0\).
This implies that the same property holds for Borel sets and we deduce that \(\mu \ll \capt_{W^{k, p}}\).
\end{proof}

When \(k = 1\) and \(p = 2\), this result was first established for signed measures by Grun-Rehomme~\cite{Gru:77}; see lemma~\ref{lemmaGrunRehommeW12} above.
The proof for signed measures requires \(\varphi\) to be chosen bounded by some fixed constant.
This is indeed possible for any \(k \in \N_*\) and \(1 < p < +\infty\) by a result of Adams and Polking~\cite{AdaPol:73};
the main ingredient in this case is to show that minimizers for an equivalent \(W^{k, p}\) capacity belong to \(L^\infty(\R^N)\) and they are bounded by a constant which do not depend on the compact set \(K\). This important result was proved independently by Adams and Meyers~\cite{AdaMey:71} and by Maz'ja and Havin~\cite{MazHav:70}.

\medskip
There is something puzzling about elements in \(\cM(\R^N) \cap (W^{k, p}(\R^N))'\).
On the one hand, as elements of \(\cM(\R^N)\) they may be seen as a continuous linear function acting on the closure of \(C_c^\infty(\R^N)\) under uniform convergence --- this is the space of continuous functions converging uniformly to \(0\) at infinity --- and they may be seen as a set function acting on Borel subsets of \(\R^N\).
On the other hand, as elements of \((W^{k, p}(\R^N))'\), they have a unique extension as a continuous linear functional acting on \(W^{k, p}(\R^N)\).
Questions related to this ambivalent behavior have been investigated by Brezis and Browder~\cites{BreBro:79,BreBro:82}.
\medskip

An alternative characterization of diffuse measures in the spirit of measures which are absolutely continuous with respect to another measure is given in lemma~\ref{lemmaDiffuseMeasureCharacterization}.
\medskip

We now turn to the proof of proposition~\ref{propositionIncreasingSequenceCapacitySeries}.  We begin with the following sublemma:

\begin{sublemma}
\label{sublemmaToolHahnBanach}
Let \(\mu \in \cM(\R^N)\) be a nonnegative measure. 
If \(F : C_c^\infty(\R^N) \to \R\) is a linear functional such that for every \(\varphi \in C_c^\infty(\R^N)\),
\[
F(\varphi) \le \int\limits_{\R^N} \varphi^+ \dif\mu,
\]
then
\begin{enumerate}[\((i)\)]
\item there exists a unique \(\nu \in \cM(\R^N)\) such that for every \(\varphi \in C_c^\infty(\R^N)\),
\[
F(\varphi) = \int\limits_{\R^N} \varphi \dif\nu,
\]
and \(0 \le \nu \le \mu\),
\item for every Borel set \(A \subset \R^N\) and for every nonnegative function \(\varphi \in C_c^\infty(\R^N)\) such that \(\varphi \ge 1\) in \(A\),
\[
\norm{\mu - \nu}_{\cM(\R^N)} 
\le \bigg( \int\limits_{\R^N} \varphi \dif\mu - F(\varphi) \bigg) + \mu(\R^N \setminus A).
\]
\end{enumerate}
\end{sublemma}

\begin{proof}
By assumption on the linear functional \(F\), for every \(\varphi \in C_c^\infty(\R^N)\) we have
\[
F(\varphi) \le C \norm{\varphi^+}_{L^\infty(\R^N)} \le C \norm{\varphi}_{L^\infty(\R^N)}.
\]
Thus,
\[
\abs{F(\varphi)} \le C \norm{\varphi}_{L^\infty(\R^N)}.
\]
By the Riesz representation theorem, there exists a unique \(\nu \in \cM(\R^N)\) satisfying the conclusion of the proposition.

For every nonnegative function \(\varphi \in C_c^\infty(\R^N)\),
\[
\int\limits_{\R^N} \varphi \dif\nu = F(\varphi) \le \int\limits_{\R^N} \varphi \dif\mu
\]
and
\[
- \int\limits_{\R^N} \varphi \dif\nu = F(- \varphi) \le 0,
\]
whence
\[
0 \le \int\limits_{\R^N} \varphi \dif\nu \le \int\limits_{\R^N} \varphi \dif\mu.
\]
We conclude that
\[
0 \le \mu \le \nu.
\]
This establishes the first assertion.

\medskip
We now prove the second assertion.
Since \(0 \le \mu \le \nu\), for every Borel set \(A \subset \R^N\),
\[
\begin{split}
\norm{\mu - \nu}_{\cM(\R^N)} 
& = (\mu - \nu)(\R^N)\\
& = (\mu - \nu)(A) + (\mu - \nu)(\R^N \setminus A) \le (\mu - \nu)(A) + \mu(\R^N \setminus A).
\end{split}
\]
If \(\varphi \in C_c^\infty(\R^N)\) is a nonnegative function such that \(\varphi \ge 1\) in \(A\), then by the Chebyshev inequality,
\[
(\mu - \nu)(A) \le \int_{\R^N} \varphi \dif(\mu - \nu) = \int_{\R^N} \varphi \dif\mu - F(\varphi).
\]
Combining both estimates, we deduce the second assertion.
\end{proof}

\begin{sublemma}
Let \(\mu \in \cM(\R^N)\) be a nonnegative measure and let
\(\Phi : C_c^ \infty(\R^N) \to \R\) be the functional defined for \(\varphi \in C_c^ \infty(\R^N)\) by
\[
\Phi(\varphi) = \int\limits_{\R^N} \varphi^+ \dif\mu.
\]
If \(\mu \ll \capt_{W^{k, p}}\), then \(\Phi\) is lower semicontinuous with respect to 
the strong topology in \(W^{k, p}(\R^N)\).
\end{sublemma}

\begin{proof}
Let \((\varphi_n)_{n \in \N}\) be a sequence converging in \(W^{k, p}(\R^N)\) to \(\varphi \in C_c^\infty(\R^N)\).
Passing to a subsequence if necessary, we may assume that the sequence \((\Phi(\varphi_n))_{n \in \N}\) converges in \([0, +\infty]\).
Let \((\varphi_{n_i})_{i \in \N}\) be a subsequence converging pointwisely to \(\varphi\) in \(\R^N \setminus E\) for some Borel set \(E \subset \R^N\) such that \(\capt_{W^{k, p}}{(E)} = 0\) (see proposition~\ref{propositionPointwiseConvergence} and remark~\ref{remarkPointwiseConvergence}). 
Since the measure \(\mu\) is diffuse with respect to the \(W^{k, p}\) capacity, \(\mu(E) = 0\). 
Thus, \((\varphi_{n_i})_{i \in \N}\) converges \(\mu\)-almost everywhere to \(\varphi\). By Fatou's lemma, we then have
\[
\Phi(\varphi) = \int\limits_{\R^N} \varphi^+ \dif\mu \le \lim_{i \to \infty} \int\limits_{\R^N} \varphi_{n_i}^+ \dif\mu = \lim_{n \to \infty}{\Phi(\varphi_n)}.
\]
This gives the conclusion.
\end{proof}

The main ingredient in the proof of proposition~\ref{propositionIncreasingSequenceCapacitySeries} is based on the following lemma:

\begin{lemma}
\label{lemmaIncreasingSequenceDiffuseMeasures}
Let \(k \in \N_*\), \(1 < p < +\infty\) and let \(\mu \in \cM(\R^N)\) be a nonnegative measure. 
If \(\mu \ll \capt_{W^{k, p}}\), then for every \(\epsilon > 0\) there exists \(\nu  \in \cM(\R^N)\) such that 
\begin{enumerate}[\((i)\)]
\item \(\nu  \in (W^{k, p}(\R^N))'\),
\item \(0 \le \nu \le \mu\),
\item \(\norm{\mu - \nu}_{\cM(\R^N)} \le \epsilon\).
\end{enumerate}
\end{lemma}

The proof we present below is due to Ancona \citelist{\cite{FeyDel:77}*{théorème~8}  \cite{BarPie:84}*{lemme~4.2}} and is based on the Hahn-Banach theorem .

\begin{proof}
Let \(\Phi : C_c^ \infty(\R^N) \to \R\) be the function defined for \(\varphi \in C_c^ \infty(\R^N)\) by
\[
\Phi(\varphi) = \int\limits_{\R^N} \varphi^+ \dif\mu.
\]
Then, \(\Phi\) is convex and by the previous lemma \(\Phi\) is lower semicontinuous with respect to the strong topology in \(W^{k, p}(\R^N)\).

Let \(\overline{\Phi} : W^{k, p}(\R^N) \to [0, +\infty]\) be the extension  of \(\Phi\) to \(W^{k, p}(\R^N)\) as a convex lower semicontinuous function.
As a \(\Gamma\)-limit this extension is given for every \(u \in W^{k, p}(\R^N)\) by
\[
\overline{\Phi}(u) =
\inf{\Big\{ \liminf_{n \to \infty}{\Phi(\varphi_n)} : \text{\((\varphi_n)_{n \in \N}\) converges to \(u\) in \(W^{k, p}(\R^N)\)} \Big\}}.
\]
It follows from the geometric form of the Hahn-Banach theorem \cite{Bre:11}*{theorem~1.11}
that \(\overline{\Phi}\) is the supremum of a family of continuous linear functionals in \(W^{k, p}(\R^N)\). 
Thus, given \(\epsilon_1 > 0\) and \(\psi \in C_c^\infty(\R^N)\), there exists \(F \in (W^{k, p}(\R^N))'\) such that 
\[
F \le \overline\Phi \quad \text{in \(W^{k, p}(\R^N)\)}
\]
and
\[
\Phi(\psi) = \overline \Phi(\psi) \le F(\psi) + \epsilon_1.
\]
In particular, \(F\) satisfies the assumptions of sublemma~\ref{sublemmaToolHahnBanach}.
Let \(\nu\) be the measure given by the sublemma.

Take \(\epsilon_1 < \epsilon\) and a compact set \(K \subset \R^N\) such that
\[
\mu(\R^N \setminus K) \le \epsilon - \epsilon_1.
\]
Choosing a nonnegative function \(\psi\) such that \(\psi \ge 1\) in \(K\), by the estimate in sublemma~\ref{sublemmaToolHahnBanach} we have
\[
\norm{\mu - \nu}_{\cM(\R^N)} 
\le \big( \Phi(\psi) - F(\psi) \big) + \mu(\R^N \setminus K) \le \epsilon_1 + (\epsilon - \epsilon_1) = \epsilon.
\]
Thus, \(\nu\) satisfies all the required properties.
\end{proof}

\begin{proof}[Proof of proposition~\ref{propositionIncreasingSequenceCapacitySeries}]
Let \((\epsilon_n)_{n \in \N}\) be a sequence of positive numbers converging to \(0\). We construct the sequence \((\nu_n)_{n \in \N}\) inductively as follows. 

By lemma~\ref{lemmaIncreasingSequenceDiffuseMeasures}, there exists \(\nu_0 \in \cM(\R^N)\) such that \(\nu_0 \in (W^{k, p}(\R^N))'\),
\[
0 \le \nu_0 \le \mu
\]
and
\[
\norm{\mu - \nu_0}_{\cM(\R^N)} \le \epsilon_0.
\]

Given \(n \in \N_*\), assume that we have defined nonnegative measures \(\nu_0, \dots, \nu_{n - 1}\) such that 
\[
0 \le \sum\limits_{i = 0}^{n-1} \nu_i \le \mu.
\]
In particular, 
\(\sum\limits_{i = 0}^{n-1} \nu_i \ll \capt_{W^{k, p}}\).

Applying lemma~\ref{lemmaIncreasingSequenceDiffuseMeasures} to the measure \(\mu - \sum\limits_{i = 0}^{n-1} \nu_i\), there exists \(\nu_n \in \cM(\R^N)\) such that \(\nu_n \in (W^{k, p}(\R^N))'\),
\[
\textstyle 0 \le \nu_n \le \mu - \sum\limits_{i = 0}^{n-1} \nu_i
\]
and
\[
\textstyle \norm{\mu - \sum\limits_{i = 0}^{n-1} \nu_i - \nu_n}_{\cM(\R^N)} \le \epsilon_n.
\]
This sequence \((\nu_n)_{n \in \N}\) has the required properties.
\end{proof}

The next corollary summarizes other equivalent characterizations of diffuse measures:

\begin{corollary}
\label{corollaryDecompositionBoccardoGallouetOrsina}
Let \(k \in \N_*\), \(1 < p < +\infty\) and let \(\mu \in \cM(\R^N)\) be a nonnegative measure. 
If \(\mu \ll \capt_{W^{k, p}}\), then 
\begin{enumerate}[\((i)\)]
\item 
there exist a nonnegative function \(f \in L^1(\R^N)\) and a signed measure \(\lambda \in \cM(\R^N) \cap (W^{k, p}(\R^N))'\) such that
\[
\mu = f + \lambda  \quad \text{in \(\cM(\R^N)\);}
\]
\label{itemDecompositionBoccardoGallouetOrsina}
\item  there exist a nonnegative measure \(\gamma \in \cM(\R^N) \cap (W^{k, p}(\R^N))'\) and a nonnegative function \(h \in L^1(\R^N; \gamma)\) such that
\[
\mu = h \gamma  \quad \text{in \(\cM(\R^N)\);}
\]
\label{itemDecompositionDalMaso}
\item for every \(\epsilon > 0\), there exists a Borel set \(E \subset \R^N\) such that \(\mu(\R^N \setminus E) \le \epsilon\) and the restriction of \(\mu\) to \(E\) satisfies
\[
\mu\lfloor_E \in (W^{k, p}(\R^N))'.
\]
\label{itemDecompositionDalMasoConsequence}
\end{enumerate}
\end{corollary}

The first decomposition is due to Boccardo, Gallouët and Orsina~\cite{BocGalOrs:96}*{theorem~2.1}, inspired by a similar decomposition of Baras and Pierre~\cite{BarPie:84}*{remarque, p.~200}.
This decomposition was proved by Boccardo, Gallouët and Orsina for \(k = 1\), but their proof works for any \(k \in \N_*\).

When \(k=2\) and \(1 < p < +\infty\), this decomposition is usually attributed to Baras and Pierre but their result concerns bounded domains \(\Omega\) in \(\R^N\). 
The decomposition of Baras and Pierre is obtained by solving the Dirichlet problem with polynomial nonlinearity and datum \(\mu\) and gives a measure \(\lambda\) in \((W^{2, p}(\Omega) \cap W^{1, p}_0(\Omega))'\). 
Note that the measure \(\lambda\) admits a natural extension as a measure in \(\R^N\) and, by the Hahn-Banach theorem, \(\lambda\) also admits an extension as a  continuous linear functional in \(W^{2, p}(\R^N)\), but both extensions need not agree in \(C_c^\infty(\R^N)\).

\medskip
The second characterization is due to Dal~Maso~\cite{Dal:83}*{theorem~2.2} and it was the main ingredient originally used by Boccardo, Gallou\"et and Orsina to prove their decomposition.
In our case, we prove them independently, based on proposition~\ref{propositionIncreasingSequenceCapacitySeries}.

\medskip
The third decomposition is proved using the Dal~Maso characterization and
has a counterpart for Hausdorff measures (see proposition~\ref{propositionStrongApproximationHausdorffMeasure}).
In the case \(k = 1\) and \(p = 2\), this result was essentially known to experts in potential theory.

Indeed, by a classical result in potential theory \citelist{\cite{Hel:69}*{theorem~6.21} \cite{BreMarPon:07}*{lemma~4.D.1}}, for every measure \(\mu\) which is diffuse with respect to the \(W^{1, 2}\) capacity, there exists a compact set \(K \subset \R^N\) with \(W^{1, 2}\) capacity as small as we wish such that the restriction \(\mu\lfloor_K\) generates a \emph{continuous} Newtonian potential \(v\). 
In dimension \(N \ge 3\), this Newtonian potential converges uniformly to zero at infinity, thus by the interpolation inequality (lemma~\ref{lemmaInterpolationLinfty}), \(\nabla v \in L^2(\R^N)\).
In fact, \(v\) decays at infinity at the same rate as the fundamental solution of the Laplacian, \(\frac{1}{\abs{x}^{N-2}}\); thus, \(v \in L^2(\Omega)\) in dimension \(N \ge 5\). We deduce in this case that \(v \in W^{1, 2}(\R^N)\), whence \(\mu\lfloor_K \in (W^{1, 2}(\R^N))'\).

\begin{proof}[Proof of corollary~\ref{corollaryDecompositionBoccardoGallouetOrsina}~\((\ref{itemDecompositionBoccardoGallouetOrsina})\)]
Let \((\nu_n)_{n \in \N}\) be a sequence satisfying the conclusion of proposition~\ref{propositionIncreasingSequenceCapacitySeries}.
Given a sequence  \((\epsilon_n)_{n \in \N}\) of positive numbers, for each \(n \in \N\) write
\[
\nu_n = \rho_{\epsilon_n} * \nu_n + (\nu_n - \rho_{\epsilon_n} * \nu_n ).
\]
Since by Fubini's theorem
\[
\norm{\rho_{\epsilon_n} * \nu_n}_{L^1(\R^N)} \le \norm{\nu_n}_{\cM(\R^N)} = \nu_n(\R^N),
\]
the series
\[
\sum_{n=0}^\infty \rho_{\epsilon_n} * \nu_n 
\]
converges in \(L^1(\R^N)\) and we denote it by \(f\).

Given a sequence \((\alpha_n)_{n \in \N}\) of positive numbers, we choose \(\epsilon_n > 0\) so that
\[
\norm{\nu_n - \rho_{\epsilon_n} * \nu_n}_{(W^{k, p}(\R^N))'} \le \alpha_n,
\]
Taking the sequence \((\alpha_n)_{n \in \N}\) so that the series \(\sum\limits_{n = 0}^\infty \alpha_n\) converges, then
\[
\sum_{n=0}^\infty(\nu_n - \rho_{\epsilon_n} * \nu_n )
\]
converges in \((W^{k, p}(\R^N))'\). 
This series also converges in \(\cM(\R^N)\) and we call its limits \(\lambda\). 
Thus, \(\mu = f + \lambda\) gives the decomposition we sought.
\end{proof}

\begin{proof}[Proof of corollary~\ref{corollaryDecompositionBoccardoGallouetOrsina}~\((\ref{itemDecompositionDalMaso})\)]
Let \((\nu_n)_{n \in \N}\) be a sequence satisfying the conclusion of proposition~\ref{propositionIncreasingSequenceCapacitySeries}. 
Given a sequence  \((\beta_n)_{n \in \N}\) of positive numbers such that both series 
\[
\sum\limits_{n = 0}^\infty \beta_n \norm{\nu_n}_{\cM(\R^N)}
\quad \text{and} \quad
\sum\limits_{n = 0}^\infty \beta_n \norm{\nu_n}_{(W^{k, p}(\R^N))'}
\]
converge, let 
\[
\gamma = \sum\limits_{n = 0}^\infty \beta_n \nu_n.
\]
Then, \(\gamma \in \cM(\R^N) \cap (W^{k, p}(\R^N))'\).
Moreover, if \(A\) is a Borel set such that \(\mu(A) = 0\), then for every \(n \in \N\), \(\nu_n(A) = 0\), whence \(\gamma \ll \mu\). 
Thus, by the Lebesgue decomposition \citelist{\cite{Fol:99}*{theorem~3.8}  \cite{WheZyg:77}*{theorem~10.38}}, there exists 
\(h \in L^1(\R^N; \gamma)\) such that
\(\mu = h \gamma\).
\end{proof}

In order to prove the third decomposition we start with the following lemma:

\begin{lemma}
Let \(k \in \N_*\), \(1 < p < +\infty\) and let \(\gamma, \mu \in \cM(\R^N)\).
If \(\abs{\mu} \le \gamma\) and if \(\gamma \in (W^{k, p}(\R^N))'\), then \(\mu \in (W^{k, p}(\R^N))'\).
\end{lemma}

\begin{proof}
For every \(\varphi \in C_c^\infty(\R^N)\), 
\[
\bigg| \int\limits_{\R^N} \varphi \dif\mu \bigg| \le \int\limits_{\R^N} \abs{\varphi} \dif\abs{\mu} \le \int\limits_{\R^N} \abs{\varphi} \dif\gamma. 
\]
By lemma~\ref{lemmaMaximumTestFunctions}, there exists \(\psi \in C_c^\infty(\R^N)\) such that
\[
\abs{\varphi} = \max{\{\varphi, -\varphi\}} \le \psi
\]
and
\[
\norm{\psi}_{W^{k, p}(\R^N)} \le C \norm{\varphi}_{W^{k, p}(\R^N)}.
\]
Thus,
\[
\bigg| \int\limits_{\R^N} \varphi \dif\mu \bigg| \le \int\limits_{\R^N} \psi \dif\gamma \le \NewConstant \norm{\psi}_{W^{k, p}(\R^N)} \le \Constant \norm{\varphi}_{W^{k, p}(\R^N)}.
\]
This implies the result.
\end{proof}

\begin{proof}[Proof of corollary~\ref{corollaryDecompositionBoccardoGallouetOrsina}~\((\ref{itemDecompositionDalMasoConsequence})\)]
Let \(\gamma \in \cM(\R^N)\)  and \(h \in L^1(\R^N; \gamma)\) be such that 
\[
\gamma \in (W^{k, p}(\R^N))'
\quad \text{and} \quad
\mu = h \gamma.
\]
Since \(\gamma \in (W^{k, p}(\R^N))'\) and for every \(n \in \N\),
\[
0 \le \chi_{\{h \le n\}} \mu \le \gamma
\]
the previous proposition yields
\[
\mu\lfloor_{\{h \le n\}} = \chi_{\{h \le n\}}\mu \in (W^{k, p}(\R^N))'.
\]
On the other hand, the set \(\bigcap\limits_{n = 0}^\infty (\R^N \setminus \{h \le n\})\) is negligible for the measure \(\mu\), whence for every \(\epsilon > 0\) there exists \(n \in \N\) such that 
\[
\mu(\R^N \setminus \{h \le n\}) \le \epsilon.
\] 
We have the conclusion with \(E = \{h \le n\}\).
\end{proof}

The Boccardo-Gallouët-Orsina decomposition behaves linearly in the sense that if \(\mu_1\) and \(\mu_2\) are diffuse measures and if
\[
\mu_1 = f_1 + \lambda_1 
\quad \text{and} \quad
\mu_2 = f_2 + \lambda_2,
\]
then for every \(\alpha_1, \alpha_2 \in \R\),
\[
\alpha_1\mu_1 + \alpha_2\mu_2 = (\alpha_1 f_1 + \alpha_2 f_2) + (\alpha_1\lambda_1 + \alpha_2\lambda_2)
\]
gives a decomposition of the diffuse measure \(\alpha_1\mu_1 + \alpha_2\mu_2\).
However, the construction of the decomposition itself is highly nonlinear since it depends on the Hahn-Banach theorem.

Already in the case \(k = 1\) and \(p = 2\) one might ask whether there is a linear construction behind, or more precisely, if there exists a linear continuous functional \(L\) from the Banach space of diffuse measures with respect to the \(W^{1, 2}\) capacity into \(L^1(\R^N) \times (W^{1, 2}(\R^N))'\) which gives the decomposition. 
Ancona~\cite{Anc:06} has showed that the answer is negative.
The main reason behind is that the decomposition is far from being unique.

\medskip
From the proof of the Boccardo-Gallouët-Orsina decomposition we have an additional control on the norms of \(f\) and \(\lambda\): for every \(\epsilon > 0\), there exist \(f \in L^1(\R^N)\) and \(\lambda \in \cM(\R^N) \cap (W^{k, p}(\R^N))'\) such that 
\[
\norm{f}_{L^1(\R^N)} \le \norm{\mu}_{\cM(\R^N)},
\]
and
\[
\norm{\lambda}_{\cM(\R^N)} \le 2 \norm{\mu}_{\cM(\R^N)} \quad \text{and} \quad
\norm{\lambda}_{(W^{k, p}(\R^N))'} \le \epsilon.
\]

We can impose a different control on the \(f\) part and on the \(\lambda\) part of the decomposition. 
For instance, for every \(\epsilon > 0\), there exists \(C > 0\) such that
\[
\norm{\overline f}_{L^1(\R^N)} \le \epsilon,
\]
and
\[
\norm{\overline \lambda}_{\cM(\R^N)} + \norm{\overline\lambda}_{(W^{k, p}(\R^N))'} \le C.
\]
The idea is to start with a decomposition \(\mu = f + \lambda\) and then, given \(\kappa > 0\), write
\[
\mu = (f - T_\kappa(f)) + (T_\kappa(f) + \lambda).
\]
By choosing \(\kappa > 0\) sufficiently small, the \(L^1\) part of the decomposition has small norm, but since we are working on the whole space \(\R^N\), \(T_\kappa(f)\) need not belong to \((W^{k, p}(\R^N))'\).

For bounded domains, the strategy above already suffices to have the estimates.
The argument in \(\R^N\) has to be slightly adapted and is contained in the following corollary:

\begin{corollary}
\label{corollaryConvergenceBoccardoGallouetOrsina}
Let \(\mu \in \cM(\R^N)\) and let \(M \ge 0\).
If \(\mu\) is a diffuse measure with respect to the capacity \(W^{k, p}\), then for every \(\epsilon > 0\) there exists \(\delta > 0\) such that if \(\varphi \in C_0^\infty(\R^N)\) satisfies 
\[
\norm{\varphi}_{L^\infty(\R^N)} \le M
\quad
\text{and}
\quad
\norm{\varphi}_{W^{k, p}(\R^N)} \le \delta,
\]
then
\[
\bigg| \int\limits_{\R^N} \varphi \dif\mu \bigg| \le \epsilon.
\]
\end{corollary}

\begin{proof}
Let \(f \in L^1({\R^N})\) and \(\lambda \in \cM({\R^N}) \cap (W^{k, p}({\R^N}))'\)
be such that
\[
\mu = f + \lambda.
\]
Given \(\kappa > 0\) and a Borel set \(A \subset \R^N\) with finite Lebesgue measure,
write
\[
\mu = \overline{f} + \overline{\lambda},
\]
where
\[
\overline{f} = f - T_\kappa(f)\chi_A 
\quad \text{and} \quad 
\overline{\lambda} = T_\kappa(f)\chi_A  + \lambda
\]
For every \(\varphi \in C_c^\infty({\R^N})\),
\[
\bigg| \int\limits_{\R^N} \varphi \dif\mu \bigg| 
\le \norm{\overline{f}}_{L^1(\R^N)} \norm{\varphi}_{L^\infty(\R^N)} 
+ \norm{\overline{\lambda}}_{(W^{k, p}(\R^N))'} \norm{\varphi}_{W^{k, p}(\R^N)} .
\]
Since the set \(A\) has finite Lebesgue measure, \(\overline{\lambda} \in (W^{k, p}(\R^N))'\).
Moreover,
\[
\norm{\overline{f}}_{L^1(\R^N)} 
\le \norm{f - T_\kappa(f)}_{L^1(\R^N)} + \norm{f}_{L^1(\R^N \setminus A)}.
\]
Given \(M \ge 0\) and \(\epsilon > 0\), we may take \(\kappa > 0\) and a Borel set \(A \subset \R^N\) with finite Lebesgue measure
such that
\[
M\norm{\overline{f}}_{L^1(\R^N)} \le \frac{\epsilon}{2}.
\]
The conclusion follows by choosing any \(\delta > 0\) such that
\[
\delta \norm{\overline\lambda}_{(W^{k, p}(\R^N))'} \le \frac{\epsilon}{2}.
\]
The proof of the corollary is complete.
\end{proof}


\chapter{Exponential nonlinearity}

We investigate the existence of solutions of the nonlinear Dirichlet problem
\[
\left\{
\begin{alignedat}{2}
- \Delta u + g(u) & = \mu &&\quad \text{in } \Omega,\\
u & = 0 &&\quad \text{on } \partial\Omega,
\end{alignedat}
\right.
\]
where the nonlinearity \(g : \R \to \R\) has exponential growth: for every \(t \in \R\),
\[
\abs{g(t)} \le C(\e^{t} + 1),
\]
for some  \(C > 0\).

\section{Two dimensional case}

The main result in dimension two is due to Vázquez~\cite{Vaz:83}*{theorem~2}:

\begin{proposition}
\label{propositionExistenceSolutionExponentialDimension2}
Let \(g : \R \to \R\) be a continuous function satisfying the sign condition, the integrability condition and such that for every \(t \in \R\),
\[
\abs{g(t)} \le C(\e^{t} + 1).
\]
If \(N = 2\), then the Dirichlet problem has a solution for every \(\mu \in \cM(\Omega)\) such that for every \(x \in \Omega\), \(\mu(\{x\}) \le 4\pi\).
\end{proposition}

Vázquez original proof relies on the John-Nirenberg inequality~\cite{JohNir:61}.
We present a simpler proof based on the following inequality established by Brezis and Merle \cite{BreMer:91}*{theorem~1}:

\begin{lemma}
Let \(N = 2\). Given a nonnegative measure  \(\mu \in \cM(\Omega)\), let \(v : \R^2 \to \R\) be the Newtonian potential generated by \(\mu\), 
\[
v(x) = 
\frac{1}{2\pi} \int\limits_{\Omega} \log{\left(\frac{d}{\abs{x - y}}\right)} \dif \mu(y),
\]
where \(d \ge \diam{\Omega}\).
If \(\norm{\mu}_{\cM(\Omega)} < 4\pi\), then \(\e^v \in L^1(\Omega)\) and
\[
\norm{\e^v}_{L^1(\Omega)} \le C
\]
for some constant \(C > 0\) depending on \(\norm{\mu}_{\cM(\Omega)}\) and \(d\).
\end{lemma}

\begin{proof}
For every \(x \in \R^N\), we write
\[
v(x) = 
 \int\limits_{\Omega} \frac{\norm{\mu}_{\cM}}{2\pi} \log{\left(\frac{d}{\abs{x - y}}\right)} \frac{\dif \mu(y)}{\norm{\mu}_{\cM}} 
 =  \int\limits_{\Omega} \log{\left(\frac{d}{\abs{x - y}}\right)^{\frac{\norm{\mu}_{\cM}}{2\pi}}} \frac{\dif \mu(y)}{\norm{\mu}_{\cM}}.
\]
Since \(\mu/\norm{\mu}_{\cM}\) is a probability measure in \(\Omega\) and \(\log\) is a concave function, by the Jensen inequality for concave functions we have
\[
v(x) \le \log\bigg[ \int\limits_{\Omega} {\left(\frac{d}{\abs{x - y}}\right)^{\frac{\norm{\mu}_{\cM}}{2\pi}}} \frac{\dif \mu(y)}{\norm{\mu}_{\cM}} \bigg].
\]
Thus,
\[
\e^{v(x)} \le \int\limits_{\Omega} {\left(\frac{d}{\abs{x - y}}\right)^{\frac{\norm{\mu}_{\cM}}{2\pi}}} \frac{\dif \mu(y)}{\norm{\mu}_{\cM}}.
\]
By Fubini's theorem,
\[
\int\limits_\Omega \e^{v(x)} \dif x 
\le \int\limits_{\Omega} \bigg[ \int\limits_{\Omega} {\left(\frac{d}{\abs{x - y}}\right)^{\frac{\norm{\mu}_{\cM}}{2\pi}}} \dif x \bigg] \frac{\dif \mu(y)}{\norm{\mu}_{\cM}}.
\]
If \(N = 2\) and if \(\norm{\mu}_{\cM(\Omega)} < 4\pi\), then
\[
\int\limits_\Omega \e^{v(x)} \dif x \le C \int\limits_{\Omega}  \frac{\dif \mu(y)}{\norm{\mu}_{\cM}} = C.
\]
This proves the estimate.
\end{proof}

An inspection of the proof shows that the \(L^1\) estimate of \(\e^v\) holds with constant \(C = \frac{\pi d^2}{1 - \norm{\mu}_{\cM(\Omega)}/4\pi}\).

\begin{lemma}
\label{lemmaExponentialDimension2}
If \(N = 2\) and if \(\mu \in \cM(\Omega)\) is such that for every \(x \in \Omega\), \(\mu(\{x\}) < 4 \pi\), then the solution of the linear Dirichlet problem
\[
\left\{
\begin{alignedat}{2}
- \Delta v & = \mu \quad && \text{in \(\Omega\),}\\
v & = 0  \quad && \text{on \(\partial\Omega\),}
\end{alignedat}
\right.
\]
satisfies \(\e^{v} \in L^1(\Omega)\).
\end{lemma}

In particular, if \(\mu\) is an \(L^1\) function, then it is always true that \(\e^v \in L^1(\Omega)\), however this is a qualitative information since there is no control \(\norm{\e^v}_{L^1(\Omega)}\) in terms of \(\norm{\mu}_{L^1(\Omega)}\).

\begin{proof}
Let \(\mu^+ = \max{\{\mu, 0\}}\).
From the assumption on \(\mu\), there exists \(\eta > 0\) such that for every \(a \in \Omega\),
\[
\mu^+(B(a; \eta) \cap \Omega) < 4\pi.
\]
Given \(a \in \Omega\), let \(v_1\) be the Newtonian potential generated by \(\mu^+\lfloor_{B(a; \eta)}\) and let \(v_2\) be the Newtonian potential generated by \(\mu^+\lfloor_{\Omega \setminus B(a; \eta)}\) with \(d \ge \diam{\Omega}\).

\begin{Claim}
The solution \(v\) of the linear Dirichlet problem satisfies \(v \le v_1 + v_2\) in \(\Omega\).
\end{Claim}

Let us temporarily assume the claim and conclude the proof.
Let \(0 < \theta < 1\).
Since \(v_2\) is a harmonic function in \(B(a; \eta)\), it is bounded in \(B(a; \theta\eta)\). Thus, 
\[
\e^v \le C \e^{v_1}
\]
in \(B(a; \theta\eta) \cap \Omega\). 
By the Brezis-Merle inequality, \(\e^{v_1} \in L^1(\Omega)\).  
We deduce that \(\e^v \in L^1(B(a; \theta\eta) \cap \Omega)\). 
To conclude the argument we cover \(\Omega\) with finitely many balls of radius \({\theta\eta}\) and we deduce that \(\e^v \in L^1(\Omega)\).

It remains to prove the claim.

\begin{proof}[Proof of the claim]
Since 
\[
-\Delta(v_1 + v_2) = \mu^+,
\]
we have
\[
\Delta(v - v_1 - v_2) \ge 0
\]
in the sense of distributions in \(\Omega\).
For \(d \ge \diam{\Omega}\), the function \(v_1 + v_2\) is positive in \(\Omega\). Thus,
\[
(v - v_1 - v_2)^+ \le v^+ \le \abs{v}.
\]
Thus, for every \(\epsilon > 0\),
\[
\frac{1}{\epsilon} \int\limits_{\{x \in \Omega : d(x, \partial\Omega) < \epsilon\}} (v - v_1 - v_2)^+
\le
\frac{1}{\epsilon} \int\limits_{\{x \in \Omega : d(x, \partial\Omega) < \epsilon\}} \abs{v}.
\]
Since \(v\) satisfies a Dirichlet problem with zero boundary condition, the quantity in the right-hand side converges to zero as \(\epsilon\) tends to zero  (proposition~\ref{propositionDirichletBoundaryCondition}), whence the quantity in the left-hand side also converges to zero.
Applying proposition~\ref{propositionDistributionC0Infty}, we deduce that
\[
\Delta(v - v_1 - v_2) \ge 0
\]
in the sense of \((C_0^\infty(\Omega))'\).
By the weak maximum principle (proposition~\ref{propositionWeakMaximumPrinciple}),  the claim follows.
\end{proof}

This concludes the proof of the lemma.
\end{proof}

\begin{corollary}
Let \(g : \R \to \R\) be a continuous function satisfying the sign condition and such that for every \(t \in \R\),
\[
\abs{g(t)} \le C(\e^{t} + 1).
\]
If \(N = 2\) and if \(\mu \in \cM(\Omega)\) is such that for every \(x \in \Omega\), \(\mu(\{x\}) < 4\pi\), then the nonlinear Dirichlet problem has a  solution.
\end{corollary}

\begin{proof}
Let \(\overline v\) be the solution of the linear Dirichlet problem with datum \(\max{\{\mu, 0\}}\).
By the previous lemma, \(\e^{\overline v} \in L^1(\Omega)\), whence \(g(\overline v) \in L^1(\Omega)\).
By the weak maximum principle (proposition~\ref{propositionWeakMaximumPrinciple}), \(\overline v \ge 0\).
By the sign condition, we deduce that \(\overline v\) is a supersolution of the nonlinear Dirichlet problem with datum  \(\mu\). 

By the weak maximum principle, the solution \(\underline v\) of the linear Dirichlet problem with datum \(\min{\{\mu, 0\}}\) is such that \(\underline{v} \le 0\), whence
\(\e^{\underline v} \in L^{\infty}(\Omega)\) and \(g(\underline{v}) \in L^1(\Omega)\).
By the sign condition, we deduce that \(\underline v\)
is a subsolution of the nonlinear Dirichlet problem with datum \(\mu\).

For every \(v \in L^1(\Omega)\) such that \(\underline{v} \le v \le \overline{v}\) we have \(g(v) \in L^1(\Omega)\).
By the method of sub and supersolution, the nonlinear Dirichlet problem with datum \(\mu\) has a solution \(u\) such that \(\underline{v} \le u \le \overline{v}\).
\end{proof}

\begin{proof}[Proof of proposition~\ref{propositionExistenceSolutionExponentialDimension2}]
Let \((\mu_n)_{n \in \N}\) be a nondecreasing sequence of measures converging strongly to \(\mu\) such that for every \(n \in \N\) and for every \(x \in \Omega\), \(\mu_n(\{x\}) < 4\pi\). For instance, we may take
\[
\mu_n = \alpha_n \max{\{\mu, 0\}} + \min{\{\mu, 0\}},
\]
where \((\alpha_n)_{n \in \N}\) is a nondecreasing sequence of numbers converging to \(1\).

As in the proof of proposition~\ref{propositionExistenceBarasPierre},
there exists a nondecreasing sequence \((u_n)_{n \in \N}\) in \(L^1(\Omega)\) such that for every \(n \in \N\), \(u_n\) is a solution of the Dirichlet problem with datum \(\mu_n\).

By the comparison principle (lemma~\ref{lemmaEstimateComparison}), every function \(u_n\) is bounded from above by the solution of the linear Dirichlet problem with datum \(\max{\{\mu, 0\}}\).
Thus, by the monotone convergence theorem the sequence \((u_n)_{n \in \N}\) converges in \(L^1(\Omega)\) to its pointwise limit \(u\).
By the absorption estimate (lemma~\ref{lemmaEstimateAbsorption}), the sequence \((g(u_n))_{n \in \N}\) is bounded in \(L^1(\Omega)\). 
The integrability condition implies that the sequence \((g(u_n))_{n \in \N}\) converges to \(g(u)\) in \(L^1(\Omega)\) (see corollary~\ref{corollaryIntegrabilityConditionMonotoneConvergence}).
We conclude that \(u\) is a solution of the nonlinear Dirichlet problem with datum \(\mu\).
\end{proof}

If the nonlinearity \(g\) is given by 
\[
g(t) = \e^t - 1,
\]
or in the case of the scalar Chern-Simons equation \cite{Yan:01},
\[
g(t) = \e^{t/2} (\e^{t/2} - 1),
\] 
then the condition given by proposition~\ref{propositionExistenceSolutionExponentialDimension2} characterizes all finite measures for which the nonlinear Dirichlet problem has a solution.
This is a consequence of the fact that if \(v\) is the solution of the linear Dirichlet problem
\[
\left\{
\begin{alignedat}{2}
- \Delta v & = \nu	&& \quad \text{in \(\Omega\),}\\
v & = 0	&& \quad \text{on \(\partial\Omega\),}\\
\end{alignedat}
\right.
\]
for some \(\nu \in \cM(\Omega)\) and if \(a \in \Omega\) is such that \(\nu(\{a\}) > 0\), then \(\nu\) has a Dirac mass at \(a\) --- a nontrivial atomic part --- and in a neighborhood of \(a\) the function \(v\) behaves like the fundamental solution of the Laplacian multiplied by \(\nu(\{a\})\),
\[
v(x) \sim \frac{\nu(\{a\})}{2\pi} \log{\frac{1}{\abs{x - a}}}.
\]
If \(\e^v \in L^1(\Omega)\), then we cannot have \(\nu(\{a\}) > 4\pi\).
The reader will find the rigorous argument in \citelist{\cite{Vaz:83}*{section~5}  \cite{BLOP:05}*{section~5}}.


\section{Higher dimensional case}

The existence of solutions in higher dimensions for exponential nonlinearities is due to Bartolucci, Leoni, Orsina and Ponce \cite{BLOP:05}*{theorem~1}:

\begin{proposition}
\label{propositionExistenceSolutionExponentialDimension3}
Let \(g : \R \to \R\) be a continuous function satisfying the sign condition, the integrability condition and such that for every \(t \in \R\),
\[
\abs{g(t)} \le C(\e^{t} + 1).
\]
If \(N \ge 3\), then the Dirichlet problem has a solution for every \(\mu \in \cM(\Omega)\) such that \(\mu \le 4 \pi \cH^{N-2}\).
\end{proposition}

The inequality \(\mu \le 4 \pi \cH^{N-2}\) is meant in the sense of measures, meaning that for every Borel set \(A \subset \Omega\),
\[
\mu(A) \le 4 \pi \cH^{N-2}(A).
\]
Since \(\cH^{N-2}\) is not a finite measure --- not even a \(\sigma\)-finite measure --- the right-hand side may be infinite. 
This is the case for every nonempty open set \(A\).
Note that every \(L^1\) function trivially satisfies this inequality since if \(\cH^{N-2}(A)\) is finite, then \(A\) is negligible with respect to the Lebesgue measure.
More generally, every measure which is diffuse with respect to the \(W^{1, 2}\) capacity also satisfies this inequality since if \(\cH^{N-2}(A)\) is finite, then \(\capt_{W^{1, 2}}{(A)} = 0\).
In this sense, we recover two cases for which the Dirichlet problem has a solution regardless of the nonlinearity \(g\) (see proposition~\ref{propositionBrezisStrauss} and proposition~\ref{propositionExistenceSolutionDiffuseMeasure}).

The conclusion of the statement is still correct when \(N = 2\) since the inequality \(\mu \le 4 \pi \cH^{0}\) amounts to saying that for every \(x \in \Omega\), 
\[
\mu(\{x\}) \le 4\pi \cH^0(\{x\}) = 4\pi
\]
and we recover Vázquez's condition.

\medskip

The main ingredient of the proof of proposition~\ref{propositionExistenceSolutionExponentialDimension3} is the following inequality established by Bartolucci-Leoni-Orsina-Ponce~\cite{BLOP:05}*{theorem~2}, which is the counterpart in higher dimensions of the Brezis-Merle inequality.

\begin{lemma}
\label{lemmaEstimationBLOP}
Let \(N \ge 3\).
Given a nonnegative measure \(\mu \in \cM(\Omega)\), let \(v : \R^N \to \R\) be the Newtonian potential generated by \(\mu\), 
\[
v(x) = 
\frac{1}{N(N-2)\omega_N} \int\limits_{\Omega} \left( \frac{1}{\abs{x - y}^{N-2}} - \frac{1}{d^{N-2}} \right) \dif\mu(y),
\]
where \(d \ge \diam{\Omega}\) and \(\omega_N\) is the volume of the unit ball in \(\R^N\).
If \(\mu \le \alpha \cH^{N-2}_\infty\) for some \(\alpha < 4\pi\), then \(\e^v \in L^1(\Omega)\) and
\[
\norm{\e^v}_{L^1(\Omega)} \le C,
\]
for some constant \(C > 0\) depending on \(N\), \(\alpha\), \(\norm{\mu}_{\cM(\Omega)}\) and \(\meas{\Omega}\).
\end{lemma}

The notation \(\cH_\infty^{N-2}\) indicates the Hausdorff content of dimension \(N-2\), whose definition is recalled in the appendix.
This is an outer measure such that for every \(x \in \R^N\) and for every \(r > 0\),
\[
\cH_\infty^{N-2}(B(x; r)) = \omega_{N-2} r^{N-2};
\]
see proposition~\ref{propositionHausdorffMeasureBalls}.

The Hausdorff content is not a measure since it is not additive.
The condition \(\mu \le \alpha \cH^{N-2}_\infty\) means that the inequality must hold for every Borel set \(A \subset \Omega\).
According to proposition~\ref{propositionHausdorffContentMorrey}, this is equivalent to have for every \(x \in \R^N\) and for every \(r > 0\),
\[
\mu(B(x; r) \cap \Omega) \le \alpha \omega_{N-2} r^{N-2}.
\]

\medskip

The identity below which we will use in the proof of the lemma gives a unified way to deal with the Newtonian potential without having to distinguish whether \(N = 2\) or \(N \ge 3\).

\begin{sublemma}
Let \(N \ge 2\) and let \(\mu \in \cM(\Omega)\) be a nonnegative measure.
If \(v : \R^N \to \R\) is the Newtonian potential generated by \(\mu\), then for every \(x \in \Omega\),
\[
v(x) = \frac{1}{N \omega_N} \int_0^d \frac{\mu(B(x; r) \cap \Omega)}{r^{N - 1}} \dif r.
\]
\end{sublemma}

This representation of the Newtonian potential can be found for instance in \cite{Fro:35}*{\S 46}.

\begin{proof}
By Cavalieri's principle \cite{Wil:07}*{corollaire~6.34}, for every nonnegative function \(f \in L^1(\Omega; \mu)\) we have
\[
\int\limits_\Omega f \dif \mu = \int_0^{+\infty} \mu(\{y \in \Omega : f(y) > t \}) \dif t.
\]
For every \(x \in \Omega\), the function \(f : \Omega \to \R\) defined for \(y \in \Omega\) by 
\[
f(y) = \frac{1}{\abs{x - y}^{N-2}} - \frac{1}{d^{N-2}}
\]
is nonnegative. Applying Cavalieri's principle, we have
\[
v(x) = \frac{1}{N (N-2) \omega_N} \int_0^{+\infty} \mu\big(\big\{y \in \Omega : \tfrac{1}{\abs{x - y}^{N-2}} - \tfrac{1}{d^{N-2}} > t \big\}\big) \dif t.
\]
Using the change of variable \(t = \tfrac{1}{r^{N-2}} - \tfrac{1}{d^{N-2}}\), we get
\[
\begin{split}
v(x) 
& = \frac{1}{N\omega_N} \int_0^d \mu\big(\{y \in \Omega : \abs{x - y}< r \}\big) \frac{1}{r^{N-1}} \dif r\\
& = \frac{1}{N\omega_N} \int_0^d \frac{\mu(B(x; r) \cap \Omega)}{r^{N-1}} \dif r.
\end{split}
\]
This gives the identity.
\end{proof}

The main ingredient in the proof of the Bartolucci-Leoni-Orsina-Ponce inequality is the following one dimensional estimate:

\begin{sublemma}
\label{sublemmaEstimationBLOP1D}
Let \(\alpha \ge 0\), \(s \ge 0\), \(d > 0\) and let \(f : [0, d] \to \R\) be a nondecreasing function. If for every \(0 \le r \le d\), 
\[
0 \le f(r) \le \alpha r^s,
\]
then
\[
\int_0^d \frac{f(r)}{r^{s + 1}} \dif r \le \log{\left(1 + C d^\alpha \int_0^d \frac{f(r)}{r^{s + 1 + \alpha}} \dif r \right)},
\]
for some constant \(C > 1\) depending on \(s\) and \(f(d)\).
\end{sublemma}

There is a natural candidate one would wish to apply this inequality to: the function \(f : [0, d] \to \R\) defined for \(r \in [0, d]\) by
\[
f(r) = \alpha r^s.
\]
In this case, both sides become infinite.
A second attempt is to consider for every parameter \(0 < \epsilon < d\), a function \(f_\epsilon : [0, d] \to \R\) defined for \(r \in [0, d]\) by
\[
f_\epsilon(r) =
\begin{cases}
0			& \text{if \(0 \le r < \epsilon\),}\\
\alpha r^s & \text{if \(\epsilon \le r \le d\).}
\end{cases}
\]
In this case,
\[
\int_0^d \frac{f_\epsilon(r)}{r^{s + 1}} \dif r = \alpha \log{\left( \frac{d}{\epsilon} \right)},
\]
while
\[
\int_0^d \frac{f_\epsilon(r)}{r^{s + 1 + \alpha}} \dif r = \frac{1}{d^\alpha} \bigg[ \bigg( \frac{d}{\epsilon} \bigg)^\alpha - 1 \bigg]
\]
and thus
\[
\log{\left(1 + C d^\alpha \int_0^d \frac{f_\epsilon(r)}{r^{s + 1 + \alpha}} \dif r \right)} = \alpha \log{\left( \frac{d}{\epsilon} \right)} + O(1),
\]
where \(O(1)\) denotes a quantity that remains bounded as \(\epsilon\) tends to zero. Thus, both terms in the inequality tend to infinity at the same rate.

\begin{proof}[Proof of sublemma~\ref{sublemmaEstimationBLOP1D}]
We first assume that \(f\) vanishes in a neighborhood of \(0\) and is smooth in \([0, d]\). 
By scaling it suffices to prove the lemma with \(d = 1\). 
Let \(g : [0, 1] \to \R\) be the function defined for \(r \in [0, 1]\) by
\[
g(r) = \sup_{0 \le t \le r}{\frac{f(t)}{t^s}}.
\]
Thus,
\[
\int_0^1 \frac{f(r)}{r^{s + 1}} \dif r \le \int_0^1 \frac{g(r)}{r} \dif r.
\]
Since \(g\) is absolutely continuous and vanishes in a neighborhood of \(0\), by integration by parts we have
\[
\int_0^1 \frac{g(r)}{r} \dif r 
=  (\log{r}) g(r) \Big|_{r = 0}^1 - \int_0^1 (\log{r}) g'(r) \dif r 
= \int_0^1 \Big(\log{\frac{1}{r}}\Big) g'(r) \dif r.
\]
Thus,
\[
\begin{split}
\int_0^1 \frac{f(r)}{r^{s + 1}} \dif r 
& \le \int_0^1 \Big(\log{\frac{1}{r}}\Big) g'(r) \dif r \\
& = \int_0^1 \Big(g(1)\log{\frac{1}{r}}\Big) \frac{g'(r)}{g(1)} \dif r 
 = \int_0^1 \Big(\log{\frac{1}{r^{g(1)}}}\Big) \frac{g'(r)}{g(1)} \dif r.
\end{split}
\]
Since \(g\) is nondecreasing and \(g(0) = 0\), the function \(g'/g(1)\) is a probability density in \([0, 1]\). By the Jensen inequality for concave functions, we have
\[
\int_0^1 \frac{f(r)}{r^{s + 1}} \dif r \le \log \left( \int_0^1 {\frac{1}{r^{g(1)}}} \frac{g'(r)}{g(1)} \dif r \right).
\]

\begin{Claim}
For almost every \(r \in [0, 1]\),
\[
0 \le g'(r) \le \frac{f'(r)}{r^s}.
\]
\end{Claim}

We temporarily assume the claim. 
From the estimate in the claim we have
\[
\int_0^1 \frac{f(r)}{r^{s + 1}} \dif r \le \log \left( \frac{1}{g(1)} \int_0^1 {\frac{f'(r)}{r^{s + g(1)}}} \dif r \right).
\]
By integration by parts,
\[
\int_0^1 \frac{f'(r)}{r^{s + g(1)}}  \dif r 
 = \left. \frac{f(r)}{r^{s + g(1)}} \right|_{r = 0}^1 + (s + g(1)) \int_0^1 \frac{ f(r)}{r^{s + 1 + g(1)}} \dif r
\]
and this implies
\[
\int_0^1 \frac{f(r)}{r^{s + 1}} \dif r \le \log \left( \frac{f(1)}{g(1)} + \left(\frac{s}{g(1)} + 1 \right) \int_0^1 \frac{f(r)}{r^{s + 1 + g(1)}} \dif r \right).
\]
Since \(f(1) \le g(1)\) and \(g(1) \le \alpha\), we get
\[
\int_0^1 \frac{f(r)}{r^{s + 1}} \dif r \le \log \left( 1 + \Big(\frac{s}{f(1)} + 1 \Big) \int_0^1 \frac{f(r)}{r^{s + 1 + \alpha}} \dif r \right).
\]
This gives the conclusion.

We now return and prove the claim:

\begin{proof}[Proof of the claim]
For almost every \(r \in [0, 1]\),
\[
0 \le g'(r) \le \left( \frac{\dif}{\dif r} \frac{f(r)}{r^s} \right)^+.
\]
The proof of this estimate relies on the fact that for every \(x < y\) in \([0, 1]\), there exist \(\tilde x \le \tilde y\) in \([x, y]\) 
such that 
\[
0 \le g(y) - g(x) \le  \left[\frac{f(\tilde y)}{\tilde y ^s} - \frac{f(\tilde x)}{\tilde x ^s}\right]^+;
\]
we refer the reader to \cite{BLOP:05}*{lemma~1} for the details. 
Since \(f\) is nonnegative and \(s \ge 0\),
\[
\frac{\dif}{\dif r} \frac{f(r)}{r^s} \le \frac{f'(r)}{r^s}.
\]
Since \(f\) is nondecreasing, the claim follows.
\end{proof}

This establish the inequality when \(f\) vanishes in a neighborhood of \(0\) and \(f\) is smooth in \([0, d]\). 
The general case follows by approximation of \(f\) by a sequence of functions \((f_n)_{n \in \N}\) of this type. 
This approximation can be done so that \((f_n)_{n \in \N}\) converges almost everywhere to \(f\) in \([0, d]\) and for every \(n \in \N\) and for every \(0 \le r \le d\), \(0 \le f_n(r) \le \alpha r^s\).
The proof of the estimate is complete.
\end{proof}

\begin{proof}[Proof of lemma~\ref{lemmaEstimationBLOP}]
We begin by writing the Newtonian potential for every \(x \in \Omega\) as
\[
v(x) = \frac{1}{N \omega_N} \int_0^d \frac{\mu(B(x; r) \cap \Omega)}{r^{N - 1}} \dif r.
\]
Given \(x \in \Omega\), we apply the previous lemma with \(s = N-2\) to the function \(f : [0, d] \to \R\) defined for \(r \in [0, d]\) by 
\[
f(r) = \frac{1}{N \omega_N} \mu(B(x; r) \cap \Omega).
\]
Since \(\mu \le \alpha \cH^{N-2}_\infty\), we have for every \(0 \le r \le d\),
\[
\begin{split}
f(r) 
= \frac{1}{N\omega_N}\mu(B(x; r) \cap \Omega) 
& \le \frac{1}{N\omega_N} \alpha \cH^{N-2}_\infty (B(x; r) \cap \Omega) \\
& \le \frac{\alpha}{N\omega_N}  \omega_{N-2} r^{N-2} = \frac{\alpha}{2\pi} r^{N-2}.
\end{split}
\]
We get
\[
v(x) \le \log{\left(1 + \NewConstant d^\frac{\alpha}{2\pi} \int_0^d \frac{\mu(B(x; r) \cap \Omega)}{r^{N - 1 + \frac{\alpha}{2\pi}}} \dif r \right)}.
\]
Thus,
\[
\e^{v(x)} \le 1 + \SameConstant d^\frac{\alpha}{2\pi}\int_0^d \frac{\mu(B(x; r) \cap \Omega)}{r^{N - 1 + \frac{\alpha}{2\pi}}} \dif r.
\]
Integrating on both sides with respect to \(x\), we get by Fubini's theorem,
\[
\int\limits_\Omega \e^{v(x)} \dif x \le \meas{\Omega} + \SameConstant d^\frac{\alpha}{2\pi} \int_0^d \bigg( \int\limits_\Omega \mu(B(x; r) \cap \Omega) \dif x \bigg) \frac{1}{r^{N - 1 + \frac{\alpha}{2\pi}}} \dif r.
\]
Applying once again Fubini's theorem,
\[
\begin{split}
\int\limits_\Omega \mu(B(x; r) \cap \Omega) \dif x 
& = \int\limits_\Omega \bigg( \int\limits_{B(x; r) \cap \Omega} \dif\mu(z) \bigg) \dif x\\
& = \int\limits_\Omega \bigg( \int\limits_{B(z; r) \cap \Omega} \dif x \bigg) \dif\mu(z).
\end{split}
\]
Thus,
\[
\int\limits_\Omega \mu(B(x; r) \cap \Omega) \dif x \le \int\limits_\Omega \omega_N r^N \dif\mu(z) = \omega_N r^N \norm{\mu}_{\cM(\Omega)}.
\]
Therefore, if \(\alpha < 4\pi\), then
\[
\int\limits_\Omega \e^{v(x)} \dif x \le \meas{\Omega} + \Constant d^\frac{\alpha}{2\pi} \int_0^d \frac{1}{r^{- 1 + \frac{\alpha}{2\pi}}} \dif r \le \Constant.
\]
This gives the conclusion.
\end{proof}

\begin{lemma}
If \(\mu \in \cM(\Omega)\) is such that \(\mu \le \alpha \cH_\delta^{N-2}\) for some \(\alpha < 4\pi\) and \(\delta > 0\), then the solution of the linear Dirichlet problem
\[
\left\{
\begin{alignedat}{2}
- \Delta v & = \mu \quad && \text{in \(\Omega\),}\\
v & = 0  \quad && \text{on \(\partial\Omega\),}
\end{alignedat}
\right.
\]
satisfies \(\e^{v} \in L^1(\Omega)\).
\end{lemma}

\begin{proof}
We begin with the following claim:

\begin{Claim}
The restriction of \(\mu\) to \(B(a; \delta) \cap \Omega\) satisfies
\[
\mu\lfloor_{B(a; \delta) \cap \Omega} \le \alpha \cH_\infty^{N-2}.
\]
\end{Claim}

We temporarily assume the claim.
Given \(a \in \Omega\), let \(v_1\) be the Newtonian potential generated by \(\mu^+\lfloor_{B(a; \delta) \cap \Omega}\) and let \(v_2\) be the Newtonian potential generated by \(\mu^+\lfloor_{\Omega \setminus B(a; \delta)}\).
As in the proof of lemma~\ref{lemmaExponentialDimension2}, by the weak maximum principle we have
\[
v \le v_1 + v_2
\]
in \(\Omega\).

Let \(0 < \theta < 1\).
Since \(v_2\) is a harmonic function in \(B(a; \eta)\), it is bounded in \(B(a; {\theta\eta})\). Thus, 
\[
\e^v \le C \e^{v_1}
\]
in \(B(a; {\theta\eta}) \cap \Omega\). 
By the Bartolucci-Leoni-Orsina-Ponce inequality, \(\e^{v_1} \in L^1(\Omega)\).
We deduce that \(\e^v \in L^1(B(a; {\theta\eta}) \cap \Omega)\). 
To conclude the argument we cover \(\Omega\) with finitely many balls of radius \({\theta\eta}\) and we deduce that \(\e^v \in L^1(\Omega)\).

\begin{proof}[Proof of the claim]
For every Borel set \(A \subset \Omega\),
\[
\mu\lfloor_{B(a; \delta) \cap \Omega}(A) = \mu(A \cap B(a; \delta)) \le \alpha \cH_\delta^{N-2}(A \cap B(a; \delta)).
\]
Since the set \(A \cap B(a; \delta)\) is contained in a ball of radius \(\delta\), for every \(s \ge 0\) we have%
\footnote{Actually, equality holds.
Given a sequence of balls \((B(x_n; r_n))_{n \in \N}\) covering \(A \cap B(a; \delta)\) without any restriction on the radii \(r_n\), we may replace each ball \(B(x_n; r_n)\) such that \(r_n > \delta\) by \(B(a; \delta)\).
Alternatively,
\(\cH_\infty^{s}(A \cap B(a; \delta))\) is at most \(\omega_s \delta^s\) so if we wish to compute \(\cH_\infty^{s}(A \cap B(a; \delta))\), it is not interesting to use balls of radii larger than \(\delta\) to cover \(A \cap B(a; \delta)\).}
\[
\cH_\delta^{s}(A \cap B(a; \delta)) \le \cH_\infty^s(A \cap B(a; \delta)),
\]
and by monotonicity of the Hausdorff content we deduce that
\[
\mu\lfloor_{B(a; \delta) \cap \Omega}(A) \le \alpha \cH_\infty^{N - 2}(A).
\]
This proves the claim.
\end{proof}

The proof of the lemma is complete.
\end{proof}

The following corollary can be proved as in the two dimensional case:

\begin{corollary}
Let \(g : \R \to \R\) be a continuous function satisfying the sign condition and such that for every \(t \in \R\),
\[
\abs{g(t)} \le C(\e^{t} + 1).
\]
If \(N \ge 3\) and if \(\mu \in \cM(\Omega)\) is such that \(\mu \le \alpha \cH_\delta^{N-2}\) for some \(\alpha < 4\pi\) and \(\delta > 0\), then the nonlinear Dirichlet problem has a solution.
\end{corollary}

In the next section we discuss the following result concerning strong approximation of measures in terms of Hausdorff outer measures:

\begin{proposition}
\label{propositionIncreasingSequenceHausdorffMeasure}
Let \(0 < \alpha < + \infty\) and \(0 \le s < + \infty\), and let \(\mu \in \cM(\Omega)\) be a nonnegative measure. 
If \(\mu \le \alpha \cH^{s}\), then there exists a sequence \((\mu_n)_{n \in \N}\) in \(\cM(\Omega)\) of nonegative measures such that
\begin{enumerate}[\((i)\)]
\item for every \(n \in \N\), there exist \(\alpha_n < \alpha\) and \(\delta_n > 0\) such that \(\mu_n \le \alpha_n \cH_{\delta_n}^s\),
\item the sequence \((\mu_n)_{n \in \N}\) is nondecreasing and converges strongly to \(\mu\) in \(\cM(\Omega)\).
\end{enumerate} 
\end{proposition}

Strictly speaking, we shall consider the case where \(\Omega = \R^N\), but the case  of an open set \(\Omega\) can be reduced to this one by extending the measure \(\mu\) to the whole space \(\R^N\).

\begin{proof}[Proof of proposition~\ref{propositionExistenceSolutionExponentialDimension3}]
By the property of strong approximation of Hausdorff measures, there exists a nondecreasing sequence  \((\mu_n)_{n \in \N}\) of measures converging strongly to \(\mu\) such that for every \(n \in \N\),
\[
\mu_n \le \alpha_n \cH_{\delta_n}^{N-2}
\]
where \(\alpha_n < 4\pi\) and \(\delta_n > 0\).
The rest of the argument is the same as in the two dimensional case.
\end{proof}

The condition given by proposition~\ref{propositionExistenceSolutionExponentialDimension3} does not characterize all measures for which the nonlinear Dirichlet problem has a solution, even in the case where the nonlinearity \(g : \R \to \R\) is given by
\[
g(t) = \e^t - 1.
\]
For every \(N \ge 3\), the author has constructed an example of a positive measure \(\mu\) supported in a compact set of zero Hausdorff measure \(\cH^{N-2}\) 
such that the solution \(v\) of the linear Dirichlet problem with datum \(\mu\) satisfies \(\e^v \in L^1(\Omega)\) \cite{Pon:04a}*{theorem~3}.
By the method of sub and supersolution (proposition~\ref{propositionMethodSubSuperSolutions}), the Dirichlet problem with exponential nonlinearity and datum \(\mu\) has a solution.
In this case, the inequality \(\mu \le 4\pi \cH^{N-2}\) is not satisfied by the support of \(\mu\).

This raises the following question:

\begin{openproblem}
Characterize the set of finite measures \(\mu\) for which the nonlinear Dirichlet problem has a solution when the nonlinearity \(g : \R \to \R\) is given by
\[
g(t) = \e^t - 1.
\]
\end{openproblem}

The answer to this problem also gives the characterization of good measures for the Chern-Simons scalar equation with nonlinearity 
\[
g(t) = \e^{t/2} (\e^{t/2} - 1).
\]
Véron~\cite{Ver:12} has recently made some progress in this direction.
He introduced a condition in terms of an Orlicz capacity which presumably includes the measures \(\mu\) constructed in \cite{Pon:04a}, not covered by proposition~\ref{propositionExistenceSolutionExponentialDimension3}.
As a drawback, Véron's condition is less transparent than ours.


\section{Strong approximation in terms of density}

The goal of this section is to prove the following result due to the author~\cite{Pon:12}. 
It was inspired by a technical lemma in \cite{BLOP:05}*{lemma~2} which was used in the proof of existence of solutions of the Dirichlet problem with exponential nonlinearity in any dimension (see proposition~\ref{propositionExistenceSolutionExponentialDimension3} above):

\begin{proposition}
\label{propositionStrongApproximationHausdorffMeasureNonMonotone}
Let \(0 < \alpha < + \infty\) and \(0 \le s < + \infty\), and let \(\mu \in \cM(\R^N)\) be a nonnegative measure. 
If \(\mu \le \alpha \cH^{s}\), then there exists a sequence \((\mu_n)_{n \in \N}\) in  \(\cM(\R^N)\) of nonnegative measures such that
\begin{enumerate}[\((i)\)]
\item for every \(n \in \N\), there exist \(\alpha_n < \alpha\) and \(\delta_n > 0\) such that \(\mu_n \le \alpha_n \cH_{\delta_n}^s\),
\item \((\mu_n)_{n \in \N}\) converges strongly to \(\mu\) in \(\cM(\R^N)\).
\end{enumerate}
\end{proposition}

The condition
\[
\mu_n \le \alpha_n \cH_{\delta_n}^s
\]
gives a uniform upper bound on the \(s\) dimensional density of \(\mu\):
for every ball \(B(x; r) \subset \R^N\) such that \(0 < r \le \delta_n\),
\[
\frac{\mu_n(B(x; r))}{\omega_s r^s} \le \alpha_n.
\]

This proposition characterizes finite measures \(\mu\) which satisfy the inequality
\[
\mu \le \alpha\cH^s
\]
in the sense that if there exists a sequence \((\mu_n)_{n \in \N}\) satisfying properties \((i)\) and \((ii)\) of the proposition, then for every Borel set \(A \subset \R^N\),
\[
\mu_n(A) \le \alpha_n \cH_{\delta_n}^s(A).
\]
Letting \(n\) tend to infinity, we get 
\[
\mu(A) \le \alpha\cH^s(A).
\]

\medskip
Compared to proposition~\ref{propositionIncreasingSequenceHausdorffMeasure}, the sequence \((\mu_n)_{n \in \N}\) above need not be monotone.
In order to get this extra conditions, the proof becomes more technical and we prefer not to present it here. 
For the sake of existence of solutions of the nonlinear Dirichlet problem, this statement is enough but it requires an additional argument which relies on the concept of reduced measure (see corollary~\ref{corollaryGoodMeasuresClosed}).
Under the additional assumption that the nonlinearity \(g\) is nondecreasing, no reduced measure is required, and this follows from the contraction property discussed in the proof of proposition~\ref{propositionBrezisStrauss}.

\medskip

Another formulation of proposition~\ref{propositionStrongApproximationHausdorffMeasureNonMonotone} is the following:

\begin{proposition}
\label{propositionStrongApproximationHausdorffMeasure}
Let \(0 < \alpha < + \infty\) and \(0 \le s < + \infty\), and let \(\mu \in \cM(\R^N)\) be a nonnegative measure.  
If \(\mu \le \alpha \cH^{s}\), then for every \(\epsilon > 0\) and for every \(\beta > \alpha\), there exists a Borel set \(E \subset \R^N\) such that
\begin{enumerate}[\((i)\)]
\item there exists \(\delta > 0\) such that \(\mu\lfloor_{E} \le \beta \cH^s_{\delta}\),
\item \(\mu(\R^N \setminus E) \le \epsilon\).
\end{enumerate}
\end{proposition}

This proposition is reminiscent of the property of uniform convergence of the Hausdorff outer measures \(\cH_\delta^s\) to the Hausdorff measure \(\cH^s\) on sets of finite Hausdorff measure (see proposition~\ref{propositionUniformConvergenceHausdorff}).

\medskip

We first prove proposition~\ref{propositionStrongApproximationHausdorffMeasure}, and then we deduce proposition~\ref{propositionStrongApproximationHausdorffMeasureNonMonotone} as a consequence.

The main ingredient is the following lemma:

\begin{lemma}
\label{lemmaHausdorffTechnical}
Let \(\mu \in \cM(\R^N)\) be a nonnegative measure.
For every nonnegative Borel outer measure \(T\), there exists a Borel set \(E \subset \R^N\) such that
\[
\mu\lfloor_{E} \le T
\quad 
\text{and}
\quad
T(\R^N \setminus E) \le \mu(\R^N \setminus E).
\]
\end{lemma}

We recall the a nonnegative outer measure \(T\) is a set function with values into \([0, +\infty]\) such that
\begin{enumerate}[\((a)\)]
\item \(T(\emptyset) = 0\),
\item if \(A \subset B\), then \(T(A) \le T(B)\),
\item for every sequence \((A_n)_{n \in \N}\), \(T\big(\bigcup\limits_{k=0}^\infty A_k\big) \le \sum\limits_{k=0}^\infty T(A_k)\).
\end{enumerate}

This lemma is essentially \cite{BLOP:05}*{lemma~2} rewritten for outer measures;
the outer measure \(T\) we have in mind is \(\beta \cH^s_{\delta}\) with \(\beta > 0\) \cite{Fol:99}*{proposition~1.10}.
When \(T\) is a finite measure, this lemma follows from the classical Hahn decomposition theorem \cite{Fol:99}*{theorem~3.3} applied to the measure \(\mu - T\), in which  case the set \(E\) may be chosen so that
\[
\mu\lfloor_{E} \le T
\quad 
\text{and}
\quad
T\lfloor_{\R^N \setminus E} \le \mu.
\]

The main idea of the proof is that if the inequality \(\mu \le T\) does not hold on every Borel set, then there exists some Borel set \(F \subset \R^N\) such that
\(T(F) < \mu(F)\) and we try to choose \(F\) so that \(\mu(F)\) is as large as possible.
Since \(\mu\) is a finite measure, we eventually exhaust the part of \(\mu\) that prevents the inequality \(\mu \le T\) to hold.

\begin{proof}[Proof of lemma~\ref{lemmaHausdorffTechnical}]
Let \(0 < \theta < 1\).
By induction, there exists a sequence \((F_n)_{n \in \N}\) of disjoint Borel sets of \(\R^N\) such that
\begin{enumerate}[\((a)\)]
\item for every \(n \in \N\), \(T(F_n) \le \mu(F_n)\),
\item for every \(n \in \N_*\), \(\mu(F_n) \ge \theta \epsilon_n\), where
\[
\textstyle \epsilon_n = \sup{\Big\{ \mu(F) : F \subset \R^N \setminus \bigcup\limits_{k=0}^{n-1} F_k \ \text{and} \ T(F) \le \mu(F) \Big\}}.
\]
\end{enumerate}

By subadditivity of \(T\) and by additivity of \(\mu\),
\[
\textstyle T\big(\bigcup\limits_{k=0}^\infty F_k\big) 
\le \displaystyle  \sum_{k=0}^\infty T(F_k) 
\le \sum_{k=0}^\infty \mu(F_k) 
= \textstyle \mu\big(\bigcup\limits_{k=0}^\infty F_k\big).
\]

We claim that
\[
\mu\lfloor_{\R^N \setminus \bigcup\limits_{k=0}^\infty F_k} \le T.
\]
Assume by contradiction that this inequality is not true.
Then, there exists a Borel set \(C \subset \R^N\) such that
\[
\textstyle T(C) < \mu\big(C \setminus \bigcup\limits_{k=0}^\infty F_k\big).
\]
Let
\[
\textstyle D = C \setminus \bigcup\limits_{k=0}^\infty F_k.
\]
By monotonicity of \(T\), we have
\[
T(D) \le T(C) < \mu(D).
\]
In particular, \(\mu(D) > 0\).
Since \(D\) is an admissible set in the definition of the numbers \(\epsilon_n\),
 for every \(n \in \N\) we have
\[
\mu(D) \le \epsilon_n.
\]
This is not possible since
\[
\theta \sum_{k=0}^\infty \epsilon_k \le \sum_{k=0}^\infty \mu(F_{k+1}) = \textstyle \mu\big(\bigcup\limits_{k=0}^\infty F_{k+1}\big) \le \mu(\R^N) < +\infty.
\]
In particular, the sequence \((\epsilon_n)_{n \in \N}\) converges to \(0\), but this contradicts the fact that \((\epsilon_n)_{n \in \N}\) is bounded from below by \(\mu(D)\).

We have the conclusion of the lemma by choosing 
\[
\textstyle \displaystyle E = \R^N \setminus \bigcup\limits_{k=0}^\infty F_k.
\qedhere
\]
\end{proof}

The following lemma allows us to bypass the lack of additivity of the outer Hausdorff measures \(\cH_\delta^s\):

\begin{lemma}
\label{lemmaHausdorffOuterMeasureAdditivity}
Let \(0 < \alpha < + \infty\) and \(0 \le s < + \infty\), let \(\nu \in \cM(\R^N)\) be a nonnegative measure, let \(\delta > 0\) and let \(F_1, \dots, F_n\) be disjoint Borel subsets of \(\R^N\).
If for every \(k \in \{1, \dots, n\}\),
\[
\nu\lfloor_{F_k} \le \alpha \cH_\delta^s,
\]
then for every \(\epsilon > 0\), there exist \(0 < \underline{\delta} \le \delta\) and a Borel set \(E \subset \bigcup\limits_{k=1}^n F_k\) such that
\[
\nu\lfloor_{E} \le \alpha \cH_{\underline{\delta}}^s
\quad
\text{and}
\quad
\textstyle\nu\big( \bigcup\limits_{k = 1}^n F_k \setminus E \big) \le \epsilon.
\]
\end{lemma}

\begin{proof}
It suffices to establish the proposition with \(\alpha = 1\).

For each \(i \in \{1, \dots, n\}\), let \(K_i \subset F_i\) be a compact subset.
For every Borel set \(A \subset \R^N\),
\[
\textstyle \nu\lfloor_{\bigcup\limits_{i = 1}^{n} K_i}(A) 
= \displaystyle \sum_{i = 1}^{n} \nu(A \cap K_i)
\le \sum_{i = 1}^{n} \cH_{\delta}^s(A \cap K_i).
\]
Let \(0 < \underline{\delta} \le \delta\) be such that for every \(i, j \in \{1, \dotsc, n\}\), if \(i \ne j\), then \(d(K_i, K_j) \ge \underline{\delta}\). 
In particular,
\[
d(A \cap K_i, A \cap K_j) \ge \underline{\delta}.
\]
By the metric subadditivity of the outer Hausdorff measure  (lemma~\ref{lemmaHausdorffOuterMeasureMetricAdditivity}),
\[
\sum_{i=1}^n \cH^s_{\underline{\delta}}(A \cap K_i) 
= \textstyle \cH^s_{\underline{\delta}}\big(\bigcup\limits_{i=1}^n (A \cap  K_i)\big). \]
We deduce that for every Borel set \(A \subset \R^N\),
\[
\begin{split}
\textstyle \nu\lfloor_{\bigcup\limits_{i = 1}^{n} K_i}(A) 
& \le \displaystyle \sum_{i = 1}^{n} \cH_{\delta}^s(A \cap K_i)\\
& \le \displaystyle \sum_{i = 1}^{n} \cH_{\underline{\delta}}^s(A \cap K_i)
= \textstyle \cH^s_{\underline{\delta}}\big(\bigcup\limits_{i=1}^n (A \cap  K_i)\big)
\le \cH^s_{\underline{\delta}}(A). 
\end{split}
\]
Thus, the set 
\[
E = \textstyle \bigcup\limits_{i = 1}^{n} K_i
\]
satisfies the first property of the statement.

We now show how to choose the compact sets \(K_i\) in order to have the second property.
Since 
\[
\textstyle 
\big(\bigcup\limits_{i=1}^{n} F_i \big) \setminus \big(\bigcup\limits_{i=1}^{n} K_i \big)
=  
\bigcup\limits_{i=1}^{n} (F_i \setminus K_i),
\]
by additivity of the measure \(\mu\),
\[
\textstyle 
\mu\big( \big(\bigcup\limits_{i=1}^{n} F_i \big) \setminus \big(\bigcup\limits_{i=1}^{n} K_i \big) \big) 
=
\displaystyle 
\sum\limits_{i=1}^{n} 
\textstyle
\mu( F_i \setminus K_i).
\]
By inner regularity of the measure \(\mu\), we may choose \(K_i \subset F_i\) such that
\[
\mu(F_i \setminus K_i) \le \frac{\epsilon}{n}.
\]
Thus,
\[
\textstyle 
\mu\big(\bigcup\limits_{i=1}^{n} F_i \setminus \bigcup\limits_{i=1}^{n} K_i\Big) 
\le n \dfrac{\epsilon}{n} = \epsilon.
\]
This proves the proposition.
\end{proof}

\begin{proof}[Proof of proposition~\ref{propositionStrongApproximationHausdorffMeasure}]
It suffices to prove the proposition with \(0 < \alpha < 1\) and \(\beta = 1\).

Let \((\delta_n)_{n \in \N}\) be a sequence of positive numbers converging to \(0\). 
We construct a sequence of Borel sets \((E_n)_{n \in \N}\) in \(\R^N\) such that
\begin{enumerate}[\((a)\)]
\item \(\mu\lfloor_{E_0} \le \cH^s_{\delta_{0}}\),
\item for every \(n \in \N_*\), 
\(\mu\lfloor_{E_{n} \setminus \bigcup\limits_{k=0}^{n-1} E_k} \le  \cH^s_{\delta_{n}}\),
\item for every \(n \in \N\), 
\(
\cH^s_{\delta_{n}}(\R^N \setminus E_n) \le \mu\big(\R^N \setminus \bigcup\limits_{k=0}^{n} E_k\big).
\)
\end{enumerate}

We proceed by induction on \(n \in \N\).
Let \(E_0 \subset \R^N\) be a Borel set satisfying the conclusion of the previous lemma with \(T = \cH_{\delta_0}^s\). 
Given \(n \in \N_*\) and Borel sets \(E_0, \dotsc, E_{n-1}\), applying the previous lemma with measure \(\mu\lfloor_{\R^N \setminus \bigcup\limits_{k=0}^{N-1} E_k}\) and outer measure \(T = \cH_{\delta_n}^s\), we obtain a Borel set \(E_{n} \subset \R^N\)
such that
\[
\mu\lfloor_{E_{n} \setminus \bigcup\limits_{k=0}^{N-1} E_k} \le \cH^s_{\delta_{n}}
\]
and
\[
\textstyle \cH^s_{\delta_{n}}(\R^m \setminus E_n) \le \mu\big(\R^N \setminus \bigcup\limits_{k=0}^{N} E_k\big).
\]

Let
\[
\textstyle C = \R^m \setminus \bigcup\limits_{k=0}^{\infty} E_k.
\]
On the one hand, since \(\big(\R^N \setminus \bigcup\limits_{k=0}^{n} E_k\big)_{n \in \N}\) is a nonincreasing sequence of sets whose limit is \(C\),
\[
\lim_{n \to \infty}{\mu\big(\R^N \setminus \textstyle\bigcup\limits_{k=0}^{n} E_k\big)} = \mu(C).
\]
On the other hand, by property \((c)\),
\[
\textstyle \cH^s_{\delta_{n}}(C) 
\le \cH^s_{\delta_{n}}(\R^N \setminus E_n) 
\le \mu\big(\R^N \setminus \bigcup\limits_{k=0}^{n} E_k\big).
\]
As \(n\) tends to infinity, we get
\[
\cH^s(C) \le \mu(C).
\]
In particular, the Hausdorff measure \(\cH^s(C)\) is finite and by the upper bound of \(\mu\) in terms of the Hausdorff measure,
\[
( 1 - \alpha) \mu(C) \le 0.
\]
Thus, \(\mu(C) = 0\), whence
\[
\lim_{n \to \infty}{\mu\big(\R^N \setminus \textstyle\bigcup\limits_{k=0}^{n} E_k\big)} = \mu(C) = 0.
\]

Given a Borel set \(E \subset \R^N \setminus \textstyle\bigcup\limits_{k=0}^{n} E_k\),
by additivity of the measure \(\mu\),
\[
\mu(\R^N \setminus E) 
= \mu\big(\R^N \setminus \textstyle\bigcup\limits_{k=0}^{n} E_k\big) + 	\mu\big(\textstyle\bigcup\limits_{k=0}^{n} E_k \setminus E \big).
\]
Given \(\epsilon > 0\), by the limit above there exists \(n \in \N\) such that
\[
\mu\big(\R^N \setminus \textstyle\bigcup\limits_{k=0}^{n} E_k\big) 
\le \dfrac{\epsilon}{2}.
\]
By the property of weak additivity of the Hausdorff outer measures (lemma~\ref{lemmaHausdorffOuterMeasureAdditivity}) applied to the sets \(E_0, E_1\setminus E_0, \dots, E_n \setminus \bigcup\limits_{k=0}^{n-1} E_k\), there exist
\[
0 < \underline{\delta} \le \min{\{\delta_0, \dots, \delta_n\}}
\quad
\text{and}
\quad
\textstyle E \subset \bigcup\limits_{k=0}^{n} E_k
\]
such that
\[
\mu\lfloor_E \le \cH_{\underline{\delta}}^s 
\quad \text{and} \quad
\mu\big(\textstyle\bigcup\limits_{k=0}^{n} E_k \setminus E\big) \le \dfrac{\epsilon}{2}.
\]
Thus,
\[
\mu(\R^N \setminus E)  \le \frac{\epsilon}{2} + \frac{\epsilon}{2} = \epsilon.
\]
This concludes the proof of the proposition.
\end{proof}

An additional argument shows that it is possible to find a set \(A\) satisfying properties \((i)\)--\((ii)\) independently of \(\beta > \alpha\), but with a parameter \(\delta > 0\) still depending on \(\beta\).

\begin{proof}[Proof of proposition~\ref{propositionStrongApproximationHausdorffMeasureNonMonotone}]
Let \((\epsilon_n)_{n \in \N}\) be a sequence of positive numbers converging to zero and let \((\beta_n)_{n \in \N}\) and \((\overline\beta_n)_{n \in \N}\) be such that for every \(n \in \N\), 
\[
\alpha < \beta_n < \overline{\beta}_n.
\]
Let \(E_n \subset \R^N\) be the Borel set and let \(\delta_n > 0\) be the parameter satisfying the conclusion of the previous proposition with parameters \(\epsilon_n\) and \(\beta_n\).
Taking the measure
\[
\mu_n = \frac{\alpha}{\overline{\beta}_n} \mu\lfloor_{E_n},
\]
we have 
\[
\mu_n \le \alpha_n \cH_{\delta_n}^s
\]
with \(\alpha_n < \alpha\). Moreover, since \(\mu_n \le \mu\),
\[
\begin{split}
\norm{\mu - \mu_n}_{\cM(\R^N)} 
& = \mu(\R^N) - \frac{\alpha}{\overline{\beta}_n} \mu(E_n) \\
& = \frac{\overline{\beta}_n - \alpha}{\overline{\beta}_n} \mu(E_n) + \mu(\R^N \setminus E_n) 
\le \frac{\overline{\beta}_n - \alpha}{\alpha} \mu(\R^N) + \epsilon_n.
\end{split}
\]
If the sequence \((\overline\beta_n)_{n \in \N}\) converges to \(\alpha\), we deduce that \((\mu_n)_{n \in \N}\) converges strongly to \(\mu\) in \(\cM(\R^N)\).
The proof of the proposition is complete.
\end{proof}

Other consequences of the results in this section and connections to the problem of removable singularities for the divergence operator are investigated by the author in \cite{Pon:12}.


\chapter{Diffuse measure data}

We show that the nonlinear Dirichlet problem
\[
\left\{
\begin{alignedat}{2}
- \Delta u + g(u) & = \mu	\quad && \text{in \(\Omega\),}\\
u & = 0 	\quad && \text{on \(\partial\Omega\),}
\end{alignedat}
\right.
\]
always has a solution when the measure \(\mu\) is diffuse with respect to the \(W^{1, 2}\) capacity.
Using a result of de~la~Vallée Poussin~\cite{Del:1915}, we also prove that this is the largest class of measures for which the Dirichlet problem has a solution regardless of the nonlinearity \(g\). 

\section{Sufficient condition}

We begin by showing that the nonlinear Dirichlet problem always has a solution if the measure \(\mu\) is diffuse:

\begin{proposition}
\label{propositionExistenceSolutionDiffuseMeasure}
Let \(\mu \in \cM(\Omega)\). 
If the measure \(\mu\) is diffuse  with respect to the \(W^{1, 2}\) capacity, then for every continuous function \(g : \R \to \R\) satisfying the sign condition, the nonlinear Dirichlet problem with datum \(\mu\) has a solution.
\end{proposition}

This statement is proved in full generality by Orsina and Ponce~\cite{OrsPon:08}*{theorem~1.2} and extends a previous result of Gallouët and Morel~\cite{GalMor:84}*{theorem~1} for \(L^1\) data.
The strategy for \(L^1\) data or for diffuse measure data is the same.
The case of nondecreasing nonlinearities was first studied by Brezis and Strauss~\cite{BreStr:73} for \(L^1\) data and by Brezis and Browder~\cite{BreBro:78}*{theorem~1} for \((W_0^{1, 2})'\) data and they can be combined to give the result for diffuse measures  \cite{BreMarPon:07}*{theorem~4.B.4}.

\begin{lemma}
Let \(\mu \in \cM(\Omega)\) and let \(u\) be the solution of the linear Dirichlet problem
\[
\left\{
\begin{alignedat}{2}
- \Delta u & = \mu &&\quad \text{in } \Omega,\\
u & = 0 &&\quad \text{on } \partial\Omega.
\end{alignedat}
\right.
\]
If \(\mu\) is a diffuse measure with respect to the \(W^{1, 2}\) capacity, then for every bounded nondecreasing continuous function \(\Phi : \R \to \R\) the quasicontinuous representative \(\Quasicontinuous{u}\) of \(u\) satisfies
\[
\int\limits_\Omega \Phi(0) \dif \mu \le \int\limits_\Omega \Phi(\Quasicontinuous{u}) \dif\mu.
\]
\end{lemma}

The strategy of the proof consists in reducing the problem to the case where \(\mu \in (W_0^{1, 2}(\Omega))'\), in which case \(u \in W_0^{1, 2}(\Omega)\).

\begin{proof}
It suffices to prove the estimate when \(\Phi(0) = 0\).

Let \((\mu_n)_{n \in \N}\) be a sequence in \(\cM(\Omega)\) converging strongly to \(\mu\) in \(\cM(\Omega)\) and such that for every \(n \in \N\), \(\mu_n \in (W_0^{1, 2}(\Omega))'\); 
the existence of such sequence follows from proposition~\ref{propositionIncreasingSequenceCapacity} applied to the positive and negative parts of \(\mu\). 
Denote by \(u_n\) the solution of the linear Dirichlet problem with datum \(\mu_n\). 

For every \(n \in \N\),
\[
\int\limits_\Omega \Phi(\Quasicontinuous u) \dif\mu 
\ge \int\limits_\Omega \Phi(\Quasicontinuous u_n) \dif\mu_n  + \int\limits_\Omega \Phi(\Quasicontinuous u_n) \dif (\mu - \mu_n) + \int\limits_\Omega \big[ \Phi(\Quasicontinuous u) - \Phi(\Quasicontinuous u_n) \big] \dif \mu.
\]
Since \(\mu_n \in (W_0^{1, 2}(\Omega))'\), \(u_n \in W_0^{1, 2}(\Omega)\). 
Assuming that \(\Phi\) is smooth and that its derivative \(\Phi'\) has compact support, then by the interpolation inequality (lemma~\ref{lemmaInterpolationLinfty}) we have
\(\Phi(u_n) \in W_0^{1, 2}(\Omega)\).

We first observe that
\[
\int\limits_\Omega \Phi(\Quasicontinuous u_n) \dif\mu_n \ge 0.
\]
Indeed,
\[
\int\limits_\Omega \Phi(\Quasicontinuous u_n) \dif\mu 
= \int\limits_\Omega \nabla \Phi(u_n) \cdot \nabla u_n 
= \int\limits_\Omega \Phi'(u_n) \abs{\nabla u_n}^2 \ge 0.
\]

Next,
\[
\left|\int\limits_\Omega \Phi(\Quasicontinuous u_n) \dif (\mu - \mu_n) \right| \le \NewConstant \norm{\mu_n - \mu}_{\cM(\Omega)}.
\]
This implies
\[
\lim_{n \to \infty}{\int\limits_\Omega \Phi(\Quasicontinuous u_n) \dif (\mu - \mu_n)} = 0.
\]

Finally, by the Boccardo-Gallouët-Orsina decomposition of diffuse measures (corollary~\ref{corollaryDecompositionBoccardoGallouetOrsina}~\((i)\)), we may write
\[
\mu = f + \lambda,
\]
where \(f \in L^1(\Omega)\) and \(\lambda \in \cM(\Omega) \cap (W_0^{1, 2}(\Omega))'\).
By the dominated convergence theorem, 
\[
\lim_{n \to \infty}{\int\limits_\Omega \big[ \Phi(u) - \Phi(u_n) \big] f} = 0.
\]
We also have
\[
\lim_{n \to \infty}{\int\limits_\Omega \big[ \Phi(\Quasicontinuous u) - \Phi(\Quasicontinuous u_n) \big] \dif \lambda} = 0.
\]
Indeed, on the one hand, by the interpolation inequality the sequence \((\Phi(u_n))_{n \in \N}\) is bounded in \(W_0^{1, 2}(\Omega)\).
On the other hand, by Stampacchia's regularity theory (proposition~\ref{propositionCompactnessLp}) the sequence \((u_n)_{n \in \N}\) converges to \(u\) in \(L^1(\Omega)\).
Thus, \((\Phi(u_n))_{n \in \N}\) converges weakly in \(W_0^{1, 2}(\Omega)\) to \(\Phi(u)\).
Since \(\lambda \in (W_0^{1, 2}(\Omega))'\), the limit above holds.
We deduce that
\[
\lim_{n \to \infty}{\int\limits_\Omega \big[ \Phi(\Quasicontinuous u) - \Phi(\Quasicontinuous u_n) \big] \dif \mu} = 0.
\]

Combining these three facts, we conclude that
\[
\int\limits_\Omega \Phi(\Quasicontinuous u) \dif\mu \ge 0
\]
The proposition follows when \(\Phi\) is smooth and \(\Phi'\) has compact support. The general case can be deduced from this one by taking a uniform approximation of \(\Phi\) by functions of this type.
\end{proof}

\begin{corollary}
\label{corollaryContractionSign}
Let \(\mu \in \cM(\Omega)\) and let \(u\) be the solution of the linear Dirichlet problem with datum \(\mu\). 
If \(\mu\) is a diffuse measure with respect to the \(W^{1, 2}\) capacity, then for every \(s \ge 0\),
\[
\int\limits_{\{\abs{\Quasicontinuous u} > s\}} \sgn{\Quasicontinuous u} \dif\mu \ge 0.
\]
\end{corollary}


We have the following capacitary estimate in the spirit of propositon~\ref{propositionCapacitaryEstimateWkp}:

\begin{lemma}
\label{lemmaCapacitaryEstimateLaplacian}
Let \(\mu \in \cM(\Omega)\). If \(u\) is the solution of the linear Dirichlet problem with datum \(\mu\), then for every \(s > 0\),
\[
\capt_{W^{1, 2}}{(\{ \abs{\Quasicontinuous u} > s \})} \le \frac{C}{s} \norm{\mu}_{\cM(\Omega)},
\]
for some constant \(C > 0\) depending on \(\Omega\).
\end{lemma}

\begin{proof}
We recall that for every \(v \in W_0^{1, 2}(\Omega)\) and for every \(s > 0\),
\[
\capt_{W^{1, 2}}{(\{\abs{\Quasicontinuous v} > s \})} \le \frac{1}{s^2} \norm{v}_{W^{1, 2}(\R^N)}^2.
\]

Given \(\kappa > 0\), by the interpolation inequality (lemma~\ref{lemmaInterpolationLinfty}) we have \(T_\kappa(u) \in W_0^{1, 2}(\Omega)\) and
\[
\norm{D T_\kappa(u)}_{L^2(\Omega)} \le \kappa^{\frac{1}{2}} \norm{\mu}_{\cM(\Omega)}^{\frac{1}{2}}.
\]
By the Poincaré inequality,
\[
\norm{T_\kappa(u)}_{W^{1, 2}(\Omega)} 
\le \NewConstant \norm{D T_\kappa(u)}_{L^2(\Omega)} 
\le \SameConstant \kappa^{\frac{1}{2}} \norm{\mu}_{\cM(\Omega)}^{\frac{1}{2}}.
\]
Thus,
\[
\capt_{W^{1, 2}}{(\{\abs{T_\kappa(\Quasicontinuous u)} > s \})} 
\le \frac{\Constant\kappa}{s^2}  \norm{\mu}_{\cM(\Omega)}.
\]
If \(\kappa \ge s\), then 
\[
\{\abs{T_\kappa(\Quasicontinuous u)} > s \} = \{\abs{\Quasicontinuous u} > s \}.
\]
Taking \(\kappa = s\), we deduce that
\[
\capt_{W^{1, 2}}{(\{\abs{\Quasicontinuous u} > s \})} \le \frac{\SameConstant}{s} \norm{\mu}_{\cM(\Omega)}.
\]
The proof of the lemma is complete.
\end{proof}

We refer the reader to \cite{PetPonPor:11}*{theorem~1.2} for the counterpart of the above estimate for the heat operator involving parabolic capacities.

\begin{lemma}
\label{lemmaDiffuseMeasureCharacterization}
For every \(\mu \in \cM(\Omega)\), \(\mu\) is a diffuse measure with respect to the \(W^{1, 2}\) capacity if and only if for every \(\epsilon > 0\) there exists \(\delta > 0\) such that if \(A \subset \Omega\) is a Borel set such that \(\capt_{W^{1, 2}}{(A)} \le \delta\), then \(\abs{\mu}(A) \le \epsilon\).
\end{lemma}

\begin{proof} 
We begin with the direct implication.
If $A \subset \Omega$ is a Borel set such that $\capt{(A)} = 0$, then for every \(\epsilon > 0\), $|\mu|(A) \le \epsilon$. 
Thus, $|\mu|(A) = 0$ and $\mu$ is a diffuse measure.

For the reverse implication, we proceed by contradiction.
Given a sequence of positive numbers \((\alpha_n)_{n \in \N}\) converging to zero,
suppose that there exist $\epsilon > 0$ and a sequence $(A_{n})_{n \in \N}$ of Borel subsets of $\Omega$ such that for every \(n \in \N\),
\[
\capt_{W^{1, 2}}{(A_{n})} \le \alpha_n 
\quad
\text{and}
\quad
\mu(A_{n}) > \epsilon.
\] 
Let
\[
\textstyle A = \bigcap\limits_{j = 0}^\infty \bigcup\limits_{n = j}^\infty A_n.
\]
For every \(j \in \N\), we have by monotonicity and by subadditivity of the capacity, 
\[
\textstyle \capt_{W^{1, 2}}{(A)} \le \capt_{W^{1, 2}}(\bigcup\limits_{n = j}^\infty A_n) 
\displaystyle
\le \sum\limits_{n = j}^\infty \capt_{W^{1, 2}}(A_n) 
\le \sum\limits_{n = j}^\infty \alpha_n.
\]
Thus, if the series \(\sum\limits_{n = 0}^\infty \alpha_n\) converges, then \(\capt_{W^{1, 2}}{(A)} = 0\) and $\mu(A) \geq \epsilon$.
This is a contradiction. 
\end{proof}

The previous lemma still holds for the \(W^{k, p}\) capacity, with the same proof.

\begin{proof}[Proof of proposition~\ref{propositionExistenceSolutionDiffuseMeasure}]
Let \((g_n)_{n \in \N}\) be a sequence of \emph{bounded} continuous functions \(g_n : \R \to \R\) satisfying the sign condition and such that \((g_n)_{n \in \N}\) converges uniformly to \(g\) in bounded subsets of \(\R\).
Since \(g_n\) is bounded, by proposition~\ref{propositionExistenceSolutionEulerLagrange} the Dirichlet problem
\[
\left\{
\begin{alignedat}{2}
- \Delta u + g_n(u) & = \mu &&\quad \text{in } \Omega,\\
u & = 0 &&\quad \text{on } \partial\Omega,
\end{alignedat}
\right.
\]
has a solution \(u_n\).

\begin{Claim}
For every \(s \ge 0\), there exists \(\NewConstant > 0\) such that for every Borel set \(E \subset \Omega\) and for every \(n \in \N\),
\[
\int\limits_E \abs{g_n(u_n)} \le \SameConstant \abs{E} + \abs{\mu}({\{\abs{\Quasicontinuous u_n} > s\}}).
\]
\end{Claim}

We postpone the proof of the claim and we conclude the proof of the proposition.
By lemma~\ref{lemmaCapacitaryEstimateLaplacian},
\[
\capt_{W^{1, 2}}{(\{\abs{\Quasicontinuous u_n} > s\})} 
\le  \frac{\Constant}{s} \norm{\mu - g_n(u_n)}_{\cM(\Omega)}.
\]
Thus, by the triangle inequality and by the absorption estimate (lemma~\ref{lemmaEstimateAbsorption}),
\[
\capt_{W^{1, 2}}{(\{\abs{\Quasicontinuous u_n} > s\})} 
 \le \frac{\Constant}{s} \norm{\mu}_{\cM(\Omega)}.
\]
This estimate combined with the claim imply that the sequence \((g_n(u_n))_{n \in \N}\) is equi-integrable. 
Indeed, for every \(\epsilon > 0\), by the characterization of diffuse measures (lemma~\ref{lemmaDiffuseMeasureCharacterization}) and by the capacitary estimate above, there exists \(s > 0\) such that for every \(n \in \N\),
\[
\abs{\mu}({\{\abs{\Quasicontinuous u_n} > s\}}) \le \frac{\epsilon}{2}.
\]
For this given \(s > 0\), take \(\delta > 0\) such that
\[
C_1 \delta \le \frac{\epsilon}{2}.
\]
Then, for every Borel set \(E \subset \Omega\) such that \(\meas{E} \le \delta\), we have
\[
\int\limits_E \abs{g_n(u_n)} \le \frac{\epsilon}{2} + \frac{\epsilon}{2} = \epsilon.
\]

By Stamapacchia's regularity theory (proposition~\ref{propositionCompactnessLp}), there exists a subsequence \((u_{n_k})_{k \in \N}\) which converges in \(L^1(\Omega)\) and almost everywhere in \(\Omega\) to some function \(u\). 
Since the sequence \((g_{n_k}(u_{n_k}))_{n \in \N}\) is equi-integrable and converges almost everywhere to \(g(u)\), we conclude by Vitali's theorem that \((g_{n_k}(u_{n_k}))_{n \in \N}\) converges to \(g(u)\) in \(L^1(\Omega)\).
Therefore, \(u\) is a solution of the nonlinear Dirichlet problem with datum \(\mu\).

It remains to establish the claim:

\begin{proof}[Proof of the claim]
For every Borel set \(E \subset \Omega\) and for every \(s \ge 0\),
\[
\int\limits_E \abs{g_n(u_n)} \le \int\limits_{E \cap \{\abs{u_n} \le s\}} \abs{g_n(u_n)} + \int\limits_{\{\abs{u_n} > s\}} \abs{g_n(u_n)}.
\]
Since the sequence \((g_n)_{n \in \N}\) is bounded in \([-s, s]\),
\[
\int\limits_{E \cap \{\abs{u_n} \le s\}} \abs{g_n(u_n)} 
\le \NewConstant \abs{E \cap \{\abs{u_n} \le s\}} 
\le \SameConstant \abs{E}.
\]
By the corollary above applied to \(\mu_n - g_n(u_n)\),
\[
\int\limits_\Omega \sgn{u_n} g_n(u_n) \le \int\limits_\Omega \sgn{\Quasicontinuous{u}_n} \dif\mu_n.
\]
Thus, by the sign condition,
\[
\int\limits_{\{\abs{u_n} > s\}} \abs{g_n(u_n)} \le \int\limits_{\{\abs{\Quasicontinuous u_n} > s\}} \sgn{\Quasicontinuous u_n}  \dif\mu \le  \abs{\mu}({\{\abs{\Quasicontinuous u_n} > s\}}).
\]
This establishes the claim.
\end{proof}

The proof of the proposition is complete.
\end{proof}


\section{Necessary condition}

Proposition~\ref{propositionExistenceSolutionDiffuseMeasure} 
gives a large class of measures for which the nonlinear Dirichlet problem always has a solution regardless of the growth rate of the nonlinearity.
One might wonder whether there are other measures which have such property. 
Brezis, Marcus and Ponce~\cite{BreMarPon:07}*{theorem~4.14} gave a negative answer to this question:

\begin{proposition}
\label{propositionExistenceSolutionDiffuseMeasureConverse}
Let \(\mu \in \cM(\Omega)\). 
If for every continuous nondecreasing function \(g : \R \to \R\) the nonlinear Dirichlet problem with datum \(\mu\) has a solution, then the measure \(\mu\) is diffuse with respect to the \(W^{1, 2}\) capacity.
\end{proposition}

The proof of this proposition relies on the following result of de la Vallée Poussin \citelist{\cite{Del:1915}*{remark~23} \cite{DelMey:75}*{théorème~II.22}}:

\begin{lemma}
For every \(f \in L^1(\Omega)\), there exists
a nonnegative nondecreasing continuous function $h : [0, +\infty) \to \R$ such that $h(0) = 0$,
\[
\lim_{t \to + \infty}{\frac{h(t)}{t}} = + \infty
\]
and
\[
h(f) \in L^1(\Omega).
\]
\end{lemma}

In the proof of the proposition we also need the following elementary property of the Legendre transform:

\begin{lemma}
Let \(h : [0, +\infty) \to \R\) be a nonnegative continuous function such that \(h(0) = 0\). 
If 
\[
\lim_{t \to + \infty}{\frac{h(t)}{t}} = + \infty,
\]
then the Legendre transform of \(h\) defined for \(t \in [0, +\infty)\) by
\[
g(t) = \sup_{s \ge 0}{\{st - h(s)\}}
\]
is a nonnegative nondecreasing continuous function such that \(g(0) = 0\).
\end{lemma}

The Legendre transform is defined in such a way that it gives the smallest function \(g: [0, +\infty) \to \R\) which satisfies Young's inequality: for every \(s \ge 0\) and \(t \ge 0\),
\[
st \le h(s) + g(t).
\]

\medskip
We also need the following identity of capacities due to Brezis, Marcus and Ponce~\cite{BreMarPon:07}*{theorem~4.E.1}:

\begin{lemma}
\label{lemmaCapacityEquivalence}
For every compact set \(K \subset \Omega\),
\begin{multline*}
\inf{\bigg\{ \int\limits_\Omega \abs{\Delta\varphi} : \varphi \in C_c^\infty(\Omega) \ \text{and} \ \varphi \ge 1 \ \text{on \(K\)} \bigg\}}\\
= 2 \inf{\bigg\{ \int\limits_\Omega \abs{\nabla\varphi}^2 : \varphi \in C_c^\infty(\Omega) \ \text{and} \ \varphi \ge 1 \ \text{on \(K\)} \bigg\}}.
\end{multline*}
\end{lemma}

The proof of this lemma relies on the observation that if \(u_K\) is the minimizer of the right-hand side in \(W_0^{1, 2}(\Omega)\) --- \(u_K\) is called the capacitary potential of \(K\) ---, then for every \(0 < \epsilon < 1\) the function 
\[
u_{K, \epsilon} = (u_K - \epsilon)^+
\] 
has compact support in \(\Omega\) and
\[
\int\limits_\Omega \abs{\Delta u_{K, \epsilon}} 
= 2 \int\limits_\Omega \abs{\Delta u_{K}}
= 2 \int\limits_\Omega \abs{\nabla u_K}^2.
\]

\begin{proof}[Proof of proposition~\ref{propositionExistenceSolutionDiffuseMeasureConverse}]
Assuming that \(g\) is nondecreasing and \(g(0) = 0\), then \(g\) satisfies the sign condition.

We first consider the case where \(\mu\) is a \emph{nonnegative} measure. Then, by the comparison principle (corollary~\ref{corollaryEstimateComparison}) we have \(u \ge 0\).

Let \(K \subset \Omega\) be a compact set and let \(\varphi \in C_c^\infty(\Omega)\) be a nonnegative functon such that \(\varphi \ge 1\) in \(K\). 
If \(u\) is the solution of Dirichlet problem with nonlinearity \(g\) and datum \(\mu\), then since the measure \(\mu\) is nonnegative,
\[
\mu(K) \le \int\limits_\Omega \varphi \dif\mu = - \int\limits_\Omega u \Delta\varphi + \int\limits_\Omega g(u)\varphi.
\]

If \(\capt_{W^{1, 2}}(K) = 0\), then by the property of equivalence of capacities (lemma~\ref{lemmaCapacityEquivalence}) there exists a sequence \((\varphi_n)_{n \in \N}\) of nonnegative functions in \(C_c^\infty(\Omega)\) such that
\begin{enumerate}[\((a)\)]
\item for every \(n \in \N\), \(\varphi_n \ge 1\) in \(K\), 
\item \((\Delta\varphi_n)_{n \in \N}\) converges to \(0\) in \(L^1(\Omega)\),
\item \((\varphi_n)_{n \in \N}\) converges to \(0\) in \(L^1(\Omega)\),
\item for every \(n \in \N\), \(0 \le \varphi_n \le 1\) in \(\Omega\).
\end{enumerate}
The third assumption follows from the linear elliptic estimate (proposition~\ref{propositionExistenceLinearDirichletProblem}) and the fourth  assertion can be achieved by a truncation argument  \cite{BreMarPon:07}*{lemma~4.E.1}.

For every \(n \in \N\),
\[
\mu(K) \le - \int\limits_\Omega u \Delta\varphi_n + \int\limits_\Omega g(u)\varphi_n.
\]

Since the sequence \((\Delta\varphi_n)_{n \in \N}\) converges in \(L^1(\Omega)\), there exist a subsequence \((\Delta\varphi_{n_k})_{k \in \N}\) and \(f \in L^1(\Omega)\) such that for every \(k \in \N\),
\[
\abs{\Delta\varphi_{n_k}} \le f
\]
and \((\Delta\varphi_{n_k})_{k \in \N}\) converges almost everywhere in \(\Omega\) to its limit \(0\).

If \(g\) is the Legendre transform of a function \(h : [0, +\infty) \to \R\),
then for every \(s, t \ge 0\),
\[
st \le g(t) + h(s).
\]
Thus,
\[
\abs{u \Delta\varphi_{n_k}} \le \abs{uf} \le g(u) + h(f).
\]
By the result of de la Vallée Poussin, there exists a nonnegative superlinear continuous function $h : [0, +\infty) \to \R$ such that \(h(0) = 0\) and
\[
h(f) \in L^1(\Omega).
\]

Take \(g\) to be the Legendre transform of \(h\) in \([0, +\infty)\).
Since the sequence \((\Delta\varphi_n)_{n \in \N}\) converges almost everywhere to \(0\), by the dominated convergence theorem,
\[
\lim_{n \to \infty}{\int\limits_\Omega u \Delta\varphi_{n_k}} = 0.
\]
Since the sequence \((\varphi_n)_{n \in \N}\) is bounded in \(L^\infty(\Omega)\) and converges to \(0\) in \(L^1(\Omega)\), by the dominated convergence theorem we have
\[
\lim_{n \to \infty}{\int\limits_\Omega g(u)\varphi_n} = 0.
\]
Hence, \(\mu(K) \le 0\). 
Since the measure \(\mu\) is nonnegative, \(\mu(K) = 0\) for every compact set \(K \subset \Omega\) such that \(\capt_{W^{1, 2}}{(K)} = 0\).
This implies that \(\mu\) is a diffuse measure.
 
\medskip
If \(\mu\) is a signed measure for which the nonlinear Dirichlet problem has a solution for every nondecreasing nonlinearity \(g\), then using reduced measures we show that both measures \(\max{\{\mu, 0\}}\) and \(\min{\{\mu, 0\}}\) have the same property (see corollary~\ref{corollaryGoodMeasuresPositivePart}), whence they are diffuse measures with respect to the \(W^{1, 2}\) capacity and the conclusion follows.
\end{proof}

In view of proposition~\ref{propositionExistenceSolutionDiffuseMeasure} and proposition~\ref{propositionExistenceSolutionDiffuseMeasureConverse}, a further question would be whether there exists some very large nonlinearity \(g\) for which the only good measures for the Dirichlet problem with this nonlinearity are the measures which are diffuse with respect to the \(W^{1, 2}\) capacity.
Perhaps the previous proposition holds because there is already one nonlinearity for which all the good measures are the diffuse measures.

The author has showed that answer is negative~\cite{Pon:04a}*{theorem~1}: 
for every continuous nondecreasing function $g : \R \to \R$, there exists a positive measure $\mu$ concentrated in a compact set of zero \(W^{1, 2}\) capacity such that $\mu$ is a good measure for the Dirichlet problem with this nonlinearity $g$.

\medskip
Another example of nonlinear Dirichlet problem in which diffuse measures are the only good measures is
\[
\left\{
\begin{alignedat}{2}
- \Delta u + u \abs{\nabla u}^2 & = \mu &&\quad \text{in } \Omega,\\
u & = 0 &&\quad \text{on } \partial\Omega.
\end{alignedat}
\right.
\]
This problem has been studied by Boccardo, Gallouët and Orsina~\cite{BocGalOrs:97}; see also \cite{MurPor:02} for further extensions.


\chapter{Reduced measures}

We would like to understand the nonnexistence mechanism behind the nonlinear Dirichlet problem
\[
\left\{
\begin{alignedat}{2}
- \Delta u + g(u) & = \mu &&\quad \text{in } \Omega,\\
u & = 0 &&\quad \text{on } \partial\Omega.
\end{alignedat}
\right.
\]

For this purpose, we study properties of the largest subsolution \(u^*\) given by the Perron method.
It turns out that the interesting tool is not the subsolution \(u^*\) itself but the reduced measure \(\mu^* \in \cM\loc(\Omega)\) defined in terms of \(u^*\) by
\[
\mu^* = - \Delta u^* + g(u^*).
\]

The fundamental property of the reduced measure  (proposition~\ref{propositionReducedMeasure}) ensures that \(\mu^* \in \cM(\Omega)\) and \(\mu^*\) is the largest good measure of the Dirichlet problem which is less than or equal to the measure \(\mu\).


\section{Existence of the reduced measure}

We begin by giving the precise definition of the reduced measure:

\begin{definition}
Let \(g : \R \to \R\) be a continuous function and let \(\mu \in \cM(\Omega)\).
If the largest subsolution \(u^*\) of the nonlinear Dirichlet problem with datum \(\mu\) exists, then the reduced measure \(\mu^*\) is the unique locally finite measure in \(\Omega\) such that
\[
\mu^* = - \Delta u^* + g(u^*)
\]
in the sense of distributions in \(\Omega\).
\end{definition}

Since \(u^*\) is a subsolution of the Dirichlet problem, 
\[
\mu^* = - \Delta u^* + g(u^*) \le \mu
\]
in the sense of distributions in \(\Omega\).
By the property of positive distributions (lemma~\ref{lemmaPositiveDistributions}), we have
\[
\mu^* \in \cM\loc(\Omega).
\]
This explains why \(\mu^*\), which a priori is only a distribution, is indeed a measure.

\medskip
We begin with a criterion that ensures that the largest subsolution \(u^*\) --- and hence \(\mu^*\) --- exists.

\begin{proposition}
\label{propositionExistenceReducedMeasure}
Let \(g : \R \to \R\) be a continuous function satisfying the sign condition and the integrability condition, and let \(\mu \in \cM(\Omega)\).
If the nonlinear Dirichlet problem with datum \(\mu\) has a subsolution, then the  largest subsolution \(u^*\) exists.
\end{proposition}

The proof is based on a uniform bound of the \(L^1\) norm of the nonlinear term \(g(u)\) for subsolutions:

\begin{lemma}
\label{lemmaEstimateAbsorptionSubsolution}
Let \(g : \R \to \R\) be a continuous function satisfying the sign condition, and let \(\mu \in \cM(\Omega)\).
\begin{enumerate}[\((i)\)]
\item If \(u\) is a subsolution of the nonlinear Dirichlet problem, then
\[
\norm{\max{\{g(u), 0\}}}_{L^1(\Omega)} \le \norm{\max{\{\mu, 0\}}}_{\cM(\Omega)}.
\]
\item If \(u\) is a supersolution of the nonlinear Dirichlet problem, then
\[
\norm{\min{\{g(u), 0\}}}_{L^1(\Omega)} \le \norm{\min{\{\mu, 0\}}}_{\cM(\Omega)}.
\]
\end{enumerate}
\end{lemma}

In particular, if \(u\) is a solution of the nonlinear Dirichlet problem, then combining both estimates we recover the absorption estimate
\[
\norm{g(u)}_{L^1(\Omega)} \le \norm{\mu}_{\cM(\Omega)}.
\]

\begin{proof}
We prove the first assertion.
Note that \(u\) is also a subsolution of the nonlinear Dirichlet problem with datum \(\mu^+ = \max{\{\mu, 0\}}\),
\[
- \Delta u + g(u) \le \mu^+
\]
in the sense of \((C_0^\infty(\overline\Omega))'\).

Let \(\Phi : [0, +\infty) \to \R\) be a smooth bounded function. 
Given \(\epsilon > 0\) and given a nonnegative function \(\zeta \in C_0^\infty(\overline{\Omega})\), let \(\psi_\epsilon : \overline\Omega \to \R\) be defined by
\[
\psi_\epsilon = \Phi\Big( \frac{\zeta}{\epsilon} \Big).
\]
On the one hand, if \(\Phi(0) = 0\) and if the function \(\Phi\) is nonnegative, then \(\psi_\epsilon \in C_0^\infty(\overline\Omega)\) and \(\psi_\epsilon \ge 0\) in \(\overline\Omega\). 
Thus, \(\psi_\epsilon\) is an admissible test function and we have
\[
- \int\limits_\Omega u \Delta\psi_\epsilon + \int\limits_{\Omega} g(u) \psi_\epsilon 
\le \int\limits_\Omega \psi_\epsilon \dif\mu^+.
\]
On the other hand,
\[
\Delta \psi_\epsilon 
= \frac{1}{\epsilon} \Phi'\Big( \frac{\zeta}{\epsilon} \Big) \Delta\zeta + \frac{1}{\epsilon^2} \Phi''\Big( \frac{\zeta}{\epsilon} \Big) \abs{\nabla\zeta}^2.
\]
Thus, if \(\Phi\) is concave and nondecreasing, and if \(\zeta\) is superharmonic in \(\Omega\),
\[
- \Delta \zeta_\epsilon \ge 0.
\]

Let us temporarily assume that the function \(u\) is nonnegative. 
In this case, we have
\[
- \int\limits_\Omega u \Delta\psi_\epsilon \ge 0,
\]
whence
\[
\int\limits_{\Omega} g(u) \psi_\epsilon 
\le \int\limits_\Omega \psi_\epsilon \dif\mu^+.
\]
We take the function \(\Phi\) so that
\[
\lim_{t \to +\infty}{\Phi(t)} = 1.
\]
As we let \(\epsilon\) tend to \(0\), by the dominated convergence theorem we deduce that
\[
\int\limits_{\Omega} g(u) 
\le \int\limits_\Omega \dif\mu^+
\]
and the estimate follows.

If the function \(u\) is not necessarily nonnegative then we may apply Kato's inequality in the spirit of proposition~\ref{propositionKatoAncona} and proposition~\ref{propositionKatoVariant} --- or the general version given by proposition~\ref{propositionKatoDiffuse} and corollary~\ref{corollaryKatoConcentrated} --- and we deduce that
\[
- \Delta u^+ + \chi_{\{u > 0\}} g(u) \le \mu^+
\]
in the sense of distributions in \(\Omega\).
By the equivalence of formulations given by proposition~\ref{propositionDistributionC0Infty}, we deduce that
\[
- \Delta u^+ + \chi_{\{u > 0\}} g(u) \le \mu^+
\]
in the sense of \((C_0^\infty(\overline\Omega))'\).
Proceeding as above, we deduce that
\[
\int\limits_{\Omega} \chi_{\{u > 0\}} g(u) 
\le \int\limits_\Omega \dif\mu^+.
\]
In view of the sign condition satisfied by \(g\), \(\chi_{\{u > 0\}} g(u) = \max{\{g(u), 0\}}\) and the conclusion follows.
\end{proof}

It is interesting to compare the proof of lemma~\ref{lemmaEstimateAbsorptionSubsolution} with the proof of the absorption estimate for solutions of the nonlinear Dirichlet problem (lemma~\ref{lemmaEstimateAbsorption}).

In the case of \emph{solutions}, we perform an approximation using the linear Dirichlet problem and the estimate follows from the fact that for nondecreasing functions \(\Phi : \R \to \R\),
\[
\int\limits_\Omega \abs{\nabla u}^2 \Phi'(u) \ge 0.
\]
In the case of \emph{subsolutions}, the counterpart of the approximation scheme is more subtle and requires the use of measure boundary traces.
Instead, we use a family of test functions in the same spirit as in the study of the Dirichlet boundary condition (propositon~\ref{propositionDirichletBoundaryCondition})
converging pointwisely to \(1\) as the parameter of the family tends to \(0\).
In this case, we are implicitly using the fact that
\[
\int\limits_\Omega \Delta u^+ \le 0.
\]

There is an interpretation of this inequality using a formal application of the divergence theorem,
\[
\int\limits_\Omega \Delta u^+ = \int\limits_{\partial\Omega} \frac{\partial u^+}{\partial n}.
\]
Since \(u\) is a subsolution of a Dirichlet problem, \(u\) is nonpositive on the boundary, thus \(u^+\) vanishes on the boundary. 
Since \(u^+\) is nonnegative, every point of \(\partial\Omega\) is a minimum point of \(u^+\), so we should expect to have \(\frac{\partial u^+}{\partial n} \le 0\) on the boundary.
This explanation can be made rigorous by considering the normal derivative of \(u\) on \(\partial\Omega\) in the sense of measures \cite{BrePon:08}.

\medskip
We also need the following counterpart of the comparison principle (lemma~\ref{lemmaEstimateComparison}); the proof requires some minor modification.

\begin{lemma}
\label{lemmaEstimateComparisonSubsolution}
Let \(g : \R \to \R\) be a continuous function satisfying the sign condition, and let
\(\mu \in \cM(\Omega)\).
\begin{enumerate}[\((i)\)]
\item If \(u\) is a subsolution of the nonlinear Dirichlet problem with datum \(\mu\) and  if \(v\) is a nonnegative supersolution of the linear Dirichlet problem with datum \(\mu\), then \(u \le v\) in \(\Omega\).
\item If \(u\) is a supersolution of the nonlinear Dirichlet problem with datum \(\mu\) and  if \(v\) is a nonpositive subsolution of the linear Dirichlet problem with datum \(\mu\), then \(u \ge v\) in \(\Omega\).
\end{enumerate}
\end{lemma}

\begin{proof}[Proof of proposition~\ref{propositionExistenceReducedMeasure}]
Let \(\overline{w}\) be the solution of the linear Dirichlet problem with datum \(\overline\mu\).
If \(\overline{\mu} \ge \max{\{\mu, 0\}}\), then by the comparison principle above, every subsolution \(u\) of the nonlinear Dirichlet problem satisfies \(u \le \overline{w}\).
By the Perron method (proposition~\ref{propositionPerronMethod}), there exists a function \(u^* \in L^1(\Omega)\) such that
\begin{enumerate}[\((a)\)]
	\item \(u^* \le \overline w\) in \(\Omega\),
	\item for every subsolution \(v\) of the nonlinear Dirichlet problem, \(v \le u^*\) in \(\Omega\),
	\item there exists a nondecreasing sequence of subsolutions \((u_n)_{n \in \N}\) of the nonlinear Dirichlet problem such that for every \(n \in \N\), \(u_n \le \overline w\) in \(\Omega\), and \((u_n)_{n \in \N}\) converges to \(u^*\) in \(L^1(\Omega)\).
	\end{enumerate}
	
We now write
\[
g(u_n) = \max{\{g(u_n), 0\}} + \min{\{g(u_n), 0\}}.
\]
By the absorption estimate (lemma~\ref{lemmaEstimateAbsorptionSubsolution}), the sequence \((\max{\{g(u_n), 0\}})_{n \in \N}\) is bounded in \(L^1(\Omega)\).
Thus, by Fatou's lemma,
\[
\max{\{g(u^*), 0\}} \in L^1(\Omega).
\] 

By the sign condition, 
\[
\min{\{g(u^*), 0\}} =  g(\chi_{\{u^* < 0\}} u^*).
\]
Since 
\[
\chi_{\{u^* < 0\}} u_0 \le \chi_{\{u^* < 0\}} u^* \le 0
\]
and since \(g(\chi_{\{u_0 < 0\}} u_0) \in L^1(\Omega)\) and \(g(0) \in L^1(\Omega)\), by the integrability condition we have
\[
\min{\{g(u^*), 0\}} =  g(\chi_{\{u^* < 0\}} u^*)  \in L^1(\Omega).
\]
We deduce that \(g(u^*) \in L^1(\Omega)\). 

Since the sequence \((u_n)_{n \in \N}\) is monotone and the integrability condition holds, we deduce from corollary~\ref{corollaryIntegrabilityConditionMonotoneConvergence} that the sequence \((g(u_n))_{n \in \N}\) converges to \(g(u^*)\) in \(L^1(\Omega)\).
Thus, \(u^*\) is a subsolution of the nonlinear Dirichlet problem with datum \(\mu\), whence it is the largest one.
\end{proof}

The integrability condition in proposition~\ref{propositionExistenceReducedMeasure} is unnecessary when \(\mu\) has a nonnegative subsolution since the sequence \((u_n)_{n \in \N}\) can be chosen to be nonnegative as well.
In view of the sign condition and Fatou's lemma, we deduce that \(u^*\) is still a subsolution of the Dirichlet problem.
We may summarize this as

\begin{proposition}
\label{propositionExistenceReducedMeasureBis}
Let \(g : \R \to \R\) be a continuous function satisfying the sign condition, and let \(\mu \in \cM(\Omega)\).
If the nonlinear Dirichlet problem with datum \(\mu\) has a \emph{nonnegative} subsolution, then the largest subsolution \(u^*\) exists.
\end{proposition}

We may also define a reduced measure in terms of the smallest supersolution:

\begin{definition}
Let \(g : \R \to \R\) be a continuous function and let \(\mu \in \cM(\Omega)\).
If the smallest supersolution \(u_*\) of the nonlinear Dirichlet problem with datum \(\mu\) exists, then the reduced measure \(\mu_*\) is the unique locally finite measure in \(\Omega\) such that
\[
\mu_* = - \Delta u_* + g(u_*)
\]
in the sense of distributions in \(\Omega\).
\end{definition}

We have the following counterparts for the existence of the smallest supersolution:

\begin{proposition}
\label{propositionExistenceReducedMeasureSuper}
Let \(g : \R \to \R\) be a continuous function satisfying the sign condition and the integrability condition, and let \(\mu \in \cM(\Omega)\).
If the nonlinear Dirichlet problem with datum \(\mu\) has a supersolution, then the smallest supersolution \(u_*\) exists.
\end{proposition}

\begin{proposition}
\label{propositionExistenceReducedMeasureBisSuper}
Let \(g : \R \to \R\) be a continuous function satisfying the sign condition, and let \(\mu \in \cM(\Omega)\).
If the Dirichlet problem with datum \(\mu\) has a \emph{nonpositive} supersolution, then the smallest supersolution \(u_*\) exists.
\end{proposition}


\section{Fundamental property}

In this section we prove the main property satisfied by the reduced measure:

\begin{proposition}
\label{propositionReducedMeasure}
Let \(g : \R \to \R\) be a continuous function satisfying the sign condition and the integrability condition.  For every \(\mu \in \cM(\Omega)\),
if the nonlinear Dirichlet problem with datum \(\mu\) has a subsolution, then \(\mu^* \in \cM(\Omega)\) and \(\mu^*\) is the largest good measure which is less than or equal to \(\mu\).
\end{proposition}

A good measure is a measure for which the nonlinear Dirichlet problem has a solution.

Our proof of the proposition relies on an approximation scheme introduced by Brezis, Marcus and Ponce~\cite{BreMarPon:07} which was used to prove the existence of the largest subsolution. 
In our case, the existence of the largest solution is obtained from the Perron method, without solving an auxiliary Dirichlet problem.

Starting from a sequence \((g_n)_{n \in \N}\) of bounded real functions converging to the nonlinearity \(g\), we solve the Dirichlet problem with nonlinearity \(g_n\).
If \((u_n)_{n \in \N}\) converges in \(L^1(\Omega)\) to a function \(u\), then in general \((g_n(u_n))_{n \in \N}\) does not converge to \(g(u)\) in \(L^1(\Omega)\). Indeed, the convergence in \(L^1(\Omega)\) of the nonlinear part implies that \(u\) is a solution of the nonlinear Dirichlet problem with datum \(\mu\), and by the counterexample of Bénilan and Brezis (proposition~\ref{propositionNonExistenceBenilanBrezis}) this cannot be always the case.

The delicate part of the proof is to identify the limit of the sequence \((g_n(u_n))_{n \in \N}\). 
In order to do this, we combine ideas from \citelist{\cite{BrePon:05}*{lemma~3} \cite{LPY:07}*{lemma~9.1} \cite{MarPon:10}*{lemma~3.1}}.

The heart of the matter is to find a sequence \((A_n)_{n \in \N}\) of Borel subsets of \(\Omega\) such that if we write
\[
g_n(u_n) = g_n(u_n)\chi_{A_n} + g_n(u_n)\chi_{\Omega \setminus A_n},
\]
then the sequence \((g_n(u_n)\chi_{A_n})_{n \in \N}\) converges to \(g(u)\) in \(L^1(\Omega)\) and the sequence \((g_n(u_n)\chi_{\Omega \setminus A_n})_{n \in \N}\) converges weakly in the sense of measures to a measure \(\gamma\) which is concentrated on a set of zero \(W^{1, 2}\) capacity.

This type of decomposition is reminiscent of the biting lemma discovered independently by Chacon and Rosenthal \cite{BroCha:80}. 
According to the biting lemma, it is always possible to decompose a sequence \((f_n)_{n \in \N}\)  of \(L^1\) functions in terms of two sequences: the first one is equi-integrable with respect to the Lebesgue measure and the second one concentrates on sets with Lebesgue measure converging to zero; we may replace the Lebesgue measure by some subadditive set function, in particular a capacity.
Connections of the biting lemma with the equation we study have been investigated in \cite{MarPon:10}.

\medskip
The first lemma relies on a diagonalization argument:

\begin{lemma}
\label{lemmaBitingLemmaAbsolutelyContinuous}
Let \((f_n)_{n \in \N}\) be a sequence in \(L^1(\Omega)\) converging pointwisely to \(f \in L^1(\Omega)\).
If \((A_k)_{k \in \N}\) is a sequence of Borel sets such that for every \(k \in \N\), the sequence \((f_n \chi_{A_k})_{n \in \N}\) converges to \(f \chi_{A_k}\) in \(L^1(\Omega)\) and if \(\lim\limits_{k \to \infty}{\meas{\Omega \setminus A_k}} = 0\), then there exists a sequence \((k_n)_{n \in \N}\) in \(\N\) diverging to \(+\infty\) such that the sequence \((f_n \chi_{A_{k_n}})_{n \in \N}\) converges to \(f\) in \(L^1(\Omega)\).
\end{lemma}

\begin{proof}
For every \(n \in \N\) and for every \(k \in \N\),
\[
\norm{f_n \chi_{A_k} - f}_{L^1(\Omega)} \le \norm{f_n \chi_{A_k} - f \chi_{A_k}}_{L^1(\Omega)} + \norm{f \chi_{\Omega \setminus A_k}}_{L^1(\Omega)}.
\]
Given a sequence \((\epsilon_i)_{i \in \N}\) of positive numbers, for every \(i \in \N\) there exists \(N_i \in \N\) such that for every \(n \ge N_i\),
\[
\norm{f_n \chi_{A_i} - f \chi_{A_i}}_{L^1(\Omega)} \le \epsilon_i.
\]
We may assume that the sequence \((N_i)_{i \in \N}\) is increasing. 
We define the sequence \((k_n)_{n \in \N}\) as follows.
If \(n < N_0\), then \(k_n = 0\).
If \(N_i \le n < N_{i + 1}\), then \(k_n = i\).
Thus, for every \(n \ge N_0\),
\[
\norm{f_n \chi_{A_{k_n}} - f \chi_{A_{k_n}}}_{L^1(\Omega)} \le \epsilon_{k_n}
\]
and this implies
\[
\norm{f_n \chi_{A_{k_n}} - f}_{L^1(\Omega)} \le \epsilon_{k_n} + \norm{f \chi_{\Omega \setminus A_{k_n}}}_{L^1(\Omega)}.
\]
By the dominated convergence theorem, the last term in the right-hand side converges to zero as \(n\) tends to infinity.
Choosing a sequence \((\epsilon_i)_{i \in \N}\) converging to zero,  the conclusion follows.
\end{proof}

Given a compact set \(K \subset \Omega\), we consider the minimization problem
\[
\inf{\bigg\{ \int\limits_\Omega \abs{\nabla u}^2 : u \in W_0^{1, 2}(\Omega) \text{ and } u \ge 1 \text{ on } K\bigg\}}.
\]
This obstacle problem has a unique solution \(u_K\) called the capacitary potential generated by \(K\).
The capacitary potential is superharmonic in \(\Omega\), is harmonic in \(\Omega \setminus K\) and is continuous in \(\Omega\) \cite{Car:67}. 
Thus, \(u_K\) satisfies the Dirichlet problem
\[
\left\{
\begin{alignedat}{2}
- \Delta u_K &= \nu_K	&&\quad \text{in \(\Omega\),}\\
u_K & = 0				&&\quad \text{on \(\partial\Omega\),}
\end{alignedat}
\right.
\]
where \(\nu_k \in \cM(\Omega)\) is a nonnegative measure supported on \(K\).
Moreover, 
\[
0 \le u_K \le 1
\]
in \(\Omega\) and
\[
\norm{u_K}_{W^{1, 2}(\Omega)} \le C \capt_{W^{1, 2}}{(K)}.
\]
 
The next tool is inspired from \cite{BrePon:05a}*{lemma~3}:

\begin{sublemma}
\label{sublemmaCapacitaryPotential}
Let \(\mu, \nu \in \cM(\Omega)\) be nonnegative measures such that for every superharmonic function \(\zeta \in C_0^\infty(\overline\Omega)\),
\[
\int\limits_\Omega \zeta \dif\nu \le \int\limits_\Omega \zeta \dif\mu.
\]
Let \(S \subset \R^N\) be a compact set and let \(\omega \subset \R^N\) be an open set such that \(S \subset \omega \subset \Omega\).
Then, for every \(\epsilon > 0\), there exists $\delta > 0$ independent of \(\nu\) such that for every Borel set $A \subset \Omega \setminus \omega$, if $\capt_{W^{1, 2}}{(A)} \le \delta$, then 
\[
\nu(A) \le \epsilon + \mu\lc(\Omega \setminus S).
\]
\end{sublemma}

\begin{proof}
We shall assume that the inequality
\[
\int\limits_\Omega \zeta \dif\nu \le \int\limits_\Omega \zeta \dif\nu.
\]
holds for every \emph{continuous} superharmonic function \(\zeta\); this property can be achieved by approximation of \(\zeta\).

For every compact set \(K \subset \Omega\), the capacitary potential of \(K\) satisfies \(u_K \ge 0\) in \(\Omega\) and \(u_K = 1\) in \(K\).
Hence,
\[
\nu(K) \le \int\limits_\Omega u_K \dif\nu.
\]
Since \(\mu = \mu\ld + \mu\lc\) and \(u_K \le 1\) in \(\Omega\),
\[
\int\limits_\Omega u_K \dif\mu 
\le \int\limits_\Omega u_K \dif\mu\ld + \int\limits_S  u_K \dif\mu\lc + \mu\lc(\Omega\setminus S). 
\]
Therefore,
\[
\nu(K) 
\le \int\limits_\Omega u_K \dif\mu\ld + \int\limits_S  u_K \dif\mu\lc + \mu\lc(\Omega\setminus S). 
\]

\Newclaim

\begin{claim}
Given \(\epsilon_1 > 0\), there exists \(\delta_1 > 0\) such that for every compact set \(K \subset \Omega\), if \(\norm{u_K}_{W^{1, 2}(\Omega)} \le \delta_1\), then
\[
\int\limits_\Omega u_K \dif\mu\ld \le \epsilon_1.
\]
\end{claim}

\begin{proof}[Proof of the claim]
By corollary~\ref{corollaryConvergenceBoccardoGallouetOrsina}, there exists  \(\delta_1 > 0\) such that for every \(\varphi \in C_c^\infty(\Omega)\) such that \(\norm{\varphi}_{L^\infty(\Omega)} \le 1\) and \(\norm{\varphi}_{W^{1, 2}(\Omega)} \le \delta_1\),
\[
\bigg| \int\limits_\Omega \varphi \dif\mu\ld \bigg| \le \epsilon_1.
\]
Since \(0 \le u_K \le 1\) in \(\Omega\),
by approximation of \(u_K\) in \(W_0^{1, 2}(\Omega)\) by functions in \(C_c^\infty(\Omega)\), we deduce that if \(\norm{u_K}_{W^{1, 2}(\Omega)} \le \delta_1\), then \(u_K\) satisfies the required property.
\end{proof}

\begin{claim}
Given \(\epsilon_2 > 0\), there exists \(\delta_2 > 0\) such that for every compact set \(K \subset \Omega \setminus \omega\), if \(\capt_{W^{1, 2}}{(K)} \le \delta_2\), then
\[
\int\limits_S u_K \dif\mu\lc \le \epsilon_2.
\]
\end{claim}

\begin{proof}[Proof of the claim]
Since \(u_K\) is harmonic in \(\omega\), by the mean value property for harmonic functions, for every \(x \in S\),
\[
u_K(x) \le \NewConstant \norm{u_K}_{L^1(\omega)},
\]
for some constant \(\SameConstant > 0\) depending on the distance \(d(S, \partial\omega)\).
Thus,
\[
\int\limits_S u_K \dif\mu\lc \le \SameConstant \norm{u_K}_{L^1(\omega)} \mu\lc(S) \le \Constant \norm{u_K}_{W^{1, 2}(\Omega)} \mu\lc(S).
\]
We have the conclusion by taking \(\delta_2 > 0\) such that \(\SameConstant \delta_2 \mu\lc(S) \le \epsilon_2\).
\end{proof}

Since
\[
\norm{u_K}_{W^{1, 2}(\Omega)} \le C \capt_{W^{1, 2}}{(K)},
\]
we deduce that if \(\delta > 0\) is such that \(C\delta \le \min{\{\delta_1, \delta_2\}}\), then for every compact set \(K \subset \Omega \setminus \omega\) such that \(\capt_{W^{1, 2}}{(K)} \le \delta\),
\[
\nu(K) \le \epsilon_1 + \epsilon_2 + \mu\lc(\Omega\setminus S). 
\]
Choosing \(\epsilon_1 + \epsilon_2 \le \epsilon\), the conclusion follows for compact sets.
Using the regularity of the \(W^{1, 2}\) capacity and the regularity of the measure  \(\nu\), we also have the estimate for open sets and for Borel sets.
\end{proof}

\begin{lemma}
\label{lemmaBitingLemmaSingular}
Let \(\mu \in \cM(\Omega)\) be a nonnegative measure and
let \((\nu_n)_{n \in \N}\) be a sequence of measures such that for every superharmonic function \(\zeta \in C_0^\infty(\overline\Omega)\),
\[
\int\limits_\Omega \zeta \dif\abs{\nu_n} \le \int\limits_\Omega \zeta \dif\mu. 
\]
If \((E_n)_{n \in \N}\) is a sequence of Borel subsets of \(\Omega\) such that 
\[
\lim\limits_{n \to \infty}{\capt_{W^{1, 2}}{(E_n)}} = 0\
\]
and if the sequence \((\nu_n \lfloor_{E_n})_{n \in \N}\) converges weakly in the sense of measures to \(\gamma\), then \(\gamma\) is a concentrated measure and
\[
\abs{\gamma} \le \mu\lc.
\]
\end{lemma}

\begin{proof}
Let \(\omega \subset \Omega\) be an open set.
By lower semicontinuity of the norm,
\[
\abs{\gamma}(\Omega \setminus \overline\omega) = \norm{\gamma}_{\cM(\Omega \setminus \overline\omega)}  \le \liminf_{n \to \infty}{\norm{\nu_n\lfloor_{E_n}}_{\cM(\Omega \setminus \overline\omega)}}.
\]
We may rewrite
\[
\norm{\nu_n\lfloor_{E_n}}_{\cM(\Omega \setminus \overline\omega)} 
= \nu_n(E_n \setminus \overline\omega).
\]

Let \(S \subset \omega\) be a compact set.
Given \(\epsilon > 0\), let \(\delta > 0\) satisfying the conclusion of the sublemma.
If \(\capt_{W^{1, 2}}{(E_n)} \le \delta\), then \(\capt_{W^{1, 2}}{(E_n \setminus \overline{\omega})} \le \delta\) and this implies
\[
\norm{\nu_n\lfloor_{E_n}}_{\cM(\Omega \setminus \overline\omega)} 
 = \nu_n(E_n \setminus \overline\omega)
\le \epsilon + \mu\lc(\Omega \setminus S).
\]
Thus,
\[
\abs{\gamma}(\Omega \setminus \overline\omega) \le \epsilon + \mu\lc(\Omega \setminus S).
\]
Since this inequality holds for every \(\epsilon > 0\),
\[
\abs{\gamma}(\Omega \setminus \overline\omega) \le \mu\lc(\Omega \setminus S).
\]
This inequality holds for every open set \(\omega\) containing \(S\).
Taking the supremum over the sets \(\omega\), by outer regularity of the measure \(\abs{\gamma}\) we have
\[
\abs{\gamma}(\Omega \setminus S) \le \mu\lc(\Omega \setminus S).
\]
This inequality holds for every compact set \(S \subset \Omega\).
Given an open set \(A \subset \Omega\), let \((S_n)_{n \in \N}\) be a nondecreasing sequence of compact sets such that 
\(\Omega \setminus A = \bigcup\limits_{n=0}^\infty S_n\), or equivalently,
\(A = \bigcap\limits_{n=0}^\infty (\Omega \setminus S_n)\).
Since the inequality above holds for each compact set \(S_n\), we deduce that
\[
\abs{\gamma}(A) \le \mu\lc(A).
\]
Thus, \(\abs{\gamma} \le \mu\lc\).
Since \(\abs{\gamma}\) is a nonnegative measure, this implies that \(\abs{\gamma}\) is a concentrated measure.
\end{proof}


\begin{proof}[Proof of proposition~\ref{propositionReducedMeasure}]
Let \((g_n)_{n \in \N}\) be the sequence of real functions defined for \(n \in \N\) by
\[
g_n = \min{\{g, n\}}.
\]
This sequence has the following properties that we need in the proof:
\begin{enumerate}[\((a)\)]
\item for every \(n \in \N\), \(g_n\) is continuous and satisfies the sign condition and the integrability condition,
\item for every \(n \in \N\), \(g_n\)is bounded from above,
\item \((g_n)_{n \in \N}\) is nondecreasing and converges uniformly to \(g\) on bounded subsets of \(\R\).
\end{enumerate}

Since the Dirichlet problem with nonlinearity \(g\) has a subsolution \(\underline{w}\) and since \(g_n \le g\), \(\underline{w}\) is also a subsolution of the Dirichlet problem with nonlinearity \(g_n\). 
Given \(\overline{\mu} \in \cM(\Omega)\), let
\(\overline{w}\) be the solution of the linear Dirichlet problem
\[
\left\{
\begin{alignedat}{2}
- \Delta \overline{w} & = \mu^+ &&\quad \text{in } \Omega,\\
\overline{w} & = 0 &&\quad \text{on } \partial\Omega.
\end{alignedat}
\right.
\]
If \(\overline{\mu} \ge \max{\{\mu, 0\}}\), then by the comparison principle (lemma~\ref{lemmaEstimateComparisonSubsolution}) every subsolution of the Dirichlet problem
\[
\left\{
\begin{alignedat}{2}
- \Delta u + g_n(u) & = \mu &&\quad \text{in } \Omega,\\
u & = 0 &&\quad \text{on } \partial\Omega,
\end{alignedat}
\right.
\]
is bounded from above by \(\overline{w}\). In particular,
\[
\underline{w} \le \overline{w}.
\]
Since \(g_n\) satisfies the integrability condition, by the Perron method and the method of sub and supersolutions (corollary~\ref{corollarySmallestLargestSolution}), the largest subsolution \(u_n\) of the Dirichlet problem with nonlinearity \(g_n\) and datum \(\mu\) exists and \(u_n\) is actually a solution of the Dirichlet problem.

The sequence \((u_n)_{n \in \N}\) is nonincreasing.
Indeed, for every \(n \in \N\), \(g_{n+1} \ge g_n\) and this implies that \(u_{n+1}\) is a subsolution of the Dirichlet problem with nonlinearity \(g_n\).
By maximality of \(u_n\), we deduce that
\[
u_{n+1} \le u_n.
\]

Since for every \(n \in \N\), \(u_n \ge \underline{w}\), by the monotone convergence theorem we deduce that the sequence \((u_n)_{n \in \N}\) converges in \(L^1(\Omega)\) to its pointwise limit \(u\).
Moreover, for every \(\zeta \in C_0^\infty(\overline{\Omega})\),
\[
\lim_{n \to \infty}{\int\limits_\Omega g_n(u_n) \zeta} = \int\limits_\Omega u \Delta\zeta + \int\limits_\Omega \zeta \dif\mu,
\]
in particular the limit exists.
Since, by the absorption estimate (lemma~\ref{lemmaEstimateAbsorption}) the sequence \((g_n(u_n))_{n \in \N}\) is bounded in \(L^1(\Omega)\), 
we deduce that the sequence \((g_n(u_n))_{n \in \N}\) converges weakly in the sense of measures in \(\Omega\).

The claim below shows that the limit of \((g_n(u_n))_{n \in \N}\) has a special form:

\begin{Claim}
\label{claimWeakConvergence}
The sequence \((g_n(u_n))_{n \in \N}\) converges weakly in the sense of measures to \(g(u) + \gamma\),  where \(\gamma\) is a measure concentrated on a set of \(W^{1, 2}\) capacity zero.
\end{Claim}

We postpone the proof of the claim. We show how the claim can be used in order to get the conclusion of the proposition.

Note that if \(v\) is a subsolution of the Dirichlet problem with nonlinearity \(g\), then
\[
v \le u.
\] 
Indeed, since \(g_n \le g\), \(v\) is a subsolution of the Dirichlet problem with nonlinearity \(g_n\). By the maximality of \(u_n\), we have for every \(n \in \N\), \(v \le u_n\). Passing to the limit, we deduce that \(v \le u\). In particular, the largest subsolution \(u^*\) of the Dirichlet problem with nonlinearity \(g\) satisfies
\[
u^* \le u.
\]

We now show that if \(\nu\) is a good measure such that \(\nu \le \mu\), then
\[
\nu \le \mu - \gamma.
\]
First of all, since the measure \(\gamma\) is concentrated in a set of zero capacity,
\[
\nu\ld \le \mu\ld = (\mu - \gamma)\ld.
\]
Denote by \(v\) a solution of the Dirichlet problem with datum \(\nu\). In particular, \(v\) is a subsolution of the Dirichlet problem with datum \(\mu\), whence \(v \le u\). 
By the claim, \(u\) satisfies
\[
-\Delta u + g(u) = \mu - \gamma
\]
in the sense of distributions in \(\Omega\).
By the inverse maximum principle (proposition~\ref{propositionInverseMaximumPrinciple}), 
\[
\nu\lc \le (\mu - \gamma)\lc.
\]
We deduce that 
\[
\nu = \nu\ld + \nu\lc \le (\mu - \gamma)\ld + (\mu - \gamma)\lc = \mu - \gamma.
\]
Similarly, one shows that the reduced measure \(\mu^*\) satisfies
\[
\mu^* \le \mu - \gamma.
\]

We now show that \(\mu^*\) is a finite measure --- up to this point we only know that \(\mu^*\) is a locally finite measure --- and for every good measure \(\nu\) such that \(\nu \le \mu\), we have 
\[
\nu \le \mu^*.
\]
We shall do this by showing that the Dirichlet problem with nonlinearity \(g\) and datum \(\max{\{\mu^*, \nu\}}\) has a solution.

We first observe that the measure \(\max{\{\mu^*, \nu\}}\) is finite since
\[
\nu \le \max{\{\mu^*, \nu\}} \le \mu
\]
and both measures \(\nu\) and \(\mu\) are finite.
We also observe that
\[
\mu^* \le \max{\{\mu^*, \nu\}} \le \mu - \gamma.
\]
Thus, \(u^*\) is a subsolution of the Dirichlet problem with datum \(\max{\{\mu^*, \nu\}}\) and \(u\) is a supersolution of the same problem. Since \(u^* \le u\) in \(\Omega\), by the method of sub and supersolutions (proposition~\ref{propositionMethodSubSuperSolutions}) there exists a solution \(\tilde u\) of the Dirichlet problem with datum \(\max{\{\mu^*, \nu\}}\) such that \(u^* \le \tilde u \le u\).

Since \(\max{\{\mu^*, \nu\}} \le \mu\), \(\tilde u\) is a subsolution of the Dirichlet problem with datum \(\mu\) and since \(u^*\) is the largest subsolution, \(\tilde u \le u^*\)  almost everywhere in \(\Omega\). Thus, \(\tilde u = u^*\). We conclude that 
\[
\mu^* = \max{\{\mu^*, \nu\}}.
\]
Therefore, \(\mu^*\) is a finite measure and \(\nu \le \mu^*\).

It remains to establish the claim:

\begin{proof}[Proof of the claim]
Let \(w\) be the solution of the Dirichlet problem
\[
\left\{
\begin{alignedat}{2}
- \Delta w & = \abs{\mu} && \quad \text{in \(\Omega\),}\\
w & = 0  && \quad \text{on \(\partial\Omega\).}
\end{alignedat}
\right.
\]
For every \(n \in \N\), by the comparison principle (lemma~\ref{lemmaEstimateComparison}), \(\abs{u_n} \le w\). 
For every \(s > 0\), since the sequence \((g_n)_{n \in \N}\) is uniformly bounded in \([-s, s]\), the sequence \((g_n(u_n) \chi_{\{w \le s\}})_{n \in \N}\) is bounded in \(L^\infty(\Omega)\).
By the dominated convergence theorem, \((g_n(u_n) \chi_{\{w \le s\}})_{n \in \N}\) converges to \(g(u) \chi_{\{w \le s\}}\) in \(L^1(\Omega)\).

By lemma~\ref{lemmaBitingLemmaAbsolutelyContinuous}, there exists a subsequence of positive numbers \((s_n)_{n \in \N}\) diverging to \(+\infty\) such that
\((g_n(u_n) \chi_{\{w \le s_n\}})_{n \in \N}\) converges to \(g(u)\) in \(L^1(\Omega)\).

Note that
\[
g_n(u_n) = g_n(u_n)\chi_{\{w \le s_n\}} + g_n(u_n)\chi_{\{w > s_n\}},
\]
We now show that the sequence \((g_n(u_n)\chi_{\{w > s_n\}})_{n \in \N}\) satisfies the assumptions of lemma~\ref{lemmaBitingLemmaSingular}.

By the capacitary estimate (lemma~\ref{lemmaCapacitaryEstimateLaplacian}), for every \(s > 0\),
\[
\capt_{W^{1, 2}}{(\{w > s\})} \le \frac{C}{s} \norm{\mu}_{\cM(\Omega)}.
\]
Thus,
\[
\lim_{n \to \infty}{\capt_{W^{1, 2}}{(\{w > s_n\})}} = 0.
\]

By Kato's inequality up to the boundary we have for every \(\zeta \in C_0^\infty(\overline\Omega)\) such that \(\zeta \ge 0\) in \(\Omega\),
\[
\int\limits_\Omega -\abs{u_n} \Delta\zeta + \int\limits_\Omega \abs{g_n(u_n)}\zeta \le \int\limits_\Omega \zeta \dif\abs{\mu}.
\]
If in addition \(\zeta\) is superharmonic in \(\Omega\),
then we have
\[
\int\limits_\Omega \abs{g_n(u_n)}\zeta \le \int\limits_\Omega \zeta \dif\abs{\mu}.
\]
In particular, 
\[
\int\limits_\Omega \abs{g_n(u_n)} \chi_{\{w > s_n\}} \zeta \le \int\limits_\Omega \zeta \dif\abs{\mu}.
\]
Thus, by lemma~\ref{lemmaBitingLemmaSingular} the sequence \((g_n(u_n)\chi_{\{w > s_n\}})_{n \in \N}\) converges weakly in the sense of measures to a measure \(\gamma\) which is concentrated in a set of \(W^{1, 2}\) capacity zero.
\end{proof}

The proof of the proposition is complete.
\end{proof}

When a nonnegative subsolution exists, we may drop the integrability condition and the proof of the fundamental property is shorter:

\begin{proposition}
\label{propositionReducedMeasureBis}
Let \(g : \R \to \R\) be a continuous function satisfying the sign condition.
For every \(\mu \in \cM(\Omega)\),
if the nonlinear Dirichlet problem with datum \(\mu\) has a \emph{nonnegative} subsolution, then \(\mu^* \in \cM(\Omega)\) and \(\mu^*\) is the largest good measure which is less than or equal to \(\mu\).
\end{proposition}

\begin{proof}
We recall that the largest solution of the Dirichlet problem with a nonlinearity satisfying the sign condition exists if there exists a nonnegative subsolution (propostion~\ref{propositionExistenceReducedMeasureBis}).
This guarantees the existence of the functions \(u_n\) in the proof of proposition~\ref{propositionExistenceReducedMeasure} and each \(u_n\) is nonnegative.

Next, the limit \(u\) of the sequence \((u_n)_{n \in \N}\) is a nonnegative solution of the Dirichlet problem with datum \(\mu - \gamma\).
By the Perron method (proposition~\ref{propositionExistenceReducedMeasureBis}), the largest subsolution \(\overline{u}\) of the Dirichlet problem  with datum \(\mu - \gamma\) exists.

By the sign condition, \((g_n(u_n))_{n \in \N}\) is a sequence of nonnegative functions, thus the concentrated measure \(\gamma\) is nonnegative.
In particular, \(\overline{u}\) is a subsolution of the nonlinear Dirichlet problem with datum \(\mu\).
By maximality of \(u^*\),
\[
\overline{u} \le u^*.
\]

On the other hand, in the proof of proposition~\ref{propositionReducedMeasure} we show that \(u^* \le u\). Thus,
\[
u^* \le u \le \overline{u}.
\]

We conclude that 
\[
\mu^* = \mu - \gamma\
\]
and this implies that \(\mu^*\) is a finite measure.
In the proof of proposition~\ref{propositionReducedMeasure}, we also show that every good measure \(\nu\) which is less than or equal to \(\mu\) satisfies \(\nu \le \mu - \gamma\), whence
\[
\nu \le \mu - \gamma = \mu^*.
\]
This proves the proposition.
\end{proof}

We state the following counterparts of the fundamental property in the case of supersolutions and the reduced measure \(\mu_*\).

\begin{proposition}
\label{propositionReducedMeasureSuper}
Let \(g : \R \to \R\) be a continuous function satisfying the sign condition and the integrability condition.  For every \(\mu \in \cM(\Omega)\),
if the nonlinear Dirichlet problem with datum \(\mu\) has a supersolution, then \(\mu_* \in \cM(\Omega)\) and \(\mu_*\) is the smallest good measure which is greater than or equal to \(\mu\).
\end{proposition}

\begin{proposition}
\label{propositionReducedMeasureBisSuper}
Let \(g : \R \to \R\) be a continuous function satisfying the sign condition.
For every \(\mu \in \cM(\Omega)\),
if the nonlinear Dirichlet problem with datum \(\mu\) has a \emph{nonpositive} supersolution, then \(\mu_* \in \cM(\Omega)\) and \(\mu_*\) is the smallest good measure which is greater than or equal to \(\mu\).
\end{proposition}


\section{Consequences}

We present some applications of the fundamental property of the reduced measures in the study of the nonlinear Dirichlet problem. 
Many of the results in this section extend \cite{BreMarPon:07}, but the the proofs based on the fundamental property are fundamentally the same.

We start with one of the most puzzling properties of the nonlinear Dirichlet problem with sign condition.
If \(\mu \in \cM(\Omega)\) is a nonnegative measure, then the function identically zero is a subsolution of the Dirichlet problem with datum \(\mu\).
By the Perron method, the largest subsolution \(u^*\) of this problem exists, in particular \(u^ \ge 0\), but we do not have a direct proof that
\[
-\Delta u^* + g(u^*) \ge 0
\]
in the sense of distributions in \(\Omega\).
The only proof we know relies on the fundamental property of reduced measures:

\begin{corollary}
Let \(g : \R \to \R\) be a continuous function satisfying the sign condition. 
If \(\mu \in \cM(\Omega)\) is a nonnegative measure, then \(\mu^*\) is a nonnegative measure.
\end{corollary}

\begin{proof}
By the sign condition, the null measure \(0\) is a good measure and, by assumption, is less than or equal to \(\mu\), by the fundamental property (proposition~\ref{propositionReducedMeasureBis}), \(0 \le \mu^*\).
\end{proof}

Another property which is closely related is the following:

\begin{corollary}
\label{corollaryGoodMeasuresPositivePart}
Let \(g : \R \to \R\) be a continuous function satisfying the sign condition. 
If \(\mu \in \cM(\Omega)\) is a good measure, then \(\max{\{\mu, 0\}}\) and \(\min{\{\mu, 0\}}\) are also good measures.
\end{corollary}

\begin{proof}
We show that \(\mu^+ = \max{\{\mu, 0\}}\) is a good measure.
Since the function identically \(0\) is a subsolution of the Dirichlet problem with datum \(\mu^+\), by the Perron method (proposition~\ref{propositionExistenceReducedMeasureBis}) the reduced measure \((\mu^+)^*\) exists.
Since the null measure \(0\) and the measure \(\mu\) are good measures less than or equal to \(\mu^+\), by the fundamental property (proposition~\ref{propositionReducedMeasureBis}) we have
\[
0 \le (\mu^+)^* 
\quad
\text{and}
\quad
\mu \le (\mu^+)^*.
\]
Thus,
\[
\mu^+ = \max{\{0, \mu\}} \le (\mu^+)^*.
\]
Therefore, \((\mu^+)^* = \mu^+\), which means that \(\mu^+ = \max{\{0, \mu\}}\) is a good measure. 
The proof that \(\min{\{0, \mu\}}\) is a good measure is similar.
\end{proof}

We have the following connection with the existence of extremal solutions:

\begin{corollary}
\label{corollaryGoodMeasuresSubSuper}
Let \(g : \R \to \R\) be a continuous function satisfying the sign condition and the integrability condition. 
IF \(\mu \in \cM(\Omega)\) is a good measure, then the nonlinear Dirichlet problem with datum \(\mu\) has a smallest and a largest solution.
\end{corollary}

\begin{proof}
If the nonlinear Dirichlet problem with datum \(\mu\) has a solution, then by the Perron method (proposition~\ref{propositionExistenceReducedMeasure} and proposition~\ref{propositionExistenceReducedMeasureSuper}) the largest subsolution \(u^*\) and the smallest supersolution \(u_*\) exist.
In particular, for every solution \(u\),
\[
u_* \le u \le u^*.
\]
By the fundamental property (proposition~\ref{propositionReducedMeasure} and proposition~\ref{propositionReducedMeasureSuper}), \(\mu^* = \mu\) and \(\mu_* = \mu\).
Thus, \(u^*\) and \(u_*\) satisfy the nonlinear Dirichlet problem with datum \(\mu\) and this gives the conclusion.
\end{proof}

We have the following connection with the method of sub and supersolution:

\begin{corollary}
\label{corollarySmallestLargestSolution}
Let \(g : \R \to \R\) be a continuous function satisfying the sign condition and the integrability condition. 
For every \(\mu \in \cM(\Omega)\), if the nonlinear Dirichlet problem with datum \(\mu\) has a subsolution and a supersolution, then there exists a solution.
\end{corollary}

The main issue in this statement is the fact that we do not assume that the subsolution \(\underline v\) and the supersolution \(\overline v\) satisfy \(\underline v \le \overline v\) in \(\Omega\).
The proof consists in showing that there exists a supersolution \(\overline v\) for which this inequality holds.

\begin{proof}
By the Perron method (proposition~\ref{propositionExistenceReducedMeasureBis}), the reduced measure \(\mu_*\) exists and satisfies 
\[
\mu \le \mu_*.
\] 
Every subsolution of the Dirichlet problem with datum \(\mu\) is a subsolution of the Dirichlet problem with datum \(\mu_*\).
By the Perron method (proposition~\ref{propositionExistenceReducedMeasure}), there exists a largest subsolution \(\overline{u}\) of the  Dirichlet problem with datum \(\mu_*\).
Since \(\mu_*\) is a good measure, by the fundamental property (proposition~\ref{propositionReducedMeasure}), \(\overline{u}\) is a solution of the Dirichlet problem with datum \(\mu_*\).
Thus, \(\overline{u}\) is a supersolution of the Dirichlet problem with datum \(\mu_*\).
On the other hand, by maximality of \(\overline{u}\), for any subsolution \(\underline{v}\) of the Dirichlet problem with datum \(\mu\), 
\[
\underline{v} \le \overline{u}.
\]
We deduce from the method of sub and supersolution (proposition~\ref{propositionMethodSubSuperSolutions}) that the Dirichlet problem with measure \(\mu\) has a solution, whence \(\mu\) is a good measure.
\end{proof}

Under the assumptions of the previous corollary, the reduced measures coincide \(\mu_* = \mu^* = \mu\), \(u^*\) is the largest solution and \(u_*\) is the smallest solution of the Dirichlet problem with datum \(\mu\).

\medskip
Corollary~\ref{corollaryGoodMeasuresPositivePart} has the following generalization:

\begin{corollary}
\label{corollaryGoodMeasuresMax}
Let \(g : \R \to \R\) be a continuous function satisfying the sign condition and the integrability condition. 
If \(\mu_1, \dots, \mu_n \in \cM(\Omega)\) are good measures, then \(\max{\{\mu_1, \dots, \mu_n\}}\) and \(\min{\{\mu_1, \dots, \mu_n\}}\) are also good measures.
\end{corollary}

\begin{proof}
For every \(i \in \{1, \dots, n\}\), \(\mu_i\) is a good measure smaller than or equal to \(\max{\{\mu_1, \dots, \mu_n\}}\). Thus, by the Perron method (proposition~\ref{propositionExistenceReducedMeasure}) the measure \((\max{\{\mu_1, \dots, \mu_n\}})^*\) exists and, by the fundamental property (proposition~\ref{propositionReducedMeasure}), satisfies
\[
\mu_i \le (\max{\{\mu_1, \dots, \mu_n\}})^*.
\]
Hence,
\[
\max{\{\mu_1, \dots, \mu_n\}} \le (\max{\{\mu_1, \dots, \mu_n\}})^*
\]
and equality holds. The proof that \(\min{\{\mu_1, \dots, \mu_n\}}\) is a good measure is similar.
\end{proof}

\begin{corollary}
Let \(g : \R \to \R\) be a continuous function satisfying the sign condition and the integrability condition, 
and let \(\mu \in \cM(\Omega)\). If there exists \(h \in L^1(\Omega)\) such that \(h + \mu\) is a good measure, then \(\mu\) is a good measure. 
\end{corollary}

\begin{proof}
Let 
\begin{align*}
\overline\lambda = \max{\{\mu\la, h\}} + \max{\{\mu\ls, 0\}}
\intertext{and}
\underline\lambda = \min{\{\mu\la, h\}} + \min{\{\mu\ls, 0\}},
\end{align*}
where \(\mu\la\) denotes the absolutely continuous part of the measure \(\mu\) with respect to the Lebesgue measure and \(\mu\ls\) the singular part.
Then, 
\[
\underline\lambda \le \mu \le \overline\lambda.
\]

Note that
\[
h + \mu \le \overline\lambda 
\quad 
\text{and}
\quad
h\la \le \overline\lambda .
\]
By assumption, \(h + \mu\) is a good measure and \(h\la\) is also a good measure; see proposition~\ref{propositionExistenceSolutionDiffuseMeasure} or proposition~\ref{propositionBrezisStrauss}.
By the Perron method (proposition~\ref{propositionExistenceReducedMeasure}), the  largest subsolution \(\overline u\) of the Dirichlet problem with datum \(\overline\lambda\) exists.
By the fundamental property (proposition~\ref{propositionReducedMeasure}), the reduced measure \((\overline\lambda)^*\) satisfies
\[
\overline\lambda = \max{\{h + \mu, \mu\la\}} \le (\overline\lambda)^*.
\]
We deduce that \(\overline\lambda = (\overline\lambda)^*\), which implies that  \(\overline u\) is a solution of the Dirichlet problem with datum \(\overline\lambda\). 

Similarly, the Dirichlet problem with datum \(\underline\lambda\) has a solution \(\underline u\). 
Since \(\underline u\) is a subsolution of the Dirichlet problem with datum \(\overline\lambda\), by maximality of \(\overline u\) we have
\[
\underline u \le \overline u.
\]
By the method of sub and supersolution (proposition~\ref{propositionMethodSubSuperSolutions}), we have the conclusion.
\end{proof}

It is also possible to obtain the following functional characterization of good measures in the spirit of \cites{BarPie:84,GalMor:84}:

\begin{corollary}
Let \(g : \R \to \R\) be a continuous function satisfying the sign condition and the integrability condition. 
For every \(\mu \in \cM(\Omega)\), \(\mu\) is a good measure if and only if there exist \(f \in L\loc^1(\Omega)\) and \(u \in L\loc^1(\Omega)\) such that \(g(u) \in L\loc^1(\Omega)\) and
\[
\mu = f - \Delta u
\]
in the sense of distributions in \(\Omega\).
\end{corollary}

A delicate issue in this characterization is that we have no control on \(f\) and \(u\) near the boundary.
Once we know that \(\mu\) is a good measure, then such decomposition holds with \(\tilde u \in W_0^{1, 1}(\Omega)\) and \(f = g(\tilde u)\).

For nondecreasing nonlinearities, this result is due to Brezis, Marcus and Ponce~\cite{BreMarPon:07}*{theorem~4.6}.
The strategy of the proof relies on the study of the problem satisfied by \(\varphi u\), where \(\varphi\) is a smooth function with compact support.

\begin{proof}
For every \(\varphi \in C_c^\infty(\Omega)\) such that \(0 \le \varphi \le 1\),
\[
\min{\{\mu, 0\}} \le \varphi\mu \le \max{\{\mu, 0\}}.
\]

\begin{Claim}
For every \(\varphi \in C_c^\infty(\Omega)\) such that \(0 \le \varphi \le 1\) in \(\Omega\), \(\varphi\mu\) is a good measure. 
\end{Claim}

We temporarily assume the claim.
Since for every \(\varphi \in C_c^\infty(\Omega)\) such that \(0 \le \varphi \le 1\),
\[
\min{\{\mu, 0\}} \le \varphi\mu \le \max{\{\mu, 0\}},
\]
and since \(\varphi\mu\) is a good measure, by the Perron method (proposition~\ref{propositionExistenceReducedMeasure} and proposition~\ref{propositionExistenceReducedMeasureSuper}), the reduced measures
\((\min{\{\mu, 0\}})_*\) and \((\max{\{\mu, 0\}})^*\) exist and by the fundamental property (proposition~\ref{propositionReducedMeasure} and proposition~\ref{propositionReducedMeasureSuper}),
\[
(\min{\{\mu, 0\}})_* \le \varphi\mu \le (\max{\{\mu, 0\}})^*.
\]

Since the function \(\varphi\) is arbitrary, this implies
\[
(\min{\{\mu, 0\}})_* \le \min{\{\mu, 0\}} 
\quad
\text{and}
\quad
\max{\{\mu, 0\}} \le (\max{\{\mu, 0\}})^*.
\]
Therefore, equality must hold in both cases.
Using the method of sub and supersolutions (proposition~\ref{propositionMethodSubSuperSolutions}), we deduce that \(\mu\) is a good measure.

\begin{proof}[Proof of the claim]
By the localization property (lemma~\ref{lemmaLocalizationMeasure}), \(\varphi u \in W_0^{1, 1}(\Omega)\) and for every \(\varphi \in C_c^\infty(\Omega)\),
\[
\varphi\mu = \varphi f - \varphi \Delta u = (\varphi f + \Delta\varphi \, u + 2 \nabla\varphi \cdot \nabla u) - \Delta(\varphi u).
\]
If in addition \(0 \le \varphi \le 1\), then
\[
\min{\{u, 0\}} \le \varphi u \le \max{\{u, 0\}},
\]
hence by the integrability condition we have \(g(\varphi u) \in L^1(\Omega)\).
Thus, \(- \Delta(\varphi u) + g(\varphi u)\) is a good measure.
Since
\[
\varphi f + \Delta\varphi \, u + 2 \nabla\varphi \cdot \nabla u - g(\varphi u) \in L^1(\Omega),
\]
by the previous corollary \(\varphi\mu\) is a good measure. 
\end{proof}

This concludes the proof of the corollary.
\end{proof}

\begin{corollary}
Let \(g : \R \to \R\) be a continuous function satisfying the sign condition and the integrability condition, and let \(\mu \in \cM(\Omega)\).
If there exists a monotone sequence \((\mu_n)_{n \in \N}\) of good measures converging strongly to \(\mu\) in \(\cM(\Omega)\), then \(\mu\) is a good measure.
\end{corollary}

\begin{proof}
For every \(n \in \N\), \(\mu_n\) is a good measure less than or equal to \(\mu\).
By the Perron method (proposition~\ref{propositionExistenceReducedMeasure}), the reduced measure \(\mu^*\) exists and by the fundamental property (proposition~\ref{propositionReducedMeasure}), we have 
\[
\mu_n \le \mu^*.
\]
As we let \(n\) tend to infinity, we deduce that \(\mu \le \mu^*\). Thus, equality holds and \(\mu\) is a good measure.
\end{proof}

It is possible to give a direct proof of the previous corollary --- without using the reduced measure --- if we know that the sequence of solutions \((u_n)_{n \in \N}\) is nondecreasing.
In this case, the conclusion follows from the absorption estimate and corollary~\ref{corollaryIntegrabilityConditionMonotoneConvergence}.
This is the strategy we have adopted to study the Dirichlet problem with polynomial and exponential nonlinearities.

\begin{corollary}
\label{corollaryGoodMeasuresClosed}
Let \(g : \R \to \R\) be a continuous function satisfying the sign condition and the integrability condition, and let \(\mu \in \cM(\Omega)\).
If there exists a sequence \((\mu_n)_{n \in \N}\) of good measures converging strongly to \(\mu\) in \(\cM(\Omega)\), then \(\mu\) is a good measure.
\end{corollary}

In this corollary, the sequence \((\mu_n)_{n \in \N}\) need not be monotone. 
When the nonlinearity \(g\) is nondecreasing, this property may be deduced from the contraction property of the Dirichlet problem (see proof of proposition~\ref{propositionBrezisStrauss}).

\begin{proof}
Since the sequence \((\mu_n)_{n \in \N}\) converges in \(\cM(\Omega)\), there exists a subsequence \((\mu_{n_k})_{k \in \N}\) such that the series
\[
\sum_{k = 0}^\infty \norm{\mu_{n_{k + 1}} - \mu_{n_{k}}}_{\cM(\Omega)}
\]
converges.
For this subsequence, there exist  \(\underline\lambda, \overline\lambda \in \cM(\Omega)\) such that for every \(k \in \N\),
\[
\underline\lambda \le \mu_{n_{k}} \le \overline\lambda;
\]
the existence of \(\underline{\lambda}\) and \(\overline{\lambda}\) follows by mimicking the proof of \cite{Bre:11}*{theorem~4.9} for \(L^p\) functions.

By the Perron method (proposition~\ref{propositionExistenceReducedMeasure} and proposition~\ref{propositionExistenceReducedMeasureSuper}), the reduced measures \(\underline{\lambda}_*\) and \(\overline{\lambda}^*\) exist.
Thus, by the fundamental property (proposition~\ref{propositionReducedMeasure} and proposition~\ref{propositionReducedMeasureSuper}),
\[
\underline\lambda_* \le \mu_{n_{k}} \le \overline\lambda^*.
\]
As \(k\) tends to infinity, we get
\[
\underline\lambda_* \le \mu \le \overline\lambda^*.
\]
Since  \(\underline{\lambda}_*\) and \(\overline{\lambda}^*\) are good measures, by corollary~\ref{corollaryGoodMeasuresSubSuper} \(\mu\) is also a good measure.
\end{proof}

\appendix
\chapter{Sobolev capacity}

Let \(k \in \N_*\) and let \(1 \le p < +\infty\). Given a compact set \(K \subset \R^N\), the \(W^{k, p}\) capacity of \(K\) is defined as
\[
\capt_{W^{k, p}}{(K)} = \inf{\left\{ \norm{\varphi}_{W^{k, p}(\R^N)}^p : \varphi \in C_c^\infty(\R^N) \ \text{and} \ \varphi \ge 1 \ \text{in} \ K \right\}}.
\]
We may extend the capacity to any subset \(A \subset \R^N\) as follows. If \(A\) is open, let
\[
\capt_{W^{k, p}}{(A)} = \sup{\Big\{ \capt_{W^{k, p}}(K) : \text{\(K \subset \R^N\) is compact and \(K \subset A\)} \Big\}},
\]
and more generally,
\[
\capt_{W^{k, p}}{(A)} = \inf{\Big\{ \capt_{W^{k, p}}(U) : \text{\(U \subset \R^N\) is open and \(A \subset U\)} \Big\}}.
\]

\section{Pointwise convergence}

If \(\capt_{W^{k, p}}(A) = 0\), then \(A\) is negligible with respect to the Lebesgue measure in \(\R^N\). 
In this sense, the \(W^{k, p}\) capacity gives a finer information concerning pointwise convergence in \(W^{k, p}(\R^N)\):

\begin{proposition}
\label{propositionPointwiseConvergence}
Let \(k \in \N_*\) and \(1 < p < +\infty\), and let \((\varphi_n)_{n \in \N}\) be a sequence in \(C_c^\infty(\R^N)\). 
If \((\varphi_n)_{n \in \N}\) converges in \(W^{k, p}(\R^N)\), then there exists a subsequence \((\varphi_{n_i})_{i \in \N}\) converging pointwisely in \(\R^N \setminus E\) for some Borel set \(E \subset \R^N\) such that \(\capt_{W^{k, p}}{(E)} = 0\).
\end{proposition}

The proof is based on the following properties of the capacity:

\begin{proposition}
Let \(k \in \N_*\) and \(1 \le p < +\infty\).
For every \(A, B \subset \R^N\) such that \(A \subset B\), 
\[
\capt_{W^{k, p}}{(A)} \le \capt_{W^{k, p}}{(B)}.
\]
\end{proposition}

\begin{proposition}
\label{propositionSubadditivityCapacity}
Let \(k \in \N_*\). 
If \(1 < p < +\infty\), then for every sequence \((A_n)_{n \in \N}\) of subsets of \(\R^N\),
\[
\textstyle \capt_{W^{k, p}}{(\bigcup\limits_{n = 0}^\infty A_n)} \le C \displaystyle \sum_{n = 0}^\infty {\capt_{W^{k, p}}{(A_n)}},
\]
for some constant \(C > 0\) depending on \(N\), \(k\) and  \(p\).
\end{proposition}

\begin{proposition}
\label{propositionCapacitaryEstimateWkp}
Let \(k \in \N_*\). 
If \(1 < p < \infty\), then for every \(\varphi \in C_c^\infty(\R^N)\) and for every \(\alpha > 0\),
\[
\capt_{W^{k, p}}{(\{\abs{\varphi} \ge \alpha\})} \le \frac{C}{\alpha^p} \norm{\varphi}_{W^{k, p}(\R^N)}^p,
\]
for some constant \(C > 0\) depending on \(N\), \(k\) and \(p\).
\end{proposition}

The set function \(\capt_{W^{k, p}}\) we have defined does not seem to be a true capacity in the sense of Choquet~\cite{Cho:54} since we do not know whether it is subadditive, or equivalently whether the conclusion of proposition above holds with constant \(C = 1\). 
This is so when \(k = 1\) and \(1 \le p < +\infty\) \cite{Wil:07}*{théorème~26.2}
but it is unlikely to be the case when \(k \ge 2\) since the constant should depend on the norm we choose to compute \(\abs{D^j u}\) pointwisely. 

The capacitary estimate of proposition~\ref{propositionCapacitaryEstimateWkp} would be a trivial consequence of the definition of the capacity if we could take the function \(\abs{\varphi}/\alpha\) as test function.
If \(k = 1\), this is almost true since \(\abs{\varphi} \in W^{1, p}(\R^N)\) and by a convolution argument we have the estimate with \(C = 1\). 
This argument is no longer possible for \(k \ge 2\) since it need not be true that \(\abs{\varphi} \in W^{k, p}(\R^N)\).

\medskip

The proofs of proposition~\ref{propositionSubadditivityCapacity} and \ref{propositionCapacitaryEstimateWkp} are based on the following lemma:

\begin{lemma}
\label{lemmaMaximumTestFunctions}
Let \(k \in \N_*\). 
If \(1 < p < +\infty\), then for every \(\varphi_1, \dots, \varphi_\ell \in C_c^\infty(\R^N)\), there exists \(\psi \in C_c^\infty(\R^N)\) such that
\[
\max{\{\varphi_1, \dots, \varphi_\ell\}} \le \psi
\]
and
\[
\norm{\psi}_{W^{k, p}(\R^N)}^p \le C \sum_{i = 1}^\ell \norm{\varphi_i}_{W^{k, p}(\R^N)}^p,
\]
for some constant \(C > 0\) depending on \(N\), \(k\) and \(p\).
\end{lemma}

\begin{proof}
Let \(G_k\) be the Bessel potential in \(\R^N\), that is to say, \(G_k : \R^N \to \R\) is a positive function whose Fourier transform is given by
\[
\widehat G_k(\xi) = (1 + 4\pi^2 \norm{x}^2)^{-\frac{k}{2}}.
\]
The Bessel potential establishes an isomorphism between \(L^p(\R^N)\) and \(W^{k, p}(\R^N)\): for each \(i \in \{1, \dots, \ell\}\), there exists \(f_i \in L^p(\R^N)\) such that
\[
\varphi_i = G_k * f_i
\]
and
\[
\norm{f_i}_{L^p(\R^N)} \le \NewConstant \norm{\varphi_i}_{W^{k, p}(\R^N)};
\]
see \cite{Ste:70}*{theorem~V.3}.

Let
\[
h = \max{\{f_1, \dots, f_\ell\}}.
\]
Then, for every \(i \in \{1, \dots, \ell\}\),
\(
\varphi_i \le G_k * h
\),
whence
\[
\max{\{\varphi_1, \dots, \varphi_\ell\}} \le G_k * h.
\]
Moreover,
\[
\int\limits_{\R^N} \abs{h}^p \le \sum_{i = 1}^\ell \int\limits_{\R^N} \abs{f_i}^p \le \Constant \sum_{i = 1}^\ell  \norm{\varphi_i}_{W^{k, p}(\R^N)}^p.
\]

If \(G_k * h \in C_c^\infty(\R^N)\), then we are done.
If this is not the case, then we may apply an approximation argument as follows.
First of all, note that \(h\) does not coincide with any of the functions \(f_i\).
Thus, for every \(x \in \R^N\) we have
\[
\max{\{\varphi_1(x), \dots, \varphi_\ell(x)\}} < G_k * h (x).
\]
For every \(i \in \{1, \dots, \ell\}\), the function \(f_i\) has fast decay at infinity, hence \(h\) also has fast decay at infinity. In particular, \(G_k * h\) is continuous in \(\R^N\). 
We may take the function \(\psi\) of the form
\[
\psi = [\rho_\epsilon * (G_k * h)] \zeta,
\]
where \(\epsilon > 0\) is such that for every \(x \in \bigcup\limits_{i = 1}^\ell \supp{\varphi_i}\)
\[
\max{\{\varphi_1 (x), \dots, \varphi_\ell(x)\}} \le \rho_\epsilon * (G_k * h)
\]
and \(\zeta \in C_c^\infty(\R^N)\) is such that \(\zeta = 1\) on \(\bigcup\limits_{i = 1}^\ell \supp{\varphi_i}\) and for some given \(\alpha > 1\),
\[
\bignorm{[\rho_\epsilon * (G_k * h)] \zeta }_{W^{k, p}(\R^N)} \le \alpha \bignorm{\rho_\epsilon * (G_k * h)}_{W^{k, p}(\R^N)}.
\]
Thus,
\[
\norm{\psi}_{W^{k, p}(\R^N)} \le \alpha \norm{G_k * h}_{W^{k, p}(\R^N)} \le \Constant \norm{h}_{L^p(\R^N)}.
\]
This gives the estimate we sought.
\end{proof}

\begin{proof}[Proof of proposition~\ref{propositionSubadditivityCapacity}]
We prove the proposition for a sequence \((A_n)_{n \in \N}\) of open subsets of \(\R^N\). Given a compact set \(K \subset \bigcup\limits_{n \in \N} A_n\), there exists finitely open sets \(A_{n_1}, \dots, A_{n_\ell} \) such that
\[
K \subset A_{n_1} \cup \dots \cup A_{n_\ell}.
\]
For each \(i \in \{1, \dots, \ell\}\), take a compact set \(K_i \subset A_{n_i}\) such that
\[
K = K_1 \cup \dots \cup K_\ell.
\]
Then, by monotonicity of the capacity,
\[
\sum_{i = 1}^\ell \capt_{W^{k, p}}{(K_i)} \le \sum_{i = 1}^\ell \capt_{W^{k, p}}{(A_{n_i})} \le \sum_{n = 0}^\infty \capt_{W^{k, p}}{(A_n)}.
\]

For each \(i \in \{1, \dots, \ell\}\), let \(\varphi_i \in C_c^\infty(\Omega)\) be such that \(\varphi_i \ge 1\) in \(K_i\). Then,
\[
\max{\{\varphi_1, \dots, \varphi_\ell\}} \ge 1 \quad \text{in \(K\)}.
\]
Let \(\psi \in C_c^\infty(\R^N)\) be a function satisfying the conclusion of the previous lemma. Since \(\psi \ge 1\) in \(K\), we deduce that
\[
\capt_{W^{k, p}}(K) \le \norm{\psi}_{W^{k, p}(\R^N)}^p \le C \sum_{i = 1}^\ell \norm{\varphi_i}_{W^{k, p}(\R^N)}^p.
\]
Taking the infimum of the right-hand side over the functions \(\varphi_i\), we deduce that
\[
\capt_{W^{k, p}}(K) \le C \sum_{i = 1}^\ell \capt_{W^{k, p}}{(K_i)} \le C \sum_{n = 0}^\infty \capt_{W^{k, p}}{(A_n)}.
\]
Since this estimate holds for every compact subset of \(\bigcup\limits_{n \in \N} A_n\), the conclusion holds when all the sets \(A_n\) are open.
\end{proof}

\begin{proof}[Proof of proposition~\ref{propositionCapacitaryEstimateWkp}]
By the previous lemma, there exists \(\psi \in C_c^\infty(\R^N)\) such that
\[
\abs{\varphi} = \max{\{\varphi, -\varphi\}} \le \psi
\]
and
\[
\norm{\psi}_{W^{k, p}(\R^N)} \le C \norm{\varphi}_{W^{k, p}(\R^N)}.
\]
Since \(\psi/\alpha \ge 1\) in the set \(\{\abs{\varphi} \ge \alpha\}\), by the definition of capacity we have the conclusion.
\end{proof}

We now prove the pointwise convergence up to sets of zero capacity:

\begin{proof}[Proof of proposition~\ref{propositionPointwiseConvergence}]
Given a sequence of positive numbers \((\alpha_i)_{i \in \N}\) and a subsequence \((\varphi_{n_i})_{i \in \N}\) of \((\varphi_n)_{n \in \N}\), let
\[
E = \bigcap\limits_{j = 0}^\infty \bigcup\limits_{i= j}^\infty \{\abs{\varphi_{n_{i+1}} - \varphi_{n_i}} \ge \alpha_i\}.
\] 
If the series \(\sum\limits_{i=0}^\infty \alpha_i\) converges, then
for every \(x \in \R^N \setminus E\),  \((\varphi_{n_i}(x))_{i \in \N}\) is a Cauchy sequence in \(\R\).

For every \(j \in \N\) we have by monotonicity and by subadditivity of the capacity,
\[
\begin{split}
\capt_{W^{k, p}}{(E)} 
& \le \textstyle\capt_{W^{k, p}}(\bigcup\limits_{i = j}^\infty \{\abs{\varphi_{n_{i+1}} - \varphi_{n_i}} \ge \alpha_i\})\\
& \le \displaystyle C \sum_{i = j}^\infty \capt_{W^{k, p}}(\{\abs{\varphi_{n_{i+1}} - \varphi_{n_i}} \ge \alpha_i\}).
\end{split}
\]
By the capacitary estimate of the sets \(\{\abs{\varphi_{n_{i+1}} - \varphi_{n_i}} \ge \alpha_i\}\) (proposition~\ref{propositionCapacitaryEstimateWkp}),
\[
\capt_{W^{k, p}}{(\{\abs{\varphi_{n_{i+1}} - \varphi_{n_i}} \ge \alpha_i\})} 
\le \frac{C}{\alpha_i^p} \norm{\varphi_{n_{i+1}} - \varphi_{n_i}}_{W^{k, p}(\R^N)}^p.
\]
Thus,
\[
\capt_{W^{k, p}}{(E)} \le C^2  \sum_{i = j}^\infty \frac{1}{\alpha_i^p} \norm{\varphi_{n_{i+1}} - \varphi_{n_i}}_{W^{k, p}(\R^N)}^p.
\]

Choosing a subsequence \((\varphi_{n_i})_{i \in \N}\) such that the series
\[
\sum_{i = 0}^\infty \frac{1}{\alpha_i^p} \norm{\varphi_{n_{i+1}} - \varphi_{n_i}}_{W^{k, p}(\R^N)}^p
\]
converges, then \(\capt_{W^{k, p}}{(E)} = 0\) and the conclusion follows.
\end{proof}

\begin{remark}
\label{remarkPointwiseConvergence}
If the sequence \((\varphi_n)_{n \in \N}\) converges in \(W^{k, p}(\R^N)\) to a continuous function \(u\), then for every \(j \in \N\), the subsequence  \((\varphi_{n_i})_{i \in \N}\) converges uniformly to \(u\) in \(\bigcup\limits_{i= j}^\infty \{\abs{\varphi_{n_{i+1}} - \varphi_{n_i}} \ge \alpha_i\}\). Hence, the function \(u\) is the pointwise limit of \((\varphi_{n_i})_{i \in \N}\) in \(\R^N \setminus E\).
\end{remark}

\section{Quasicontinuous representative}

We summarize in this section some results related to quasicontinuous representatives.

We recall the definition of a quasicontinuous function:

\begin{definition}
A function \(v : \Omega \to \R\) is quasicontinuous with respect to the \(W^{k, p}\) capacity if for every \(\epsilon > 0\) there exists a Borel set \(E \subset \Omega\) such that \(\capt_{W^{k, p}}{(E)} \le \epsilon\) and \(v\) is continuous in \(\Omega\setminus E\).
\end{definition}

The existence of quasicontinuous representatives follows from the pointwise convergence we proved in the previous section:

\begin{proposition}
\label{propositionExistenceQuasicontinuousRepresentative}
Let \(k \in \N_*\) and \(1 < p < +\infty\).
For every \(u \in W^{k, p}(\Omega)\), there exists a quasicontinuous function \(\Quasicontinuous u : \Omega \to \R\) with respect to the \(W^{k, p}\) capacity such that \(u = \Quasicontinuous u\) almost everywhere in \(\Omega\).
\end{proposition}

We may state the following capacitary estimate for the level sets of \(W^{k, p}\) functions:

\begin{corollary}
\label{corollaryCapacitaryEstimateWkpBis}
Let \(k \in \N_*\).
If \(1 < p < \infty\), then for every \(u \in W^{k, p}(\R^N)\) and for every \(\alpha > 0\),
\[
\capt_{W^{k, p}}{(\{\abs{\Quasicontinuous u} > \alpha\})} \le \frac{C}{\alpha^p} \norm{u}_{W^{k, p}(\R^N)}^p,
\]
for some constant \(C > 0\) depending on \(N\), \(k\) and \(p\).
\end{corollary}

Although solutions of the linear Dirichlet problem need not belong to \(W_0^{1, 2}(\Omega)\), they have a \(W^{1, 2}\) quasicontinuous representative \cite{BrePon:03}*{lemma~1}:

\begin{proposition}
\label{propositionExistenceQuasicontinuousRepresentativeLaplacian}
Let \(\mu \in \cM(\Omega)\). If \(u\) is the solution of the Dirichlet problem
\[
\left\{
\begin{alignedat}{2}
-\Delta u &= \mu && \quad \text{in \(\Omega\),}\\
u &= 0 && \quad  \text{on \(\partial\Omega\),}
\end{alignedat}
\right.
\]
then there exists a quasicontinuous function \(\Quasicontinuous{u} : \Omega \to \R\) with respect to the \(W^{1, 2}\) capacity such that \(\Quasicontinuous{u} = u\) almost everywhere in \(\Omega\). 
\end{proposition}

By the interpolation inequality (lemma~\ref{lemmaInterpolationLinfty}), for every \(\kappa > 0\) we have \(T_\kappa(u) \in W_0^{1, 2}(\Omega)\).
The quasicontinuous representative \(\Quasicontinuous{u}\) is defined so that \(T_\kappa(\Quasicontinuous{u})\) is a quasicontinuous representative of \(T_\kappa(u)\) with respect to the \(W^{1, 2}\) capacity.


\chapter{Hausdorff measure}

Let \(0 \le s < +\infty\) and \(0 < \delta \le +\infty\).
Given a set \(A \subset \R^N\), define the Hausdorff outer measures of dimension \(s\),
\[
\cH_\delta^s(A) = \inf{\bigg\{ \sum_{n = 0}^\infty \omega_s r_n^s : \textstyle A \subset \bigcup\limits_{n = 0}^\infty B(x_n; r_n) \ \text{and}\ 0 \le r_n \le \delta \bigg\}},
\]
where 
\[
\omega_s = \frac{\pi^{\frac{N}{2}}}{\Gamma(\frac{N}{2} + 1)}
\]
and \(\Gamma\) denotes the Gamma function.
When \(s\) is an integer, \(\omega_s\) is the volume of the unit ball in \(\R^s\).
The outer measure \(\cH_\infty^s(A)\) is usually called the Hausdorff content of \(A\) of dimension~\(s\). 

The Hausdorff measure of dimension \(s\) of \(A\) is then defined as the limit
\[
\cH^s(A) = \lim_{\delta \to 0}{\cH^s_\delta(A)}.
\]

We have adopted Hausdorff's original definition~\cite{Hau:1918}*{definition~1} of Hausdorff measures.
Nowadays, this definition is usually referred to as the spherical Hausdorff measure since we are covering the set \(A\) using balls instead of arbitrary bounded sets.
The equivalence between the \(N\) dimensional Hausdorff measure and the Lebesgue measure in \(\R^N\) is more transparent using Hausdorff's definition since it does not require the isodiametric inequality~\cite{EvaGar:92}*{section~2.2}%
\footnote{An alternative proof based on the fact that \(\cH^N\) is invariant under translations can be found in \cite{Fol:99}*{proposition~11.20}.}.

\section{Density estimate}

The Hausdorff outer measures \(\cH_\delta^s\) have some properties that we would naively expect from the Hausdorff measure \(\cH^s\) but the Hausdorff measure itself does not have.
For instance, they are finite on bounded subsets of \(\R^N\) and they measure balls in \(\R^N\) as if they were balls of dimension \(s\):

\begin{proposition}
\label{propositionHausdorffMeasureBalls}
Let \(0 \le s < +\infty\) and \(0 < \delta \le +\infty\).
For every \(x \in \R^N\) and for every \(0 \le r \le \delta\),
\begin{enumerate}[\((i)\)]
\item if \(0 \le s \le N\), then
\(\cH_\delta^s(B(x; r)) = \omega_s r^s\),
\item if \(s > N\), then \(\cH_\delta^s(B(x; r)) = 0\).
\end{enumerate}
\end{proposition}

\begin{proof}
We begin with the case \(0 \le s \le N\).
The inequality
\[
\cH_\delta^s(B(x; r))
\le \omega_s r^s
\]
is clear for every \(0 \le r \le \delta\). 
The reverse inequality follows from the fact that for every \(r > 0\), if \((B(x_n; r_n))_{n \in \N}\) is a sequence of balls covering \(B(x; r)\), then by subadditivity of the Lebesgue measure
\[
r^N \le \sum_{n=0}^\infty r_n^N
\]
and if in addition we have \(0 \le \frac{s}{N}  \le 1\), then
\[
r^s = (r^N)^{s/N} \le \bigg(\sum_{n=0}^\infty r_n^N\bigg)^{s/N} \le \sum_{n=0}^\infty r_n^s.
\]
Minimizing over all possible coverings of \(B(x; r)\), the reverse inequality follows.

\medskip
We now consider the case \(s > N\).
Given \(0 < \underline{\delta} \le \delta\), if \((B(x_n; r_n))_{n \in \N}\) is a  sequence of balls covering \(B(x; r)\) such that for every \(n \in \N\), \(0 \le r_n \le \underline{\delta}\), then
\[
\cH_\delta^s(B(x; r)) 
\le \sum_{n = 0}^\infty \omega_s r_n^s 
= \sum_{n = 0}^\infty \omega_s r_n^{s-N} r_n^N 
\le \omega_s \underline{\delta}^{s-N}  \sum_{n = 0}^\infty r_n^N.
\]
We may choose the sequence \((B(x_n; r_n))_{n \in \N}\) so that
\[
\sum_{n = 0}^\infty r_n^N \le r^N + 1.
\]
Thus,
\[
\cH_\delta^s(B(x; r)) \le \omega_s \underline{\delta}^{s-N} (r^N + 1)
\]
Letting \(\underline{\delta}\) tend to \(0\), we have the conclusion.
\end{proof}

We can use the Hausdorff outer measures \(\cH_\delta^s\) to give density estimates of the type
\[
\nu(B(x; r)) \le C r^s
\]
for a nonnegative measure \(\nu \in \cM(\R^N)\).
Indeed, if the inequality
\[
\nu \le \alpha \cH_\delta^s
\]
holds, then for every \(x \in \R^N\) and \(0 \le r \le \delta\), we may apply this estimate on the ball \(B(x; r)\) to get
\[
\nu(B(x; r)) \le \alpha\omega_s r^s.
\]

If we are only interested in this density estimate, the inequality 
\(\nu \le \alpha \cH_\delta^s\) seems too strong since it gives an information on all Borel subsets of \(\R^N\).
As it was observed by Frostman~\cite{Fro:35}, these conditions are actually equivalent:

\begin{proposition}
\label{propositionHausdorffContentMorrey}
Let \(0 < \alpha < + \infty\) and \(0 \le s < + \infty\), and let \(\nu \in \cM(\R^N)\) be a nonnegative measure. 
Then, \(\nu \le \alpha \cH_\delta^s\) if and only if for every \(x \in \R^N\) and for every \(0 \le r \le \delta\), 
\[
\nu(B(x; r)) \le \alpha\omega_s r^s.
\]
\end{proposition}

This property appears in the proof of \cite{Fro:35}*{\S 46, theorem~2}.

\begin{proof}
It suffices to establish the result for \(\alpha = 1\).

For the direct implication, if \(\nu \le \cH_\delta^s\), then for every \(x \in \R^N\) and for every \(r \ge 0\), 
\[
\nu(B(x; r)) \le \cH_\delta^s(B(x; r)).
\]
If in addition \(r \le \delta\), then 
\[
\cH_\delta^s(B(x; r)) \le \omega_s r^s
\]
and the conclusion follows.

Conversely, given a Borel set \(A \subset \R^N\), consider a sequence  of balls \((B(x_n; r_n))_{n \in \N}\) covering \(A\). 
Since the measure \(\nu\) is subadditive,
\[
\nu(A) \le \sum_{n=0}^\infty \nu(B(x_n; r_n)).
\]
If for every \(n \in \N\), \(r_n \le \delta\), then by the assumption on the measure \(\nu\) on each ball \(B(x_n; r_n)\),
\[
\nu(A) \le \sum_{n=0}^\infty \omega_s r_n^s.
\]
Taking the infimum over all sequences of balls \((B(x_n; r_n))_{n \in \N}\) covering \(A\), the conclusion follows.
\end{proof}

Density conditions of the type
\[
\nu(B(x; r)) \le C r^s
\]
are common in the literature:
\begin{enumerate}[\((a)\)]
\item a nonnnegative measure \(\nu \in \cM(\R^N)\) belongs to \((BV(\R^N))'\) if and only if there exists \(C > 0\) such that
\[
\nu(B(x; r)) \le C r^{N-1};
\]
see \cite{MeyZie:77}*{theorem~4.7},
\item if a nonnegative measure \(\nu \in \cM(\R^N)\) satisfies
\[
\nu(B(x; r)) \le C r^{N - 2 + \alpha}
\]
for some \(0 < \alpha < 1\), then the Newtonian potential generated by \(\nu\) is Hölder continuous with exponent \(\alpha\); see \citelist{\cite{Car:67}*{chapter~VII} \cite{DupPonPor:06}*{lemma~3}}.
\end{enumerate}

\section{Uniform approximation}

The Hausdorff measure \(\cH^s\) is not \(\sigma\)-finite when \(0 \le s < N\), but it naturally induces finite Borel measures:

\begin{proposition}
Let \(0 \le s < + \infty\) and let \(A \subset \R^N\) be a Borel set.
If \(\cH^s(A) < +\infty\), then \(A\) is \(\cH^s\) measurable and \(\cH^s\lfloor_A \in \cM(\R^N)\).
\end{proposition}

This proposition is a consequence of Carathéodory's theorem on measurable sets \cite{Fol:99}*{proposition~11.16} since the Hausdorff measure \(\cH^s\) satisfies the property of metric additivity \citelist{ \cite{Rog:70}*{theorem~16} \cite{Fol:99}*{proof of proposition~11.17}}:

\begin{lemma}
\label{lemmaHausdorffOuterMeasureMetricAdditivity}
Let \(0 \le s < + \infty\) and \(0 \le \delta < + \infty\).
For every disjoint sets \(A_1, A_2 \subset \R^N\), if \(d(A_1, A_2) \ge \delta\), then
\[
\cH^s_{\delta}(A_1 \cup A_2) = \cH^s_{\delta}(A_1) + \cH^s_{\delta}(A_2).
\]
\end{lemma}

\begin{corollary}
\label{corollaryHausdorffMeasureMetricAdditivity}
Let \(0 \le s < + \infty\).
For every disjoint sets \(A_1, A_2 \subset \R^N\), if \(d(A_1, A_2) > 0\), then
\[
\cH^s(A_1 \cup A_2) = \cH^s(A_1) + \cH^s(A_2).
\]
\end{corollary}

The proof of the lemma relies on the fact that an open ball of radius at most \(d(A_1, A_2)\) cannot intersect simultaneously \(A_1\) and \(A_2\).

\medskip

The Hausdorff outer measures \(\cH^s_\delta\) converge uniformly to the Hausdorff measure \(\cH^s\) on sets of finite Hausdorff measure:

\begin{proposition}
\label{propositionUniformConvergenceHausdorff}
Let \(0 \le s < + \infty\) and let \(A \subset \R^N\) be a Borel set.
If \(\cH^s(A) < +\infty\), then for every \(\epsilon > 0\) there exists \(\delta > 0\) such that for every \(B \subset A\),
\[
0 \le \cH^s(B) - \cH^s_\delta(B) \le \epsilon.
\]
\end{proposition}

This property appears in the book of Falconer~\cite{Fal:86}*{lemma~1.7}.
It is surprising that this uniform approximation of the Hausdorff measure is not mentioned more often in elementary books on Hausdorff measure.

We follow the strategy of the proof from \cite{Fal:86}:

\begin{proof}
Since \(A\) is measurable and \(\cH^s(A) < +\infty\), we have for every \(B \subset A\),
\[
\cH^s(B) = \cH^s(A) - \cH^s(A\setminus B).
\]
Thus, for every \(\delta > 0\), 
\[
\cH^s(B) \le \cH^s(A) - \cH_\delta^s(A\setminus B).
\]
Given \(\epsilon > 0\), we choose \(\delta > 0\) such that
\[
\cH^s(A) \le \cH_\delta^s(A) + \epsilon.
\]
Hence,
\[
\cH^s(B) \le \cH_\delta^s(A) - \cH_\delta^s(A\setminus B) + \epsilon.
\]

Let \((B(x_n; r_n))_{n \in \N}\) be any sequence of balls covering \(B\) with \(0 \le r_n \le \delta\) and let \((B(y_n; t_n))_{n \in \N}\) be any sequence of balls covering \(A \setminus B\) such that \(0 \le t_n \le \delta\).
Combining both sequences, we obtain a covering of \(A\). 
Thus,
\[
\cH_\delta^s(A)
\le \sum_{n= 0}^\infty \omega_s r_n^s + \sum_{n = 0}^\infty \omega_s t_n^s
\]
and this implies
\[
\cH^s(B) \le \sum_{n= 0}^\infty \omega_s r_n^s + \sum_{n = 0}^\infty \omega_s t_n^s - \cH_\delta^s(A\setminus B) + \epsilon.
\]
Since the sequences \((B(x_n; r_n))_{n \in \N}\) and \((B(y_n; t_n))_{n \in \N}\) are independent of each other, we can take the infimum over all possible choices of sequences to get
\[
\cH^s(B) 
\le \cH_\delta^s(B) + \cH_\delta^s(B\setminus A) - \cH_\delta^s(A\setminus B) + \epsilon 
=  \cH_\delta^s(B) + \epsilon
.
\]
Since \(\cH_\delta^s(B) \le \cH^s(B)\), the conclusion follows.
\end{proof}


\backmatter


\def\cprime{$'$}


\end{document}